\documentclass{amsart}

\title{Logarithmic Riemann--Hilbert correspondences for rigid varieties}

\author{Hansheng Diao, Kai-Wen Lan, Ruochuan Liu, and Xinwen Zhu}
\address{Yau Mathematical Sciences Center, Tsinghua University, Beijing 100084, China}
\email{hdiao@mail.tsinghua.edu.cn}
\address{University of Minnesota, 127 Vincent Hall, 206 Church Street SE, Minneapolis, MN 55455, USA}
\email{kwlan@math.umn.edu}
\address{Beijing International Center for Mathematical Research, Peking University, 5 Yi He Yuan Road, Beijing 100871, China}
\email{liuruochuan@math.pku.edu.cn}
\address{California Institute of Technology, 1200 East California Boulevard, Pasadena, CA 91125, USA}
\email{xzhu@caltech.edu}
\thanks{K.-W.\@\xspace Lan was partially supported by the National Science Foundation under agreement No.\@\xspace DMS-1352216, by an Alfred P.\@\xspace Sloan Research Fellowship, and by a Simons Fellowship in Mathematics.  R.\@\xspace Liu was partially supported by the National Natural Science Foundation of China under agreement Nos.\@\xspace NSFC-11571017 and NSFC-11725101, and by the Tencent Foundation.  X.\@\xspace Zhu was partially supported by the National Science Foundation under agreement Nos.\@\xspace DMS-1602092 and DMS-1902239, by an Alfred P.\@\xspace Sloan Research Fellowship, and by a Simons Fellowship in Mathematics.  Any opinions, findings, and conclusions or recommendations expressed in this writing are those of the authors, and do not necessarily reflect the views of the funding organizations.}
\subjclass[2010]{Primary 14F40, 14G22; Secondary 14D07, 14F30, 14G35}

\usepackage{amsfonts, latexsym, amssymb, mathrsfs}

\usepackage[driverfallback=hypertex, unicode]{hyperref}
\usepackage{color}

\usepackage{upref}
\usepackage[shortlabels]{enumitem}
\usepackage{verbatim}
\usepackage{xspace}

\usepackage[cmtip, all]{xy}
\begingroup\expandafter\expandafter\expandafter\endgroup
\expandafter\ifx\csname IncludeInRelease\endcsname\relax
\usepackage{fixltx2e}
\fi

\newtheorem{thm}[equation]{{Theorem}}
\newtheorem{cor}[equation]{{Corollary}}
\newtheorem{lemma}[equation]{{Lemma}}
\newtheorem{prop}[equation]{{Proposition}}
\newtheorem{conj}[equation]{{Conjecture}}
\newtheorem{exam}[equation]{{Example}}

\theoremstyle{definition}
\newtheorem{defn}[equation]{{Definition}}

\theoremstyle{remark}
\newtheorem{rk}[equation]{{Remark}}

\newcommand{\quash}[1]{}

\newcommand{\Ainf}{A_{\inf}}

\newcommand{\AAinf}{\bA_{\inf}}
\newcommand{\AAinfX}[1]{\bA_{\inf, {#1}}}
\newcommand{\BBinf}{\bB_{\inf}}
\newcommand{\BBinfX}[1]{\bB_{\inf, {#1}}}

\newcommand{\BdR}{B_\dR}
\newcommand{\BdRp}{B_\dR^+}
\newcommand{\BBdRp}{\bB_\dR^+}
\newcommand{\BBdRpX}[1]{\bB_{\dR, {#1}}^+}
\newcommand{\BBdR}{\bB_\dR}
\newcommand{\BBdRX}[1]{\bB_{\dR, {#1}}}

\newcommand{\OBdRpX}[1]{\cO\bB_{\dR, {#1}}^+}
\newcommand{\OBdlp}{\cO\bB_{\dR, \log}^+}
\newcommand{\OBdlpX}[1]{\cO\bB_{\dR, \log, {#1}}^+}
\newcommand{\OBdR}{\cO\bB_\dR}

\newcommand{\OBdl}{\cO\bB_{\dR, \log}}
\newcommand{\OBdlX}[1]{\cO\bB_{\dR, \log, {#1}}}
\newcommand{\OCl}{\cO\bC_{\log}}
\newcommand{\OClX}[1]{\cO\bC_{\log, {#1}}}

\newcommand{\DdR}{D_\dR}
\newcommand{\DdRalg}{D_\dR^\alg}
\newcommand{\Ddl}{D_{\dR, \log}}
\newcommand{\Ddlalg}{D_{\dR, \log}^\alg}

\newcommand{\ho}{\widehat{\otimes}}

\newcommand{\GrSh}[1]{\underline{#1}}

\newcommand{\bA}{\mathbb{A}}
\newcommand{\bB}{\mathbb{B}}
\newcommand{\bC}{\mathbb{C}}
\newcommand{\bD}{\mathbb{D}}
\newcommand{\bE}{\mathbb{E}}

\newcommand{\bG}{\mathbb{G}}

\newcommand{\bJ}{\mathbb{J}}
\newcommand{\bK}{\mathbb{K}}
\newcommand{\bL}{\mathbb{L}}
\newcommand{\bM}{\mathbb{M}}

\newcommand{\bQ}{\mathbb{Q}}
\newcommand{\bR}{\mathbb{R}}

\newcommand{\bT}{\mathbb{T}}

\newcommand{\bZ}{\mathbb{Z}}

\newcommand{\cB}{\mathcal{B}}

\newcommand{\cE}{\mathcal{E}}
\newcommand{\cF}{\mathcal{F}}

\newcommand{\cH}{\mathcal{H}}
\newcommand{\cI}{\mathcal{I}}

\newcommand{\cL}{\mathcal{L}}
\newcommand{\cM}{\mathcal{M}}

\newcommand{\cO}{\mathcal{O}}

\newcommand{\cX}{\mathcal{X}}
\newcommand{\cY}{\mathcal{Y}}
\newcommand{\cZ}{\mathcal{Z}}

\newcommand{\OP}[1]{\operatorname{#1}}

\newcommand{\Aut}{\OP{Aut}}
\newcommand{\Hom}{\OP{Hom}}
\newcommand{\Spec}{\OP{Spec}}
\newcommand{\Spa}{\OP{Spa}}

\newcommand{\Em}{\hookrightarrow}                   % embedding
\newcommand{\Surj}{\twoheadrightarrow}              % surjection
\newcommand{\Mi}{\stackrel{\sim}{\to}}              % mapping isomorphically
\newcommand{\Mapn}[1]{\stackrel{#1}{\to}}           % mapping with name
\newcommand{\Emn}[1]{\stackrel{#1}{\Em}}            % embedding with name
\newcommand{\Surn}[1]{\stackrel{#1}{\Surj}}         % surjection with name
\newcommand{\Misn}[1]{\stackrel{#1}{\Mi}}           % mapping isomorphically with name

\newcommand{\can}{\Utext{can.}}                     % canonical morphism

\newcommand{\bAi}{{\bA_f}}                          % finite adeles

\newcommand{\Grp}[1]{\mathrm{#1}}                   % algebraic group we consider
\newcommand{\ad}{\Utext{ad}}                        % adjoint group
\newcommand{\der}{\Utext{der}}                      % derived group
\newcommand{\nc}{\Utext{nc}}                        % noncompact factor

\newcommand{\Grpmu}{\boldsymbol{\mu}}               % group scheme symbol for \mu
\newcommand{\Gm}[1]{\mathbf{G}_{\Utext{m}, {#1}}}   % G_{m, #1}

\newcommand{\GL}{\mathrm{GL}}                       % general linear group
\newcommand{\SU}{\mathrm{SU}}                       % special unitary group

\newcommand{\pr}{\OP{pr}}                           % projection
\newcommand{\rank}{\OP{rk}}                         % rank

\newcommand{\Shdom}{\mathsf{X}}                     % Hermitian symmetric domain
\newcommand{\Lquot}{\backslash}                     % left quotient
\newcommand{\hd}{h}                                 % h
\newcommand{\hc}{\mu}                               % mu
\newcommand{\DelS}{\mathbf{S}}                      % Deligne torus
\newcommand{\Fl}{{\mathcal{F}\ell}}                 % flag variety

\newcommand{\Art}{\OP{Art}}                         % Artin map
\newcommand{\Gal}{\OP{Gal}}                         % Galois groups
\newcommand{\Id}{\OP{Id}}                           % identity
\newcommand{\Ind}{\OP{Ind}}                         % induction
\newcommand{\Res}{\OP{Res}}                         % restriction of scalar or residue
\newcommand{\ab}{\Utext{ab}}                        % maximal abelian extension
\newcommand{\alg}{\Utext{alg}}                      % algebraization
\newcommand{\free}{\Utext{free}}                    % free
\newcommand{\geom}{\Utext{geom}}                    % geometric
\newcommand{\gp}{\Utext{gp}}                        % group
\newcommand{\inv}{\Utext{inv}}                      % invertible elements
\newcommand{\red}{\Utext{red}}                      % reduced
\newcommand{\ur}{\Utext{ur}}                        % maximal unramified extension

\newcommand{\Talg}[1]{\langle{#1}\rangle}           % notation for Tate algebra etc

\newcommand{\mono}[1]{e^{#1}}                       % notation for monoid algebra
\newcommand{\monob}[1]{\mathbf{e}^{#1}}             % notation for monoid algebra with boldface
\newcommand{\monon}[1]{y_{#1}}                      % notation for infinitesimal monoid algebra

\newcommand{\ReFl}{E}                               % reflex field
\newcommand{\Sh}{\mathrm{Sh}}                       % Shimura variety
\newcommand{\Model}{\Sh}                            % model
\newcommand{\levcp}{K}                              % level defined by open compact

\newcommand{\AV}{\mathrm{A}}                        % abelian scheme
\newcommand{\AVstr}{\mathrm{f}}                     % structure morphism of abelian scheme
\newcommand{\Pol}{\lambda}                          % polarization

\newcommand{\an}{\Utext{an}}                        % complex analytification
\newcommand{\tops}{\Utext{top}}                     % underlying topological space

\newcommand{\dualsign}{{\vee}}                      % notation for dual
\newcommand{\dual}[1]{{#1}^\dualsign}               % dual
\newcommand{\Ex}{\wedge}                            % exterior power

\newcommand{\Tor}{\Utext{tor}}                      % toroidal compactification
\newcommand{\Torcpt}[1]{{#1}^\Tor}                  % toroidal compactification
\newcommand{\NCD}{D}                                % normal crossings divisor on boundary

\newcommand{\M}{\mathrm{M}}                         % matrices
\newcommand{\Proj}{\mathrm{Proj}}                   % projective modules
\newcommand{\Rep}{\mathrm{Rep}}                     % representations
\newcommand{\rep}{V}                                % representation of G

\newcommand{\B}{\Utext{B}}                          % Betti
\newcommand{\dR}{\Utext{dR}}                        % de Rham
\newcommand{\Hdg}{\Utext{Hodge}}                    % Hodge
\newcommand{\Hi}{\Utext{Higgs}}                     % Higgs

\newcommand{\Fil}{\Utext{Fil}}                      % filtration
\newcommand{\gr}{\OP{gr}}

\newcommand{\et}{\Utext{\'et}}                      % etale
\newcommand{\ket}{\Utext{k\'et}}                    % Kummer etale
\newcommand{\proket}{\Utext{prok\'et}}              % pro-Kummer etale

\newcommand{\BSh}[1]{{}_{\B}\GrSh{#1}}              % Betti local system
\newcommand{\dRSh}[1]{{}_{\dR}\GrSh{#1}}            % de Rham local system
\newcommand{\etSh}[1]{{}_{\et}\GrSh{#1}}            % etale local system
\newcommand{\pdRSh}[1]{{}_{p\Utext{-}\dR}\GrSh{#1}} % p-adically constructed de Rham local system
\newcommand{\pBSh}[1]{{}_{p\Utext{-}\B}\GrSh{#1}}   % p-adically constructed Betti local system
\newcommand{\canext}{{\Utext{can}}}                 % canonical extension

\newcommand{\Frep}{\rho}                            % fundamental group representation
\newcommand{\Hrep}{\pi}                             % Hecke action

\newcommand{\AC}[1]{\overline{#1}}                  % algebraic closure or geometric point
\newcommand{\ACMap}{\iota}                          % map between algebraic closures

\newcommand{\BFp}{k}                                % p-adic base field
\newcommand{\Coef}{L}                               % coefficient field
\newcommand{\CoefMap}{\tau}                         % map of coefficient field

\newcommand{\RH}{\mathcal{RH}}                      % Riemann--Hilbert correspondence
\newcommand{\RHl}{\mathcal{RH}_{\log}}              % log Riemann--Hilbert correspondence
\newcommand{\RHlp}{\mathcal{RH}_{\log}^+}           % log Riemann--Hilbert correspondence (integral)
\newcommand{\DR}{\mathit{DR}}                       % de Rham complex
\newcommand{\DRl}{\mathit{DR}_{\log}}               % log de Rham complex
\newcommand{\Hl}{\mathcal{H}_{\log}}                % log Simpson correspondence
\newcommand{\Hil}{\mathit{Higgs}_{\log}}            % log Higgs complex

\newcommand{\unip}{\Utext{\textit{u}}}              % unipotent
\newcommand{\qunip}{\Utext{\textit{qu}}}            % quasi-unipotent

\newcommand{\Utext}[1]{\text{\rm #1}}               % upright math text
\newcommand{\Refenum}[1]{\Pth{\textrm{#1}}}
\newcommand{\Refeq}[1]{\Pth{#1}}

\newcommand{\Pth}[1]{{\rm (}#1{\rm )}}              % parenthesis ( )
\newcommand{\Qtn}[1]{``#1''}                        % quotation `` ''

\newcommand{\parenthesis}[1]{\Pth{#1}}              % parenthesis (old)

\newcommand{\aCh}{Ch.\@\xspace}                     % chapter
\newcommand{\aChs}{Ch.\@\xspace}                    % chapters
\newcommand{\aSec}{Sec.\@\xspace}                   % section
\newcommand{\aSecs}{Sec.\@\xspace}                  % sections

\newcommand{\aDef}{Def.\@\xspace}                   % definition
\newcommand{\aDefs}{Def.\@\xspace}                  % definitions
\newcommand{\aLem}{Lem.\@\xspace}                   % lemma
\newcommand{\aLems}{Lem.\@\xspace}                  % lemmas
\newcommand{\aProp}{Prop.\@\xspace}                 % proposition
\newcommand{\aProps}{Prop.\@\xspace}                % propositions
\newcommand{\aThm}{Thm.\@\xspace}                   % theorem
\newcommand{\aThms}{Thm.\@\xspace}                  % theorems
\newcommand{\aCor}{Cor.\@\xspace}                   % corollary
\newcommand{\aCors}{Cor.\@\xspace}                  % corollaries
\newcommand{\aRem}{Rem.\@\xspace}                   % remark
\newcommand{\aEx}{Ex.\@\xspace}                     % example
\newcommand{\aExs}{Ex.\@\xspace}                    % examples
\newcommand{\resp}{resp.\@\xspace}                  % resp.
\newcommand{\ie}{i.e.\@\xspace}                     % i.e.
\newcommand{\eg}{e.g.\@\xspace}                     % e.g.
\newcommand{\etc}{etc\xspace}                       % etc.

\newcommand{\Refcf}{cf.\@\xspace}                   % 'compare' for the purpose of reference

\newcommand{\logadicdefmonoid}{2.1.1} % def-monoid
\newcommand{\logadicremmonoidcat}{2.1.2} % rem-monoid-cat
\newcommand{\logadicdeflogstr}{2.2.2} % def-log-str
\newcommand{\logadiclemstrnoe}{2.2.13} % lem-str-noe
\newcommand{\logadiclemmonoidalgperf}{2.2.15} % lem-monoid-alg-perf
\newcommand{\logadicdefPlog}{2.2.17} % def-P-log
\newcommand{\logadicdefimm}{2.2.23} % def-imm
\newcommand{\logadicdefchart}{2.3.1} % def-chart
\newcommand{\logadicremdefchart}{2.3.2} % rem-def-chart
\newcommand{\logadicdeflogadicspfs}{2.3.5} % def-log-adic-sp-fs
\newcommand{\logadicexlogadicspncd}{2.3.17} % ex-log-adic-sp-ncd
\newcommand{\logadicexlogadicspncdstrictclimm}{2.3.18} % ex-log-adic-sp-ncd-strict-cl-imm
\newcommand{\logadicpropchartmorexist}{2.3.21} % prop-chart-mor-exist
\newcommand{\logadicpropchartmorexistfine}{2.3.22} % prop-chart-mor-exist-fine
\newcommand{\logadicpropfiberprodlogadic}{2.3.27} % prop-fiber-prod-log-adic
\newcommand{\logadicdeflogsm}{3.1.1} % def-log-sm
\newcommand{\logadicproplogsmchart}{3.1.4} % prop-log-sm-chart
\newcommand{\logadicdeflogsmbasefield}{3.1.9} % def-log-sm-base-field
\newcommand{\logadicproptoricchart}{3.1.10} % prop-toric-chart
\newcommand{\logadiccorsmtoricchart}{3.1.11} % cor-sm-toric-chart
\newcommand{\logadicdeftoricchart}{3.1.12} % def-toric-chart
\newcommand{\logadicexlogadicspncdchart}{3.1.13} % ex-log-adic-sp-ncd-chart
\newcommand{\logadicproplogdiffmonoid}{3.2.25} % prop-log-diff-monoid
\newcommand{\logadiccorlogdiffmonoidstrictclimm}{3.2.29} % cor-log-diff-monoid-strict-cl-imm
\newcommand{\logadicdeflogdiffsheaf}{3.3.6} % def-log-diff-sheaf
\newcommand{\logadicthmlogdiffsheaffund}{3.3.17} % thm-log-diff-sheaf-fund
\newcommand{\logadiccorlogdiffsheafstrictclimm}{3.3.18} % cor-log-diff-sheaf-strict-cl-imm
\newcommand{\logadicdeflogdiffsheafex}{3.3.19} % def-log-diff-sheaf-ex
\newcommand{\logadicsecketsite}{4.1} % sec-ket-site
\newcommand{\logadiclemAbhyankarbasic}{4.2.5} % lem-Abhyankar-basic
\newcommand{\logadicdefketlocconst}{4.4.14} % def-ket-loc-const
\newcommand{\logadiclemclimmketmor}{4.5.4} % lem-cl-imm-ket-mor
\newcommand{\logadicthmpurity}{4.6.1} % thm-purity
\newcommand{\logadicdefproketsite}{5.1.2} % def-proket-site
\newcommand{\logadicpropproketvsketadj}{5.1.7} % prop-proket-vs-ket-adj
\newcommand{\propdirimketvsproket}{5.2.1} % prop-dir-im-ket-vs-proket
\newcommand{\logadicseclogaffperf}{5.3} % sec-log-aff-perf
\newcommand{\logadicdeflogaffperf}{5.3.1} % def-log-aff-perf
\newcommand{\logadicremunifcompllim}{5.3.2} % rem-unif-compl-lim
\newcommand{\logadicremdeflogaffperfsp}{5.3.5} % rem-def-log-aff-perf-sp
\newcommand{\logadicproplogaffperfbasis}{5.3.12} % prop-log-aff-perf-basis
\newcommand{\logadicsecproketsheaves}{5.4} % sec-proket-sheaves
\newcommand{\logadicdefproketsheaves}{5.4.1} % def-proket-sheaves
\newcommand{\logadicthmalmostvanhat}{5.4.3} % thm-almost-van-hat
\newcommand{\logadicfinfreeOhatmod}{5.4.4} % thm-loc-syst-vs-proj
\newcommand{\logadicpropsurjOhatclimm}{5.4.5} % prop-surj-O-hat-cl-imm
\newcommand{\logadicsectoricchart}{6.1} % sec-toric-chart
\newcommand{\logadiceqgeomtowerGal}{6.1.4} % eq-geom-tower-Gal
\newcommand{\logadiclemGalcohafg}{6.1.7} % lem-Gal-coh-afg
\newcommand{\logadiclemScholzeref}{6.1.9} % lem-Scholze-ref
\newcommand{\logadicthmprimcomp}{6.2.1} % thm-prim-comp
\newcommand{\logadicremcohfin}{6.2.2} % rem-coh-fin
\newcommand{\logadicdefketlisse}{6.3.1} % def-ket-lisse
\newcommand{\logadicdefproketlocsyst}{6.3.2} % def-proket-loc-syst
\newcommand{\logadiclemproketlisse}{6.3.3} % lem-proket-lisse
\newcommand{\logadiccorpuritylisse}{6.3.4} % cor-purity-lisse
\newcommand{\logadiclissepr}{6.3.5} % cor-lisse-pr
\newcommand{\logadiclemQplocclimm}{6.3.6} % lem-Q-p-loc-cl-imm
\newcommand{\logadicdefunipqunipmonod}{6.3.7} % def-unip-qunip-monod
\newcommand{\logadicrkunipqunipmonodalgcl}{6.3.13} % rk-unip-qunip-monod-alg-cl
\newcommand{\logadicdefnearby}{6.4.1} % def-nearby
\newcommand{\logadiclemnearbycohstab}{6.4.2} % lem-nearby-coh-stab
\newcommand{\logadiclemearby}{6.4.3} % lem-nearby-ket-cov
\newcommand{\logadicpropnearby}{6.4.4} % prop-nearby

\begin{document}

\begin{abstract}
    On any smooth algebraic variety over a $p$-adic local field, we construct a tensor functor from the category of de Rham $p$-adic \'etale local systems to the category of filtered algebraic vector bundles with integrable connections satisfying the Griffiths transversality, which we view as a $p$-adic analogue of Deligne's classical Riemann--Hilbert correspondence.  A crucial step is to construct canonical extensions of the desired connections to suitable compactifications of the algebraic variety with logarithmic poles along the boundary, in a precise sense characterized by the eigenvalues of residues; hence the title of the paper.  As an application, we show that this $p$-adic Riemann--Hilbert functor is compatible with the classical one over all Shimura varieties, for local systems attached to representations of the associated reductive algebraic groups.
\end{abstract}

\maketitle

\tableofcontents

\numberwithin{equation}{section}

\section{Introduction}\label{sec-intro}

Let $X$ be a connected smooth complex algebraic variety, $X^\an$ the associated analytic space and $X^\tops$ the underlying topological space.  The classical Riemann--Hilbert correspondence establishes \Pth{tensor} equivalences among the following:
\begin{itemize}
    \item the category of finite-dimensional complex representations of $\pi_1(X^\tops, x)$ \Pth{where $x$ is a chosen based point}, which by a well-known topological construction, is equivalent to the category of local systems \Pth{\ie, locally constant sheaves} of finite-dimensional $\bC$-vector spaces on $X^\tops$;

    \item the category of vector bundles with integrable connections on $X^\an$; and

    \item the category of vector bundles with integrable connections on $X$, with regular singularities at infinity \Pth{which we shall simply call \emph{regular integrable connections}, in what follows}.
\end{itemize}
The equivalence of the first and second categories is a simple consequence of the Frobenius theorem: for a local system $\bL$ on $X^\tops$, the associated vector bundle with an integrable connection is $(\bL \otimes_\bC \cO_{X^\an}, 1 \otimes d)$; and conversely, for a vector bundle with an integrable connection on $X^\an$, its sheaf of horizontal sections is a local system on $X^\tops$.  The equivalence of the second and third categories, however, is a deep theorem due to Deligne \cite{Deligne:1970-EDR}.

An analogous Riemann--Hilbert correspondence for varieties over a $p$-adic field is long desired but remains rather mysterious until recently.  The situation is far more complicated.  Let $X$ be a smooth algebraic variety over $\bQ_p$, for example.  In this setting, the second and third categories remain meaningful, and it is natural to replace the first with the category of $p$-adic \'etale local systems on $X$.  However, after this replacement, one cannot expect an equivalence between the first and the second categories, as can already be seen when $X$ is a point.  Moreover, in general, the natural analytification functor from the third to the second category is not an equivalence either.  Nevertheless, one of the main goals of this paper is to prove the following result \Pth{as one step towards the $p$-adic Riemann--Hilbert correspondence}:

\begin{thm}\label{thm-intro-main}
    Let $X$ be a smooth algebraic variety over a $p$-adic field $k$ \Pth{see Notation and Conventions}.  Then there is a tensor functor $\DdRalg$ from the category of de Rham $p$-adic \'etale local systems $\bL$ on $X$ to the category of algebraic vector bundles on $X$ with regular integrable connections and decreasing filtrations satisfying the Griffiths transversality.  In addition, there is a canonical comparison isomorphism
    \begin{equation}\label{eq-thm-intro-main}
        H^i_\et\bigl(X_{\AC{k}}, \bL\bigr) \otimes_{\bQ_p} \BdR \cong H^i_\dR\bigl(X, \DdRalg(\bL)\bigr) \otimes_k \BdR
    \end{equation}
    compatible with the canonical filtrations and the actions of $\Gal(\AC{k} / k)$ on both sides.
\end{thm}

Here $\BdR$ is Fontaine's $p$-adic period ring, and $H_\dR$ is the algebraic de Rham cohomology.  The notion of \emph{de Rham $p$-adic \'etale local systems} was first introduced by Scholze in \cite[\aDef 8.3]{Scholze:2013-phtra} \Pth{generalizing earlier work of Brinon \cite{Brinon:2008-RPR}} using some relative de Rham period sheaf.  However, it turns out that this notion satisfies a rather surprising rigidity property: by \cite[\aThm 3.9]{Liu/Zhu:2017-rrhpl}, a $p$-adic \'etale local system $\bL$ on $X^\an$ is de Rham if and only if, on each connected component of $X$, there exists \emph{some} classical point $x$ such that, for some \Pth{and hence every} geometric point $\AC{x}$ over $x$, the corresponding $p$-adic representation $\bL_{\AC{x}}$ of the absolute Galois group of the residue field of $x$ is de Rham in the classical sense.  In this situation, it follows that the same is also true at \emph{every} classical point $x$ of $X$.  Note that the functor denoted by $\DdR$ in \cite[\aThm 3.9]{Liu/Zhu:2017-rrhpl} is the composition of the functor $\DdRalg$ in Theorem \ref{thm-intro-main} with the analytification functor.  Compared with \cite{Scholze:2013-phtra} and \cite{Liu/Zhu:2017-rrhpl}, the algebraicity of the integrable connection is an important new contribution of this paper.  In particular, this allows us to go further to compare the $p$-adic theory with Deligne's complex theory mentioned above, as we shall see shortly.

Theorem \ref{thm-intro-main} also includes a new \emph{de Rham comparison isomorphism} for smooth algebraic varieties over $k$ with nontrivial coefficients, which implies that $H^i_\et\bigl(X_{\AC{k}}, \bL\bigr)$ is a \emph{de Rham} representation of $\Gal(\AC{k} / k)$.  The de Rham comparison for smooth varieties has a long history, which we shall not attempt to review---see, for example, \cite{Faltings:1989-ccpgr, Faltings:2002-aee, Tsuji:1999-pacst, Kisin:2002-pspec, Niziol:2008-scvkt, Niziol:2009-upapm, Yamashita:2011-phtov, Beilinson:2012-ppddc, Andreatta/Iovita:2013-cisfs, Scholze:2013-phtra, Scholze:2016-phtra-corr, Colmez/Niziol:2017-scpnc, Li/Pan:2019-ldrco}.  All these earlier works either imposed some strong assumption on the coefficient $\bL$ \Pth{and in fact most works assumed that $\bL$ is trivial} or assumed that the variety $X$ is proper.  But we require \emph{neither}.  In this generality, without first constructing the corresponding $\DdRalg(\bL)$ as in Theorem \ref{thm-intro-main}, it was not even clear how to formulate the comparison isomorphism!  Once $\DdRalg(\bL)$ is constructed, we can adapt Scholze's approach in \cite{Scholze:2013-phtra} and obtain the desired comparison for arbitrary nontrivial coefficients on arbitrary smooth varieties.

Besides taking cohomology, the functor $\DdRalg$ is compatible with many other operations of sheaves.  For example, it commutes with taking nearby cycles in the simplest situation where our formulation is available.
\begin{thm}\label{thm-intro-nearby}
    Let $f: X \to \bA^1$ be a smooth morphism and let $D := f^{-1}(0)$.  Let $\bL$ be a de Rham $p$-adic \'etale local systems on $X - D$.  Then there is a canonical isomorphism of vector bundles with integrable connections
    \[
        \DdRalg\bigl(R\Psi_f(\bL)\bigr) \cong R\Psi_f\bigl(\DdRalg(\bL)\bigr)
    \]
    on $D$, compatible with the filtrations on both sides.  \Pth{Here $R\Psi_f$ at the two sides of the isomorphism denotes the nearby cycle functors in the \'etale and $D$-module settings, respectively.}  In particular, $R\Psi_f(\bL)$ is a de Rham local system on $D_\et$.
\end{thm}

See Theorem \ref{thm-nearby-comp-alg} for a slightly more general statement, and see Corollary \ref{cor-nearby-curve} for a concrete interpretation when $X$ is a smooth curve over $k$.  These results suggest a strong relation between our work and classical Hodge theory, such as Schmid's theorem on limit Hodge structures \cite{Schmid:1973-vhssp}.  Another manifestation of this relation is that we prove some cohomology vanishing results for $p$-adic algebraic varieties \Pth{see Theorem \ref{thm-van-gen}}, similar to those that can be obtained from complex Hodge theory.  In particular, we obtain a new proof of the Kodaira--Akizuki--Nakano vanishing theorem \Pth{together with some generalizations} by $p$-adic Hodge-theoretic methods.

We shall call the functor $\DdRalg$ in Theorem \ref{thm-intro-main} the \Pth{algebraic} \emph{$p$-adic Riemann--Hilbert functor}.  It is natural to ask whether this functor is compatible with Deligne's classical Riemann--Hilbert correspondence in a suitable sense.  We shall formulate our expectation in the Conjecture \ref{conj-intro} below.  Let us begin with some preparations.

Let $X$ be a smooth algebraic variety over a number field $E$.  We fix an isomorphism $\ACMap: \AC{\bQ}_p \Mi \bC$ and a homomorphism $\sigma: E \to \bC$, and write $\sigma X = X \otimes_{E, \sigma} \bC$.  There is a tensor functor from the category of $p$-adic \'etale local systems $\bL$ on $X$ to the category of \Pth{vector bundles with} regular integrable connections on $\sigma X$ as follows.  Note that $\bL|_{\sigma X}$ is an \'etale local system on $\sigma X$, corresponding to a $p$-adic representation of the \'etale fundamental group of each connected component of $\sigma X$, which is the profinite completion of the fundamental group of the corresponding connected component of $(\sigma X)^\tops$.  Then $\bL|_{\sigma X} \otimes_{\bQ_p, \ACMap} \bC$ can be regarded as a classical local system on $(\sigma X)^\tops$, denoted by $\ACMap \bL_\sigma$.  By Deligne's Riemann--Hilbert correspondence, we obtain a regular integrable connection on $\sigma X$.  On the other hand, $\ACMap^{-1} \circ \sigma: E \to \AC{\bQ}_p$ determines a $p$-adic place $v$.  Let $E_v$ be the completion of $E$ with respect to $v$, and assume that $\bL|_{X_{E_v}}$ is de Rham.  Then $\DdRalg(\bL|_{X_{E_v}}) \otimes_{E_v, \ACMap} \bC$ is another regular integrable connection, with an additional decreasing filtration $\Fil^\bullet$ satisfying the Griffiths transversality.  We would like to compare the above two constructions.  In order to do so, we need to impose a further restriction on $\bL$.

We say that $\bL$ is \emph{geometric} if, at each geometric point $\AC{x}$ above a closed point $x$ of $X$, the $p$-adic representation $\bL_{\AC{x}}$ of $\Gal\bigl(\AC{k(x)} / k(x)\bigr)$ is \emph{geometric in the sense of Fontaine--Mazur} \Pth{see \cite[Part I, \S 1]{Fontaine/Mazur:1993-ggr}}.  Note that geometric $p$-adic \'etale local systems on $X$ form a full tensor subcategory of the category of all \'etale local systems.  If $\bL$ is geometric, then $\bL|_{X_{E_v}}$ is de Rham \Pth{by \cite[\aThm 3.9]{Liu/Zhu:2017-rrhpl}}.

\begin{conj}\label{conj-intro}
    The above two tensor functors from the category of geometric $p$-adic \'etale local systems on $X$ to the category of regular integrable connections on $\sigma X$ are canonically isomorphic.  In addition, $(\DdRalg(\bL|_{X_{E_v}}) \otimes_{E_v, \ACMap} \bC, \Fil^\bullet)$ is a complex variation of Hodge structures.
\end{conj}

This is closely related to a relative version of the Fontaine--Mazur conjecture proposed in \cite{Liu/Zhu:2017-rrhpl}, but it might be more approachable because it is stated purely in terms of sheaves.  Even so, it seems to be currently out of reach.  Nevertheless, in the case of Shimura varieties, we can partially verify this conjecture.  Let $(\Grp{G}, \Shdom)$ be a Shimura datum, $\levcp \subset \Grp{G}(\bAi)$ a neat open compact subgroup, and $\Model_\levcp = \Sh_\levcp(\Grp{G}, \Shdom)$ the associated Shimura variety, defined over the reflex field $\ReFl = \ReFl(\Grp{G}, \Shdom)$.  Let $\Grp{G}^c$ be the quotient of $\Grp{G}$ by the minimal subtorus $Z_s(\Grp{G})$ of the center $Z(\Grp{G})$ of $\Grp{G}$ such that the torus $Z(\Grp{G})^\circ / Z_s(\Grp{G})$ has the same split ranks over $\bQ$ and $\bR$.  Recall that there is a tensor functor from the category $\Rep_{\bQ_p}(\Grp{G}^c)$ of algebraic representations of $\Grp{G}^c$ over $\bQ_p$ to the category of $p$-adic \'etale local systems on $\Model_\levcp$ \Pth{see, for example, \cite[\aSec 3]{Lan/Stroh:2018-ncaes-2} or \cite[\aSec 4.2]{Liu/Zhu:2017-rrhpl}}, whose essential image consists of only certain \emph{geometric} $p$-adic \'etale local systems \Pth{see \cite[\aThm 1.2]{Liu/Zhu:2017-rrhpl}}.

\begin{thm}\label{thm-intro-Sh}
    The conjecture holds for the \Pth{$p$-adic} \'etale local systems on $\Model_\levcp$ coming from $\Rep_{\bQ_p}(\Grp{G}^c)$ as above.
\end{thm}

Note that this theorem applies to \emph{all} Shimura varieties, on which \'etale local systems are not \Pth{yet} known to be related to motives in general.  Crucial ingredients in our proof include Margulis's \emph{superrigidity theorem} \cite{Margulis:1991-DSL}, and a construction credited to Piatetski-Shapiro by Borovoi \cite{Borovoi:1983/84-lcccs} and Milne \cite{Milne:1983-aacsv}.  This theorem itself has applications to the arithmetic of Shimura varieties, such as the following:
\begin{cor}\label{cor-intro-GM}
     The Grothendieck--Messing period map for $\Model_\levcp$ is \'etale.
\end{cor}

This implies that the local geometry of Shimura varieties is controlled by the moduli spaces of $p$-adic shtukas constructed by Scholze \Pth{see \cite[\aSec 23.3]{Scholze/Weinstein:2020-BLG}}.  \Pth{From the definitions, it was not clear how these moduli spaces are related to Shimura varieties.}  Some other applications of Theorem \ref{thm-intro-Sh} will appear in \cite{Lan/Liu/Zhu:dcpdr}.

Now let us explain our strategy for proving Theorem {\ref{thm-intro-main}}.  As mentioned above, in \cite[\aThm 3.9]{Liu/Zhu:2017-rrhpl}, a tensor functor $D^\an_\dR$ was constructed from the category of de Rham $p$-adic \'etale local systems on $X^\an$ to the category of filtered vector bundles on $X^\an$ with integrable connections satisfying the Griffiths transversality.  In order to prove Theorem {\ref{thm-intro-main}, a natural idea is to fix a smooth compactification $\overline{X}$ of $X$ with a normal crossings boundary divisor, and extend the filtered vector bundles with integrable connections in \emph{loc.~cit.} to filtered vector bundles on $\overline{X}^\an$ with integrable log connections \Pth{\ie, connections with log poles along the boundary divisor}.  However, rather unlike the complex analytic situation in \cite{Deligne:1970-EDR}, \emph{not} every integrable connection on $X^\an$ is extendable and hence algebraizable \Pth{see \cite[\aCh 4, \aRem 6.8.3]{Andre/Baldassarri:2001-DDA} or \cite[\aRem 34.6.3]{Andre/Baldassarri/Cailotto:2020-DDA(2)} for some counter-example}.

Instead, we shall directly construct a functor from the category of de Rham $p$-adic \'etale local systems on $X^\an$ to the category of filtered vector bundles with integrable log connections on $\overline{X}^\an$.  We shall work in the realm of log analytic geometry as in \cite{Diao/Lan/Liu/Zhu:lasfr}, and construct a \emph{log Riemann--Hilbert correspondence}, which is a crucial step in this paper.  Compared with \cite{Liu/Zhu:2017-rrhpl}, many new ingredients and essential new ideas are needed for this construction, and many new difficulties have to be overcome.  Let us begin with a rough summary.  Based on \cite{Diao/Lan/Liu/Zhu:lasfr}, the starting point is the construction of the log de Rham period sheaf $\OBdl$, generalizing the de Rham period sheaf $\OBdR$ as in \cite{Scholze:2013-phtra} and \cite{Liu/Zhu:2017-rrhpl}.  Then we define the log Riemann--Hilbert functors in a way similar to \cite{Liu/Zhu:2017-rrhpl}.  However, for our purpose, we also need to develop a very general formalism of decompletion systems, generalizing the one introduced in \cite[\aSec 5]{Kedlaya/Liu:2016-RPH-2} and many other classical works.  After these, the major divergence of the methods from \cite{Liu/Zhu:2017-rrhpl} occurs.  One of the key facts used in \emph{loc.~cit.}, that a coherent module with an integrable connection is automatically locally free, completely breaks down for log connections in general.  Our new idea is to study a collection of important invariants attached to a log connection---\ie, the \emph{residues} along the irreducible components of the boundary, using the above-mentioned decompletion formalism.  This allows us to prove a lot of favorable properties of the log connections constructed from the de Rham local systems.  In particular, we can canonically extend the filtration on the integrable connection \Pth{as constructed in \cite{Scholze:2013-phtra} and \cite{Liu/Zhu:2017-rrhpl}} to the boundary.  Moreover, the residues play an essential role in our study of nearby cycles, as in Theorem \ref{thm-intro-nearby}.

Let us also mention that, in the classical setting over $\bC$, Illusie--Kato--Nakayama \cite{Illusie/Kato/Nakayama:2005-qulrh, Illusie/Kato/Nakayama:2007-qulrh-err} developed a theory of quasi-unipotent log Riemann--Hilbert correspondence, which obtained as a byproduct Deligne's Riemann--Hilbert correspondence for local systems with quasi-unipotent monodromy at infinity.

We now explain our construction in more details.  We will work over a smooth rigid analytic variety $Y$ over $k$ \Pth{viewed as an adic space over $\Spa(k, \cO_k)$}, together with a normal crossings divisor $D \subset Y$, and view $Y$ as a log adic space by equipping it with the natural log structure defined by $D$.  \Pth{For applications to our previous setup, we take $Y = \overline{X}^\an$, and take $D$ to be the analytification of the boundary divisor $\overline{X} - X$ with its reduced subscheme structure.}  Any Kummer \'etale $\bZ_p$-local system $\bL$ on $Y$ induces a $\widehat{\bZ}_p$-local system $\widehat{\bL}$ on $Y_\proket$.  Let $\mu: Y_\proket \to Y_\an$ denote the natural projection from the pro-Kummer \'etale site to the analytic site.  \Pth{Note that as a subscript \Qtn{$\an$} means the analytic site, while as a superscript it means the analytification of an algebraic object.}  Let $\Omega^1_Y(\log D)$ denote the sheaf of differentials with log poles along $D$, as usual.  The following theorem is an abbreviated version of Theorem \ref{thm-log-RH-arith}, from which Theorem \ref{thm-intro-main} will be deduced.
\begin{thm}\label{thm-intro-log-RH}
    Let $Y$ and $\mu$ be as above.  Consider the functor $\Ddl$, which sends a Kummer \'etale $\bZ_p$-local system $\bL$ on $Y$ to
    \[
        \Ddl(\bL) := \mu_*(\widehat{\bL} \otimes_{\widehat{\bZ}_p} \OBdl).
    \]
    Then $\Ddl(\bL)$ is a vector bundle on $Y_\an$ equipped with an integrable log connection
    \[
         \nabla_\bL: \Ddl(\bL) \to \Ddl(\bL) \otimes_{\cO_Y} \Omega^1_Y(\log D)
    \]
    and a decreasing filtration \Pth{by coherent subsheaves} satisfying the Griffiths transversality, which extends the vector bundle $\DdR(\bL)$ with its integrable connection in \cite[\aThm 3.9]{Liu/Zhu:2017-rrhpl}.  Moreover, all eigenvalues of the residues of $\Ddl(\bL)$ along the irreducible components of $D$ are rational numbers in $[0, 1)$.  In particular, $(\Ddl(\bL), \nabla_\bL)$ is the \emph{canonical extension} of the $(\DdR(\bL), \nabla_\bL)$; \ie, the unique \Pth{if existent} extension of $(\DdR(\bL), \nabla_\bL)$ with such eigenvalues of residues.

    If $\bL|_{Y - D}$ is a \emph{de Rham} \'etale $\bZ_p$-local system, then $\gr \Ddl(\bL)$ is a vector bundle on $Y_\an$ of rank $\rank_{\bZ_p}(\bL)$, and we have the de Rham \Pth{\resp Hodge--Tate} comparison isomorphism between the Kummer \'etale cohomology of $\bL$ and the log de Rham \Pth{\resp log Hodge} cohomology of $\Ddl(\bL)$.
\end{thm}

Note that, unlike the functors $\DdRalg$ and $\DdR$ in Theorem \ref{thm-intro-main} and \cite[\aThm 3.9]{Liu/Zhu:2017-rrhpl}, the functor $\Ddl$ fails to be a tensor functor in general, as the eigenvalues of the residues of $\Ddl(\bL_1) \otimes_{\cO_Y} \Ddl(\bL_2)$ might be outside $[0, 1)$, and therefore $\Ddl(\bL_1) \otimes_{\cO_Y} \Ddl(\bL_2)$ might not be isomorphic to $\Ddl(\bL_1 \otimes_{\widehat{\bZ}_p} \bL_2)$.  This failure is caused by the failure of the surjectivity of the canonical morphism
\begin{equation}\label{eq-fail-surj}
    \mu^{-1}\bigl(\Ddl(\bL)\bigr) \otimes_{\cO_{Y_\proket}} \OBdl \to \widehat{\bL} \otimes_{\widehat{\bZ}_p} \OBdl
\end{equation}
in general, even when $\rank_{\cO_Y}(\Ddl(\bL)) = \rank_{\bZ_p}(\bL)$.  This phenomenon is not present in the usual comparison theorems in $p$-adic Hodge theory, but is consistent with the complex Riemann--Hilbert correspondence.  Nevertheless, we will see in Theorem \ref{thm-unip-vs-nilp} that $\Ddl$ restricts to a \emph{tensor functor} from the subcategory of de Rham local systems with \emph{unipotent geometric monodromy} along the boundary to the subcategory of integrable log connections with \emph{nilpotent residues} along the boundary.

We will deduce Theorem \ref{thm-intro-log-RH} from a geometric log Riemann--Hilbert correspondence.  Let $Y$ be as above, let $K$ be a perfectoid field over $k$ containing all roots of unity, and let $\Gal(K / k)$ abusively denote the group of continuous field automorphisms of $K$ over $k$.  Let $\BdRp = \BBdRp(K, \cO_K)$ and $\BdR = \BBdR(K, \cO_K)$ \Pth{as in \cite[\aSec 3.1]{Liu/Zhu:2017-rrhpl}}, and consider the ringed spaces $\cY^+ = (Y_\an, \cO_Y \ho_k \BdRp)$ and $\cY = (Y_\an, \cO_Y \ho_k \BdR)$, where $\cO_Y \ho_k \BdRp$ and $\cO_Y \ho_k \BdR$ are sheaves on $Y_\an$ which we interpret as the rings of functions on the not-yet-defined base changes \Qtn{$Y \ho_k \BdRp$} and \Qtn{$Y \ho_k \BdR$}, respectively.  Let $\mu': {Y_\proket}_{/Y_K} \to Y_\an$ denote the natural morphism of sites.  The following theorem is an abbreviated version of Theorem \ref{thm-log-RH-geom}.
\begin{thm}\label{thm-intro-log-RH-geom}
    The functor
    \[
        \RHl(\bL) := R\mu'_*(\widehat{\bL} \otimes_{\widehat{\bZ}_p} \OBdl)
    \]
    is an exact functor from the category of Kummer \'etale $\bZ_p$-local systems on $Y$ to the category of $\Gal(K / k)$-equivariant vector bundles on $\cY$, equipped with an integrable log connection $\nabla_\bL: \RHl(\bL) \to \RHl(\bL) \otimes_{\cO_Y} \Omega^1_Y(\log D)$, and a decreasing filtration \Pth{by locally free $\cO_Y \ho_k \BdRp$-submodules} satisfying the Griffiths transversality.  Moreover, we have $H^i_\ket\bigl(Y_{\AC{k}}, \bL\bigr) \otimes_{\bZ_p} B_\dR \cong H^i_{\log \dR}\bigl(\cY, \RHl(\bL)\bigr)$ when $K = \widehat{\AC{k}}$.
\end{thm}

Let $\RHlp(\bL) := R\mu'_*(\widehat{\bL} \otimes_{\widehat{\bZ}_p} \Fil^0 \OBdl)$, which is an $\cO_Y \ho_k \BdRp$-lattice in $\RHl(\bL)$ equipped \Pth{as in \cite[\aRem 3.2]{Liu/Zhu:2017-rrhpl}} with the $t$-connection $\nabla_\bL^+ := t \nabla_\bL: \RHlp(\bL) \to \RHlp(\bL) \otimes_{\cO_Y} \Omega^1_Y(\log D)(1)$, where $t \in \BdR$ is an element on which $\Gal(K / k)$ acts via the cyclotomic character.  By reduction modulo $t$, we obtain the log $p$-adic Simpson functor $\Hl$, constructed in much greater generality by Faltings \cite{Faltings:2005-psc} and Abbes--Gros--Tsuji \cite{Abbes/Gros/Tsuji:2016-pSC}.  See Theorem \ref{thm-log-Simp} for more details.

Compared with the situation in \cite{Liu/Zhu:2017-rrhpl}, the proof of Theorem \ref{thm-intro-log-RH-geom} requires some decompletion statement beyond the scope of \cite{Kedlaya/Liu:2016-RPH-2}.  We have therefore developed a general formalism in Appendix \ref{sec-decompl}, which might be of some independent interest.

Now we explain how to deduce Theorem \ref{thm-intro-log-RH} from Theorem \ref{thm-intro-log-RH-geom}, where some essential new ideas of this paper appear.  By using the above-mentioned decompletion statement and an argument similar to the one in \cite{Liu/Zhu:2017-rrhpl}, it is not difficult to show that $\Ddl(\bL) \cong (\RHl(\bL))^{\Gal(K / k)}$ is a coherent sheaf on $Y_\an$.  However, unlike the situation in \cite{Liu/Zhu:2017-rrhpl}, the existence of a log connection does not guarantee the local freeness of $\Ddl(\bL)$.  A priori, only its reflexivity is clear.  Nevertheless, there is a collection of important invariants attached to a log connection; \ie, the \emph{residues} along the irreducible components of $D$.  Somewhat surprisingly, by using the decompletion formalism again, we find that all the eigenvalues of the residues are \emph{rational numbers} in $[0, 1)$.  Together with the reflexivity and a general fact about log connections, this allows us to conclude that $\Ddl(\bL)$ is indeed locally free.  We remark that residues also plays a vital role in deducing the comparison of cohomology in Theorem \ref{thm-intro-log-RH} from the comparison of cohomology in Theorem \ref{thm-intro-log-RH-geom}.

Finally, our results on residues also allow us to define $V$-filtrations and study nearby cycles in the $p$-adic setting, which in turn allows us to deduce Theorem \ref{thm-intro-nearby}.

\subsection*{Outline of this paper}

Let us briefly describe the organization of this paper, and highlight the main themes in each section.

In Section \ref{sec-period-sheaves}, we study the \emph{log de Rham period sheaves}, generalizing the usual ones studied in \cite[\aSec 5]{Brinon:2008-RPR}, \cite[\aSec 6]{Scholze:2013-phtra}, and \cite{Scholze:2016-phtra-corr}, with a subtle difference---see Remark \ref{rem-mod-period-sheaf}.  In Section \ref{sec-log-adic-sp}, we recall some notation and basic results for log adic spaces developed in \cite{Diao/Lan/Liu/Zhu:lasfr}.  In Section \ref{sec-OBdl}, we present the general definitions of these log de Rham period sheaves.  In Section \ref{sec-OBdl-explicit}, we describe their structures in detail, when there are good local coordinates.  In Section \ref{sec-OBdl-conseq}, we record some consequences, including the \emph{Poincar\'e lemma}.  We note that results of this section hold for a class of log adic spaces larger than those considered in Theorems \ref{thm-intro-log-RH} and \ref{thm-intro-log-RH-geom}.  This extra generality is useful for many applications \Pth{see, \eg, \cite{Lan/Liu/Zhu:dcpdr}}.

In Section \ref{sec-log-RH}, we establish the geometric and arithmetic versions of \emph{log $p$-adic Riemann--Hilbert correspondences}, as well as the \emph{log $p$-adic Simpson correspondence}, as explained above.  We introduce some general terminologies for filtered vector bundles with log connections \Qtn{relative to $\BdR$} in Section \ref{sec-log-conn-BdR}, and state the main results in Section \ref{sec-log-RH-thm}, which are more detailed versions of Theorems \ref{thm-intro-log-RH} and \ref{thm-intro-log-RH-geom}.  The proofs are given in the following subsections.  In Section \ref{sec-coh}, we show that we obtain \emph{coherent sheaves} in the various constructions.  In Section \ref{sec-calc-res}, we calculate the \emph{eigenvalues of residues} of the connections along the boundary.  This is the technical heart of this paper.  In particular, it justifies the local freeness of the above coherent sheaves.  In Section \ref{sec-mor}, we show that our correspondences are compatible with certain pullbacks and pushforwards, by using our results on residues and the known compatibilities in \cite{Scholze:2013-phtra} and \cite{Liu/Zhu:2017-rrhpl}.  In Section \ref{sec-comp-coh}, we establish the comparisons of cohomology in our main theorems.  In Section \ref{sec-comp-nearby}, we show that the formation of quasi-unipotent nearby cycles \Pth{in the rigid analytic setting, as introduced in \cite[\aSec \logadicdefnearby]{Diao/Lan/Liu/Zhu:lasfr}} is compatible with the log Riemann--Hilbert functors.

In Section \ref{sec-alg}, we present our main results for algebraic varieties.  In Section \ref{sec-DdR-alg}, we construct the $p$-adic Riemann--Hilbert functor and prove Theorem \ref{thm-intro-main}.  We also establish the corresponding log Hodge--Tate comparison and the degeneration of \Pth{log} Hodge--de Rham spectral sequences, and record the latter results in Theorem \ref{thm-HT-degen-comp}.  In Section \ref{sec-van-gen}, we present some vanishing theorem for $p$-adic algebraic varieties, by adapting complex Hodge-theoretic arguments in \cite{Suh:2018-vmhma} using our $p$-adic results.  In Section \ref{sec-dR-bd}, we show that the formation of algebraic \Pth{quasi-unipotent} nearby cycles is compatible with our $p$-adic Riemann--Hilbert functor.

In Section \ref{sec-Sh-var}, we compare two constructions of filtered vector bundles with regular connections on Shimura varieties, and deduce Theorem \ref{thm-intro-Sh} and Corollary \ref{cor-intro-GM}.  In Section \ref{sec-loc-syst-setup}, we begin with the overall setup.  In Section \ref{sec-loc-syst-constr}, we explain the two constructions, one complex analytic and the other $p$-adic.  In Section \ref{sec-loc-syst-main-comp}, we state our main comparison theorem on these two constructions, and record some consequences.  In Section \ref{sec-proof-comp-prelim}, we reduce the theorem to a technical statement concerning representations of fundamental groups, which are then verified in the remaining two subsections.  This section can be read largely independent of the rest of the paper.

In Appendix \ref{sec-decompl}, we generalize the decompletion formalism in \cite[\aSec 5]{Kedlaya/Liu:2016-RPH-2}.  In Appendix \ref{sec-decompl-results}, we introduce and study the notions of \emph{decompletion systems} and \emph{decompleting triples}.  In Appendix \ref{sec-decompl-ex}, we present three examples which play crucial roles in Section \ref{sec-log-RH} \Pth{in the proof of coherence and the calculation of residues}.  The appendix can also be read largely independent of the rest of the paper.

\subsection*{Acknowledgements}

We would like to thank Kiran Kedlaya, Koji Shimizu, and Daxin Xu for helpful conversations, and thank the Beijing International Center for Mathematical Research, the Morningside Center of Mathematics, and the California Institute of Technology for their hospitality.  Some important ideas occurred to us when we were participants of the activities at the Mathematical Sciences Research Institute and the Oberwolfach Research Institute for Mathematics, and we would like to thank these institutions for providing stimulating working environments.  Finally, we would like to thank Yihang Zhu and the anonymous referees for many helpful comments that helped us correct and improve earlier versions of this paper.

\subsection*{Notation and conventions}

Unless otherwise specified, we always denote by $k$ a nonarchimedean local field \Pth{\ie, a field complete with respect to a nontrivial nonarchimedean multiplicative norm $|\cdot|: k \to \bR_{\geq 0}$} with residue field $\kappa$ of characteristic $p > 0$, and $\cO_k$ denotes the ring of integers in $k$.  We also denote by $k^+ \subset \cO_k$ an open valuation ring, whose choice depends on the context.  Sometimes, we choose a \emph{pseudo-uniformizer} \Pth{\ie, a topological nilpotent unit} $\varpi$ of $k$ contained in $k^+$.

By a \emph{locally noetherian adic space} over $k$, we mean an adic space $X$ over $\Spa(k, k^+)$ that admits an open covering by affinoids $U_i = \Spa(A_i, A_i^+)$ where each $A_i$ is strongly noetherian.  A \emph{noetherian adic space} over $k$ is a qcqs locally noetherian adic space over $k$.  If $X$ is locally noetherian, we denote by $X_\an$ its analytic site, by $X_\et$ its associated \'etale site, and by $\lambda: X_\et \to X_\an$ the natural projection of sites.  We shall regard rigid analytic varieties as adic spaces topologically of finite type over $\Spa(k, \cO_k)$ \Pth{as in \cite{Huber:1996-ERA}}, in which case we will work with $k^+ = \cO_k$.

By default, monoids are assumed to be commutative, and the monoid operations are written additively \Pth{rather than multiplicatively}, unless otherwise specified.  For a monoid $P$, let $P^\gp$ denote its group completion.  If $R$ is a commutative ring with unit and $P$ is a monoid, we denote by $R[P]$ the monoid algebra over $R$ associated with $P$.  The image of $a \in P$ in $R[P]$ will be denoted by $\mono{a}$.

Group cohomology will always mean continuous group cohomology.

As in \cite{Scholze:2013-phtra}, many of our results will be over nonarchimedean local fields $k$ that are \emph{discrete valuation fields of mixed characteristic $(0, p)$ with perfect residue fields}.  For the sake of simplicity, we will abusively call such fields \emph{$p$-adic fields}.  The main results of \cite{Liu/Zhu:2017-rrhpl} work over such fields.

\numberwithin{equation}{subsection}

\section{Log de Rham period sheaves}\label{sec-period-sheaves}

In this section, we define and study the log de Rham period sheaves, generalizing the usual ones studied in \cite[\aSec 5]{Brinon:2008-RPR}, \cite[\aSec 6]{Scholze:2013-phtra}, and \cite{Scholze:2016-phtra-corr}.  We shall assume that $k$ is of characteristic zero and residue characteristic $p > 0$.

\subsection{Basics of log adic spaces}\label{sec-log-adic-sp}

We begin with a summary of some notation and basic results for log adic spaces developed in the companion paper \cite{Diao/Lan/Liu/Zhu:lasfr}, in slightly less generality than the one in \emph{loc.~cit.}, for the sake of simplicity.

Let $X$ be any \'etale sheafy adic space; \ie, $X$ admits a well-defined \'etale site $X_\et$ and the \'etale structure presheaf $\cO_{X_\et}: U \mapsto \cO_U(U)$ is a sheaf.  A \emph{pre-log structure} on $X$ is a pair $(\cM_X, \alpha)$ consisting of a sheaf $\cM_X$ of monoids on $X_\et$ and a morphism $\alpha: \cM_X \to \cO_{X_\et}$ of sheaves of monoids \Pth{for the natural multiplicative monoid structure of $\cO_{X_\et}$}.  Such a pair is a \emph{log structure} if $\alpha$ induces an isomorphism $\alpha^{-1}(\cO_{X_\et}^\times) \Mi \cO_{X_\et}^\times$, in which case $(X, \cM_X, \alpha)$ is a \emph{log adic space}.  The log structure is \emph{trivial} when $\alpha^{-1}(\cO_{X_\et}^\times) = \cM_X$.  For simplicity, we shall often write $(X, \cM_X)$ or $X$, when the context is clear.  Moreover, we have the notions of morphisms between log structures and between log adic spaces, of the log structure associated with a pre-log structure, and of pullbacks of log structures.  These are analogous to the similar notions for schemes---see \cite[\aDef \logadicdeflogstr]{Diao/Lan/Liu/Zhu:lasfr} for more details.  A log adic space is \emph{noetherian} \Pth{\resp \emph{locally noetherian}} if its underlying adic space is.

For example, when $P$ is a monoid such that either $P$ is finitely generated, or that $k$ is perfectoid and $P$ is uniquely $p$-divisible, then we know that $Y := \Spa(k[P], k^+[P]) \cong \Spa(k\Talg{P}, k^+\Talg{P})$ is an \'etale sheafy adic space over $\Spa(k, k^+)$ \Pth{see \cite[\aLems \logadiclemstrnoe{} and \logadiclemmonoidalgperf]{Diao/Lan/Liu/Zhu:lasfr}}.  By abuse of notation, we shall sometimes denote by simply $P$ the constant sheaf $P_Y$ on $Y$ associated with the monoid $P$.  Then we have the canonical log structure $P^{\log}$ on $Y$ associated with the pre-log structure $P \to \cO_{Y_\et}$ induced by $a \mapsto \mono{a} \in k\Talg{P}$ \Pth{see \cite[\aDef \logadicdefPlog]{Diao/Lan/Liu/Zhu:lasfr}}.

\begin{exam}\label{ex-log-adic-sp-toric}
    When $P \cong \bZ_{\geq 0}^n$ for some $n \geq 0$, we have $\Spa(k\Talg{P}, k^+\Talg{P}) \cong \bD^n := \Spa(k\Talg{T_1, \ldots, T_n}, k^+\Talg{T_1, \ldots, T_n})$, the \emph{$n$-dimensional unit disc}, with the log structure on $\bD^n$ induced by $\bZ_{\geq 0}^n \to k\Talg{T_1, \ldots, T_n}: (a_1, \ldots, a_n) \mapsto T_1^{a_1} \cdots T_n^{a_n}$.
\end{exam}

Given any log adic space $X$ and any monoid $P$, a \emph{chart of $X$ modeled on $P$} is a morphism of sheaves of monoids $\theta: P_X \to \cM_X$ such that $\alpha\bigl(\theta(P_X)\bigr) \subset \cO_{X_\et}^+$ and such that the log structure associated with the pre-log structure $\alpha\circ\theta: P_X \to \cO_{X_\et}$ is isomorphic to $\cM_X$ \Pth{see \cite[\aDef \logadicdefchart]{Diao/Lan/Liu/Zhu:lasfr}}.  When $X$ is defined over $\Spa(k, k^+)$ and when $\Spa(k\Talg{P}, k^+\Talg{P})$ is defined as a log adic space as above, this is equivalent to having a strict morphism \Pth{see \cite[\aRem \logadicremdefchart]{Diao/Lan/Liu/Zhu:lasfr}} from $X$ to $\Spa(k\Talg{P}, k^+\Talg{P})$.  We say a log adic space is \emph{fs} if it \'etale locally admits charts models on monoids that are \emph{fs}, \ie, \emph{finitely generated}, \emph{integral}, and \emph{saturated} \Pth{see \cite[\aDefs \logadicdefmonoid{} and \logadicdeflogadicspfs]{Diao/Lan/Liu/Zhu:lasfr}}.  We also have the notion of fs charts of morphisms between fs log adic spaces \Pth{see \cite[\aProps \logadicpropchartmorexist{} and \logadicpropchartmorexistfine]{Diao/Lan/Liu/Zhu:lasfr}}.  By \cite[\aProp \logadicpropfiberprodlogadic]{Diao/Lan/Liu/Zhu:lasfr}, fiber products exist in the category of locally noetherian fs log adic spaces whenever the fiber products of underlying adic spaces exist \Pth{although the underlying adic spaces of the fiber products of fs log adic spaces might differ from the latter}.

We say that a morphism of locally noetherian fs log adic spaces is \emph{strictly \'etale} if the underlying morphism of adic spaces is \'etale.  By using the notion of charts, we can define \emph{log smooth} and \emph{log \'etale} morphisms of locally noetherian fs adic spaces \Pth{see \cite[\aDef \logadicdeflogsm]{Diao/Lan/Liu/Zhu:lasfr}}.  When $X$ is log smooth over $\Spa(k, \cO_k)$, where the latter is equipped with the trivial log structure, we simply say that $X$ is log smooth over $k$ \Pth{see \cite[\aDef \logadicdeflogsmbasefield]{Diao/Lan/Liu/Zhu:lasfr}}.  By \cite[\aProp \logadicproptoricchart]{Diao/Lan/Liu/Zhu:lasfr}, a log smooth fs log adic space $X$ over $k$ \'etale locally admits strictly \'etale morphisms $X \to \Spa(k\Talg{P}, k^+\Talg{P})$ which provide charts modeled on \emph{toric} monoids $P$, \ie, fs monoids that are \emph{sharp} in the sense that the subgroups $P^\inv$ of invertible elements of $P$ are trivial.  When the underlying adic space of $X$ is smooth, we may assume in the above that $P \cong \bZ_{\geq 0}^n$ for some $n \geq 0$, so that $\Spa(k\Talg{P}, k^+\Talg{P}) \cong \bD^n$ \Pth{see \cite[\aCor \logadiccorsmtoricchart]{Diao/Lan/Liu/Zhu:lasfr}}.  We call any strictly \'etale morphism $X \to \Spa(k\Talg{P}, k^+\Talg{P})$ \Pth{\resp $X \to \bD^n$} as above a \emph{toric chart} \Pth{\resp \emph{smooth toric chart}} \Pth{see \cite[\aDef \logadicdeftoricchart]{Diao/Lan/Liu/Zhu:lasfr}}.

We will mostly apply the general theory to the following class of fs log adic spaces \Pth{see \cite[\aExs \logadicexlogadicspncd{} and \logadicexlogadicspncdchart]{Diao/Lan/Liu/Zhu:lasfr} for more details}:
\begin{exam}\label{ex-log-adic-sp-ncd}
    Let $X$ be a smooth rigid analytic variety over $k$.  A \emph{\Pth{reduced} normal crossings divisor} $D$ of $X$ is given by a closed immersion $\imath: D \Em X$ over $k$ that is \'etale locally---or equivalently, analytic locally, up to replacing $k$ with a finite extension---of the form $S \times \{ T_1 \cdots T_r = 0 \} \Em S \times \bD^r$, for some smooth $S$ over $k$.  We equip $X$ with the log structure $\cM_X := \{ f \in \cO_{X_\et} : \Utext{$f$ is invertible on $X - D$} \}$, where $\alpha: \cM_X \to \cO_{X_\et}$ is the natural inclusion, which makes $(X, \cM_X)$ a log smooth noetherian fs adic space over $k$.  \'Etale locally, when $X \cong S \times \bD^r$, the log structure on $X$ is the pullback of the one on $\bD^r$ \Pth{see Example \ref{ex-log-adic-sp-toric}}, and we have smooth toric charts $X \to \bD^n$ such that $\imath: D \Em X$ is the pullback of $\{ T_1 \cdots T_r = 0 \} \Em \bD^n$, where $T_1, \ldots, T_n$ are the coordinates of $\bD^n$, for some $0 \leq r \leq n$.
\end{exam}

For each locally noetherian fs log adic space $X$, we have the \emph{Kummer \'etale site} $X_\ket$, as in \cite[\aSec \logadicsecketsite]{Diao/Lan/Liu/Zhu:lasfr}, whose objects are fs log adic spaces \emph{Kummer \'etale} over $X$.  A typical example of Kummer \'etale morphisms is, for each integer $m \geq 1$, the ramified cover $\bD_m^n := \Spa(k\Talg{T_1^{\frac{1}{m}}, \ldots, T_n^{\frac{1}{m}}}, k^+\Talg{T_1^{\frac{1}{m}}, \ldots, T_n^{\frac{1}{m}}}) \to \bD^n = \Spa(k\Talg{T_1, \ldots, T_n}, k^+\Talg{T_1, \ldots, T_n})$.  We also have the \emph{pro-Kummer \'etale site} $X_\proket$, as in \cite[\aDef \logadicdefproketsite]{Diao/Lan/Liu/Zhu:lasfr} \Pth{which generalizes \cite{Scholze:2016-phtra-corr}}.  Then we have natural projections of sites $\upsilon_X: X_\proket \to X_\ket$, $\varepsilon_{X, \et}: X_\ket \to X_\et$, and $\nu_X = \varepsilon_{X, \et} \circ \upsilon_X: X_\proket \to X_\et$, where $\varepsilon_{X, \et}$ is an isomorphism when the log structure is trivial.  For any morphism $f: X \to Y$ of locally noetherian fs log adic spaces, we have compatible canonical morphisms of sites $f_\ket: X_\ket \to Y_\ket$ and $f_\proket: X_\proket \to Y_\proket$.

We can naturally define locally constant sheaves and torsion local systems on $X_\ket$ \Pth{see \cite[\aDef \logadicdefketlocconst]{Diao/Lan/Liu/Zhu:lasfr}}.  We define a \emph{$\bZ_p$-local system} \Pth{or \emph{lisse $\bZ_p$-sheaf}} on $X_\ket$ to be an inverse system of $\bZ / p^n$-modules $\bL = (\bL_n)_{n \geq 1}$ on $X_\ket$ such that each $\bL_n$ is a locally constant sheaf which are locally \Pth{on $X_\ket$} associated with finitely generated $\bZ / p^n$-modules, and such that the inverse system is isomorphic in the pro-category to an inverse system in which $\bL_{n + 1} / p^n \cong \bL_n$.  We define a \emph{$\bQ_p$-local system} \Pth{or \emph{lisse $\bQ_p$-sheaf}} on $X_\ket$ to be an object of the stack associated with the fibered category of isogeny lisse $\bZ_p$-sheaves.  \Pth{See \cite[\aDef \logadicdefketlisse]{Diao/Lan/Liu/Zhu:lasfr}.}  Let $\widehat{\bZ}_p := \varprojlim_n (\bZ / p^n)$ as a sheaf of rings on $X_\proket$, and let $\widehat{\bQ}_p := \widehat{\bZ}_p[\frac{1}{p}]$.  A \emph{$\widehat{\bZ}_p$-local system} on $X_\proket$ is a sheaf of $\widehat{\bZ}_p$-modules on $X_\proket$ that is locally isomorphic to $L \otimes_{\bZ_p} \widehat{\bZ}_p$ for some finitely generated $\bZ_p$-modules $L$.  The notion of \emph{$\widehat{\bQ}_p$-local systems} on $X_\proket$ is defined similarly.  \Pth{See \cite[\aDef \logadicdefproketlocsyst{} and \aLem \logadiclemproketlisse]{Diao/Lan/Liu/Zhu:lasfr}}.  Functors on $\bQ_p$- \Pth{\resp $\widehat{\bQ}_p$-} local systems naturally extend to functors on $\bZ_p$- \Pth{\resp $\widehat{\bZ}_p$-} local systems, which we shall abusively denote by the same symbols, for simplicity.

\subsection{Definitions of period sheaves}\label{sec-OBdl}

Let $(R, R^+)$ be a perfectoid affinoid algebra over $k$, with $(R^\flat, R^{\flat+})$ its tilt.  Recall that there are the period rings
\[
    \AAinf(R, R^+) := W(R^{\flat+}) \quad \Utext{and} \quad \BBinf(R, R^+) := \AAinf(R, R^+)[\tfrac{1}{p}].
\]
It is well known that there is a natural surjective map
\begin{equation}\label{eq-theta}
    \theta: \AAinf(R, R^+) \to R^+,
\end{equation}
whose kernel is a principal ideal generated by some $\xi \in \AAinf(R, R^+)$, which is not a zero divisor \Pth{see, for example, \cite[\aLem 3.6.3]{Kedlaya/Liu:2015-RPH}}.  We define
\[
    \BBdRp(R, R^+) := \varprojlim_r \bigl( \BBinf(R, R^+) / \xi^r \bigr) \quad \Utext{and} \quad \BBdR(R, R^+) := \BBdRp(R, R^+)[\xi^{-1}].
\]
We shall equip $\BBdR(R, R^+)$ with the filtration $\Fil^r \BBdR(R, R^+) := \xi^r \BBdRp(R, R^+)$, for all $r \in \bZ$.  This filtration is separated, complete, and independent of the choice of $\xi$.  Therefore, for all $r \in \bZ$, we have a canonical isomorphism
\begin{equation}\label{eq-gr-BdR}
    \gr^r \BBdR(R, R^+) \cong \xi^r R.
\end{equation}

Now let $X$ be a locally noetherian fs log adic space over $\Spa(\bQ_p, \bZ_p)$.  As in \cite[\aDef \logadicdefproketsheaves]{Diao/Lan/Liu/Zhu:lasfr} \Pth{which generalizes \cite[\aDefs 4.1 and 5.10]{Scholze:2013-phtra}}, we have:
\begin{itemize}
    \item $\cO_{X_\proket}^? = \upsilon^{-1}(\cO_{X_\ket}^?)$, where $? = \emptyset$ or $+$;
    \item $\widehat{\cO}_{X_\proket}^+ = \varprojlim_n \bigl( \cO_{X_\proket}^+ / p^n \bigr)$ and $\widehat{\cO}_{X_\proket} = \widehat{\cO}_{X_\proket}^+[\frac{1}{p}]$;
    \item $\widehat{\cO}_{X_\proket}^{\flat?} = \varprojlim_\Phi \widehat{\cO}_{X_\proket}^?$, where $? = \emptyset$ or $+$;
    \item $\alpha: \cM_{X_\proket} := \upsilon^{-1}(\cM_{X_\ket}) \to \cO_{X_\proket}$; and
    \item $\alpha^\flat: \cM_{X_\proket}^\flat := \varprojlim_{a \mapsto a^p} \cM_{X_\proket} \to \widehat{\cO}_{X_\proket}^\flat$.
\end{itemize}
We shall sometimes omit the subscripts \Qtn{$\proket$} or \Qtn{$X_\proket$}.
\begin{defn}\label{def-BdR}
    We define the following sheaves on $X_\proket$:
    \begin{enumerate}
        \item\label{def-BdR-1}  Let $\AAinfX{X} := W(\widehat{\cO}_{X^\flat_\proket}^+)$ and $\BBinfX{X} := \AAinfX{X}[\frac{1}{p}]$, where the latter is equipped with a natural map $\theta: \BBinfX{X} \to \widehat{\cO}_{X_\proket}$.

        \item\label{def-BdR-2}  Let $\BBdRpX{X} := \varprojlim_r \bigl( \BBinfX{X} / (\ker \theta)^r \bigr)$, and $\BBdRX{X} := \BBdRpX{X}[t^{-1}]$, where $t$ is any generator of $(\ker \theta) \BBdRpX{X}$.  \Pth{We will choose some $t$ in \Refeq{\ref{eq-choice-t}} below.}

        \item\label{def-BdR-3}  The filtration on $\BBdRpX{X}$ is given by $\Fil^r \BBdRpX{X} := (\ker \theta)^r \BBdRpX{X}$.  It induces a filtration on $\BBdRX{X}$ given by $\Fil^r \BBdRX{X} := \sum_{s \geq -r} t^{-s} \Fil^{r + s} \BBdRpX{X}$.
    \end{enumerate}
    We shall omit the subscript $X$ in the notation when the context is clear.
\end{defn}

\begin{prop}\label{prop-BdR-van}
    Suppose that $U \in X_\proket$ is log affinoid perfectoid, with associated perfectoid space $\widehat{U} = \Spa(R, R^+)$, as in \cite[\aDef \logadicdeflogaffperf{} and \aRem \logadicremdeflogaffperfsp]{Diao/Lan/Liu/Zhu:lasfr}.
    \begin{enumerate}
        \item\label{prop-BdR-van-1}  We have a canonical isomorphism $\AAinf(U) \cong \AAinf(R, R^+)$, and similar isomorphisms for $\BBinf$, $\BBdRp$, and $\BBdR$.

        \item\label{prop-BdR-van-2}  $H^j(U, \BBdRp) = 0$ and $H^j(U, \BBdR) = 0$ for all $j > 0$.
    \end{enumerate}
\end{prop}
\begin{proof}
    The proof is essentially the same as in the one of \cite[\aThm 6.5]{Scholze:2013-phtra}, with the input \cite[\aLem 5.10]{Scholze:2013-phtra} there replaced with \cite[\aProp \logadicthmalmostvanhat]{Diao/Lan/Liu/Zhu:lasfr}.
\end{proof}

\begin{rk}\label{rem-BdR-t-loc}
    By the previous discussion and Proposition \ref{prop-BdR-van}\Refenum{\ref{prop-BdR-van-1}} \Pth{for $\AAinf$}, the element $t$ in Definition \ref{def-BdR}\Refenum{\ref{def-BdR-2}} exists locally on $X_\proket$ and is not a zero divisor.  Therefore, the sheaf $\BBdR$ and its filtration are indeed well defined.
\end{rk}

\begin{cor}\label{cor-gr-BdR}
    If $X$ is over a perfectoid field $k$ \Pth{over $\bQ_p$} containing all roots of unity, then $\gr^\bullet \BBdR \cong \oplus_{r \in \bZ} \, \bigl( \widehat{\cO}_{X_\proket}(r) \bigr)$.
\end{cor}
\begin{proof}
    This follows from \Refeq{\ref{eq-gr-BdR}} and Proposition \ref{prop-BdR-van}, as in \cite[\aCor 6.4]{Scholze:2013-phtra}.
\end{proof}

\begin{cor}\label{cor-BdR-cl-imm}
    Suppose that $\imath: Z \to X$ is a strict closed immersion, as in \cite[\aDef \logadicdefimm]{Diao/Lan/Liu/Zhu:lasfr}.  Then $\BBdRX{X} \to \imath_{\proket, *}(\BBdRX{Z})$ is surjective.  More precisely, its evaluation at every log affinoid perfectoid object $U$ in $X_\proket$ is surjective.
\end{cor}
\begin{proof}
    This follows from \cite[\aProp \logadicpropsurjOhatclimm]{Diao/Lan/Liu/Zhu:lasfr} and Proposition \ref{prop-BdR-van}.
\end{proof}

Now let $X$ be a locally noetherian fs log adic space over $\Spa(k, k^+)$.  We shall construct $\OBdlpX{X}$, a log version of the \emph{geometric de Rham period sheaves} $\OBdRpX{X}$ introduced in \cite{Brinon:2008-RPR, Scholze:2013-phtra, Scholze:2016-phtra-corr}.  As log affinoid perfectoid objects form a basis $\cB$ of $X_\proket$ \Pth{see \cite[\aProp \logadicproplogaffperfbasis]{Diao/Lan/Liu/Zhu:lasfr}}, it suffices to define $\OBdlpX{X}$ as a sheaf associated with a presheaf on $\cB$.

We adopt the notation in \cite[\aSecs \logadicseclogaffperf{} and \logadicsecproketsheaves]{Diao/Lan/Liu/Zhu:lasfr}.  Let $U = \varprojlim_{i \in I} U_i \in X_\proket$ be a log affinoid perfectoid object, with $U_i = (\Spa(R_i, R_i^+), \cM_i, \alpha_i)$, for each $i \in I$, and with associated perfectoid space $\widehat{U} = \Spa(R, R^+)$, where $(R, R^+)$ is the $p$-adic completion of $\varinjlim_{i \in I} (R_i, R_i^+)$, which is perfectoid.  By \cite[\aThm \logadicthmalmostvanhat]{Diao/Lan/Liu/Zhu:lasfr}, $\bigl(\widehat{\cO}(U), \widehat{\cO}^+(U)\bigr) = (R, R^+)$, and $\bigl(\widehat{\cO}^\flat(U), \widehat{\cO}^{\flat+}(U)\bigr)$ is its tilt $(R^\flat, R^{\flat+})$.  Let us write:
\begin{itemize}
    \item $M_i := \cM_i(U_i)$, for each $i \in I$;
    \item $M := \cM_{X_\proket}(U) = \varinjlim_{i \in I} M_i$; and
    \item $M^\flat := \cM_{X_\proket^\flat}(U) = \varprojlim_{a \mapsto a^p} M$.
\end{itemize}
Recall that we have $\alpha_i: M_i \to R_i$, for each $i \in I$, and $\alpha^\flat: M^\flat \to R^\flat$ \Pth{see Section \ref{sec-log-adic-sp} and \cite[\aDef \logadicdefproketsheaves]{Diao/Lan/Liu/Zhu:lasfr}}.  For each $r \geq 1$, we have a multiplicative map $R^\flat \to W(R^{\flat+})[\frac{1}{p}] / \xi^r$ induced by $R^{\flat+} \to W(R^{\flat+})$, which we still denote by $f \mapsto [f]$.  Then the composition of this map with $\alpha^\flat: M^\flat \to R^\flat$ induces a map
\[
\begin{split}
    \widetilde{\alpha}_{i, r}: M_i \times_M M^\flat & \to \bigl( R_i \ho_{W(\kappa)} (W(R^{\flat+}) / \xi^r) \bigr)[M_i \times_M M^\flat] \\
    a = (a', a'') & \mapsto \bigl(\alpha_i(a') \otimes 1\bigr) - \bigl(1 \otimes [\alpha^\flat(a'')]\bigr) \, \monob{a},
\end{split}
\]
where $\monob{a}$ denotes \Pth{in boldface, unlike in our convention} the element of the monoid algebra corresponding to $a = (a', a'') \in M_i \times_M M^\flat$.  Let
\begin{equation}\label{def-S-i-r}
    S_{i, r} := \bigl( R_i \ho_{W(\kappa)} (W(R^{\flat+}) / \xi^r) \bigr)[M_i \times_M M^\flat] \big/ \bigl( \widetilde{\alpha}_{i, r}(a) \bigr)_{a \in M_i \times_M M^\flat}.
\end{equation}
By abuse of notation, we shall sometimes drop tensor products with $1$ in the notation, and write $\alpha_i(a') = [\alpha^\flat(a'')] \, \monob{a}$ in $S_{i, r}$.  There is a natural map
\begin{equation}\label{eq-theta-log}
    \theta_{\log}: S_{i, r} \to R
\end{equation}
induced by the natural maps $R_i \to R$ and \Refeq{\ref{eq-theta}} such that $\theta_{\log}(\monob{a}) = 1$, which is well defined because $\theta([\alpha^\flat(a'')]) = \alpha_i(a')$ in $R$, for all $(a', a'') \in M_i \times_M M^\flat$.  Let
\[
    \widehat{S}_i := \varprojlim_{r, s} \bigl( S_{i, r} / (\ker \theta_{\log})^s \bigr),
\]
equipped with a canonically induced map $\theta_{\log}: \widehat{S}_i \to R$.  Note that $\varinjlim_{i \in I} \widehat{S}_i$ depends only on $U \in X_\proket$, but not on the presentation $U = \varprojlim_{i \in I} U_i$.

\begin{defn}\phantomsection\label{def-OBdl}
    \begin{enumerate}
        \item The geometric de Rham period sheaf $\OBdlpX{X}$ on $X_\proket$ is the sheaf associated with the presheaf sending $U$ to $\varinjlim_{i \in I} \widehat{S}_i$, equipped with the filtration $\Fil^r \OBdlpX{X} := (\ker \theta_{\log})^r \OBdlpX{X}$.

        \item We define the filtration on $\OBdlpX{X}[t^{-1}]$, where $t$ is the same as in Definition \ref{def-BdR}\Refenum{\ref{def-BdR-2}}, by $\Fil^r (\OBdlpX{X}[t^{-1}]) := \sum_{s \geq -r} t^{-s} \Fil^{r + s} \OBdlpX{X}$.

        \item Let $\OBdlX{X}$ be the completion of $\OBdlpX{X}[t^{-1}]$ with respect to the above filtration, equipped with the induced filtration.  Then we have
            \[
                \Fil^r \OBdlX{X} = \varprojlim_{s \geq 0} \bigl( \Fil^r (\OBdlpX{X}[t^{-1}]) / \Fil^{r + s} (\OBdlpX{X}[t^{-1}]) \bigr),
            \]
            and $\OBdlX{X} = \cup_{r \in \bZ} \, \Fil^r \OBdlX{X}$.  Let $\OClX{X} := \gr^0 \OBdlX{X}$.
    \end{enumerate}
    We shall omit the subscript $X$ in the notation when the context is clear.
\end{defn}

Note that $\Fil^0 \OBdl$ is a sheaf of rings and $\OBdl = (\Fil^0 \OBdl)[t^{-1}]$.  \Pth{However, $\OBdl \neq \OBdlp[t^{-1}]$ in general.}

\begin{rk}\label{rem-mod-period-sheaf}
    Even if the log structure is trivial, the definition of $\OBdl$ given here is slightly different from the definitions of $\OBdR$ in \cite[\aSec 5]{Brinon:2008-RPR}, \cite[\aSec 6]{Scholze:2013-phtra}, and \cite{Scholze:2016-phtra-corr}, as we perform an additional completion with respect to the filtration.  This modification is necessary because the sheaves $\OBdR$ defined in \emph{loc.~cit.} are not complete with respect to the filtrations---we thank Koji Shimizu for pointing out this.  We will see in Corollary \ref{cor-log-dR-cplx} below that the Poincar\'e lemma still holds and, with the new definition of $\OBdR$, all the previous arguments in \emph{loc.~cit.} \Pth{and also those in \cite{Liu/Zhu:2017-rrhpl}} remain essentially unchanged.
\end{rk}

\begin{rk}\label{rem-def-OBdl-applicability}
    Although $\OBdl$ is defined for any locally noetherian fs log adic space $X$ over $\Spa(k, k^+)$, we shall only use it for log smooth ones over $p$-adic fields, or some of their closed subspaces with induced log structures.
\end{rk}

The period sheaf $\OBdlp$ is equipped with a natural log connection.  Let $\Omega^{\log}_X$ be the sheaf of log differentials, as in \cite[\aDef \logadicdeflogdiffsheaf]{Diao/Lan/Liu/Zhu:lasfr}.  By abuse of notation, its pullbacks to $X_\et$, $X_\ket$, and $X_\proket$ will still be denoted by the same symbols.

Note that there is a unique $\BBdRp(U) / \xi^r$-linear log connection
\begin{equation}\label{eq-conn-S-i}
    \nabla: S_{i, r} \to S_{i, r} \otimes_{R_i} \Omega^{\log}_X(U_i)
\end{equation}
extending $d: R_i \to \Omega^{\log}_X(U_i)$ and $\delta: M_i \to \Omega^{\log}_X(U_i)$ such that
\begin{equation}\label{eq-conn-e-a}
    \nabla(\monob{a}) = \monob{a} \, \delta(a'),
\end{equation}
for all $a = (a', a'') \in M_i \times_M M^\flat$.  Essentially by definition, we have
\[
    \nabla\bigl( (\ker \theta_{\log})^s \bigr) \subset (\ker \theta_{\log})^{s - 1} \otimes_{R_i} \Omega^{\log}_X(U_i)
\]
for all $s \geq 1$.  By taking $\ker(\theta_{\log})$-adic completion, inverse limit over $r$, and direct limit over $i$, the above log connection \Refeq{\ref{eq-conn-S-i}} extends to a $\BBdRp$-linear log connection
\begin{equation}\label{eq-conn-OBdlp}
    \nabla: \OBdlp \to \OBdlp \otimes_{\cO_{X_\proket}} \Omega^{\log}_X.
\end{equation}
Since $t \in \BBdRp$, \Refeq{\ref{eq-conn-OBdlp}} further extends to a $\BBdR$-linear log connection
\begin{equation}\label{eq-conn-OBdlp-t}
    \nabla: \OBdlp[t^{-1}] \to \OBdlp[t^{-1}] \otimes_{\cO_{X_\proket}} \Omega^{\log}_X,
\end{equation}
satisfying $\nabla\bigl( \Fil^r (\OBdlp[t^{-1}]) \bigr) \subset \bigl( \Fil^{r - 1} (\OBdlp[t^{-1}]) \bigr) \otimes_{\cO_{X_\proket}} \Omega^{\log}_X$, for all $r \in \bZ$.  Therefore, \Refeq{\ref{eq-conn-OBdlp-t}} also extends to a $\BBdR$-linear log connection
\begin{equation}\label{eq-conn-OBdl}
    \nabla: \OBdl \to \OBdl \otimes_{\cO_{X_\proket}} \Omega^{\log}_X,
\end{equation}
satisfying $\nabla(\Fil^r \OBdl) \subset ( \Fil^{r - 1} \OBdl ) \otimes_{\cO_{X_\proket}} \Omega^{\log}_X$, for all $r \in \bZ$.

\subsection{Local study of $\BBdR$ and $\OBdl$}\label{sec-OBdl-explicit}

In this subsection, we study $\BBdR$ and $\OBdl$ when there are good local coordinates.  These results are similar to \cite[\aSec 5]{Brinon:2008-RPR}, \cite[\aSec 6]{Scholze:2013-phtra}, and \cite{Scholze:2016-phtra-corr}.  We assume that $k$ is a $p$-adic field, and let $\AC{k}$ be a fixed algebraic closure.  For each $m \geq 1$, we denote by $\Grpmu_m$ \Pth{\resp $\Grpmu_\infty = \cup_m \, \Grpmu_m$} the group of $m$-th \Pth{\resp all} roots of unity in $\AC{k}$.  Let $k_m = k(\Grpmu_m) \subset \AC{k}$, for all $m \geq 1$; let $k_\infty = k(\Grpmu_\infty) = \cup_m \, k_m$ in $\AC{k}$; and let $\widehat{k}_\infty$ be the $p$-adic completion of $k_\infty$.  Then $\widehat{k}_\infty$ is a perfectoid field.  Let $\widehat{k}_\infty^\flat$ denote its tilt.  We shall denote by $k_m^+$, $\widehat{k}_\infty^+$, and $\widehat{k}_\infty^{\flat+}$ the rings of integers in $k_m$, $\widehat{k}_\infty$, and $\widehat{k}_\infty^\flat$, respectively.  Let
\[
    \Ainf := \AAinf(\widehat{k}_\infty, \widehat{k}_\infty^+).
\]
Fix an isomorphism of abelian groups
\begin{equation}\label{eq-zeta}
    \zeta: \bQ / \bZ \Mi \Grpmu_\infty,
\end{equation}
and write $\zeta_n := \zeta(\frac{1}{n})$, for all $n \in \bZ_{\geq 1}$.  Then $\zeta(\frac{m}{n}) = \zeta_n^m$, for all $\frac{m}{n} \in \bQ$.  Define
\[
    \epsilon: \bQ \to (\widehat{k}_\infty^{\flat+})^\times: \; y \mapsto (\zeta(y), \zeta(\tfrac{y}{p}), \zeta(\tfrac{y}{p^2}), \ldots).
\]
We shall also write $\zeta^y = \zeta(y)$ and $\epsilon^y = \epsilon(y)$.  \Pth{In particular, $\epsilon = \epsilon(1)$.}  Note that $\varpi^\flat := (\epsilon - 1) / (\epsilon^{\frac{1}{p}} - 1)$ is a pseudo-uniformizer of $\widehat{k}_\infty^\flat$, and
\[
    \xi := ([\epsilon] - 1) / ([\epsilon^{\frac{1}{p}}] - 1)
\]
generates the kernel of $\theta: \Ainf \to \widehat{k}_\infty^+$.  Consider
\begin{equation}\label{eq-choice-t}
    t := \log([\epsilon]) \in \BdRp := \BBdRp(\widehat{k}_\infty, \widehat{k}_\infty^+) = \varprojlim_r \bigl( \Ainf[\tfrac{1}{p}] / \xi^r \bigr).
\end{equation}
Let $k \to \BdRp$ be the unique embedding such that the composition of $k \to \BdRp \to \widehat{k}_\infty$ is the natural homomorphism $k \to \widehat{k}_\infty$.

Let $P$ be a toric monoid which decomposes as a direct sum $P = \overline{P} \oplus Q$ of toric monoids.  \Pth{Here $\overline{P}$ means $P / Q$ as in \cite[\aRem \logadicremmonoidcat]{Diao/Lan/Liu/Zhu:lasfr}, rather than the sharp quotient of $P$ as in \cite[\aDef \logadicdefmonoid]{Diao/Lan/Liu/Zhu:lasfr}.}  We shall denote by $(Q - \{0\})$ the ideal of $k\Talg{P}$ \Pth{and other similar algebras} generated by $\{ \mono{a} \}_{a \in Q - \{0\}}$.  The surjection $k\Talg{P} \Surj k\Talg{P} / (Q - \{0\}) \cong k\Talg{\overline{P}}$ induces a strict closed immersion
\[
    \bE := \Spa(k\Talg{\overline{P}}, k^+\Talg{\overline{P}}) \Em \Spa(k\Talg{P}, k^+\Talg{P}),
\]
where $\bE$ is equipped with the log structure pulled back from $\Spa(k\Talg{P}, k^+\Talg{P})$ \Pth{see \cite[\aDef \logadicdefPlog]{Diao/Lan/Liu/Zhu:lasfr}}.  For each $m \in \bZ_{\geq 1}$, let $\frac{1}{m} P$ be the toric monoid such that $P \Em \frac{1}{m} P$ can be identified with the $m$-th multiple map $[m]: P \to P$, and let $\frac{1}{m} Q$ and $\frac{1}{m} \overline{P}$ be defined similarly.  Let $P_{\bQ_{\geq 0}} := \varinjlim_m \bigl(\frac{1}{m} P\bigr)$ and $P^\gp_\bQ := (P_{\bQ_{\geq 0}})^\gp \cong P^\gp \otimes_\bZ \bQ$, and let $Q_{\bQ_{\geq 0}}$, $Q^\gp_\bQ$, $\overline{P}_{\bQ_{\geq 0}}$, and $\overline{P}^\gp_\bQ$ be defined similarly.  By pulling back a standard toric tower over $\Spa(k\Talg{P}, k^+\Talg{P})$ \Pth{\Refcf{} \cite[\aSec \logadicsectoricchart]{Diao/Lan/Liu/Zhu:lasfr}}, we obtain a pro-Kummer \'etale cover $\widetilde{\bE} := \varprojlim_m \bE_m \to \bE$, with strict closed immersions
\[
    \bE_m := \bE \times_{\Spa(k\Talg{P}, k^+\Talg{P})} \Spa(k_m\Talg{\tfrac{1}{m} P}, k_m^+\Talg{\tfrac{1}{m} P}) \Em \Spa(k_m\Talg{\tfrac{1}{m} P}, k_m^+\Talg{\tfrac{1}{m} P})
\]
and transition morphisms $\bE_{m'} \to \bE_m$ \Pth{for $m | m'$} induced by the natural inclusions $\frac{1}{m} P \Em \frac{1}{m'} P$.  Note that $\bE_m \cong \Spa(k_m\Talg{\tfrac{1}{m} \overline{P}}, k_m^+\Talg{\tfrac{1}{m} \overline{P}})$ as adic spaces, because all nilpotent elements of $k_m\Talg{\frac{1}{m}P} / (Q - \{0\})$ are integral over $k_m^+\Talg{\frac{1}{m}P} / (Q - \{0\})$ and hence $p$-divisible.  Then $\widetilde{\bE}$ is log affinoid perfectoid with associated perfectoid space
\[
    \widehat{\widetilde{\bE}} = \Spa(\widehat{k}_\infty\Talg{\overline{P}_{\bQ_{\geq 0}}}, \widehat{k}_\infty^+\Talg{\overline{P}_{\bQ_{\geq 0}}}).
\]

For $a \in P_{\bQ_{\geq 0}}$, we denote by $T^a \in k^+\Talg{P_{\bQ_{\geq 0}}}$ the corresponding element \Pth{as opposed to $\mono{a}$ as usual}.  Let $\overline{T}^a$ denote the image of $T^a$ in $k^+\Talg{\overline{P}_{\bQ_{\geq 0}}}$.  Note that $\overline{T}^a = 0$ when $a \not\in \overline{P}_{\bQ_{\geq 0}}$.  Moreover, we denote by $\frac{1}{m} a$ the unique element in $P_{\bQ_{\geq 0}}$ such that $m(\frac{1}{m} a) = a$, so that $\bigl(T^{\frac{1}{m} a}\bigr)^m = T^a$ in $k^+\Talg{P_{\bQ_{\geq 0}}}$.  Let $T^{a \flat} := (T^a, T^{\frac{1}{p} a}, \ldots) \in (\widehat{k}_\infty^+\Talg{P_{\bQ_{\geq 0}}})^\flat$, and let $\overline{T}^{a\flat}$ denote its image in $(\widehat{k}_\infty^+\Talg{\overline{P}_{\bQ_{\geq 0}}})^\flat$.  Again, note that $\overline{T}^{a\flat} = 0$ when $a \not\in \overline{P}_{\bQ_{\geq 0}}$.  The Galois group $\Gamma = \Aut(\widetilde{\bE} / \bE_{\widehat{k}_\infty})$ has a natural action on $\widehat{k}_\infty^+\Talg{\overline{P}_{\bQ_{\geq 0}}}$ given by
\begin{equation}\label{eq-act-coord}
    \gamma(\overline{T}^{\overline{a}}) = \gamma(\overline{a}) \, \overline{T}^{\overline{a}},
\end{equation}
for all $\gamma \in \Gamma$ and $\overline{a} \in \overline{P}_{\bQ_{\geq 0}} \subset \overline{P}^\gp_\bQ \subset P^\gp_\bQ$, where $\gamma(\overline{a})$ is the element of $\Grpmu_\infty$ given by $\Aut(\widetilde{\bE} / \bE_{\widehat{k}_\infty}) \cong \Hom(P^\gp_\bQ / P^\gp, \Grpmu_\infty)$ \Pth{\Refcf{} \cite[(\logadiceqgeomtowerGal)]{Diao/Lan/Liu/Zhu:lasfr}}.

For each $r \geq 1$, we view $\BdRp / \xi^r$ as a Tate $k$-algebra \Pth{in the sense of \cite[\aDef 2.6]{Scholze:2012-ps}} with a ring of definition $\Ainf / \xi^r$ \Pth{with its $p$-adic topology}, and view
\[
    (\BdRp / \xi^r)\Talg{\tfrac{1}{m} P} / (Q - \{0\}) = (\BdRp / \xi^r) \ho_k \bigl(k\Talg{\tfrac{1}{m} P} / (Q - \{0\})\bigr)
\]
as a Tate algebra as well.  The completed direct limit \Pth{over $m$} of these algebras is canonically isomorphic to the completed direct limit of $(\BdRp / \xi^r)\Talg{\frac{1}{m} \overline{P}}$, which we denote by $(\BdRp / \xi^r)\Talg{\overline{P}_{\bQ_{\geq 0}}}$.  By \Refeq{\ref{eq-gr-BdR}}, there is a canonical isomorphism
\begin{equation}\label{eq-P-to-Ainf}
    (\BdRp / \xi^r)\Talg{\overline{P}_{\bQ_{\geq 0}}} \Mi \BBdRp(\widehat{k}_\infty\Talg{\overline{P}_{\bQ_{\geq 0}}}, \widehat{k}_\infty^+\Talg{\overline{P}_{\bQ_{\geq 0}}}) / \xi^r
\end{equation}
sending $\mono{\overline{a}}$ to $[\overline{T}^{\overline{a}\flat}]$, for all $\overline{a} \in \overline{P}_{\bQ_{\geq 0}}$.  Then \Refeq{\ref{eq-P-to-Ainf}} is $\Gamma$-equivariant if we equip $(\BdRp / \xi^r)\Talg{\overline{P}_{\bQ_{\geq 0}}}$ with the action of $\Gamma$ defined by
\begin{equation}\label{eq-Gamma-act-mono-alg}
    \gamma(\mono{\overline{a}}) = [(\gamma(\overline{a}), \gamma(\tfrac{1}{p} \overline{a}), \ldots)] \, \mono{\overline{a}},
\end{equation}
for all $\gamma \in \Gamma$ and $\overline{a} \in \overline{P}_{\bQ_{\geq 0}}$, which reduces modulo $\xi$ to the action \Refeq{\ref{eq-act-coord}}.

Now suppose that $X = \Spa(A, A^+)$ is an affinoid fs log adic space with a strictly \'etale morphism $X \to \bE$.  Let $\widetilde{X}$ be the pullback of $\widetilde{\bE}$ under $X \to \bE$.  Then $\widetilde{X} \in X_\proket$ is also log affinoid perfectoid with $\widehat{\widetilde{X}} = \Spa(\widehat{A}_\infty, \widehat{A}_\infty^+)$ the associated perfectoid space, and $\widetilde{X} \to X_{\widehat{k}_\infty}$ is a Galois cover with Galois group $\Gamma$.

Consider the sheaf of monoid algebras $\BBdRp|_{\widetilde{X}}[P]$.  Let $\mathfrak{M} \subset \BBdRp|_{\widetilde{X}}[P]$ denote the sheaf of ideals generated by $\{ \mono{a} - 1 \}_{a \in P}$, and let
\[
    \BBdRp|_{\widetilde{X}}[[P - 1]] := \varprojlim_r \bigl( \BBdRp|_{\widetilde{X}}[P] / \mathfrak{M}^r \bigr).
\]
Note that $\mono{a} \in 1 + \mathfrak{M} \subset \bigl( \BBdRp|_{\widetilde{X}}[[P - 1]] \bigr)^\times$, for all $a \in P$.  Therefore, we have a monoid homomorphism $P_{\bQ_{\geq 0}} \to \bigl( \BBdRp|_{\widetilde{X}}[[P - 1]] \bigr)^\times \subset \BBdRp|_{\widetilde{X}}[[P - 1]]$ \Pth{with respect to the multiplicative structure on $\BBdRp|_{\widetilde{X}}[[P - 1]]$} defined by sending $\frac{1}{m} a$, where $m \in \bZ_{\geq 1}$ and $a \in P$, to the formal power series of $\bigl(1 + (\mono{a} - 1) \bigr)^{\frac{1}{m}}$, which we abusively still denote by $\mono{\frac{1}{m} a}$.  On the other hand, consider the monoid homomorphism
\[
    P \to \BBdRp|_{\widetilde{X}}[[P - 1]]: \; a \mapsto \log(\mono{a}) := \sum_{l = 1}^\infty (-1)^{l - 1} \tfrac{1}{l} (\mono{a} - 1)^l
\]
\Pth{with respect to the additive structure on $\BBdRp|_{\widetilde{X}}[[P - 1]]$}, which uniquely extends to a group homomorphism $P^\gp \to \BBdRp|_{\widetilde{X}}[[P - 1]]: \; a \mapsto \monon{a}$ such that
\[
    \monon{a} = \log(\mono{a^+}) - \log(\mono{a^-})
\]
when $a = a^+ - a^-$ for some $a^+, a^- \in P$.  Then the above homomorphism further extends linearly to a $\bQ$-vector space homomorphism $P^\gp_\bQ \to \BBdRp|_{\widetilde{X}}[[P - 1]]: \; a \mapsto \monon{a}$.  Since $\monon{a} - (\mono{a} - 1) \in \mathfrak{M}^2$ for all $a \in P$, if we choose a $\bZ$-basis $\{ a_1, \ldots, a_n \}$ of $P^\gp$, and write $\monon{j} = \monon{a_j}$, for each $j = 1, \ldots, n$, then we have a canonical isomorphism
\begin{equation}\label{eq-BBdRp-pow}
    \BBdRp|_{\widetilde{X}}[[\monon{1}, \ldots, \monon{n}]] \Mi \BBdRp|_{\widetilde{X}}[[P - 1]]: \; \monon{j} \mapsto \monon{j}
\end{equation}
of $\BBdRp|_{\widetilde{X}}$-algebras, matching the ideals $(\monon{1}, \ldots, \monon{n})^r$ and $(\xi, \monon{1}, \ldots, \monon{n})^r$ of the source with the ideals $\mathfrak{M}^r$ and $(\xi, \mathfrak{M})^r$ of the target, respectively, for all $r \in \bZ_{\geq 0}$.  We similarly define $\BBdRp|_{\widetilde{X}}[[\overline{P} - 1]]$ and $\BBdRp|_{\widetilde{X}}[[\overline{P} - 1]][[Q - 1]]$, and the decomposition $P = \overline{P} \oplus Q$ induces a canonical isomorphism $\BBdRp|_{\widetilde{X}}[[P - 1]] \cong \BBdRp|_{\widetilde{X}}[[\overline{P} - 1]][[Q - 1]]$.

\begin{lemma}\label{lem-O-str-BdR}
    There is a unique morphism of sheaves
    \begin{equation}\label{eq-lem-O-str-BdR}
        v: \cO_{X_\proket}|_{\widetilde{X}} \to \BBdRp|_{\widetilde{X}}[[\overline{P} - 1]]
    \end{equation}
    satisfying the following conditions:
    \begin{enumerate}
        \item The composition $\cO_{X_\proket}|_{\widetilde{X}} \xrightarrow{v} \BBdRp|_{\widetilde{X}}[[\overline{P} - 1]] \xrightarrow{\mono{\overline{a}} \mapsto 1, \, \xi \mapsto 0} \widehat{\cO}_{X_\proket}|_{\widetilde{X}}$ is the natural map.

        \item The composition $P_{\bQ_{\geq 0}, X_\proket}|_{\widetilde{X}} \to \cO_{X_\proket}|_{\widetilde{X}} \xrightarrow{v} \BBdRp|_{\widetilde{X}}[[\overline{P} - 1]]$ is induced by $a \mapsto [\overline{T}^{a\flat}] \, \mono{\overline{a}}$, where $P_{\bQ_{\geq 0}, X_\proket}$ denotes the constant sheaf of monoids on $X_\proket$ associated with $P_{\bQ_{\geq 0}}$, and where $\overline{a}$ denotes the image of $a$ in $\overline{P}$.
    \end{enumerate}
\end{lemma}
\begin{proof}
    Consider any log affinoid perfectoid object $U = \varprojlim_{i \in I} U_i \in {X_\proket}_{/\widetilde{X}}$, as in \cite[\aDef \logadicdeflogaffperf]{Diao/Lan/Liu/Zhu:lasfr}, with $U_i = (\Spa(R_i, R_i^+), \cM_i, \alpha_i)$.  We would like to construct a compatible family of maps
    \begin{equation}\label{eq-lem-O-str-BdR-lifting}
        v_i: R_i \to \BBdRp(U)[[\overline{P} - 1]],
    \end{equation}
    indexed by $i \in I$.  For each $i \in I$, by \cite[\aLem \logadiclemAbhyankarbasic]{Diao/Lan/Liu/Zhu:lasfr}, $U_i \times_\bE \bE_{m_i} \to \bE_{m_i}$ is strictly \'etale for some $m_i$.  Given $U \to \widetilde{X}$, some $U_{i'} \to U_i$ factors through $U_{i'} \to U_i \times_\bE \bE_{m_i}$, so that $R_i \to R_{i'}$ factors as $R_i \to B := \cO(U_i \times_\bE \bE_{m_i}) \to R_{i'}$.  We have a structure homomorphism $k_{m_i}[\frac{1}{m_i} P] / (Q - \{0\}) \to B$ induced by the second projection $U_i \times_\bE \bE_{m_i} \to \bE_{m_i}$, and a canonical $k_{m_i}$-algebra homomorphism $k_{m_i}[\frac{1}{m_i} P] / (Q - \{0\}) \to \BBdRp(U)[[\overline{P} - 1]]$ sending the image of $\mono{a}$ to $[\overline{T}^{a\flat}] \, \mono{\overline{a}}$, for all $a \in \frac{1}{m_i} P$, which fit into the following commutative diagram of solid arrows
    \[
        \xymatrix{ & {k_{m_i}[\frac{1}{m_i} P] / (Q - \{0\})} \ar[rr] \ar[d] & & {\BBdRp(U)[[\overline{P} - 1]]} \ar[d]^-{\mono{a} \mapsto 1, \, \xi \mapsto 0} \\
        {R_i} \ar[r] & {B} \ar[r] \ar@{.>}[urr] & {R_{i'}} \ar[r] & {\widehat{\cO}_{X_\proket}(U).} }
    \]
    By \cite[\aCor 1.7.3(iii)]{Huber:1996-ERA}, there is a finitely generated $k^+[\frac{1}{m_i} P] / (Q - \{0\})$-algebra $B_0^+$ such that $B_0 := B_0^+[\frac{1}{p}]$ is \'etale over $k[\frac{1}{m_i} P] / (Q - \{0\})$ and such that $B$ is the $p$-adic completion of $B_0$.  Then it follows from \cite[\aLem 6.11]{Scholze:2013-phtra} that there is a unique continuous lifting $B \to \BBdRp(U)[[\overline{P} - 1]]$, denoted by the dotted arrow above, making the whole diagram commute.  Then the composition of $R_i \to B\to  \BBdRp(U)[[\overline{P} - 1]]$ gives the desired \Refeq{\ref{eq-lem-O-str-BdR-lifting}}.  It is clear that such $v_i$'s, for all $i \in I$, are independent of the choices, compatible with each other, and define the desired \Refeq{\ref{eq-lem-O-str-BdR}}.
\end{proof}

The following lemma generalizes the isomorphism \Refeq{\ref{eq-P-to-Ainf}} for $X = \Spa(A, A^+)$.  Note that \Refeq{\ref{eq-lem-O-str-BdR}} induces a map $A \to (\BBdRp(\widetilde{X}) / \xi^r)[[\overline{P} - 1]] \xrightarrow{\mono{\overline{a}} \mapsto 1} \BBdRp(\widetilde{X}) / \xi^r$.  Together with \Refeq{\ref{eq-P-to-Ainf}}, it induces a canonical map
\begin{equation}\label{eq-str-BdR}
    \bigl(A \ho_k (\BdRp / \xi^r)\bigr) \ho_{(\BdRp / \xi^r)\Talg{\overline{P}}} (\BdRp / \xi^r)\Talg{\overline{P}_{\bQ_{\geq 0}}} \to \BBdRp(\widetilde{X}) / \xi^r.
\end{equation}
\begin{lemma}\label{lem-str-BdR}
    The map \Refeq{\ref{eq-str-BdR}} is an isomorphism.  Furthermore, the $\Gamma$-action on $(\BdRp / \xi^r)\Talg{\overline{P}}$ defined in \Refeq{\ref{eq-Gamma-act-mono-alg}} uniquely extends to a continuous $\Gamma$-action on $A \ho_k (\BdRp / \xi^r)$, which is trivial modulo $\xi$ and makes \Refeq{\ref{eq-str-BdR}} $\Gamma$-equivariant.
\end{lemma}
\begin{proof}
    Since \Refeq{\ref{eq-str-BdR}} is compatible with the filtrations induced by multiplication by the powers of $\xi$, by considering the associated graded pieces, it suffices to show that $(A \ho_k \widehat{k}_\infty) \ho_{\widehat{k}_\infty\Talg{\overline{P}}} \widehat{k}_\infty\Talg{\overline{P}_{\bQ_{\geq 0}}} \cong \widehat{A}_\infty$.  Then the same argument as in the proof of \cite[\aLem 6.18]{Scholze:2013-phtra} applies, with its input \cite[\aLem 4.5]{Scholze:2013-phtra} for the tower $\widetilde{\bT}^n \to \bT^n$ replaced with \cite[\aLem \logadiclemScholzeref]{Diao/Lan/Liu/Zhu:lasfr} for the tower $\varprojlim_m \, \Spa(k_m\Talg{\frac{1}{m} \overline{P}}, k_m^+\Talg{\frac{1}{m} \overline{P}} \to \bE$.

    The unique existence of a continuous $\Gamma$-action on $A \ho_k (\BdRp / \xi^r)$ extending the $\Gamma$-action on $(\BdRp / \xi^r)\Talg{\overline{P}}$ follows from an argument similar to the one in the proof of Lemma \ref{lem-O-str-BdR}.  Concretely, since the map $k\Talg{\overline{P}} \to A$ arises as the $p$-adic completion of an \'etale morphism $k[\overline{P}] \to A_0$ of finite type, the $\Gamma$-action on $(\BdRp / \xi^r)[\overline{P}]$ uniquely extends to a continuous $\Gamma$-action on $A_0 \otimes_k (\BdRp / \xi^r)$, and further uniquely extends to a continuous $\Gamma$-action after $p$-adic completion, with desired properties.
\end{proof}

\begin{lemma}\label{lem-M-to-BdRp-pow}
    There exists a map of sheaves of monoids
    \begin{equation}\label{eq-lem-M-to-BdRp-pow}
        \beta: \cM^\flat|_{\widetilde X} \to (\BBdRp|_{\widetilde{X}}[[P - 1]])^\times
    \end{equation}
    satisfying the following conditions:
    \begin{enumerate}
        \item For all $a \in \cM^\flat$, we have $v\bigl(\alpha(a^\sharp)\bigr) = [\alpha^\flat(a)] \, \beta(a)$, where we denote by $a \mapsto a^\sharp$ the natural projection $\cM^\flat \to \cM$.

        \item The composition of $\beta$ with the canonical map $\BBdRp|_{\widetilde{X}}[[P - 1]] \xrightarrow{\mono{a} \mapsto 1, \, \xi \mapsto 0} \widehat{\cO}_{X_\proket}|_{\widetilde{X}}$ is the constant $1$.

        \item The restriction of $\beta$ to $Q_{\bQ_{\geq 0}}$ \Pth{as a constant sheaf} is given by $a \mapsto \mono{a}$.
    \end{enumerate}
\end{lemma}
\begin{proof}
    The sheaf $\cM^\flat|_{\widetilde X}$ is generated by $P_{\bQ_{\geq 0}} \cong \overline{P}_{\bQ_{\geq 0}} \oplus Q_{\bQ_{\geq 0}}$ \Pth{as constant sheaves} and $\varprojlim\limits_{f \mapsto f^p} \cO_{X_\proket}^\times|_{\widetilde{X}}$.  We need to define the map for $\overline{a} \in \overline{P}_{\bQ_{\geq 0}}$.  If we write \Pth{locally} $\alpha^\flat(\overline{a}) = h \, \overline{T}^{\overline{a}_0 \flat}$, for some section $h$ of $\varprojlim\limits_{f \mapsto f^p} \cO_{X_\proket}^\times|_{\widetilde{X}} \subset \widehat{\cO}_{X^\flat_\proket}^\times$ and $\overline{a}_0 \in \overline{P}_{\bQ_{\geq 0}}$, then $\alpha(\overline{a}^\sharp) = \bigl(\alpha^\flat(\overline{a})\bigr)^\sharp = h^\sharp \, \overline{T}^{\overline{a}_0}$, and the conditions of the lemma are satisfied by the local section $\beta(\overline{a}) := \frac{v(h^\sharp)}{[h]} \mono{\overline{a}_0}$ of $(\BBdRp[[P - 1]])^\times|_{\widetilde{X}}$.  This expression is independent of the local choices, and hence globalizes and defines the desired map $\beta$.
\end{proof}

Next, we give an explicit description of $\OBdlp$ on the localized site ${X_\proket}_{/\widetilde{X}}$.  Let $U = \varprojlim_{i \in I} U_i \in {X_\proket}_{/\widetilde{X}}$ be a log affinoid perfectoid object, with $U_i = (\Spa(R_i, R_i^+), \cM_i, \alpha_i)$, as in the proof of Lemma \ref{lem-O-str-BdR}.  Let $S_{i, r}$ be as in \Refeq{\ref{def-S-i-r}}.  Note that, for $a = (a', a'')$ with $a'' \in Q_{\bQ_{\geq 0}} - \{0\}$, we have $\alpha^\flat(a'') = 0$ and hence the relation $\alpha_i(a') = [\alpha^\flat(a'')] \, \monob{a}$ reduces to simply $\alpha_i(a') = 0$ in $S_{i, r}$, with no constraint on $\monob{a}$, in which case we can view $\monob{a}$ as a free variable.  Consider the map
\[
    (\BBdRp(U) / \xi^r)[P] \to S_{i, r}: \; \mono{a} \mapsto \monob{(a, a)}, \; \Utext{for all $a \in P$},
\]
which sends $(\xi, \mathfrak{M})$ to $\ker \theta_{\log}$, and therefore induces a map $\BBdRp(U)[[P - 1]] \to \widehat{S}_i$.  By taking completion and sheafification, we obtain a map
\begin{equation}\label{eq-OBdlp-loc}
    \BBdRp|_{\widetilde{X}}[[P - 1]] \to \OBdlp|_{\widetilde{X}}
\end{equation}
on ${X_\proket}_{/\widetilde{X}}$.  If we define $\Fil^r\BBdRp|_{\widetilde{X}}[[P - 1]] := (\xi, \mathfrak{M})^r\BBdRp|_{\widetilde{X}}[[P - 1]]$, then \Refeq{\ref{eq-OBdlp-loc}} is compatible with the filtrations on both sides \Pth{see Definition \ref{def-OBdl}}.

The following is a log analogue of \cite[\aProp 5.2.2]{Brinon:2008-RPR} and \cite[\aProp 6.10]{Scholze:2013-phtra}, formulated in terms of charts and monoids:
\begin{prop}\label{prop-OBdlp-loc}
    The map \Refeq{\ref{eq-OBdlp-loc}} is an isomorphism of filtered sheaves.
\end{prop}
\begin{proof}
    Let $U = \varprojlim_{i \in I} U_i \in {X_\proket}_{/\widetilde{X}}$ be a log affinoid perfectoid object, with $U_i = (\Spa(R_i, R_i^+), \cM_i, \alpha_i)$, as in Lemma \ref{lem-O-str-BdR}.  For $i \in I$ and $r \geq 1$, the map \Refeq{\ref{eq-lem-O-str-BdR}} induces a natural map $R_i \to (\BBdRp(U) / \xi^r)[[\overline{P} - 1]]$.  Together with the map \Refeq{\ref{eq-lem-M-to-BdRp-pow}}, these maps induce a ring homomorphism
    \begin{equation}\label{eq-prop-OBdlp-loc-mono-alg}
        \bigl(R_i \ho_{W(\kappa)} (\BBdRp(U) / \xi^r)\bigr)[M_i \times_M M^\flat] \to (\BBdRp(U) / \xi^r)[[P - 1]],
    \end{equation}
    sending $\monob{a}$ to $\beta(a'')$, for all $a = (a', a'') \in M_i \times_M M^\flat$.  By Lemmas \ref{lem-O-str-BdR} and \ref{lem-M-to-BdRp-pow}, this map factors through $S_{i, r} \to (\BBdRp(U) / \xi^r)[[P - 1]]$, and its composition with $(\BBdRp(U) / \xi^r)[[P - 1]] \xrightarrow{\mono{a} \mapsto 1, \, \xi \mapsto 0} \widehat{R}_\infty$ is the map $\theta_{\log}$, where $S_{i, r}$ and $\theta_{\log}$ are as in \Refeq{\ref{def-S-i-r}} and \Refeq{\ref{eq-theta-log}}, respectively.  Therefore, this map sends $\ker \theta_{\log}$ to $(\xi, \mathfrak{M})$.  By taking $\ker \theta_{\log}$-adic completion, inverse limit over $r \geq 1$, and direct limit over $i \in I$, we obtain a map $\OBdlp(U) \to \BBdRp(U)[[P - 1]]$, whose pre-composition with the map $\BBdRp(U)[[P - 1]] \to \OBdlp(U)$ given by \Refeq{\ref{eq-OBdlp-loc}} is the identity map, because it is $\BBdRp(U)$-linear and sends $\mono{a}$ to $\mono{a}$, for all $a \in P$.  On the other hand, the post-composition of \Refeq{\ref{eq-lem-O-str-BdR-lifting}} with $\BBdRp(U)[[P - 1]] \to \OBdlp(U)$ is the natural map $R_i \to \OBdlp(U)$, because $k[P_{\bQ_{\geq 0}}] \to \BBdRp(U)[[\overline{P} - 1]] \to \widehat{S}_i$ sends $\mono{a} \mapsto [\overline{T}^{a\flat}] \, \mono{\overline{a}} \mapsto [\overline{T}^{a\flat}] \, \monob{(\overline{a}, \overline{a})} = \overline{T}^a$ and the map $k[\frac{1}{m_i} P] / (Q - \{0\}) \to B_0$ in the proof of Lemma \ref{lem-O-str-BdR} is \'etale.  Consequently, the map $S_{i, r} \to \widehat{S}_i / \xi^r$ induced by the composition of \Refeq{\ref{eq-prop-OBdlp-loc-mono-alg}} with $(\BBdRp(U) / \xi^r)[[P - 1]] \to \widehat{S}_i / \xi^r$ coincides with the natural map, because both maps send the image of $\monob{a}$ in $S_{i, r}$ to the same further image in $\widehat{S}_i / \xi^r$, for all $a \in M_i \times_M M^\flat$.  Thus, the composition of $\OBdlp(U) \to \BBdRp(U)[[P - 1]] \to \OBdlp(U)$ is also the identity map, as desired.
\end{proof}

As a byproduct of the proof, we see that, as in \cite{Scholze:2016-phtra-corr}, for $U = \varprojlim_{i \in I} U_i$ as above, the natural map $\widehat{S}_i \to \OBdlp(U)$ \Pth{as in Definition \ref{def-OBdl}} is already an isomorphism, for each $i$.

Now, as before, let us fix a $\bZ$-basis $\{ a_1, \ldots, a_n \}$ of $P^\gp$, and write $\monon{j} = \monon{a_j}$, for each $j = 1, \ldots, n$, so that we have $\BBdRp|_{\widetilde{X}}[[P - 1]] \cong \BBdRp|_{\widetilde{X}}[[\monon{1}, \ldots, \monon{n}]]$ as in \Refeq{\ref{eq-BBdRp-pow}}.

\begin{cor}\label{cor-OBdlp-loc-gr}
    The isomorphism \Refeq{\ref{eq-OBdlp-loc}} induces isomorphisms
    \[
    \begin{split}
        & \Fil^r \OBdl  \cong t^r \BBdRp\{W_1, \ldots, W_n\} \\
        & := \Bigl\{ t^r \sum_{\Lambda \in \bZ_{\geq 0}^n} b_\Lambda \, W^\Lambda \in \BBdRp[[W_1, \ldots, W_n]] : \Utext{$b_\Lambda \rightarrow 0$, $t$-adically, as $|\Lambda| \rightarrow \infty$} \Bigr\}
    \end{split}
    \]
    over ${X_\proket}_{/\widetilde{X}}$, for all $r \in \bZ$, where we have the variable
    \begin{equation}\label{eq-def-W-i}
        W_j := t^{-1} \monon{j},
    \end{equation}
    for each $1 \leq j \leq n$, and the monomial
    \begin{equation}\label{eq-def-W-I}
        W^\Lambda := W_1^{\Lambda_1} \cdots W_n^{\Lambda_n},
    \end{equation}
    with $|\Lambda| := |\Lambda_1| + \cdots + |\Lambda_n| = \Lambda_1 + \cdots + \Lambda_n$, for each exponent $\Lambda = (\Lambda_1, \ldots, \Lambda_n) \in \bZ_{\geq 0}^n$.  \Pth{Here we denote by $\{W_1, \ldots, W_n\}$ the ring of power series that are $t$-adically convergent, which is similar to the notation $\Talg{W_1, \ldots, W_n}$ for the ring of power series that are $p$-adically convergent.}  Thus, $\gr^r \OBdl \cong t^r \widehat{\cO}_{X_\proket}[W_1, \ldots, W_n]$, for all $r \in \bZ$, and $\gr^\bullet \OBdl \cong \widehat{\cO}_{X_\proket}[t^\pm, W_1, \ldots, W_n]$.
\end{cor}

By comparing the constructions, we obtain the following:
\begin{cor}\label{cor-OBdRp-loc-cl-imm}
    Suppose that $\imath: Z \to X$ is a strict closed immersion of log adic spaces such that the underlying morphism of adic spaces is the pullback of the closed immersion $\Spa(k\Talg{\overline{P} / Q'}, k^+\Talg{\overline{P} / Q'}) \Em \bE = \Spa(k\Talg{\overline{P}}, k^+\Talg{\overline{P}})$ induced by $k\Talg{\overline{P}} \Surj k\Talg{\overline{P}} / (Q' - \{0\}) \cong k\Talg{\overline{P} / Q'}$, for some direct summand $Q'$ of $\overline{P}$.  Let $\widetilde{Z} := \widetilde{X} \times_X Z$.  Suppose that $U \in {X_\proket}_{/\widetilde{X}}$ is log affinoid perfectoid, whose pullback $V: = U \times_X Z$ is log affinoid perfectoid in ${Z_\proket}_{/\widetilde{Z}}$.  Then, for each $r \geq 1$, the canonical surjection $\BBdRpX{X}(U) / \xi^r \Surj \BBdRpX{Z}(V) / \xi^r$ \Pth{\Refcf{} Corollary \ref{cor-BdR-cl-imm}} induces
    \begin{equation}\label{eq-cor-OBdRp-loc-cl-imm-BdR-mod-xi-r}
        \bigl(\BBdRpX{X}(U) \big/ \xi^r\bigr) \big/ ([T^{sa\flat}])^\wedge_{s \in \bQ_{> 0}, a \in Q' - \{0\}} \Mi \bigl(\BBdRpX{Z}(V) \big/ \xi^r\bigr),
    \end{equation}
    where $([T^{sa\flat}])^\wedge_{s \in \bQ_{> 0}, a \in Q' - \{0\}}$ denotes the $p$-adic completion of the ideal generated by $\{ [T^{sa\flat}] \}_{s \in \bQ_{> 0}, a \in Q' - \{0\}}$.  In addition, the canonical isomorphisms $\BBdRpX{Z}|_{\widetilde{Z}}[[P - 1]] \cong \OBdlpX{Z}|_{\widetilde{Z}}$ and $\BBdRpX{X}|_{\widetilde{X}}[[P - 1]] \cong \OBdlpX{X}|_{\widetilde{X}}$ given by Proposition \ref{prop-OBdlp-loc} are compatible with each other via pullback and pushforward.
\end{cor}

\subsection{Consequences}\label{sec-OBdl-conseq}

\begin{rk}\label{rem-OBdl-conseq-setting}
    Let $k$ be a $p$-adic field.  We may apply the calculations in Section \ref{sec-OBdl-explicit} in the following two cases:
    \begin{enumerate}
        \item  When $X$ is log smooth over $k$, \'etale locally there are \emph{toric charts} $X \to \bE = \Spa(k\Talg{P}, k^+\Talg{P})$ \Pth{with $Q = 0$}, as in \cite[\aDef \logadicdeftoricchart]{Diao/Lan/Liu/Zhu:lasfr}.

        \item  Let $Y$ be smooth over $k$, with log structure defined by a normal crossings divisor $E \Em Y$ as in Example \ref{ex-log-adic-sp-ncd}, and let $X$ be a smooth intersection of irreducible components of $E$, equipped with the log structure pulled back from $Y$, as in \cite[\aEx \logadicexlogadicspncdstrictclimm]{Diao/Lan/Liu/Zhu:lasfr}.  Then, \'etale locally, there is a toric chart of $Y$ as above inducing a strictly \'etale morphism $X \to \bE = \Spa(k\Talg{P / Q}, k^+\Talg{P /Q})$ \Pth{for some direct summand $Q$ of $P$}.
    \end{enumerate}
    In both cases, the sheaves of log differentials $\Omega^{\log}_X$ and $\Omega^{\log, \bullet}_X = \Ex^\bullet \Omega^{\log}_X$ are defined as in \cite[\aDefs \logadicdeflogdiffsheaf{} and \logadicdeflogdiffsheafex]{Diao/Lan/Liu/Zhu:lasfr}, and are known to be vector bundles on $X$, by \cite[\aThm \logadicthmlogdiffsheaffund{} and \aCor \logadiccorlogdiffsheafstrictclimm]{Diao/Lan/Liu/Zhu:lasfr}.  \Pth{As before, by abuse of notation, their pullbacks to $X_\et$, $X_\ket$, and $X_\proket$ will still be denoted by the same symbols.}
\end{rk}

In particular, we have the following \emph{Poincar\'e lemma} for $\OBdlp$ and $\OBdl$, with the log connections defined at the end of Section \ref{sec-OBdl}:
\begin{cor}\label{cor-log-dR-cplx}
    Let $X$ be as in Remark \ref{rem-OBdl-conseq-setting}.
    \begin{enumerate}
        \item\label{cor-log-dR-cplx-1}  We have an exact complex $0 \to \BBdRp \to \OBdlp \Mapn{\nabla} \OBdlp \otimes \Omega^{\log, 1}_X \Mapn{\nabla} \OBdlp \otimes \Omega^{\log, 2}_X \to \cdots$.

        \item\label{cor-log-dR-cplx-2}  The above statement holds with $\BBdRp$ and $\OBdlp$ replaced with $\BBdR$ and $\OBdl$, respectively.

        \item\label{cor-log-dR-cplx-3}  The subcomplex $0 \to \Fil^r \BBdR \to \Fil^r \OBdl \Mapn{\nabla} (\Fil^{r - 1} \OBdl) \otimes \Omega^{\log, 1}_X \Mapn{\nabla} (\Fil^{r - 2} \OBdl) \otimes \Omega^{\log, 2}_X \cdots$ of the complex for $\BBdR$ and $\OBdl$ is also exact, for each $r \in \bZ$.

        \item\label{cor-log-dR-cplx-4}  For each $r \in \bZ$, the quotient complex $0 \to \gr^r \BBdR \to \gr^r \OBdl \Mapn{\nabla} (\gr^{r - 1} \OBdl) \otimes \Omega^{\log, 1}_X \Mapn{\nabla} (\gr^{r - 2} \OBdl) \otimes \Omega^{\log, 2}_X \cdots$ of the previous complex is exact, and can be identified with the complex $0 \to \cO_{X_\proket}(r) \to \OCl(r) \Mapn{\nabla} ( \OCl(r) ) \otimes \Omega^{\log, 1}_X(-1) \Mapn{\nabla} ( \OCl(r) ) \otimes \Omega^{\log, 2}_X(-2) \cdots$.
    \end{enumerate}
    \Pth{All the above tensor products are over $\cO_{X_\proket}$, which we omitted for simplicity.}
\end{cor}
\begin{proof}
    In both cases of Remark \ref{rem-OBdl-conseq-setting}, up to \'etale localization on $X$, we may assume that there exists a strictly \'etale morphism $X \to \bE = \Spa(k\Talg{\overline{P}}, k^+\Talg{\overline{P}})$, and then pass to $\widetilde{X}$ pro-Kummer \'etale locally, as in Section \ref{sec-OBdl-explicit}.  Choose a $\bZ$-basis $\{ a_1, \ldots, a_n \}$ of $P^\gp$, and write $a_j = a_j^+ - a_j^-$ for some $a_j^+, a_j^- \in P$, for each $j = 1, \ldots, n$.  By Proposition \ref{prop-OBdlp-loc} and Corollary \ref{cor-OBdlp-loc-gr}, it suffices to prove the exactness of the complexes by using $\BBdRp|_{\widetilde{X}}[[\monon{1}, \ldots, \monon{n}]]$ and $\BBdR|_{\widetilde{X}}\{W_1, \ldots, W_n\}$ in place of $\OBdlp$ and $\OBdl$, respectively.  Note that the isomorphism \Refeq{\ref{eq-OBdlp-loc}} matches $\monon{j}$ with $\log(\monob{(a_j^+, a_j^+)}) - \log(\monob{(a_j^-, a_j^-)})$.  By \cite[\aThm \logadicthmlogdiffsheaffund, \aCor \logadiccorlogdiffsheafstrictclimm, \aProp \logadicproplogdiffmonoid, and \aCor \logadiccorlogdiffmonoidstrictclimm]{Diao/Lan/Liu/Zhu:lasfr}, $\Omega^{\log}_X = \oplus_{j = 1}^n \bigl( \cO_X \, \delta(a_j) \bigr)$, and hence \Pth{because of \Refeq{\ref{eq-conn-e-a}}}
    \begin{equation}\label{eq-conn-X-i}
        \nabla(\monon{j}) = \nabla\bigl(\log(\monob{(a_j^+, a_j^+)})\bigr) - \nabla\bigl(\log(\monob{(a_j^-, a_j^-)})\bigr) = \delta(a_j^+) - \delta(a_j^-) = \delta(a_j)
    \end{equation}
    and \Pth{because of \Refeq{\ref{eq-def-W-i}}}
    \begin{equation}\label{eq-conn-W-i}
        \nabla(W_j) = t^{-1} \delta(a_j).
    \end{equation}
    The exactness then follows from a straightforward calculation.  \Pth{Note that the $t$-adic convergence condition on power series is not affected by taking anti-derivatives.}
\end{proof}

By combining Corollaries \ref{cor-gr-BdR} and \ref{cor-log-dR-cplx}, we obtain the \emph{log Faltings's extension}:
\begin{cor}\label{cor-log-Falt-ext}
    We have a short exact sequence of sheaves of $\widehat{\cO}_{X_\proket}$-modules
    \[
        0 \to \widehat{\cO}_{X_\proket}(1) \to \gr^1 \OBdlp \to \widehat{\cO}_{X_\proket} \otimes_{\cO_{X_\proket}} \Omega^{\log}_X \to 0.
    \]
\end{cor}

Finally, suppose that $X$ and $X'$ are both as in Remark \ref{rem-OBdl-conseq-setting}, and that $f: X \to X'$ is a log smooth morphism.  Then we have a canonical short exact sequence $0 \to f^*(\Omega^{\log}_{X'}) \to \Omega^{\log}_X \to \Omega^{\log}_{X / X'} \to 0$ of vector bundles on $X$, by \cite[\aThm \logadicthmlogdiffsheaffund{} and \aCor \logadiccorlogdiffsheafstrictclimm]{Diao/Lan/Liu/Zhu:lasfr}; and we shall write $\Omega^{\log, \bullet}_{X / Y} = \Ex^\bullet \Omega^{\log}_{X / Y}$, as in \cite[\aDef \logadicdeflogdiffsheafex]{Diao/Lan/Liu/Zhu:lasfr}.  In this case, the log de Rham complex $(\Omega^{\log, \bullet}_X, \nabla)$ induces the \emph{relative log de Rham complex} $(\Omega^{\log, \bullet}_{X / X'}, \nabla)$, and we have the following \emph{relative Poincar\'e lemma}:
\begin{cor}\label{cor-log-dR-cplx-rel}
    With $f: X \to X'$ as above, we have an exact complex $0 \to \BBdRpX{X} \otimes_{f_\proket^{-1}(\BBdRpX{X'})} f_\proket^{-1}(\OBdlpX{X'}) \to \OBdlpX{X} \Mapn{\nabla} \OBdlpX{X} \otimes \Omega^{\log, 1}_{X / X'} \Mapn{\nabla} \OBdlpX{X} \otimes \Omega^{\log, 2}_{X / X'} \to \cdots$.  Similarly, we have an exact complex with $\BBdRpX{X}$, $\BBdRpX{X'}$, $\OBdlpX{X}$, and $\OBdlpX{X'}$ replaced with $\BBdRX{X}$, $\BBdRX{X'}$, $\OBdlX{X}$, and $\OBdlX{X'}$, respectively, which is strictly compatible with all the filtrations.
\end{cor}
\begin{proof}
    By \cite[\aProps \logadicproplogsmchart{} and \logadicproptoricchart]{Diao/Lan/Liu/Zhu:lasfr}, up to \'etale localization on $X$ and $X'$, we may assume that $f: X \to X'$ admits an injective sharp fs chart $P' \Em P$.  Then we can compatibly define log affinoid perfectoid objects $\widetilde{X} \to X$ and $\widetilde{X}' \to X'$, as in Section \ref{sec-OBdl-explicit}, with a morphism $\widetilde{X} \to \widetilde{X}'$ lifting $f: X \to X'$.  Let $\{ a_1, \ldots, a_n \}$ and $\{ a'_1, \ldots, a'_{n'} \}$ be $\bZ$-bases of $P^\gp$ and $(P')^\gp$, respectively, and define $\{ \monon{1}, \ldots, \monon{n} \}$, $\{ \monon{1}', \ldots, \monon{n'}' \}$, $\{ W_1, \ldots, W_n \}$, and $\{ W'_1, \ldots, W'_{n'} \}$, as in the proof of Corollary \ref{cor-log-dR-cplx}.  Hence, it suffices to prove the exactness of the complex $0 \to \BBdRpX{X}|_{\widetilde{X}}[[\monon{1}', \ldots, \monon{n}']] \to \BBdRpX{X}|_{\widetilde{X}}[[\monon{1}, \ldots, \monon{n}]] \to \BBdRpX{X}|_{\widetilde{X}}[[\monon{1}, \ldots, \monon{n}]] \otimes \Omega^{\log, 1}_{X / X'}|_{\widetilde{X}} \to \cdots$ and the analogous one for $\BBdRX{X}|_{\widetilde{X}}\{W'_1, \ldots, W'_{n'}\}$ and $\BBdRX{X}|_{\widetilde{X}}\{W_1, \ldots, W_n\} \otimes \Omega^{\log, \bullet}_{X / X'}|_{\widetilde{X}}$.  Since $(P')^\gp$ and $P^\gp$ are both finitely generated free abelian groups, $(P')^\gp_\bQ$ is noncanonically a direct factor of $P^\gp_\bQ$.  Hence, there exist elements $a'_{n' + 1}, \ldots, a'_n$ of $P$ whose images in $P^\gp_\bQ / (P')^\gp_\bQ$ form a $\bQ$-basis.  By \cite[\aThm \logadicthmlogdiffsheaffund, \aCor \logadiccorlogdiffsheafstrictclimm, \aProp \logadicproplogdiffmonoid, and \aCor \logadiccorlogdiffmonoidstrictclimm]{Diao/Lan/Liu/Zhu:lasfr}, we have $\Omega^{\log}_X = \oplus_{j = 1}^n \bigl( \cO_X \, \delta(a_j) \bigr)$, $f^*(\Omega^{\log}_{X'}) = \oplus_{j = 1}^{n'} \bigl( \cO_X \, \delta(a'_j) \bigr)$, and $\Omega^{\log}_{X / X'} = \oplus_{j = n'+1}^n \bigl( \cO_X \, \delta(a'_j) \bigr)$; and a $\bQ$-linear combination of $\monon{1}, \ldots, \monon{n}$ is annihilated by $\nabla: \BBdRpX{X}|_{\widetilde{X}}[[\monon{1}, \ldots, \monon{n}]] \to \BBdRpX{X}|_{\widetilde{X}}[[\monon{1}, \ldots, \monon{n}]] \otimes \Omega^{\log}_{X / X'}|_{\widetilde{X}}$ exactly when it lies in the $\bQ$-linear span of $\monon{1}', \ldots, \monon{n'}'$.  We have a similar statement for $\BBdRX{X}|_{\widetilde{X}}\{W_1, \ldots, W_n\}$ and $W'_1, \ldots, W'_{n'}$.  Thus, the exactness of the complexes follows from a straightforward calculation, as in the proof of Corollary \ref{cor-log-dR-cplx}.
\end{proof}

\section{Log Riemann--Hilbert Correspondences}\label{sec-log-RH}

In this section, we establish our log $p$-adic Riemann--Hilbert and Simpson correspondences.  Let $k$ be a $p$-adic field, with a fixed algebraic closure $\AC{k}$.  Let $K$ be a perfectoid field containing $k_\infty = k(\Grpmu_\infty) \subset \AC{k}$, and let $\Gal(K / k)$ abusively denote the group of continuous field automorphisms of $K$ over $k$.

\subsection{Filtered log connections \Qtn{relative to {$\BdR$}}}\label{sec-log-conn-BdR}

Let us begin with a few definitions and constructions for a general locally noetherian adic space $X$ over $k$.
\begin{defn}\phantomsection\label{def-OXBdR}
    \begin{enumerate}
        \item\label{def-OXBdR-1}  As in \cite[\aSec 3.1]{Liu/Zhu:2017-rrhpl}, let
            \begin{equation}\label{eq-def-BdR}
                \BdRp = \BBdRp(K, \cO_K) \quad \Utext{and} \quad \BdR = \BBdR(K, \cO_K)
            \end{equation}
            \Pth{the first replacing \Refeq{\ref{eq-choice-t}} from now on}.  Let $t = \log([\epsilon]) \in \BdRp$, as in \Refeq{\ref{eq-choice-t}}.  Then the homomorphism $k \to K$ lifts uniquely to $k \to \BdRp$.

        \item\label{def-OXBdR-2}  For each integer $r \geq 1$, we define $\cO_X \ho_k (\BdRp / t^r)$ to be the sheaf on $X_\an$ associated with the presheaf which assigns to each affinoid open subset $U = \Spa(A, A^+) \subset X$ the ring $A \ho_k (\BdRp / t^r)$.  Then we define
            \[
                \cO_X \ho_k \BdRp = \varprojlim_r \bigl( \cO_X \ho_k (\BdRp / t^r) \bigr) \quad \Utext{and} \quad \cO_X \ho_k \BdR = (\cO_X \ho_k \BdRp)[t^{-1}].
            \]

        \item\label{def-OXBdR-3}  The filtrations on $\cO_X \ho_k \BdRp$ and $\cO_X \ho_k \BdR$ are defined by setting
            \[
                \Fil^r(\cO_X \ho_k \BdRp) = t^r (\cO_X \ho_k \BdRp) \quad \Utext{and} \quad \Fil^r(\cO_X \ho_k \BdR) = t^{-s} \Fil^{r + s}(\cO_X \ho_k \BdRp)
            \]
            for some \Pth{and hence every} $s \geq -r$.  Then we define
            \[
                (\cO_X \ho_k \BdR)^{[a, b]} = \Fil^a(\cO_X \ho_k \BdR) / \Fil^{b + 1}(\cO_X \ho_k \BdR),
            \]
            for any $-\infty \leq a \leq b \leq \infty$.  In particular, $\gr^r (\cO_X \ho_k \BdR) = (\cO_X \ho_k \BdR)^{[r, r]}$.

          \item\label{def-OXBdR-4}  By replacing affinoid open subsets $U \subset X$ in \Refenum{\ref{def-OXBdR-2}} with general \'etale morphisms $U \to X$ from affinoid adic spaces, we similarly define the sheaves $\cO_{X_\et} \ho_k (\BdRp / t^r)$, for all integers $r \geq 1$; $\cO_{X_\et} \ho_k \BdRp$; and $\cO_{X_\et} \ho_k \BdR$ on $X_\et$.  They are equipped with similarly defined filtrations.
    \end{enumerate}
\end{defn}

\begin{rk}\label{rem-cX-top}
    These sheaves were introduced slightly differently in \cite[\aSec 3.1]{Liu/Zhu:2017-rrhpl} as sheaves on $X_{K, \an}$ and $X_{K, \et}$.  But since $X_{K, \an}$ \Pth{\resp $X_{K, \et}$} is generated by base changes of objects of $X_{k', \an}$ \Pth{\resp $X_{k', \et}$}, for all finite extensions $k'$ of $k$ \Pth{see, \eg, \cite[\aLem 2.5]{Liu/Zhu:2017-rrhpl}}, the categories of finite locally free $\cO_X \ho_k (\BdRp / t^r)$-modules, $\cO_X \ho_k \BdRp$-modules, $\cO_{X_\et} \ho_k (\BdRp / t^r)$-modules, and $\cO_{X_\et} \ho_k \BdRp$-modules are naturally equivalent to the corresponding categories introduced in \cite[\aDef 3.5]{Liu/Zhu:2017-rrhpl}.  For example, the category of finite locally free $\cO_X \ho_k K$-modules \Pth{\ie, $\gr^0 (\cO_X \ho_k \BdRp)$-modules} on $X_\an$ is equivalent to the category of vector bundles on $X_{K, \an}$.
\end{rk}

Thanks to Remark \ref{rem-cX-top}, the arguments in the proofs of \cite[\aLems 3.1 and 3.2, \aProp 3.3, and \aCor 3.4]{Liu/Zhu:2017-rrhpl} also apply in the current setting and give the following:
\begin{lemma}\label{lem-OXBdR-mod}
    Recall that $\lambda: X_\et \to X_\an$ denotes the natural projection of sites.
    \begin{enumerate}
        \item\label{lem-OXBdR-mod-1}  If $X = \Spa(A, A^+)$ is affinoid, then
            \[
                H^i\bigl(X_\et, \cO_{X_\et} \ho_k (\BdRp / t^r)\bigr) =
                \begin{cases}
                    A \ho_k (\BdRp / t^r), & \Utext{if $i = 0$}; \\
                    0, & \Utext{if $i > 0$}.
                \end{cases}
            \]

        \item\label{lem-OXBdR-mod-2}  There is a canonical isomorphism $\gr^r (\cO_{X_\et} \ho_k \BdR) \cong \cO_{X_\et} \ho_k K(r)$.

        \item\label{lem-OXBdR-mod-3}  There are canonical isomorphisms
            \[
                \cO_X \ho_k (\BdRp / t^r) \cong \lambda_*\bigl(\cO_{X_\et} \ho_k (\BdRp / t^r)\bigr) \cong R\lambda_*\bigl(\cO_{X_\et} \ho_k (\BdRp / t^r)\bigr),
            \]
            which in turn induce, for $? = \emptyset$ and $+$, isomorphisms
            \[
                \cO_X \ho_k \BdR^? \cong \lambda_*(\cO_{X_\et} \ho_k \BdR^?) \cong R\lambda_*(\cO_{X_\et} \ho_k \BdR^?).
            \]

        \item\label{lem-OXBdR-mod-4}  If $X = \Spa(A, A^+)$ is affinoid, then we have canonical equivalences among the categories of finite projective $A \ho_k \BdRp$-modules; of finite locally free $\cO_X \ho_k \BdRp$-modules; and of finite locally free $\cO_{X_\et} \ho_k \BdRp$-modules.

        \item\label{lem-OXBdR-mod-5}  The pushforward $\lambda_*$ induces an equivalence from the category of finite locally free $\cO_{X_\et} \ho_k (\BdRp / t^r)$-modules \Pth{\resp $\cO_{X_\et} \ho_k \BdRp$-modules} to the category of finite locally free $\cO_X \ho_k (\BdRp / t^r)$-modules \Pth{\resp $\cO_X \ho_k \BdRp$-modules}.
    \end{enumerate}
\end{lemma}

As in \cite[\aSec 3.1]{Liu/Zhu:2017-rrhpl}, for $? = \emptyset$ or $+$, we can define the ringed space
\begin{equation}\label{eq-def-cX}
    \cX^? = (X_\an, \cO_X \ho_k \BdR^?),
\end{equation}
where $\cO_X \ho_k \BdRp$ and $\cO_X \ho_k \BdR$ are as in Definition \ref{def-OXBdR}\Refenum{\ref{def-OXBdR-2}}.  They should be interpreted as the \Pth{not-yet-defined} base changes of $X$ under $k \to \BdRp$ and $k \to \BdR$, respectively.  Then we have $\cO_{\cX^+} = \cO_X \ho_k \BdRp$ and $\cO_\cX = \cO_X \ho_k \BdR$.

Following \cite[\aDef 3.5]{Liu/Zhu:2017-rrhpl}, we call a finite locally free $\cO_X \ho_k \BdRp$-module a \emph{vector bundle} on $\cX^+$.  By considering such objects over open subspaces of $X$, these objects form a stack on $X_\an$.  By passing to the $t$-isogeny category, we obtain the stack of vector bundles on open subspaces of $\cX$.  Then the category of vector bundles on $\cX$ is the groupoid of global sections of this stack.  Note that, unlike in \cite[\aDef 3.5]{Liu/Zhu:2017-rrhpl}, we do not require that a vector bundle on $\cX$ comes from a vector bundle on $\cX^+$ via a global extension of scalars \Pth{although this extra generality will not be needed in the following}.  Clearly, there is a faithful functor from the category of vector bundles on $\cX$ to the category of $\cO_X \ho_k \BdR$-modules.

Hence, for each vector bundle $\cE$ on $X_\an$, the sheaf $\cE \ho_k \BdRp$ \Pth{\resp $\cE \ho_k \BdR$} \Pth{with its obvious meaning} is a vector bundle on $\cX^+$ \Pth{\resp $\cX$}.  More generally, if $\cE$ is a vector bundle on $X_\an$, and if $\cM$ is a vector bundle on $\cX^+$ \Pth{\resp $\cX$}, then we may regard $\cE \otimes_{\cO_X} \cM$ as a vector bundle on $\cX^+$ \Pth{\resp $\cX$}.

Now let $X$ be a log smooth fs log adic space over $k$.  Let $\Omega^{\log}_X$ and $\Omega^{\log, \bullet}_X = \Ex^\bullet \, \Omega^{\log}_X$ be the sheaves of log differentials on $X_\an$, as in \cite[\aDefs \logadicdeflogdiffsheaf{} and \logadicdeflogdiffsheafex]{Diao/Lan/Liu/Zhu:lasfr}.
\begin{defn}\label{def-log-diff-sheaf-ex-cX}
    For $? = \emptyset$ or $+$, let $\Omega^{\log}_{\cX^? / \BdR^?} := \Omega^{\log}_X \ho_k \BdR^?$ and $\Omega^{\log, \bullet}_{\cX^? / \BdR^?} := \Omega^{\log, \bullet}_X \ho_k \BdR^?$, called the \emph{sheaves of relative log differentials} on $\cX^?$ over $\BdR^?$.
\end{defn}

For $? = \emptyset$ or $+$, there is a natural $\BdR^?$-linear differential map $d: \cO_{\cX^?} \to \Omega^{\log}_{\cX^? / \BdR^?}$ inducing differential maps on $\Omega^{\log, \bullet}_{\cX^? / \BdR^?}$, extending the ones on $\cO_X$ and $\Omega^{\log, \bullet}_X$.

\begin{defn}\phantomsection\label{def-log-conn-etc}
    \begin{enumerate}
        \item\label{def-log-conn-etc-1}  A \emph{log connection} on a vector bundle $\cE$ on $\cX$ is a $\BdR$-linear map of sheaves $\nabla: \cE \to \cE \otimes_{\cO_\cX} \Omega^{\log}_{\cX / \BdR}$ satisfying the usual Leibniz rule.  We say that $\nabla$ is \emph{integrable} if $\nabla^2 = 0$, in which case we have the \emph{log de Rham complex} $\DRl(\cE) = (\cE \otimes_{\cO_\cX} \Omega^{\log, \bullet}_{\cX / \BdR}, \nabla)$ and the \emph{log de Rham cohomology} $H_{\log \dR}^i(\cX, \cE) := H^i(\cX, \DRl(\cE))$.

        \item\label{def-log-conn-etc-2}  Let $t = \log([\epsilon]) \in \BdRp$ be as in \Refeq{\ref{eq-choice-t}}.  A \emph{log} $t$-\emph{connection} on a vector bundle $\cE^+$ on $\cX^+$ is a $\BdRp$-linear map of sheaves $\nabla^+: \cE^+ \to \cE^+ \otimes_{\cO_{\cX^+}} \Omega^{\log}_{\cX^+ / \BdRp}$ satisfying the \Pth{modified} Leibniz rule $\nabla^+(f e) = (t e) \otimes df + f \nabla^+(e)$, for all $f \in \cO_{\cX^+}$ and $e \in \cE^+$.  We say $\nabla^+$ is \emph{integrable} if $(\nabla^+)^2 = 0$, in which case we have a similar log de Rham complex \Pth{as above}.

        \item\label{def-log-conn-etc-3}  A \emph{log Higgs bundle} on $X_K$ is a vector bundle $E$ on $X_{K, \an}$ equipped with an $\cO_{X_K}$-linear map of sheaves $\theta: E \to E \otimes_{\cO_{X_K}} \Omega^{\log}_{X_K}(-1)$ such that $\theta \wedge \theta = 0$.  \Pth{We shall often omit the subscript \Qtn{$\an$} in the following, when there is no risk of confusion.}  Then we have the \emph{log Higgs complex} $\Hil(E) = (E \otimes_{\cO_{X_K}} \Omega^{\log, \bullet}_{X_K}(-\bullet), \theta)$ \Pth{where the two $\bullet$ are equal to each other} and the \emph{log Higgs cohomology} $H_{\log \Hi}^i(X_K, E) := H^i(X_K, \Hil(\cE))$.

        \item\label{def-log-conn-etc-4}  A \emph{log connection} on a coherent sheaf $E$ on $X$ is a $k$-linear map of sheaves $\nabla: E \to E \otimes_{\cO_X} \Omega^{\log}_X$ satisfying the usual Leibniz rule.  We say that $\nabla$ is \emph{integrable} if $\nabla^2 = 0$, in which case we have the \emph{log de Rham complex} $\DRl(E) = (E \otimes_{\cO_X} \Omega^{\log, \bullet}_X, \nabla)$ and the \emph{log de Rham cohomology} $H_{\log \dR}^i(X, E) := H^i(X, \DRl(E))$.

            Suppose that $E$ is equipped with a decreasing filtration by coherent subsheaves $\Fil^\bullet E$ satisfying the \Pth{usual} Griffiths transversality condition $\nabla( \Fil^r E ) \subset ( \Fil^{r - 1} E ) \otimes_{\cO_X} \Omega^{\log}_X$, for all $r$.  Then the complex $\DRl(E)$ admits a filtration defined by $\Fil^r \DRl(E) := \bigl( ( \Fil^{r - \bullet} E ) \otimes_{\cO_X} \Omega^{\log, \bullet}_X, \nabla \bigr)$, with the two $\bullet$ equal to each other, and with $\nabla$ respecting the filtration and inducing $\cO_X$-linear morphisms on the graded pieces.  The graded pieces form a complex $\gr \DRl(E)$ with $\cO_X$-linear differentials, and we also have the \emph{log Hodge cohomology} $H^{a, b}_{\log \Hdg}\bigl(X, E\bigr) := H^{a + b}\bigl(X, \gr^a \DRl(E)\bigr)$.

            The log de Rham cohomology and the log Hodge cohomology are related by the \emph{\Pth{log} Hodge--de Rham spectral sequence} \Pth{associated with the filtration $\Fil^\bullet \DRl(E)$ above} $E_1^{a, b} = H^{a, b}_{\log \Hdg}\bigl(X, E\bigr) \Rightarrow H^{a + b}_{\log \dR}\bigl(X, E\bigr)$.
    \end{enumerate}
\end{defn}

The following two lemmas are clear.

\begin{lemma}\label{lem-conn-vs-t-conn}
    The functor
    \[
        (\cE^+, \nabla^+) \mapsto (\cE, \nabla, \{ \Fil^r \}_{r \geq 0}) := (\cE^+ \otimes_{\BdRp} \BdR, t^{-1} \nabla^+, \{ t^r \cE^+ \}_{r \geq 0})
    \]
    is an equivalence of categories from the category of vector bundles with integrable log $t$-connections on $\cX^+$ to the category of vector bundles with integrable log connections $(\cE, \nabla)$ on $\cX$ that are equipped with filtrations $\{ \Fil^r \}_{r \geq 0}$ by locally free $\cO_X \ho_k \BdRp$-submodules satisfying, for all $r \geq 1$, the condition $\Fil^r \cE = t (\Fil^{r - 1} \cE)$ and the Griffiths transversality condition $\nabla(\Fil^r \cE) \subset (\Fil^{r - 1} \cE) \otimes_{\cO_{\cX^+}} \Omega^{\log}_{\cX^+ / \BdRp}$.
\end{lemma}

\begin{lemma}\label{lem-t-conn-Higgs}
    The functor $(\cE^+, \nabla^+) \mapsto (\cE^+ / t, \nabla^+)$, where $\nabla^+$ abusively also denotes its induced map on $\cE^+ / t$, is a functor from the category of vector bundles with integrable log $t$-connections on $\cX^+$ to the category of log Higgs bundles on $X_K$.
\end{lemma}

\subsection{Statements of theorems}\label{sec-log-RH-thm}

Let us now state the main theorems of this section.  Let $k$, $\AC{k}$, $k_\infty$, and $K$ be as in the beginning of this Section \ref{sec-log-RH}, and let $X$ be any log adic space over $k$ as in Example \ref{ex-log-adic-sp-ncd}, with its log structure induced by a normal crossings divisor $D$.  Let $U := X - D$.  Given any $\bQ_p$-local system $\bL$, recall that we say $\bL|_{U_\et}$ has \emph{unipotent geometric monodromy} along $D$ \Pth{see \cite[\aDef \logadicdefunipqunipmonod{} and \aRem \logadicrkunipqunipmonodalgcl]{Diao/Lan/Liu/Zhu:lasfr}} when $\pi_1^\ket\bigl(X(\xi), \widetilde{\xi}\bigr)$ acts unipotently on the stalk $\bL_{\widetilde{\xi}}$, for each log geometric points $\xi$ of $X$ lying above each geometric point $\xi$ of $D$, where the log structure of the strict localization $X(\xi)$ is pulled back from $X$.  Let
\begin{equation}\label{eq-mu-prime}
    \mu': {X_\proket}_{/X_K} \to X_\an.
\end{equation}
be the natural projection of sites.  For a $\bQ_p$-local system $\bL$ on $X_\ket$, let $\widehat{\bL}$ be the corresponding $\widehat{\bQ}_p$-local system on $X_\proket$, as in \cite[\aLem \logadiclemproketlisse]{Diao/Lan/Liu/Zhu:lasfr}, and consider
\begin{equation}\label{eq-def-RHl}
    \RHl(\bL) := R\mu'_*(\widehat{\bL} \otimes_{\widehat{\bQ}_p} \OBdl).
\end{equation}

\begin{thm}\phantomsection\label{thm-log-RH-geom}
    \begin{enumerate}
        \item\label{thm-log-RH-geom-main}  The assignment $\bL \mapsto \RHl(\bL)$ is an exact functor from the category of $\bQ_p$-local systems on $X_\ket$ to the category of $\Gal(K / k)$-equivariant vector bundles on $\cX$ equipped with integrable log connections $\nabla_\bL: \RHl(\bL) \to \RHl(\bL) \otimes_{\cO_\cX} \Omega^{\log}_{\cX / \BdR}$ and decreasing filtrations \Pth{by locally free $\cO_X \ho_k \BdRp$-submodules} satisfying the Griffiths transversality, defined by $\Fil^r \RHl(\bL) := \mu'_*(\widehat{\bL} \otimes_{\widehat{\bQ}_p} \Fil^r \OBdl)$, for all $r \in \bZ$.

        \item\label{thm-log-RH-geom-res}  For each irreducible component $Z$ \Pth{defined as in \cite{Conrad:1999-icrs}} of the normal crossings divisor $D$, let $\Res_Z(\nabla_\bL)$ denote the residue of the log connection $\nabla_\bL$ along $Z$ \Pth{see Section \ref{sec-calc-res} below for details on the definition of residues}.  If $Z_{\AC{k}}$ is irreducible \Pth{which we may always assume, up to replacing $k$ with a finite extension}, then all the eigenvalues of $\Res_Z(\nabla_\bL)$ are in $\bQ \cap [0, 1)$.

        \item\label{thm-log-RH-geom-comp}  Assume that $X$ is proper over $k$, and that $K = \widehat{\AC{k}}$.  Let $\bL$ be a $\bZ_p$-local system on $X_\ket$.  Then there is a canonical $\Gal(K / k)$-equivariant isomorphism
            \[
                H^i\bigl(X_{K, \ket}, \bL\bigr) \otimes_{\bZ_p} B_\dR \cong H^i_{\log \dR}\bigl(\cX, \RHl(\bL)\bigr),
            \]
            for each $i \geq 0$, compatible with the filtrations on both sides, where the right-hand side is as in Definition \ref{def-log-conn-etc}\Refenum{\ref{def-log-conn-etc-1}}.

        \item\label{thm-log-RH-geom-mor}  Suppose that $Y$ is another log adic space whose log structure is defined by some normal crossings divisor $E$ as in Example \ref{ex-log-adic-sp-ncd}, and that $h: Y \to X$ is a morphism of log adic spaces.  For any pair of irreducible components $Z$ and $W$ of $D$ and $E$, respectively, let $m_{W Z} \in \bZ_{\geq 0}$ denote the multiplicity of $W$ in the divisor $h^{-1}(Z)$; and let $n_Z$ be $0$ \Pth{\resp $1$} if $\bL|_{U_\et}$ has \Pth{\resp does not have} unipotent geometric monodromy along $Z$.  Assume that, for each irreducible component $W$ of $E$, we have $\sum_Z \, m_{W Z} \, n_Z \leq 1$, where the sum is over all irreducible components $Z$ of $D$.  Then there is a canonical $\Gal(K / k)$-equivariant isomorphism $h^*\bigl(\RHl(\bL), \nabla_\bL\bigr) \Mi \bigl(\RHl(h^{-1}(\bL)), \nabla_{h^{-1}(\bL)}\bigr)$, compatible with the filtrations on both sides.
    \end{enumerate}
\end{thm}

As a byproduct, we obtain the log $p$-adic Simpson functor in our setting.  We refer to \cite{Faltings:2005-psc, Abbes/Gros/Tsuji:2016-pSC} for more general and thorough treatments.
\begin{thm}\phantomsection\label{thm-log-Simp}
    \begin{enumerate}
        \item\label{thm-log-Simp-main}  There is a natural functor $\Hl$ from the category of $\bQ_p$-local systems $\bL$ on $X_\ket$ to the category of $\Gal(K / k)$-equivariant log Higgs bundles $\theta_\bL: \Hl(\bL) \to \Hl(\bL) \otimes_{\cO_{X_K}} \Omega^{\log}_{X_K}(-1)$ on $X_{K, \an}$.  Concretely, by Lemma \ref{lem-conn-vs-t-conn}, $\RHl^+ := \Fil^0 \RHl$ is a functor from the category of $\bQ_p$-local systems on $X_\ket$ to the category of $\Gal(K / k)$-equivariant vector bundles with integrable log $t$-connections on $\cX^+$.  Then, by Lemma \ref{lem-t-conn-Higgs}, $\Hl := \gr^0 \RHl = \RHl^+ / t$ is the desired functor.

        \item\label{thm-log-Simp-comp}  Under the same assumption as in Theorem \ref{thm-log-RH-geom}\Refenum{\ref{thm-log-RH-geom-comp}}, there is a canonical $\Gal(K / k)$-equivariant isomorphism
            \[
                H^i\bigl(X_{K, \ket}, \bL\bigr) \otimes_{\bZ_p} K \cong H^i_{\log \Hi}\bigl(X_{K, \an}, \Hl(\bL)\bigr),
            \]
            for each $i \geq 0$, where $H^i_{\log \Hi}\bigl(X_{K, \an}, \Hl(\bL)\bigr)$ is as in Definition \ref{def-log-conn-etc}\Refenum{\ref{def-log-conn-etc-3}}.

        \item\label{thm-log-Simp-mor}  Under the same assumption as in Theorem \ref{thm-log-RH-geom}\Refenum{\ref{thm-log-RH-geom-mor}}, there is a canonical $\Gal(K / k)$-equivariant isomorphism
            \[
                h^*\bigl(\Hl(\bL), \theta_\bL\bigr) \Mi \bigl(\Hl(h^{-1}(\bL)), \theta_{h^{-1}(\bL)}\bigr).
            \]
    \end{enumerate}
\end{thm}

We also have an arithmetic log $p$-adic Riemann--Hilbert functor.  Consider the natural projection of sites
\begin{equation}\label{eq-mu}
    \mu: X_\proket \to X_\an.
\end{equation}
For any $\bQ_p$-local system $\bL$ on $X_\ket$, consider
\begin{equation}\label{eq-def-Ddl}
    \Ddl(\bL) := \mu_*(\widehat{\bL} \otimes_{\widehat{\bQ}_p} \OBdl).
\end{equation}

\begin{thm}\phantomsection\label{thm-log-RH-arith}
     \begin{enumerate}
        \item\label{thm-log-RH-arith-main}  The assignment $\bL \mapsto \Ddl(\bL)$ defines a functor from the category of $\bQ_p$-local systems on $X_\ket$ to the category of vector bundles on $X_\an$ with integrable log connections $\nabla_\bL: \Ddl(\bL) \to \Ddl(\bL) \otimes_{\cO_X} \Omega^{\log}_X$ and decreasing filtrations $\Fil^\bullet \Ddl(\bL)$ \Pth{by coherent subsheaves} satisfying the \Pth{usual} Griffiths transversality.

        \item\label{thm-log-RH-arith-res}  For each irreducible component $Z$ \Pth{defined as in \cite{Conrad:1999-icrs}} of the normal crossings divisor $D$, all eigenvalues of the residue $\Res_Z(\nabla_\bL)$ are in $\bQ \cap [0, 1)$.  If the restriction of $\bL$ to $U_\ket \cong U_\et$ is \emph{de Rham} \Pth{as reviewed in the introduction}, then $\gr \Ddl(\bL)$ is a vector bundle on $X$ of rank $\rank_{\bQ_p}(\bL)$.

        \item\label{thm-log-RH-arith-comp}  Assume that $X$ is proper over $k$, that $K = \widehat{\AC{k}}$, and that $\bL$ is a $\bZ_p$-local system on $X_\ket$ whose restriction to $U_\et$ is \emph{de Rham}.  Then, for each $i \geq 0$, there is a canonical $\Gal(K / k)$-equivariant isomorphism
            \begin{equation}\label{eq-thm-log-RH-arith-comp-dR}
                H^i\bigl(X_{K, \ket}, \bL\bigr) \otimes_{\bZ_p} B_\dR \cong H^i_{\log \dR}\bigl(X_\an, \Ddl(\bL)\bigr) \otimes_k B_\dR
            \end{equation}
            compatible with the filtrations on both sides.  Moreover, the \Pth{log} Hodge--de Rham spectral sequence for $\Ddl(\bL)$ degenerates on the $E_1$ page, and there is also a canonical $\Gal(K / k)$-equivariant isomorphism
            \begin{equation}\label{eq-thm-log-RH-arith-comp-HT}
                H^i\bigl(X_{K, \ket}, \bL\bigr) \otimes_{\bZ_p} K \cong \oplus_{a + b = i} \, \Bigl( H^{a, b}_{\log \Hdg}\bigl(X_\an, \Ddl(\bL)\bigr) \otimes_k K(-a) \Bigr),
            \end{equation}
            for each $i \geq 0$, which can be identified with the $0$-th graded piece of the isomorphism \Refeq{\ref{eq-thm-log-RH-arith-comp-dR}}, giving the \emph{\Pth{log} Hodge--Tate decomposition}.

        \item\label{thm-log-RH-arith-mor}  Under the same assumption as in Theorem \ref{thm-log-RH-geom}\Refenum{\ref{thm-log-RH-geom-mor}}, there is a canonical isomorphism $h^*\bigl(\Ddl(\bL), \nabla_\bL\bigr) \Mi \bigl(\Ddl(h^{-1}(\bL)), \nabla_{h^{-1}(\bL)}\bigr)$, compatible with the filtrations on both sides.

        \item\label{thm-log-RH-arith-push}  Suppose that $Y$ is another log adic space with its log structure defined by a normal crossings divisor $E \Em Y$ as in Example \ref{ex-log-adic-sp-ncd}.  Let $V = Y - E$.  Let $f: X \to Y$ be a proper log smooth morphism that restricts to a proper smooth morphism $f|_U: U \to V$.  Let $\bL$ be a $\bZ_p$-local system on $X_\ket$ that is de Rham when restricted to $U_\ket \cong U_\et$.  Then $R^i f_{\ket, *}(\bL)$ is a $\bZ_p$-local system on $Y_\ket$ that is \emph{de Rham} when restricted to $V_\ket \cong V_\et$, for each $i \geq 0$.  Moreover, we have a canonical isomorphism
            \[
                \bigl(\Ddl(R^i f_{\ket, *}(\bL)), \nabla_{R^i f_{\ket, *}(\bL)}\bigr) \cong \bigl(R^i f_{\log \dR, *}(\Ddl(\bL), \nabla_\bL)\bigr)_\free,
            \]
            compatible with the filtrations on both sides, where $R^i f_{\log \dR, *}$ denotes the usual relative analogue of the log de Rham cohomology, and where the subscript \Qtn{$\free$} denotes the $\cO_Y$-torsion-free quotient.
     \end{enumerate}
\end{thm}

By Theorem \ref{thm-log-RH-arith}\Refenum{\ref{thm-log-RH-arith-comp}} and \cite[\aCor \logadiccorpuritylisse]{Diao/Lan/Liu/Zhu:lasfr}, we obtain the following:
\begin{cor}\label{cor-comp-open}
    Let $Y$ be a smooth rigid analytic variety over $k$, and let $K = \widehat{\AC{k}}$.  Assume that $Y$ admits a proper smooth compactification $Y \Em \overline{Y}$ such that $\overline{Y} - Y$ is a normal crossings divisor.  Let $\bL$ be a \emph{de Rham} $\bZ_p$-local system on $Y_\et$, with its extension $\overline{\bL} := \jmath_{\ket, *}(\bL)$ to a $\bZ_p$-local system on $\overline{Y}_\ket$.  Then $H^i\bigl(Y_{K, \et}, \bL\bigr)$ is a \emph{finite} $\bZ_p$-module, and there is a canonical $\Gal(K / k)$-equivariant isomorphism
    \begin{equation}\label{eq-cor-comp-open}
        H^i\bigl(Y_{K, \et}, \bL\bigr) \otimes_{\bZ_p} \BdR \cong H^i_{\log \dR}\bigl(\overline{Y}_\an, \Ddl(\overline{\bL})\bigr) \otimes_k B_\dR,
    \end{equation}
    compatible with the filtrations on both sides.  Moreover, the \Pth{log} Hodge--de Rham spectral sequence for $\Ddl(\overline{\bL})$ degenerates on the $E_1$ page, and the $0$-th graded piece of \Refeq{\ref{eq-cor-comp-open}} is also a canonical $\Gal(K / k)$-equivariant isomorphism
    \[
        H^i\bigl(Y_{K, \et}, \bL\bigr) \otimes_{\bZ_p} K \cong \oplus_{a + b = i} \, \Bigl( H^{a, b}_{\log \Hdg}\bigl(\overline{Y}_\an, \Ddl(\overline{\bL})\bigr) \otimes_k K(-a) \Bigr).
    \]
\end{cor}
Note that, as explained in \cite[\aRem \logadicremcohfin]{Diao/Lan/Liu/Zhu:lasfr}, the finiteness of $H^i\bigl(Y_{K, \et}, \bL\bigr)$ as a $\bZ_p$-module does not hold in general for an arbitrary smooth rigid analytic variety $Y$ \Pth{that is not Zariski open in some proper rigid analytic variety}.

As mentioned in the introduction, due to the failure of the surjectivity of \Refeq{\ref{eq-fail-surj}}, $\Ddl$ is not a tensor functor in general, and we have similar failures for $\RHl$ and $\Hl$.  Nevertheless, we still have the following:
\begin{thm}\phantomsection\label{thm-unip-vs-nilp}
    \begin{enumerate}
        \item\label{thm-unip-vs-nilp-1} The functor $\RHl$ \Pth{\resp $\Hl$} restricts to a \emph{tensor functor} from the category of $\bQ_p$-local systems on $X_\ket$ whose restrictions to $U_\et$ have \emph{unipotent} geometric monodromy along $D$ to the category of filtered $\Gal(K / k)$-equivariant vector bundles on $\cX$ equipped with integrable log connections with \emph{nilpotent} residues along $D$ \Pth{\resp the category of $\Gal(K / k)$-equivariant log Higgs bundles on $X_{K, \an}$}.

        \item\label{thm-unip-vs-nilp-2} The functor $\Ddl$ restricts to a tensor functor from the category of $\bQ_p$-local systems on $X_\ket$ whose restrictions to $U_\et$ are \emph{de Rham} and have \emph{unipotent} geometric monodromy along $D$ to the category of filtered vector bundles on $X_\an$ equipped with integrable log connections with \emph{nilpotent} residues along $D$.
    \end{enumerate}
\end{thm}

\subsection{Coherence}\label{sec-coh}

In this subsection, we prove Theorems \ref{thm-log-RH-geom}\Refenum{\ref{thm-log-RH-geom-main}} and \ref{thm-log-Simp}\Refenum{\ref{thm-log-Simp-main}}, and show that $\Ddl(\bL)$ is a torsion-free reflexive coherent sheaf on $X_\an$.

By factoring $\mu'$ as ${X_\proket}_{/X_K} \cong X_{K, \proket} \to X_{K, \et} \to X_{K, \an} \to X_\an$, we see that $\RHl(\bL)$ admits a natural $\Gal(K / k)$-action.  We need to show that $R\mu'_*(\widehat{\bL} \otimes_{\widehat{\bQ}_p} \Fil^r \OBdl)$ is a locally free $\cO_X \ho_k \BdRp$-module of rank $\rank_{\bQ_p}(\bL)$, for every $r$.  Assuming this, it follows that
\[
    \RHl(\bL) = R\mu'_*(\widehat{\bL} \otimes_{\widehat{\bQ}_p} \OBdl) \cong R\mu'_*(\widehat{\bL} \otimes_{\widehat{\bQ}_p} \Fil^0 \OBdl)[t^{-1}]
\]
is a vector bundle of rank $\rank_{\bQ_p}(\bL)$ on $\cX$, equipped with the filtration
\[
    \Fil^r \RHl(\bL) := \mu'_*(\widehat{\bL} \otimes_{\widehat{\bQ}_p} \Fil^r \OBdl).
\]
by locally free $\cO_X \ho_k \BdRp$-submodules.  Consider the integrable log connection
\[
    \nabla: \widehat{\bL} \otimes_{\widehat{\bQ}_p} \OBdl \to \widehat{\bL} \otimes_{\widehat{\bQ}_p} \OBdl \otimes_{\cO_{X_\proket}} \Omega^{\log}_X
\]
formed by tensoring the one on $\OBdl$ with $\widehat{\bL}$.  By the projection formula
\begin{equation}\label{eq-RHl-proj}
    R\mu'_*(\widehat{\bL} \otimes_{\widehat{\bQ}_p} \OBdl \otimes_{\cO_{X_\proket}} \Omega^{\log, \bullet}_X) \cong R\mu'_*(\widehat{\bL} \otimes_{\widehat{\bQ}_p} \OBdl) \otimes_{\cO_X} \Omega^{\log, \bullet}_X,
\end{equation}
we obtain a log connection $\nabla_\bL: \RHl(\bL) \to \RHl(\bL) \otimes_{\cO_X} \Omega^{\log}_X$. The integrability of $\nabla_\bL$ and the Griffiths transversality with respect to the filtration $\Fil^\bullet \RHl(\bL)$ follow from the corresponding properties of the connection \Refeq{\ref{eq-conn-OBdl}}.

In what follows, we shall denote by $Z$ either the whole $X$ or an open subspace of a smooth intersection of irreducible components of $D$, equipped with the log structure pulled back from $X$, which fits into the second case of Remark \ref{rem-OBdl-conseq-setting}.

\begin{lemma}\label{lem-OXBdR-a-b}
    Let $Z$ be as above.  For any $-\infty \leq a < b \leq \infty$, there is a natural isomorphism $(\cO_Z \ho_k \BdR)^{[a, b]} \cong R\mu'_{Z, *}(\OBdlX{Z}^{[a, b]})$.
\end{lemma}
\begin{proof}
    By Lemma \ref{lem-OXBdR-mod}\Refenum{\ref{lem-OXBdR-mod-3}}, it suffices to prove the analogue for the morphism $\nu'_Z: {Z_\proket}_{/Z_K} \to Z_\et$ \Pth{instead of $\mu'_Z$}.  By using Corollary \ref{cor-OBdlp-loc-gr}, the argument is similar to the ones in the proofs of \cite[\aProp 6.16(i)]{Scholze:2013-phtra} and \cite[\aLem 3.7]{Liu/Zhu:2017-rrhpl}.
\end{proof}

By the same arguments as in the proofs of \cite[\aThms 2.1(i) and 3.8(i)]{Liu/Zhu:2017-rrhpl}, in order to show that $R\mu'_*(\widehat{\bL} \otimes_{\widehat{\bQ}_p} \Fil^r \OBdl)$ is a locally free $\cO_X \ho_k \BdRp$-module of rank $\rank_{\bQ_p}(\bL)$, for every $r$, it suffices to prove the following:
\begin{prop}\label{prop-L-OCl}
    Let $\bL$ be a $\bQ_p$-local system on $X_\ket$.  Let $Z$ be as above, and let $\widehat{\bL}_Z$ denote the pullback of $\widehat{\bL}$ under $Z_\proket \to X_\proket$.
    \begin{enumerate}
        \item $R^i\mu'_{Z, *}(\widehat{\bL}_Z \otimes_{\widehat{\bQ}_p} \OClX{Z}) = 0$, for all $i > 0$.

        \item $\mu'_{Z, *}(\widehat{\bL}_Z \otimes_{\widehat{\bQ}_p} \OClX{Z})$ is a finite locally free $\gr^0 (\cO_X \ho_k \BdR)$-module, whose rank is equal to $\rank_{\bQ_p}(\bL)$ if $Z = X$.
    \end{enumerate}
\end{prop}

For simplicity, we may assume that $K = \widehat{k}_\infty$, so that $\Gal(K / k)$ is identified with an open subgroup of $\widehat{\bZ}^\times$ via the cyclotomic character $\chi$.  \Pth{The assertions for larger perfectoid fields then follow by base change.}  By Lemma \ref{lem-OXBdR-mod}\Refenum{\ref{lem-OXBdR-mod-5}}, it suffices to prove similar statements for the projection of sites $\nu'_Z: {Z_\proket}_{/Z_K} \to Z_\et$ \Pth{instead of $\mu'_Z: {Z_\proket}_{/Z_K} \to Z_\an$}.  Since such statements are \'etale local in nature, we may assume that $X = \Spa(R, R^+)$ is an affinoid log adic space over $\Spa(k, k^+)$, where $k^+ = \cO_k$, with a smooth toric chart $X \to \bE := \Spa(k\Talg{P}, k^+\Talg{P})$ \Pth{see \cite[\aCor \logadiccorsmtoricchart{} and \aDef \logadicdeftoricchart]{Diao/Lan/Liu/Zhu:lasfr}}, where $P = \bZ_{\geq 0}^n = \oplus_{j = 1}^n (\bZ_{\geq 0} \, a_j)$.  We shall write $T_j = \mono{a_j}$, for each $j$.  Note that this fits into the setup in Section \ref{sec-OBdl-explicit}, with $Q = 0$ there, and we may assume that $Z$ is defined by $T_1 = \cdots = T_l = 0$, for some $l \leq n$.  Therefore, we have a log affinoid perfectoid object $\widetilde{X}$ in $X_\proket$ \Pth{\resp $\widetilde{Z}$ in $Z_\proket$} obtained by pulling back $\widetilde{\bE} := \varprojlim_m \bE_m \to \bE$, where $\bE_m := \Spa(k_m\Talg{\tfrac{1}{m} P}, k_m^+\Talg{\tfrac{1}{m} P})$, and we shall write $T_j^{\frac{1}{m}} = \mono{\frac{1}{m} a_j}$, for each $j$.  Then $\widetilde{X} \to X_{k_\infty}$ is a Galois pro-Kummer \'etale cover with Galois group $\Gamma_\geom \cong (\widehat{\bZ}(1))^n$, and $\widetilde{X} \to X$ is also a Galois pro-Kummer \'etale cover, whose Galois group $\Gamma$ fits into a short exact sequence
\begin{equation}\label{eq-fund-grp-Gamma}
    1 \to \Gamma_\geom \to \Gamma \to \Gal(k_\infty / k) \to 1,
\end{equation}
with $\Gal(k_\infty / k)$ acting on $\Gamma_\geom \cong (\widehat{\bZ}(1))^n$ via the cyclotomic character $\chi: \Gal(k_\infty / k) \to \widehat{\bZ}^\times$.  The same is true for the pullbacks $\widetilde{Z} \to Z_{k_\infty}$ and $\widetilde{Z} \to Z$.

Let $R_K := R \ho_k K$.  Also, let $\overline{R} := R / (T_1, \ldots, T_l)$ and $\overline{R}_K := \overline{R} \ho_k K$.  By Corollary \ref{cor-OBdlp-loc-gr}, we have $\OClX{Z}|_{\widetilde{Z}} \cong \widehat{\cO}_{Z_\proket}|_{\widetilde{Z}}[W_1, \ldots, W_n]$, where $W_j = t^{-1} \monon{j} = t^{-1} \log(\mono{a_j})$ in the notation there, for all $j = 1, \ldots, n$.  Let
\[
    \cL_Z := \widehat{\bL}_Z \otimes_{\widehat{\bQ}_p} \widehat{\cO}_{Z_\proket},
\]
which is a locally free $\widehat{\cO}_{Z_\proket}$-module of rank $\rank_{\bQ_p}(\bL)$.  Then
\[
    (\widehat{\bL}_Z \otimes_{\widehat{\bQ}_p} \OClX{Z})|_{\widetilde{Z}} \cong \cL_Z|_{\widetilde{Z}}[W_1, \ldots, W_n].
\]

Note that $R^i\nu'_{Z, *}(\widehat{\bL}_Z \otimes_{\widehat{\bQ}_p} \OClX{Z})$ is the sheaf on $Z_\et$ associated with the presheaf
\[
    Y \mapsto H^i({Z_\proket}_{/Y_K}, \widehat{\bL}_Z \otimes_{\widehat{\bQ}_p} \OClX{Z}).
\]
In order to prove Proposition \ref{prop-L-OCl}, it suffices to prove the following two statements:
\begin{enumerate}[label=(\alph*), ref=\alph*]
    \item\label{prop-L-OCl-a}  $H^0({Z_\proket}_{/Z_K}, \widehat{\bL}_Z \otimes_{\widehat{\bQ}_p} \OClX{Z})$ is a finite projective $\overline{R}_K$-module, of rank $\rank_{\bQ_p}(\bL)$ if $Z = X$; and $H^i({Z_\proket}_{/Z_K}, \widehat{\bL}_Z \otimes_{\widehat{\bQ}_p} \OClX{Z}) = 0$, for all $i > 0$.

    \item\label{prop-L-OCl-b}  Let $Y = \Spa(S, S^+) \to Z$ be a composition of rational embeddings and finite \'etale morphisms, and let $\widehat{\bL}_Y$ denote the pullback of $\widehat{\bL}_Z$ under $Y_\proket \to Z_\proket$.  Then we have a canonical isomorphism of $S_K$-modules
        \[
            H^0({Z_\proket}_{/Z_K}, \widehat{\bL}_Z \otimes_{\widehat{\bQ}_p} \OClX{Z}) \otimes_{\overline{R}_K} S_K \Mi
            H^0({Y_\proket}_{/Y_K}, \widehat{\bL}_Y \otimes_{\widehat{\bQ}_p} \OClX{Y}).
        \]
\end{enumerate}

Our approach to proving \Refenum{\ref{prop-L-OCl-a}} and \Refenum{\ref{prop-L-OCl-b}} is similar to the one in the proof of \cite[\aThm 2.1]{Liu/Zhu:2017-rrhpl}.  We will only explain the new ingredients here, and refer to \cite{Liu/Zhu:2017-rrhpl} for more details.  For any $Y$ as in \Refenum{\ref{prop-L-OCl-b}}, we endow it with the induced log structure.  Then $\widetilde{Y} := Y \times_Z \widetilde{Z} \in Z_\proket$, where $\widetilde{Z} \to Z$ is as above, is log affinoid perfectoid; and $\widetilde{Y} \to Y_{k_\infty}$ is also a Galois pro-Kummer \'etale cover with Galois group $\Gamma_\geom$.

By Corollary \ref{cor-OBdlp-loc-gr} and \cite[\aThm \logadicfinfreeOhatmod]{Diao/Lan/Liu/Zhu:lasfr}, and by the same arguments as in the proofs of \cite[\aCor 2.4, and \aLem 2.7]{Liu/Zhu:2017-rrhpl}, we obtain the following lemma:
\begin{lemma}\label{lem-L-OBdl-a-b-van}
    Let $\bM$ be a $\bQ_p$-local system on $Z_\ket$.
    \begin{enumerate}
        \item Let $U$ be log affinoid perfectoid object in ${Z_\proket}_{/\widetilde{Z}_K}$.  For any $-\infty \leq a \leq b \leq \infty$, and for each $i > 0$, we have $H^i({Z_\proket}_{/U}, \widehat{\bM} \otimes_{\widehat{\bQ}_p} \OBdlX{Z}^{[a, b]}) = 0$.

        \item $H^i\bigl(\Gamma_\geom, (\widehat{\bM} \otimes_{\widehat{\bQ}_p} \OClX{Z})(\widetilde{Y})\bigr) \cong H^i\bigl({Z_\proket}_{/Y_K}, \widehat{\bM} \otimes_{\widehat{\bQ}_p} \OClX{Z}\bigr)$, for all $i \geq 0$.
    \end{enumerate}
\end{lemma}

Consider the topological basis $\{ \gamma_1, \ldots, \gamma_n \}$ of $\Gamma_\geom \cong (\widehat{\bZ}(1))^n$ given by pulling back the image in $\widehat{\bZ}^n$ of the standard basis $\{ a_1, \ldots, a_n \}$ of $\bZ^n$ via the isomorphism $(\widehat{\bZ}(1))^n \Mi \widehat{\bZ}^n$ induced by \Refeq{\ref{eq-zeta}}, which is characterized by the property that
\begin{equation}\label{eq-def-gamma-j}
    \gamma_j \, T_{j'}^{\frac{1}{m}} = \zeta_m^{\delta_{jj'}} T_{j'}^{\frac{1}{m}},
\end{equation}
for all $1 \leq j, j' \leq n$ and $m \geq 1$ \Pth{\Refcf{} \cite[(\logadiceqgeomtowerGal)]{Diao/Lan/Liu/Zhu:lasfr} and \Refeq{\ref{eq-act-coord}}}.

For each $m \geq 1$, write $X_{K, m} = \Spa(R_{K, m}, R_{K, m}^+) := X_K \times_{\bE_K} (\bE_m)_K$, and write $Z_{K, m} := Z_K \times_{\bE_K} (\bE_m)_K = \Spa(\overline{R}_{K, m}, \overline{R}_{K, m}^+)$, for some uniquely determined complete Huber pairs $(R_{K, m}, R_{K, m}^+)$ and $(\overline{R}_{K, m}, \overline{R}_{K, m}^+)$.  Note that $\overline{R}_{K, m} \cong R_{K, m} / (T_1^{\frac{1}{m}}, \ldots, T_l^{\frac{1}{m}})$ \Pth{\Refcf{} Section \ref{sec-OBdl-explicit}}.  Let $\widehat{R}_{K, \infty}$ and $\widehat{\overline{R}}_{K, \infty}$ be the $p$-adic completions of $\varinjlim_m R_{K, m}$ and $\varinjlim_m \overline{R}_{K, m}$, respectively \Pth{\Refcf{} \cite[\aRem \logadicremunifcompllim]{Diao/Lan/Liu/Zhu:lasfr}}.  By Theorem \ref{thm-geom-decompl}, $\bigl(\{ R_{K, m} \}_{m \geq 1}, \widehat{R}_{K, \infty}, \Gamma_\geom \rtimes \Gal(K / k)\bigr)$ is a decompletion system.  Therefore, as in Definition \ref{def-decompl-syst}, $L_\infty = \cL_X(\widetilde{X})$ has a \emph{model} over $R_{K, m_0}$, for some $m_0 \geq 1$, \ie, a finite projective $R_{K, m_0}$-module $L_{m_0}(X_K)$, necessarily of rank $\rank_{\bQ_p}(\bL)$, with a continuous $R_{K, m_0}$-semilinear action of $\Gamma_\geom \rtimes \Gal(K / k)$ such that $L_{m_0}(X_K) \otimes_{R_{K, m_0}} \widehat{R}_{K, \infty} \cong \cL_X(\widetilde{X})$; and we may assume that it is \emph{good}, \ie, $H^i\bigl(\Gamma_\geom, L_{m_0}(X_K)\bigr) \Mi H^i\bigl(\Gamma_\geom, \cL_X(\widetilde{X})\bigr)$, for all $i \geq 0$.  Note that $\Gamma_\geom \cong (\widehat{\bZ}(1))^n$ acts on $\overline{R}_{K, m}$ via the last $n - l$ factors $\overline{\Gamma}_\geom \cong (\widehat{\bZ}(1))^{n - l}$ \Pth{see \Refeq{\ref{eq-def-gamma-j}}}.  By Theorem \ref{thm-geom-decompl} again, $\bigl(\{ \overline{R}_{K, m} \}_{m \geq 1}, \widehat{\overline{R}}_{K, \infty}, \overline{\Gamma}_\geom \rtimes \Gal(K / k)\bigr)$ is also a decompletion system.  Since $\cL_Z(\widetilde{Z}) \cong \cL_X(\widetilde{X}) \otimes_{\widehat{R}_{K, \infty}} \widehat{\overline{R}}_{K, \infty}$ by \cite[\aLem \logadiclemQplocclimm]{Diao/Lan/Liu/Zhu:lasfr},
\begin{equation}\label{eq-def-L-m-0-Z-K}
    L_{m_0}(Z_K) := L_{m_0}(X_K) \otimes_{R_{K, m}} \overline{R}_{K, m} \cong L_{m_0}(X_K) / (T_1^{\frac{1}{m_0}}, \ldots, T_l^{\frac{1}{m_0}})
\end{equation}
is a model of $\cL_Z(\widetilde{Z})$.  Up to enlarging $m_0$ \Pth{and replacing $L_{m_0}(X_K)$ with its base change, accordingly}, we may assume that $L_{m_0}(Z_K)$ is also a good model.

\begin{lemma}\label{lem-L-m-0-qunip}
    The $R_K$-linear representation of $\Gamma_\geom$ on $L_{m_0}(X_K)$ is quasi-unipotent; \ie, a finite-index subgroup of $\Gamma_\geom$ acts unipotently on $L_{m_0}(X_K)$.  By base change, the same holds for the $\overline{R}_K$-linear representation of $\Gamma_\geom$ on $L_{m_0}(Z_K)$.
\end{lemma}
\begin{proof}
    Let $k' := \widehat{k^\ur}$, where $k^\ur$ is the maximal unramified extension of $k$ in $\AC{k}$.  Let $k'_{p^l} := k'(\Grpmu_{p^l}) \subset \AC{k}$, $R'_{p^l} := R \ho_{k\Talg{T_1, \ldots, T_n}} \, k'_{p^l}\Talg{T_1^{\frac{1}{m_0}}, \ldots, T_n^{\frac{1}{m_0}}}$, and $\Gamma'_{p^l} := \Gal(K / k'_{p^l})$, for each $l \geq 0$.  Since $R_{K, m_0}$ is canonically isomorphic to the completion of $\varinjlim_l R'_{p^l}$, by Theorem \ref{thm-arith-decompl} and Remark \ref{rem-arith-decompl-DVF}, $(\{ R'_{p^l} \}_{l \geq 0}, R_{K, m_0}, \Gamma'_1)$ is a decompletion system.  By Definition \ref{def-decompl-syst} \Pth{with $L_\infty = L_{m_0}(X_K)$} and Remark \ref{rem-def-decompl-syst-equiv-cat}, since $\Gamma_\geom$ is topologically finitely generated, for some sufficiently large $l_0$, there exists an $R'_{p^{l_0}}$-submodule $L_{p^{l_0}}$ with a continuous $\Gamma_\geom \rtimes \Gamma'_1$-action and a canonical isomorphism $L_{p^{l_0}} \otimes_{R'_{p^{l_0}}} R_{K, m_0} \Mi L_{m_0}(X_K)$ of $\Gamma_\geom \rtimes \Gamma'_1$-modules.  Then the same argument as in the proof of \cite[\aLem 2.15]{Liu/Zhu:2017-rrhpl} works here.
\end{proof}

By Lemma \ref{lem-L-m-0-qunip}, we obtain decompositions
\begin{equation}\label{eq-L-m-0-tau}
    L_{m_0}(X_K) = \oplus_\tau \, L_{m_0, \tau}(X_K) \quad \Utext{and} \quad L_{m_0}(Z_K) = \oplus_\tau \, L_{m_0, \tau}(Z_K),
\end{equation}
where $\tau$ are characters of $\Gamma_\geom$ of finite order and the subscript \Qtn{$\tau$} denotes the maximal $K$-subspaces on which $\gamma - \tau(\gamma)$ acts nilpotently, for all $\gamma \in \Gamma$.  Then each $L_{m_0, \tau}(X_K)$ \Pth{\resp $L_{m_0, \tau}(Z_K)$} is a finite projective $R_K$-module \Pth{\resp $\overline{R}_K$-module} stable under the action of $\Gamma_\geom$.  Consider, in particular, the unipotent parts
\begin{equation}\label{eq-def-L-X-K-etc}
    L(X_K) := L_{m_0, 1}(X_K) \quad \Utext{and} \quad L(Z_K) := L_{m_0, 1}(Z_K).
\end{equation}
Up to enlarging $m_0$ as before, we may assume that the order of every $\tau$ in \Refeq{\ref{eq-L-m-0-tau}} divides $m_0$.  For each such $\tau$, there exists some monomial $T^{a_\tau}$ in $R_{K, m_0}$, with $a_\tau$ in $\frac{1}{m_0} \bZ_{\geq 0}^n$, on which $\Gamma_\geom$ acts via $\tau$.  Since all monomials $T^a$ with $a \in \frac{1}{m_0}\bZ^n$ are in $R_{K, m_0}[T_1^{-1}, \ldots, T_n^{-1}]$, it follows that
\begin{equation}\label{eq-L-m-0}
    L(X_K) \otimes_{R_K} R_{K, m_0}[T_1^{-1}, \ldots, T_n^{-1}] \cong L_{m_0}(X_K)[T_1^{-1}, \ldots, T_n^{-1}],
\end{equation}
and that the rank of $L(X_K)$ as a finite projective $R_K$-module is $\rank_{\bQ_p}(\bL)$.

\begin{rk}\label{rem-fail-surj-L-m-0}
    However, the natural map
    \begin{equation}\label{eq-fail-surj-L-m-0}
        L(X_K) \otimes_{R_K} R_{K, m_0} \to L_{m_0}(X_K)
    \end{equation}
    might not be an isomorphism in general.  This is the source of the failure of the surjectivity of \Refeq{\ref{eq-fail-surj}} mentioned in the introduction.
\end{rk}

\begin{rk}\label{rem-compat-L-m-0}
    In general, the two decompositions in \Refeq{\ref{eq-L-m-0-tau}} are not compatible via base change from $R_K$ to $\overline{R}_K \cong R_K / (T_1, \ldots, T_l)$.  Nevertheless, since the induced morphisms $L_{m_0, \tau}(X_K) / (T_1, \ldots, T_l) \to L_{m_0, \tau'}(Z_K)$ are zero whenever $\tau \neq \tau'$, we have a canonical surjection $L(X_K) / (T_1, \ldots, T_l) \Surj L(Z_K)$.
\end{rk}

For each $\tau \neq 1$, there exists some $j$ such that $\gamma_j - 1: L_{m_0, \tau}(Z_K) \to L_{m_0, \tau}(Z_K)$ is invertible, and so $H^i\bigl(\Gamma_\geom, L_{m_0, \tau}(Z_K)\bigr) = 0$, for all $i \geq 0$.  Hence, we have $H^i\bigl(\Gamma_\geom, L(Z_K)\bigr) \cong H^i\bigl(\Gamma_\geom, L_{m_0}(Z_K)\bigr)$, and the following lemma follows from essentially the same argument as in the proof of \cite[\aLem 2.9]{Liu/Zhu:2017-rrhpl}:
\begin{lemma}\label{lem-L-OCl-coh}
    There is a canonical $\Gal(K / k)$-equivariant isomorphism
    \[
        H^i({Z_\proket}_{/Z_K}, \widehat{\bL}_Z \otimes_{\widehat{\bQ}_p} \OClX{Z}) \cong
        \begin{cases}
            L(Z_K), & \Utext{if $i = 0$}; \\
            0, & \Utext{if $i > 0$}.
        \end{cases}
    \]
\end{lemma}

By Definition \ref{def-decompl-syst}, up to enlarging $m_0$, the formation of $L_{m_0}(Z_K)$ is compatible with base changes under compositions of rational embeddings and finite \'etale morphisms $Y \to Z$.  The same is true for the formation of the direct summands $L_{m_0, \tau}(Z_K)$ in the decomposition \Refeq{\ref{eq-L-m-0-tau}}.  These yield the following:
\begin{lemma}\label{lem-L-descent}
    The formation of the finite projective $\overline{R}_K$-module $L(Z_K)$, which is of rank $\rank_{\bQ_p}(\bL)$ when $Z = X$, is compatible with base changes under compositions of rational embeddings and finite \'etale morphisms $Y \to Z$.
\end{lemma}

Thus, we have established the statements \Refenum{\ref{prop-L-OCl-a}} and \Refenum{\ref{prop-L-OCl-b}} above, and completed the proofs of Proposition \ref{prop-L-OCl} and hence also of Theorems \ref{thm-log-RH-geom}\Refenum{\ref{thm-log-RH-geom-main}} and \ref{thm-log-Simp}\Refenum{\ref{thm-log-Simp-main}}.  \Pth{The cases where $Z \neq X$ will be also useful in Section \ref{sec-comp-nearby} and in \cite{Lan/Liu/Zhu:dcpdr}.}

Next, we move to the arithmetic situation.  We will only consider $Z = X$.
\begin{lemma}\label{lem-Ddl-coh}
    The sheaf $\Ddl(\bL)$ is a coherent sheaf on $X_\an$.
\end{lemma}
\begin{proof}
    For simplicity, we may still assume that $K = \widehat{k}_\infty$.  Again, to show the coherence of $\Ddl(\bL)$, we shall consider the projection $\nu: X_\proket \to X_\et$ instead, and we may assume that $X = \Spa(R, R^+)$ admits a smooth toric chart.  Note that this modified $\Ddl(\bL)$ is the sheaf on $X_\et$ associated with the presheaf
    \[
        Y \mapsto H^0\bigl({X_\proket}_{/Y}, \widehat{\bL} \otimes_{\widehat{\bQ}_p} \OBdl\bigr) = H^0\bigl(\Gal(K / k), \RHl(\bL)(Y)\bigr).
    \]
    From the proof of Theorem \ref{thm-log-RH-geom}\Refenum{\ref{thm-log-RH-geom-main}}, we know that
    \[
        \gr^r \RHl(\bL) \cong \mu'_*(\widehat{\bL} \otimes_{\widehat{\bQ}_p} \OCl)(r).
    \]
    It suffices to prove the following two statements \Pth{parallel to \Refenum{\ref{prop-L-OCl-a}} and \Refenum{\ref{prop-L-OCl-b}} above}:
    \begin{enumerate}[label=(\alph*$^\prime$), ref=\alph*$^\prime$]
        \item\label{lem-Ddl-coh-a}  The $R$-module $H^0\bigl(\Gal(K / k), \mu'_*(\widehat{\bL} \otimes_{\widehat{\bQ}_p} \OCl)(r)(X)\bigr)$ is finitely generated, and vanishes for $|r| \gg 0$.

        \item\label{lem-Ddl-coh-b}  If $Y = \Spa(S, S^+) \to X = \Spa(R, R^+)$ is a composition of rational localizations and finite \'etale morphisms, then we have a canonical isomorphism
            \[
            \begin{split}
                & H^0\bigl(\Gal(K / k), \mu'_*(\widehat{\bL} \otimes_{\widehat{\bQ}_p} \OCl)(r)(X)\bigr) \otimes_R S \\
                & \Mi H^0\bigl(\Gal(K / k), \mu'_*(\widehat{\bL} \otimes_{\widehat{\bQ}_p} \OCl)(r)(Y)\bigr).
            \end{split}
            \]
    \end{enumerate}

    Let $k_{p^\infty} := \cup_l \, k_{p^l}$ in $\AC{k}$, with $p$-adic completion $\widehat{k}_{p^\infty}$.  By assumption, $K = \widehat{k}_\infty$.  Hence, there are extensions $k \supset E_1 \subset E_2 \subset \bC_p := \widehat{\AC{\bQ}}_p$ over $\bQ_p$ such that $K \cong \widehat{k}_{p^\infty} \ho_{E_1} E_2$, and we can deduce from \cite[\aProps 4.1.1, 3.1.4, and 3.3.1]{Berger/Colmez:2008-frdrm} that $\bigl(L(X_K)(r)\bigr)^{\Gal(K / \widehat{k}_{p^\infty})} \cong \bigl(L(X_K)^{\Gal(K / \widehat{k}_{p^\infty})}\bigr)(r)$ is a finite projective $R_{\widehat{k}_{p^\infty}}$-module satisfying $\bigr(L(X_K)(r)\bigr)^{\Gal(K / \widehat{k}_{p^\infty})} \otimes_{R_{\widehat{k}_{p^\infty}}} R_K \cong L(X_K)(r)$, for all $r \in \bZ$.  Also, we have $\bigl(L(X_K)(r)\bigr)^{\Gal(K / \widehat{k}_{p^\infty})} \otimes_{R_{\widehat{k}_{p^\infty}}} S_{\widehat{k}_{p^\infty}} \cong \bigl(L(Y_K)(r)\bigr)^{\Gal(K / \widehat{k}_{p^\infty})}$, because $\bigl(L(X_K)(r)\bigr) \otimes_{R_K} S_K \cong L(Y_K)(r)$.  By Theorem \ref{thm-arith-decompl} and Corollary \ref{cor-decompl-char-twist} \Pth{with $\{ \psi_s \}$ there given by all powers of the cyclotomic character of $\Gal(k_{p^\infty} / k) \to \bZ_p^\times$}, the finite projective $R_{\widehat{k}_{p^\infty}}$-module $L(X_K)^{\Gal(K / \widehat{k}_{p^\infty})}$, with its induced action of $\Gal(\widehat{k}_{p^\infty} / k) \cong \Gal(k_{p^\infty} / k)$, descends to a finite projective $R_{k_{p^{l_0}}}$-module $L := L_{k_{p^{l_0}}}(X)$, for some $l_0 \geq 0$, such that $L(r)$ is a \emph{good model} \Pth{see Definition \ref{def-decompl-syst}\Refenum{\ref{def-decompl-syst-2}}} of $\bigl(L(X_K)(r)\bigr)^{\Gal(K / \widehat{k}_{p^\infty})}$, for all $r \in \bZ$, in the sense that $H^i\bigl(\Gal(k_{p^\infty} / k), L(r)\bigr) \Mi H^i\bigl(\Gal(k_{p^\infty} / k), \bigl(L(X_K)(r)\bigr)^{\Gal(K / \widehat{k}_{p^\infty})}\bigr)$, for all $i \geq 0$.  Consequently, we have
    \[
        H^0\bigl(\Gal(K / k), \mu'_*(\widehat{\bL} \otimes_{\widehat{\bQ}_p} \OCl)(r)(X)\bigr) \cong H^0\bigl(\Gal(k_{p^\infty} / k), L(r)\bigr),
    \]
    which is clearly a finitely generated $R$-module, and vanishes when $|r| \gg 0$.

    As for the statement \Refenum{\ref{lem-Ddl-coh-b}}, up to enlarging $l_0$, we may assume in addition that $L(r) \otimes_R S$ is a good model of $\bigl(L(Y_K)(r)\bigr)^{\Gal(K / \widehat{k}_{p^\infty})}$.  Hence, it suffices to show that
    \[
        H^0\bigl(\Gal(k_{p^\infty} / k), L(r)\bigr) \otimes_R S \Mi H^0\bigl(\Gal(k_{p^\infty}/ k), L(r) \otimes_R S\bigr).
    \]
    Thus, the desired base change property follows from the exactness of the complex $0 \to H^0(\Gal(k_{p^\infty} / k), L) \to L \xrightarrow{(\delta_1 - 1, \ldots, \delta_s - 1)} L^s$, where $\delta_1, \dots, \delta_s$ \Pth{for some $s \leq 2$, in fact} are topological generators of $\Gal(k_{p^\infty} / k)$, and from the flatness of $R \to S$.
\end{proof}

\begin{lemma}\label{lem-Ddl-reflex}
    The coherent sheaf $\Ddl(\bL)$ on $X_\an$ is reflexive.
\end{lemma}
\begin{proof}
    Being torsion-free, $\Ddl(\bL)$ is locally free outside some locus $X_0$ of codimension at least two in $X$.  Let $\jmath: X - X_0 \to X$ denote the canonical open immersion.  We claim that $\RHl(\bL) \cong \jmath_* \, \jmath^* \bigl(\RHl(\bL)\bigr)$.  Since $\RHl(\bL)$ is locally free, we may work locally and assume that it is isomorphic to $(\cO_X \ho_k \BdR)^n$ for some $n \geq 0$.  By using the filtration on $\cO_X \ho_k \BdR$ in Definition \ref{def-OXBdR}, it suffices to treat the case of $\cO_{X_K}$, which follows from \cite[\aCor 2.2.4]{Kisin:1999-afzlc}.  By taking $\Gal(K / k)$-invariants, we obtain a similar canonical isomorphism $\Ddl(\bL) \cong \jmath_* \, \jmath^* \bigl(\Ddl(\bL)\bigr)$.  Since $\Ddl(\bL)$ is coherent and $\jmath^* \Ddl(\bL)$ is locally free, it follows that $\Ddl(\bL)$ is reflexive, by the same argument as in the proof of \cite[\aProp 7]{Serre:1966-pfac}.
\end{proof}

\subsection{Calculation of residues}\label{sec-calc-res}

The main goal of this subsection is to prove Theorems \ref{thm-log-RH-geom}\Refenum{\ref{thm-log-RH-geom-res}}, \ref{thm-log-RH-arith}\Refenum{\ref{thm-log-RH-arith-main}}--\Refenum{\ref{thm-log-RH-arith-res}}, and \ref{thm-unip-vs-nilp}.

Let us first review the definition of residues for log connections and some basic properties.  We shall only consider the case where $X$ is as Example \ref{ex-log-adic-sp-ncd}, although the definition can be given more generally.  We first suppose that $F$ is a vector bundle on $X_\an$ equipped with an integrable log connection $\nabla: F \to F \otimes_{\cO_X} \Omega^{\log}_X$.

Let $Z \subset D$ be an \emph{irreducible component} \Pth{\ie, the image of a connected component of the normalization of $D$, as in \cite{Conrad:1999-icrs}}.  To define the residue $\Res_Z(\nabla)$ of $\nabla$ along $Z$, we may shrink $X$ and assume that $Z$ is smooth and connected.  Locally on $X$, up to enlarging $k$, we may assume that there is a smooth toric chart $X \to \bD^n$, where $\bD^n$ is as in Example \ref{ex-log-adic-sp-toric}, such that $Z = \{ T_1 = 0 \}$.  Let $\imath: Z \Em X$ denote the closed immersion, and let $F|_Z$ denote the $\cO$-module pullback $\imath^*(F)$.  Then there is an $\cO_Z$-linear endomorphism
\begin{equation}\label{eq-def-res}
    \Res_Z(\nabla) := \nabla (T_1\tfrac{\partial}{\partial T_1}) \bmod T_1: F|_Z \to F|_Z,
\end{equation}
where $T_1 \frac{\partial}{\partial T_1}$ denotes the dual of $\frac{dT_1}{T_1}$.  As in the classical situation, this operator does not depend on the choice of the coordinate $T_1$.  Also, its formation is compatible with rational localizations, and hence is a well-defined endomorphism of $F|_Z$.

Consider $Z$ as a smooth rigid analytic variety by itself, which is equipped with the normal crossings divisor $D' = \cup_j \, (D_j \cap Z)$, where the $D_j$'s are irreducible components of $D$ other than $Z$.  Then $Z$ admits the structure of a log adic space, defined by $D'$, as in Example \ref{ex-log-adic-sp-ncd}.  Again as in the classical situation, the pullback $F|_Z$ is equipped with a log connection $\nabla': F|_Z \to F|_Z \otimes_{\cO_Z} \Omega^{\log}_Z$, and the residue $\Res_Z(\nabla)$ is horizontal with respect to $\nabla'$.  As a result, the characteristic polynomial $P_Z(x)$ of $\Res_Z(\nabla)$ is constant over $Z$ and lies in $k_Z[x]$, where $k_Z$ is the algebraic closure of $k$ in $\Gamma(Z, \cO_Z)$.  Thus, the \emph{eigenvalues} of $\Res_Z(\nabla)$ \Pth{\ie, the \emph{roots} of $P_Z(x)$} are algebraic over $k$.  For each root $\alpha$ of $P_Z(x)$ in a finite extension $k'$ of $k$, let
\begin{equation}\label{eq-res-gen-eigen}
    F|_{Z \otimes_k k'}^\alpha \subset (F|_Z) \otimes_k k'
\end{equation}
be the corresponding generalized eigenspace of $\Res_Z(\nabla)$.  This is a direct summand \Pth{and hence a quotient} of $(F|_Z) \otimes_k k'$, which is preserved by the log connection $\nabla'$.

Given any vector bundle with an integrable log connection $(\cF, \nabla)$ on $\cX = (X_\an, \cO_X \ho_k \BdR)$, by using Lemma \ref{lem-OXBdR-mod}, the above discussions carry through.  Specifically, the coefficients of the characteristic polynomial $P_Z(x)$ of $\Res_Z(\nabla)$ are constant over $\cZ = (Z_\an, \cO_Z \ho_k \BdR)$, and therefore lie in $k_Z \otimes_k \BdR$, where $k_Z$ is as above.  In particular, if $k_Z = k$, then $P_Z(x) \in \BdR[x]$; and we can similarly define $(\cF|_{\cZ \otimes_{\BdR} B'}^\alpha, \nabla')$, for each root $\alpha$ of $P_Z(x)$ in a finite extension $B'$ of $\BdR$.

Now suppose that $F$ is a torsion-free coherent $\cO_X$-module equipped with an integrable log connection $\nabla$.  Let $U$ be the maximal open subset of $X$ such that $F|_U$ is a vector bundle, which is the complement of an analytic closed subvariety $X_0$ of $X$ of codimension at least two.  In particular, $X_0$ cannot contain any irreducible component of $D$.  Hence, by replacing $X$ with $U$, we can proceed as above and attach a polynomial $P_Z(x) \in k_Z[x]$ to each irreducible component $Z \subset D$.

Now we begin the proof of Theorem \ref{thm-log-RH-geom}\Refenum{\ref{thm-log-RH-geom-res}}.  Since the question is local, we may assume that $X = \Spa(R, R^+)$ is affinoid and admits a smooth toric chart $X \to \bD^n$; that the close subspace $Z_i = \{ T_i = 0 \}$ of $X$ is irreducible when nonempty, for each $1 \leq i \leq n$; and that $\RHl(\bL)$ is free.  By Proposition \ref{prop-OBdlp-loc} and Corollary \ref{cor-OBdlp-loc-gr},
\[
    \RHl(\bL)(X) \cong H^0\bigl(\Gamma_\geom, (\widehat{\bL} \otimes_{\widehat{\bQ}_p} \BBdR)(\widetilde{X})\{W_1, \ldots, W_n\}\bigr),
\]
where each $W_j = t^{-1} \monon{j}$ is defined as in \Refeq{\ref{eq-def-W-i}}.  Let $N_\infty := (\widehat{\bL} \otimes_{\widehat{\bQ}_p} \BBdRp)(\widetilde{X})$, which is a module over $\BBdRp(\widetilde{X}) \cong \BBdRp(\widehat{R}_\infty, \widehat{R}_\infty^+)$ \Pth{see Proposition \ref{prop-BdR-van}}.  Then we have
\[
\begin{split}
    & \Fil^0(\widehat{\bL} \otimes_{\widehat{\bQ}_p} \OBdl)(\widetilde{X}) \cong N_\infty\{W_1, \ldots, W_n\} \\
    & = \Bigl\{ \sum_{\Lambda \in \bZ_{\geq 0}^n} b_\Lambda \, W^\Lambda : \Utext{$b_\Lambda \in N_\infty$, $b_\Lambda \rightarrow 0$, $t$-adically, as $|\Lambda| \rightarrow \infty$} \Bigr\} \\
    & = \Bigl\{ \sum_{\Lambda \in \bZ_{\geq 0}^n} c_\Lambda \tbinom{W}{\Lambda} : \Utext{$c_\Lambda \in N_\infty$, $c_\Lambda \rightarrow 0$, $t$-adically, as $|\Lambda| \rightarrow \infty$} \Bigr\},
\end{split}
\]
where $W^\Lambda$ is as in \Refeq{\ref{eq-def-W-I}}, and $\binom{W}{\Lambda} := \binom{W_1}{\Lambda_1} \cdots \binom{W_n}{\Lambda_n}$, for each $\Lambda = (\Lambda_1, \ldots, \Lambda_n)$.

Recall that we have chosen the topological basis $\{ \gamma_1, \ldots, \gamma_n \}$ of $\Gamma_\geom \cong (\widehat{\bZ}(1))^n$ satisfying the characterizing property \Refeq{\ref{eq-def-gamma-j}}.  For each $\Lambda = (\Lambda_1, \ldots, \Lambda_n) \in \bZ_{\geq 0}^n$, let us write $(\gamma - 1)^\Lambda$ for $(\gamma_1 - 1)^{\Lambda_1} \cdots (\gamma_n - 1)^{\Lambda_n}$.
\begin{lemma}\phantomsection\label{lem-Gamma-geom-inv}
    \begin{enumerate}
        \item\label{lem-Gamma-geom-inv-1}  If $\sum c_\Lambda \binom{W}{\Lambda} \in N_\infty\{W_1, \ldots, W_n\}$ is $\Gamma_\geom$-invariant, then $(\gamma - 1)^\Lambda c_0 \rightarrow 0$, $t$-adically, as $|\Lambda| \rightarrow \infty$, and $c_\Lambda = (\gamma - 1)^\Lambda c_0$ for all $\Lambda \in \bZ_{\geq 0}^n$.

        \item\label{lem-Gamma-geom-inv-2}  Let
            \begin{equation}\label{eq-def-N+}
                N^+ := \bigl\{ c \in N_\infty : \Utext{$(\gamma - 1)^\Lambda c \rightarrow 0$, $t$-adically, as $|\Lambda| \rightarrow \infty$} \bigr\}.
            \end{equation}
            Then the map $N_\infty\{W_1, \ldots, W_n\} \to N_\infty$ sending all $W_1, \ldots, W_n$ to zero induces a canonical isomorphism
            \begin{equation}\label{eq-RHlp-N+}
                \eta: \RHlp(\bL)(X) \cong (N_\infty\{W_1, \ldots, W_n\})^{\Gamma_\geom} \cong N^+,
            \end{equation}
            with the inverse map given by $c \mapsto \sum_{\Lambda \in \bZ_{\geq 0}^n} (\gamma - 1)^\Lambda c \, \binom{W}{\Lambda}$.

        \item\label{lem-Gamma-geom-inv-3}  Let $N := N^+ \otimes_{\BdRp} \BdR \cong N^+[t^{-1}]$.  Then the above isomorphism $\eta$ induces a canonical isomorphism $\RHl(\bL)(X) \cong N$, which we still denote by $\eta$.
    \end{enumerate}
\end{lemma}
\begin{proof}
    We have $\gamma_i^{-1} W_j = W_j + \delta_{ij}$.  \Pth{Note that the $W_j$ defined in Corollary \ref{cor-OBdlp-loc-gr} differs from the $V_j$ defined in the proof of \cite[\aProp 6.16]{Scholze:2013-phtra} by a sign, and therefore we need $\gamma_i^{-1}$ in our formula rather than the $\gamma_i$ as in \cite[\aLem 6.17]{Scholze:2013-phtra}.}  This implies that $\gamma_i^{-1} \binom{W_i}{j} = \binom{W_i + 1}{j} = \binom{W_i}{j} + \binom{W_i}{j - 1}$, and so $(\gamma_i^{-1} - 1) \bigl( \sum_\Lambda c_\Lambda \binom{W}{\Lambda} \bigr) = \sum_\Lambda \bigl( \gamma_i^{-1} c_{\Lambda + e_i} + \gamma_i^{-1} c_\Lambda - c_\Lambda \bigr) \binom{W}{\Lambda}$, where $e_i = (0, \ldots, 0, 1, 0, \ldots, 0)$ has only the $i$-th entry equal to 1.  Therefore, $c_\Lambda - \gamma_i^{-1} c_\Lambda = \gamma_i^{-1} c_{\Lambda + e_i}$, or, equivalently, $\gamma_i c_\Lambda - c_\Lambda = c_{\Lambda + e_i}$, for all $i$ and $\Lambda$.  In particular, this implies that $(\gamma - 1)^\Lambda c_0 = c_\Lambda$, which goes to 0 as $|\Lambda| \rightarrow \infty$.  This proves \Refenum{\ref{lem-Gamma-geom-inv-1}}.  Then \Refenum{\ref{lem-Gamma-geom-inv-2}} and \Refenum{\ref{lem-Gamma-geom-inv-3}} also follow easily.
\end{proof}

By the proof of Theorem \ref{thm-log-Simp}\Refenum{\ref{thm-log-Simp-main}} in Section \ref{sec-coh}, $\RHlp(\bL)(X)$ is a finite projective $R \ho_k \BdRp$-module.  Note that the natural action of $\Gamma_\geom$ on $N_\infty$ preserves $N^+$, and by transport of structure gives an action of $\Gamma_\geom$ on $\RHl(\bL)(X)$.  This action is closely related to the residues, as we shall see.  Recall from Lemma \ref{lem-str-BdR} that, if we define the $\Gamma_\geom$-action of $R \ho_k \BdR$ by requiring that $\gamma_i(T_j) = [\epsilon]^{\delta_{ij}} T_j$ and that the action becomes trivial modulo $\xi$, then the embedding $R \ho_k \BdR \to \BBdR(\widehat{R}_\infty, \widehat{R}_\infty^+)$ sending $T_i$ to $[T_i^\flat]$ is $\Gamma_\geom$-equivariant.  Via this embedding, we may regard $N_\infty$ as an $R \ho_k \BdRp$-module, and $N^+$ as an $R \ho_k \BdRp$-submodule of $N_\infty$.

\begin{lemma}\label{lem-RHlp-N+}
    The isomorphism \Refeq{\ref{eq-RHlp-N+}} is an isomorphism of $R \ho_k \BdRp$-modules, where the actions of $R \ho_k \BdRp$ on $\RHlp(\bL)(X)$ and $N_\infty$ are as explained above.
\end{lemma}
\begin{proof}
    This follows from the fact that the map \Refeq{\ref{eq-str-BdR}} \Pth{which is an isomorphism by Lemma \ref{lem-str-BdR}} is obtained from the map \Refeq{\ref{eq-lem-O-str-BdR}} via $\mono{a} \mapsto 1$, for all $a \in P$.
\end{proof}

Hence, the action of $\gamma_i$ on $N^+ / T_i N^+$ is $(R / T_i) \ho_k \BdRp$-linear, and induces an $(R / T_i) \ho_k \BdR$-linear action on $N / T_i N$.

\begin{lemma}\label{lem-res-explicit}
    Under the isomorphism $\RHl(\bL)(X) \cong N$ given by Lemma \ref{lem-Gamma-geom-inv}, the residue of the connection $\nabla_\bL$ of $\RHl(\bL)$ along $Z_i = \{ T_i = 0 \}$ corresponds to the endomorphism $t^{-1} \log(\gamma_i)$ of $N / T_i N$.
\end{lemma}
\begin{proof}
    Let us expand elements of $N_\infty\{W_1, \ldots, W_n\}$ in the basis $\{ W^\Lambda \}_\Lambda$ instead of $\{ \binom{W}{\Lambda} \}_\Lambda$.  Suppose that $c_0 \in N$ and $\eta^{-1}(c_0) = \sum_\Lambda c_\Lambda \binom{W}{\Lambda} = \sum_\Lambda b_\Lambda \, W^\Lambda$.  Then, by the definition of residues as in \Refeq{\ref{eq-def-res}}, by Lemma \ref{lem-Gamma-geom-inv}, and by \Refeq{\ref{eq-conn-X-i}} and \Refeq{\ref{eq-conn-W-i}}, we obtain the identities $\eta\bigl((\Res_{Z_i}(\nabla_\bL))(\eta^{-1}(c_0))\bigr) = t^{-1} b_{e_i} = t^{-1} \sum_{a = 1}^\infty (-1)^{a - 1} \frac{1}{a} c_{a e_i} = t^{-1} \sum_{a = 1}^\infty (-1)^{a - 1} \frac{1}{a} (\gamma_i - 1)^a c_0 = t^{-1} \log(\gamma_i) (c_0)$, as desired.
\end{proof}

\begin{rk}\label{rem-res-explicit}
    The definitions of both $t$ and $\gamma_i$ \Pth{in \Refeq{\ref{eq-choice-t}} and \Refeq{\ref{eq-def-gamma-j}}} depend on the choice of $\zeta: \bQ / \bZ \Mi \Grpmu_\infty$ in \Refeq{\ref{eq-zeta}}, but $t^{-1} \log(\gamma_i)$ does not.
\end{rk}

To proceed further, we need the following lemma, which follows from \cite[\aLem \logadiclemQplocclimm]{Diao/Lan/Liu/Zhu:lasfr} by induction on $r$.
\begin{lemma}\label{lem-BdR-restr-cl-imm}
    Let $\imath: Z \to Y$ be a strict closed immersion of locally noetherian fs log adic spaces over $\Spa(\bQ_p, \bZ_p)$.  Let $\widehat{\bM}$ be a $\widehat{\bQ}_p$-local system on $Y_\proket$.  Then
    \[
    \begin{split}
        & \bigl(\imath_\proket^{-1}(\widehat{\bM}) \otimes_{\widehat{\bQ}_p} (\bB_{\dR, Z} / \xi^r)\bigr)(U \times_Y Z) \\
        & \cong \bigl(\widehat{\bM} \otimes_{\widehat{\bQ}_p} (\bB_{\dR, Y} / \xi^r)\bigr)(U) \otimes_{(\bB_{\dR, Y} / \xi^r)(U)} (\bB_{\dR, Z} / \xi^r)(U \times_Y Z).
    \end{split}
    \]
    for every $r \geq 1$ and every log affinoid perfectoid object $U$ of $Y_\proket$.
\end{lemma}

Let $Z_i$ be the \Pth{possibly empty} smooth divisor on $X$ defined by $T_i = 0$.  Equip $Z_i$ with the pullback of the log structure of $X$, and denote by $Z_i^\partial$ log adic space thus obtained.  Then the canonical morphism of log adic spaces $\imath: Z_i^\partial \Em X$ is a strict closed immersion.  Consider the log affinoid perfectoid object $\widetilde{Z}_i^\partial := Z_i^\partial \times_X \widetilde{X} \cong Z_i \times_{\bD^n} \widetilde{\bD}^n$ in $(Z_i^\partial)_\proket$ \Pth{as in Corollary \ref{cor-OBdRp-loc-cl-imm}}, with associated perfectoid space $\widehat{\widetilde{Z}}{}_i^\partial$.  By \Refeq{\ref{eq-cor-OBdRp-loc-cl-imm-BdR-mod-xi-r}} and Lemma \ref{lem-BdR-restr-cl-imm}, we have a canonical isomorphism
\begin{equation}\label{lem-L-B-dR-Z-i}
    \bigl(N_\infty / \xi^r\bigr) \big/ ([T_i^{s\flat}])^\wedge_{s \in \bQ_{> 0}} \cong \bigl(\imath_\proket^{-1}(\widehat{\bL}) \otimes_{\widehat{\bQ}_p} \BBdRpX{Z^\partial_i}\bigr)(\widetilde{Z}_i^\partial) \big/ \xi^r.
\end{equation}

Let $\AC{B}_\dR$ denote a fixed algebraic closure of $\BdR$ extending the fixed algebraic closure $\AC{k}$ of $k$, and let $\AC{B}_\dR^+$ denote the integral closure of $\BdRp$ in $\AC{B}_\dR$.
\begin{lemma}\label{lem-poss-eval-N-infty-xi-r}
    If $\gamma_i v = x v$ for some nonzero $v \in \bigl(N_\infty / \xi^r\bigr) \big/ ([T_i^{s\flat}])^\wedge_{s \in \bQ_{> 0}}$ and some $x \in \AC{B}_\dR^+ / \xi^r$, then $x = \zeta^y$ for some $y \in \bQ$.
\end{lemma}
\begin{proof}
    Up to replacing $k$ with a finite extension, we may assume that $Z_i$ contains a $k$-point $z$.  Let $z^\partial$ denote $z$ equipped with the log structure pulled back from $X$.  Let $\widetilde{z}^\partial := z^\partial \times_X \widetilde{X}$, with associated perfectoid space $\widehat{\widetilde{z}}{}^\partial$.  Then $\gamma_i v = x v$ still holds in the base change of $\bigl(N_\infty / \xi^r\bigr) \big/ ([T_i^{s\flat}])^\wedge_{s \in \bQ_{> 0}}$ along $(\BBdRpX{Z_i^\partial} / \xi^r)(\widetilde{Z}^\partial_i) \to (\BBdRpX{z^\partial} / \xi^r)(\widetilde{z}^\partial)$, which is isomorphic to $\bigl(\widehat{\bL}|_{z^\partial} \otimes_{\widehat{\bQ}_p} (\BBdRpX{z^\partial} / \xi^r)\bigr)(\widetilde{z}^\partial)$, by Lemma \ref{lem-BdR-restr-cl-imm}.  Note that $\gamma_i$ acts trivially on $\widehat{\widetilde{z}}{}^\partial$.  Hence, $\widehat{\bL}|_{\widehat{\widetilde{z}}{}^\partial}$ is equipped with an action of $\gamma_i$, which is quasi-unipotent because it extends to a continuous one of $\widehat{\bZ}(1) \rtimes \Gal(k_\infty / k)$ \Pth{see \Refeq{\ref{eq-fund-grp-Gamma}}}, and the same argument as in the proof of \cite[\aLem 2.15]{Liu/Zhu:2017-rrhpl} also works here.
\end{proof}

Next we use the decompletion over $\BBdRp / \xi^r$ established in Section \ref{sec-geom-tower-deform} to descend $N_\infty / \xi^r = (\widehat{\bL} \otimes_{\widehat{\bQ}_p} \BBdRp)(\widetilde{X}) / \xi^r$ to some finite level.

\begin{lemma}\label{lem-decompl-BdR-mod}
    For each $r \geq 1$, there exist some $m \geq 1$ and a finite projective $\bB_{r, m}$-module $N_{r, m}$ \Pth{where $\bB_{r, m}$ is as in \Refeq{\ref{eq-def-B-r-m}}}, equipped with a semilinear $\Gamma$-action, such that $N_\infty / \xi^r \cong N_{r, m} \otimes_{\bB_{r, m}} \bigl(\BBdRp(\widetilde{X}) / \xi^r\bigr)$ as $\BBdRp(\widetilde{X}) / \xi^r$-modules with semilinear $\Gamma$-actions.  In addition, up to replacing $m$ with a multiple \Pth{and replacing $N_{r, m}$ with its base change, accordingly}, we may assume that $N^+ / \xi^r \subset N_{r, m}$ \Pth{as submodules of $N_\infty / \xi^r$}, and that $N_{r, m} \big/ (N^+ / \xi^r)$ is $T_i$-torsion-free.
\end{lemma}
\begin{proof}
    The first statement follows from Lemma \ref{lem-str-BdR} and Theorem \ref{thm-BdR-decompl}.  As for the second statement, we may assume that $H^j(\Gamma, N_{r, m}) \to H^i(\Gamma, N_\infty / \xi^r)$ is an isomorphism for $j = 0, 1$ \Pth{by Definition \ref{def-decompl-syst}\Refenum{\ref{def-decompl-syst-2}}}, so that $H^0(\Gamma, (N_\infty / \xi^r) / N_{r, m}) = 0$.  Then the whole $\Gamma$ acts unipotently on each element of $N^+ / \xi^r$ \Pth{by \Refeq{\ref{eq-def-N+}}}, while each nonzero element of $(N_\infty / \xi^r) / N_{r, m}$ lies outside the kernel of $\gamma - 1$ for some $\gamma \in \Gamma$.  It follows that $N^+ / \xi^r \subset N_{r, m} \subset N_\infty / \xi^r$, as desired.  As for the last statement, it suffices to show that $(N_\infty / \xi^r) \big/ (N^+ / \xi^r)$ is $T_i$-torsion-free.  By the definition of $N_\infty$, and by \Refeq{\ref{eq-P-to-Ainf}}, $N_\infty / \xi^r$ is $T_i$-torsion-free.  By the definition of $N^+ / \xi^r$, it remains to note that, for each $c \in N_\infty / \xi^r$, we have $(\gamma - 1)^\Lambda (T_i \, c) = 0$ for some $\Lambda \in \bZ_{\geq 0}^n$ if and only if $(\gamma - 1)^{\Lambda'}(c) = 0$ for some $\Lambda' \in \bZ_{\geq 0}^n$, since $(\gamma_j - 1) (T_i \, c) = [\epsilon]^{\delta_{ij}} \, T_i \, \gamma_i(c) - T_i \, c = T_i \bigl(([\epsilon]^{\delta_{ij}} - 1) \gamma_i + (\gamma_i - 1)\bigr) c$ and $[\epsilon] - 1 \in (\xi) \subset \BdRp$.
\end{proof}

\begin{lemma}\label{lem-poss-eval-N+-xi-r-T}
    If $\gamma_i v = x v$ for some nonzero $v \in (N^+ / \xi^r) \big/ (T_i)$ and some $x \in \AC{B}_\dR^+ / \xi^r$, then $x = \zeta^y [\epsilon^z]$ for some $y \in \bQ$ and $z \in \bQ \cap [0, 1)$.
\end{lemma}
\begin{proof}
    By Lemma \ref{lem-decompl-BdR-mod}, we may assume that $v \in (N^+ / \xi^r) \big/ (T_i) \subset N_{r, m} \big/ T_i N_{r, m}$ for some $m$ and $N_{r, m}$.  Consider the filtration $T_i^{\frac{a}{m}} N_{r, m} \big/ T_i N_{r, m} \subset N_{r, m} \big/ T_i N_{r, m}$, with $0 \leq a \leq m$.  Since $v \neq 0$, there exists some $0 \leq a < m$ such that the image $\overline{v}$ of $v$ in $T_i^{\frac{a}{m}} N_{r, m} \big/ T_i^{\frac{a + 1}{m}} N_{r, m}$ is nonzero, which also satisfies $\gamma_i \overline{v} = x \overline{v}$.  By Lemma \ref{lem-decompl-BdR-mod} again, the natural embedding $N_{r, m} \Em N_\infty / \xi^r$ induces by restriction to $T_i^{\frac{a}{m}} N_{r, m}$ and by factoring out a multiplication by $T_i^{\frac{a}{m}}$ a well-defined embedding $T_i^{\frac{a}{m}} N_{r, m} \big/ T_i^{\frac{a + 1}{m}} N_{r, m} \Em (N_\infty / \xi^r) \big/ ([T_i^{s\flat}])^\wedge_{s \in \bQ_{> 0}}$, and the nonzero image $w$ of $\overline{v}$ in $(N_\infty / \xi^r) \big/ ([T_i^{s\flat}])^\wedge_{s \in \bQ_{> 0}}$ satisfies the twisted relation $\gamma_i w = [\epsilon^{-\frac{a}{m}}] x w$ because $\gamma_i \, T_i^{\frac{1}{m}} = [\epsilon^{\frac{1}{m}}] \, T_i^{\frac{1}{m}}$ \Pth{\Refcf{} \Refeq{\ref{eq-Gamma-act-mono-alg}}}.  Thus, by Lemma \ref{lem-poss-eval-N-infty-xi-r}, we have $x = \zeta^y [\epsilon^z]$ with $y \in \bQ$ and $z := \frac{a}{m} \in \frac{1}{m}\bZ \cap [0, 1) \subset \bQ \cap [0, 1)$, as desired.
\end{proof}

Finally, let us finish the proof of Theorem \ref{thm-log-RH-geom}\Refenum{\ref{thm-log-RH-geom-res}}.  Let $k_{Z_i}$ be the algebraic closure of $k$ in $\Gamma(Z_i, \cO_{Z_i})$.  For our purpose, we may replace $k$ with a finite extension, replace $X$ with an open subspace, and replace $Z_i$ accordingly, so that $k_{Z_i} = k$ and the eigenvalues of the residue along $Z_i$ are in $\AC{B}_\dR$.  By Lemma \ref{lem-res-explicit}, it suffices to show that the eigenvalues of $t^{-1} \log(\gamma_i)$ are in $\bQ \cap [0, 1)$.  Let $v \in N / T_i N$ be an eigenvector of $t^{-1} \log(\gamma_i)$ with eigenvalue $\widetilde{x} \in \AC{B}_\dR$.  Up to multiplying $v$ by a power of $t$, we may assume that $v \in N^+ / T_i N^+$.  Since the action of $\gamma_i$ on $N^+ / T_i N^+$ is $\BdRp$-linear, $v$ is an eigenvector of $\gamma_i$ with eigenvalue $\exp(\widetilde{x} t)$ in $\AC{B}_\dR^+$.  By Lemma \ref{lem-poss-eval-N+-xi-r-T}, and by the assumption that $(\gamma_i - 1)^l \, v \rightarrow 0$, $t$-adically, as $l \rightarrow \infty$, it is easy to see that $\exp(\widetilde{x} t) = \zeta^y [\epsilon^z]$, with $y + z \in \bZ$ and $z \in \bQ \cap [0, 1)$.  Therefore, we may assume that $-y = z \in \bQ \cap [0, 1)$.  Thus, the eigenvalues of $t^{-1} \log(\gamma_i)$ are of the form $t^{-1} \log(\zeta^{-z}[\epsilon^z]) = z \in \bQ \cap [0, 1)$, which verifies Theorem \ref{thm-log-RH-geom}\Refenum{\ref{thm-log-RH-geom-res}}, as desired.

\begin{rk}\label{rem-compat-L-m-0-res}
    By the proof Lemma \ref{lem-poss-eval-N+-xi-r-T}, the surjection $L(X_K) / (T_1) \Surj L(Z_K)$ in Remark \ref{rem-compat-L-m-0} \Pth{when $l = 1$ there} is the evaluation on $\widetilde{X}$ of $\bigl(\gr^0  \RHl(\bL)\bigr)|_{D_1} \to \gr^0 \bigl(\RHl(\bL)|_{D_1}^0\bigr)$ \Pth{\Refcf{} \Refeq{\ref{eq-res-gen-eigen}}}, where $\RHl(\bL)|_{D_1}^0$ is interpreted as a quotient of $\RHl(\bL)|_{D_1}$ and equipped with the canonically induced filtration.
\end{rk}

\begin{prop}\label{prop-res-Ddl}
    For the connection $\nabla_\bL: \Ddl(\bL) \to \Ddl(\bL) \otimes_{\cO_X} \Omega^{\log}_X$, all eigenvalues of $\Res_{\{ T_i = 0 \}}(\nabla_\bL)$ are in $\bQ \cap [0, 1)$.
\end{prop}
\begin{proof}
    Note that $\Ddl(\bL)(X) \cong \RHl(\bL)(\widetilde X)^{\Gal(K / k)}$, and the isomorphism $\eta$ in \Refeq{\ref{eq-RHlp-N+}} is $\Gal(K / k)$-equivariant.  Therefore, $\Ddl(\bL)(X) \cong N^{\Gal(K / k)}$.  In addition, the residue $\Res_{\{ T_i = 0 \}}(\nabla_\bL)$ is still given by $t^{-1} \log(\gamma_i)$ as in Lemma \ref{lem-res-explicit}.  Then the proposition follows from the arguments just explained above.
\end{proof}

In order to complete the proof of Theorem \ref{thm-log-RH-arith}\Refenum{\ref{thm-log-RH-arith-main}}, it remains to apply the following proposition to conclude that $\Ddl(\bL)$ is a vector bundle.

\begin{prop}\label{prop-loc-free}
    A torsion-free coherent $\cO_X$-module $F$ with an integrable log connection $\nabla: F \to F \otimes_{\cO_X} \Omega^{\log}_X$ is locally free when the following conditions hold:
    \begin{enumerate}
        \item\label{prop-loc-free-1}  $F$ is reflexive \Pth{\ie, isomorphic to its bidual}.

        \item\label{prop-loc-free-2}  For every $i$, all eigenvalues of $\Res_{\{ T_i = 0 \}}(\nabla)$ are in $\bQ \cap [0, 1)$.
    \end{enumerate}
\end{prop}
\begin{proof}
    This follows from the same argument as in the proof of \cite[\aCh 1, \aProp 4.5]{Andre/Baldassarri:2001-DDA} or \cite[\aLem 11.5.1]{Andre/Baldassarri/Cailotto:2020-DDA(2)}.  More precisely, it suffices to note that, under the assumptions, the completion of the stalk of $\cE$ at each classical point of $X$ is free, by \cite[\aCh 1, \aLem 4.6.1]{Andre/Baldassarri:2001-DDA} or \cite[the proof of \aLem 11.5.1]{Andre/Baldassarri/Cailotto:2020-DDA(2)}.
\end{proof}

To apply Proposition \ref{prop-loc-free} to $\Ddl(\bL)$, it suffices to note that the condition \Refenum{\ref{prop-loc-free-1}} is satisfied by Lemma \ref{lem-Ddl-reflex}, and the condition \Refenum{\ref{prop-loc-free-2}} is satisfied by Proposition \ref{prop-res-Ddl}.  The proof of Theorem \ref{thm-log-RH-arith}\Refenum{\ref{thm-log-RH-arith-main}} is now complete.

\begin{prop}\label{prop-conn-isom-res}
    Suppose that $F$ \Pth{\resp $F'$} is a locally free \Pth{\resp torsion-free} coherent $\cO_X$-module, with an integrable log connection $\nabla$ \Pth{\resp $\nabla'$} as in Definition \ref{def-log-conn-etc}\Refenum{\ref{def-log-conn-etc-4}}, whose residues along the irreducible components of $D$ all have eigenvalues in $\bQ \cap [0, 1)$.  Then any morphism $(F, \nabla) \to (F', \nabla')$ whose restriction to $U = X - D$ is an isomorphism is necessarily an isomorphism over the whole $X$.  The same is true if we replace $\cO_X$-modules with $\cO_X \ho_k \BdR$-modules.
\end{prop}
\begin{proof}
    Let $(F'', \nabla'') := ((F', \nabla')^\dualsign)^\dualsign$, where $F''$ is the double $\cO_X$-dual of $F'$, which is by definition a reflexive coherent $\cO_X$-module, and where $\nabla''$ is the induced log integrable connection, whose residues along the irreducible components of $D$ also have eigenvalues in $\bQ \cap [0, 1)$.  Hence, by Proposition \ref{prop-loc-free}, $F''$ is \emph{locally free}.  Since the restriction of the given morphism $(F, \nabla) \to (F', \nabla')$ to the dense subspace $U$ is an isomorphism, we have injective morphisms $(F, \nabla) \to (F', \nabla') \to (F'', \nabla'')$, and it suffices to show that their composition is an isomorphism over $X$.  Therefore, we can replace $F'$ with $F''$, and assume that both $F$ and $F'$ are locally free.  Thus, by working locally, we may replace $X$ with its affinoid open subspaces which admit strictly \'etale morphisms to $\bD^n$ as in Example \ref{ex-log-adic-sp-ncd}, and assume that both $F$ and $F'$ are free of rank $d$.  Then, with respect to the chosen bases, the map $F \to F'$ is represented by a matrix $A$ in $M_d\bigl(\cO_X(X)\bigr)$, which is invertible outside $D$.  In order to show that it is invertible over $X$, it suffices to show that the entries of $A^{-1}$, which are a priori analytic functions on $X$ \emph{meromorphic} along $D$, are everywhere regular analytic functions on $X$.  But this is classical---see, for example, the proof of \cite[\aCh 1, \aProp 4.7]{Andre/Baldassarri:2001-DDA} or \cite[the uniqueness assertion of \aThm 11.2.2 in \aSec 11.4, based on \aProp-\aDef 10.2.5]{Andre/Baldassarri/Cailotto:2020-DDA(2)}.  Moreover, by Lemma \ref{lem-OXBdR-mod}, the above arguments also apply to integrable log connections on $\cX$ \Pth{as in Definition \ref{def-log-conn-etc}\Refenum{\ref{def-log-conn-etc-1}}}.
\end{proof}

As usual, we define a decreasing filtration on $\Ddl(\bL)$ by setting
\[
    \Fil^\bullet \Ddl(\bL) := \bigl(\Fil^\bullet \RHl(\bL)\bigr)^{\Gal(K / k)}.
\]

\begin{lemma}\label{lem-Ddl-to-RHl-fil-strict}
    We endow $\Ddl(\bL) \ho_k \BdR$ with the usual product filtration.  Then the canonical morphism
    \begin{equation}\label{eq-lem-Ddl-to-RHl-fil-strict}
        \Ddl(\bL) \ho_k \BdR \to \RHl(\bL)
    \end{equation}
    defined by adjunction is injective \Pth{by definition} and strictly compatible with the filtrations on both sides.  That is, for each $r$, \Refeq{\ref{eq-lem-Ddl-to-RHl-fil-strict}} induces an injective morphism
    \begin{equation}\label{eq-lem-Ddl-to-RHl-fil-strict-gr}
        \gr^r \bigl( \Ddl(\bL) \ho_k \BdR \bigr) \to \gr^r \RHl(\bL).
    \end{equation}
\end{lemma}
\begin{proof}
    Since $\Ddl(\bL) \cong (\RHl(\bL))^{\Gal(K / k)}$, the left-hand side of \Refeq{\ref{eq-lem-Ddl-to-RHl-fil-strict-gr}} can be identified with $\oplus_{a + b = r} \, \Bigl( \bigl(\gr^a \bigl((\RHl(\bL))^{\Gal(K / k)}\bigr) \bigr) \otimes_k K(b) \Bigr)$, while the right-hand side of \Refeq{\ref{eq-lem-Ddl-to-RHl-fil-strict-gr}} contains $\oplus_{a + b = r} \, \Bigl( \bigl(\gr^a \RHl(\bL)\bigr)^{\Gal(K / k)} \otimes_k K(b) \Bigr)$ as a subspace, where we have direct sums in such forms because of the $\Gal(K / k)$-actions.  Thus, it suffices to note that the canonical morphism $\gr^a \bigl( (\RHl(\bL))^{\Gal(K / k)} \bigr) \to \bigl(\gr^a \RHl(\bL)\bigr)^{\Gal(K / k)}$ is injective, for each $a$, essentially by definition.
\end{proof}

\begin{cor}\label{cor-Ddl-to-RHl-fil-strict}
    If $\bL|_{U_\et}$ is de Rham, then \Refeq{\ref{eq-lem-Ddl-to-RHl-fil-strict}} is an isomorphism of vector bundles on $\cX$, compatible with the log connections and filtrations on both sides.
\end{cor}
\begin{proof}
     Since $\bL|_{U_\et}$ is de Rham, by \cite[\aCor 3.12(ii)]{Liu/Zhu:2017-rrhpl}, the restriction of \Refeq{\ref{eq-lem-Ddl-to-RHl-fil-strict}} to $U$ is an isomorphism.  By Proposition \ref{prop-conn-isom-res} and Theorems \ref{thm-log-RH-geom}\Refenum{\ref{thm-log-RH-geom-res}} and \ref{thm-log-RH-arith}\Refenum{\ref{thm-log-RH-arith-main}}, the morphism \Refeq{\ref{eq-lem-Ddl-to-RHl-fil-strict}} is an isomorphism, compatible with the log connections.  By Lemma \ref{lem-Ddl-to-RHl-fil-strict}, it is also compatible with the filtrations.
\end{proof}

\begin{cor}\label{cor-Ddl-gr-vec-bdl}
    If $\bL|_{U_\et}$ is de Rham, then $\gr \Ddl(\bL)$ is a vector bundle of rank $\rank_{\bQ_p}(\bL)$.
\end{cor}
\begin{proof}
    By Corollary \ref{cor-Ddl-to-RHl-fil-strict}, $\oplus_a \Bigl(\bigl(\gr^a \Ddl(\bL)\bigr) \ho_k K(-a)\Bigr) \Mi \gr^0 \RHl(\bL) \cong \Hl(\bL)$.  Since $\Hl(\bL)$ is a vector bundle on $X_K$ by Theorem \ref{thm-log-Simp}\Refenum{\ref{thm-log-Simp-main}}, this shows that $\gr \Ddl(\bL)$ is a vector bundles on $X$ of rank equal to that of $\Hl(\bL)$, which is in turn equal to $\rank_{\bQ_p}(\bL)$ by the proof of Theorem \ref{thm-log-Simp}\Refenum{\ref{thm-log-Simp-main}} in Section \ref{sec-coh}.
\end{proof}

Thus, by Proposition \ref{prop-res-Ddl} and Corollary \ref{cor-Ddl-gr-vec-bdl}, the proof of Theorem \ref{thm-log-RH-arith}\Refenum{\ref{thm-log-RH-arith-res}} is also complete.  We conclude this subsection with the following:
\begin{proof}[Proof of Theorem \ref{thm-unip-vs-nilp}]
    Given any $\bQ_p$-local system $\bL$ on $X_\ket$ such that $\bL|_{U_\et}$ has unipotent geometric monodromy along $D$, by definition \Pth{see the paragraph preceding Theorem \ref{thm-unip-vs-nilp}}, the action of $\gamma_i$ as in Lemma \ref{lem-res-explicit} on any stalk of $\bL|_{\widehat{\widetilde{Z}}{}_i^\partial}$ is \emph{unipotent}.  Consequently, $x = 1$ in Lemmas \ref{lem-poss-eval-N-infty-xi-r} and \ref{lem-poss-eval-N+-xi-r-T}, and the residues of $\RHl(\bL)$ \Pth{by Lemma \ref{lem-res-explicit}} are all \emph{nilpotent} \Pth{\ie, have zero eigenvalues}.  For such $\bL$, in the paragraph preceding Remark \ref{rem-fail-surj-L-m-0}, $L_{m_0, \tau}(X_K) \neq 0$ can happen in \Refeq{\ref{eq-L-m-0-tau}} only when $\tau(\gamma_i) = 1$ for all $i$ such that $Z_i = \{ T_i = 0 \} \neq \emptyset$ on $X = \Spa(R, R^+)$.  Hence, up to enlarging $m_0$, the corresponding monomial $T^{a_\tau}$ there is invertible in $R_{K, m_0}$.  In this case, the morphism \Refeq{\ref{eq-fail-surj-L-m-0}} is an isomorphism, just like \Refeq{\ref{eq-L-m-0}}; and the canonical morphisms $(\mu')^{-1}\bigl(\Hl(\bL)\bigr) \otimes_{\cO_{X_\proket}} \OCl \to \widehat{\bL} \otimes_{\widehat{\bZ}_p} \OCl$ and $(\mu')^{-1} \bigl(\RHl(\bL)\bigr) \otimes_{(\mu')^{-1}(\cO_\cX)} \OBdl \to \widehat{\bL} \otimes_{\widehat{\bZ}_p} \OBdl$ are also isomorphisms \Pth{\Refcf{} \cite[\aThms 2.1(ii) and 3.8(iii)]{Liu/Zhu:2017-rrhpl}}.  Therefore, we can argue as in the proofs of \cite[\aThms 2.1(iv) and 3.8(i)]{Liu/Zhu:2017-rrhpl} that $\Hl$ and $\RHl$ restrict to tensor functors.  This proves part \Refenum{\ref{thm-unip-vs-nilp-1}} of the theorem.

    As for part \Refenum{\ref{thm-unip-vs-nilp-2}}, if $\bL|_{U_\et}$ is \emph{de Rham} and has unipotent geometric monodromy along $D$, then the residues of $\Ddl(\bL)(X)$ are nilpotent, by the proofs of part \Refenum{\ref{thm-unip-vs-nilp-1}} above and of Proposition \ref{prop-res-Ddl}; and \Refeq{\ref{eq-lem-Ddl-to-RHl-fil-strict}} is an isomorphism, by Corollary \ref{cor-Ddl-to-RHl-fil-strict}.  In this case, the canonical morphism $\mu^{-1}\bigl(\Ddl(\bL)\bigr) \otimes_{\cO_{X_\proket}} \OBdl \to \widehat{\bL} \otimes_{\widehat{\bZ}_p} \OBdl$ \Pth{\Refcf{} \Refeq{\ref{eq-fail-surj}}} is also an isomorphism, and we can conclude as in \cite[\aThm 3.9(v)]{Liu/Zhu:2017-rrhpl} that $\Ddl$ also restricts to a tensor functor.
\end{proof}

\subsection{Compatibility with pullbacks and pushforwards}\label{sec-mor}

In this subsection, we prove Theorems \ref{thm-log-RH-geom}\Refenum{\ref{thm-log-RH-geom-mor}}, \ref{thm-log-Simp}\Refenum{\ref{thm-log-Simp-mor}}, and \ref{thm-log-RH-arith}\Refenum{\ref{thm-log-RH-arith-mor}}\Refenum{\ref{thm-log-RH-arith-push}}.  We shall omit the explicit verifications of the $\Gal(K / k)$-equivariance of the adjunction morphisms, because they are obvious from the constructions of the functors $\RHl$ and $\Hl$ \Pth{\Refcf{} Section \ref{sec-coh}}.

Let us begin with pullbacks.  Let $\cY$ be defined by $Y$ as in \Refeq{\ref{eq-def-cX}}.  Let $h: Y \to X$ be as in the statements of the theorems.  Let $E$ be the normal crossings divisor defining the log structure on $Y$, as in Example \ref{ex-log-adic-sp-ncd}, and let $V := Y - E$.

\begin{rk}\label{rem-compat-triv}
    Theorems \ref{thm-log-RH-geom}\Refenum{\ref{thm-log-RH-geom-mor}}, \ref{thm-log-Simp}\Refenum{\ref{thm-log-Simp-mor}}, and \ref{thm-log-RH-arith}\Refenum{\ref{thm-log-RH-arith-mor}} are obvious when $h: Y \to X$ is an open immersion.  Moreover, when the log structure is trivial, the functors $\RHl$, $\Hl$, and $\Ddl$ coincide with their analogues $\RH$, $\cH$, and $\DdR$ in \cite[\aThms 3.8, 2.1, and 3.9]{Liu/Zhu:2017-rrhpl}.  \Pth{See also Remark \ref{rem-mod-period-sheaf}.}
\end{rk}

\begin{lemma}\label{lem-mor-induced-pull}
    In the above setting, we have $h^{-1}(D) \subset E$ set-theoretically.
\end{lemma}
\begin{proof}
    By the definition of the log structures $\cM_X$ and $\cM_Y$ of $X$ and $Y$, respectively, as in Example \ref{ex-log-adic-sp-ncd}, the map $h^\sharp: h^{-1}(\cM_X) \to \cM_Y$ between log structures is defined only when $h^{-1}(D) \subset E$ set-theoretically.  Hence, the lemma follows.
\end{proof}

\begin{lemma}\label{lem-compat-fil-pull}
    The canonical morphism
    \begin{equation}\label{eq-lem-compat-pull-log-Simp}
        h^*\bigl(\Hl(\bL)\bigr) \to \Hl\bigl(h^{-1}(\bL)\bigr),
    \end{equation}
    defined by adjunction is injective.  The similarly defined morphisms
    \begin{equation}\label{eq-lem-compat-pull-RHl}
        h^*\bigl(\RHl(\bL)\bigr) \to \RHl\bigl(h^{-1}(\bL)\bigr)
    \end{equation}
    and
    \begin{equation}\label{eq-lem-compat-pull-Ddl}
        h^*\bigl(\Ddl(\bL)\bigr) \to \Ddl\bigl(h^{-1}(\bL)\bigr)
    \end{equation}
    are injective and strictly compatible with the filtrations on their both sides.
\end{lemma}
\begin{proof}
    Since $V_K$ is dense in $Y_K$, and since $h^*\bigl(\Hl(\bL)\bigr)$ is a vector bundle on $Y_K$ by Theorem \ref{thm-log-Simp}\Refenum{\ref{thm-log-Simp-main}}, the morphism \Refeq{\ref{eq-lem-compat-pull-log-Simp}} is injective because the corresponding morphisms for $h|_V: V \to U$ is an isomorphism by \cite[\aThm 2.1(iii)]{Liu/Zhu:2017-rrhpl}.  Since $\gr^r \RHl(\bL) \cong \Hl(\bL)(r)$ and $\gr^r \RHl\bigl(h^{-1}(\bL)\bigr) \cong \Hl\bigl(h^{-1}(\bL)\bigr)(r)$, for all $r$, the statement for \Refeq{\ref{eq-lem-compat-pull-RHl}} follows from that for \Refeq{\ref{eq-lem-compat-pull-log-Simp}}.  Also, the statement for \Refeq{\ref{eq-lem-compat-pull-Ddl}} follows from those for \Refeq{\ref{eq-lem-compat-pull-log-Simp}} and \Refeq{\ref{eq-lem-Ddl-to-RHl-fil-strict-gr}}, by Lemma \ref{lem-Ddl-to-RHl-fil-strict} \Pth{and its proof}.
\end{proof}

\begin{cor}\label{cor-compat-conn-pull}
    Under the same assumption as in Theorem \ref{thm-log-RH-geom}\Refenum{\ref{thm-log-RH-geom-mor}}, the canonical morphisms \Refeq{\ref{eq-lem-compat-pull-RHl}} and \Refeq{\ref{eq-lem-compat-pull-Ddl}} are isomorphisms compatible with the log connections and filtrations on both sides.  In this case, the canonical morphism \Refeq{\ref{eq-lem-compat-pull-log-Simp}}, which can be identified with the $0$-th graded piece of \Refeq{\ref{eq-lem-compat-pull-RHl}}, is an isomorphism compatible with the log Higgs fields on both sides.
\end{cor}
\begin{proof}
    For any pair of irreducible components $Z$ and $W$ of $D$ and $E$, respectively, let $m_{W Z}$ be as in Theorem \ref{thm-log-RH-geom}\Refenum{\ref{thm-log-RH-geom-mor}}; and let $h_{W Z}: W \to Z$ denote the induced morphism, when $m_{W Z} \geq 1$ \Pth{\ie, $h(W) \subset Z$}.  Let $\nabla_\bL$ denote the log connection for either $\RHl(\bL)$ or $\Ddl(\bL)$.  For each $W$, by working with local coordinates as in Section \ref{sec-calc-res}, we see that $\Res_W\bigl(h^*(\nabla_\bL)\bigr) = \sum_{h(W) \subset Z} \, m_{W Z} \, h_{W Z}^*\bigl(\Res_Z(\nabla_\bL)\bigr)$; and that $h_{W Z}^*\bigl(\Res_Z(\nabla_\bL)\bigr)$ and $h_{W Z'}^*\bigl(\Res_{Z'}(\nabla_\bL)\bigr)$ commute, for any $Z$ and $Z'$ as above.  Hence, by Theorem \ref{thm-unip-vs-nilp}, the assumption in Theorem \ref{thm-log-RH-geom}\Refenum{\ref{thm-log-RH-geom-mor}} ensures that there is at most one non-nilpotent summand $m_{W Z_0} \, h_{W Z_0}^*\bigl(\Res_Z(\nabla_\bL)\bigr)$, in which case $m_{W Z_0} = 1$.  As a result, by Theorems \ref{thm-log-RH-geom}\Refenum{\ref{thm-log-RH-geom-res}} and \ref{thm-log-RH-arith}\Refenum{\ref{thm-log-RH-arith-res}}, the eigenvalues of $\Res_W\bigl(h^*(\nabla_\bL)\bigr)$ belong to $\bQ \cap [0, 1)$ as those of $\Res_{Z_0}(\nabla_\bL)$ do.  Also by these theorems, the eigenvalues of the residues of the log connections for $\RHl\bigl(h^{-1}(\bL)\bigr)$ and $\Ddl\bigl(h^{-1}(\bL)\bigr)$ belong to $\bQ \cap [0, 1)$.  Thus, by Proposition \ref{prop-conn-isom-res} and Theorems \ref{thm-log-RH-geom}\Refenum{\ref{thm-log-RH-geom-main}} and \ref{thm-log-RH-arith}\Refenum{\ref{thm-log-RH-arith-main}}, the assertions for \Refeq{\ref{eq-lem-compat-pull-RHl}} and \Refeq{\ref{eq-lem-compat-pull-Ddl}} follow from the corresponding ones for $h|_V: V \to U$ in \cite[\aThms 3.8(iv) and 3.9(ii)]{Liu/Zhu:2017-rrhpl}, and the assertion for \Refeq{\ref{eq-lem-compat-pull-log-Simp}} follows from the one for \Refeq{\ref{eq-lem-compat-pull-RHl}}, as desired.
\end{proof}

Thus, we have finished the proofs of Theorems \ref{thm-log-RH-geom}\Refenum{\ref{thm-log-RH-geom-mor}}, \ref{thm-log-Simp}\Refenum{\ref{thm-log-Simp-mor}}, and \ref{thm-log-RH-arith}\Refenum{\ref{thm-log-RH-arith-mor}}.

Next, let us turn to Theorem \ref{thm-log-RH-arith}\Refenum{\ref{thm-log-RH-arith-push}}.  Let $f: X \to Y$ be as in the statement of the theorem.  Let $\cY$ be defined by $Y$ as in \Refeq{\ref{eq-def-cX}}.  Since $f^{-1}(E) \subset D$ by the same argument as in the proof of Lemma \ref{lem-mor-induced-pull}, and since $f|_U: U \to V$ is proper smooth, we must have $D = f^{-1}(E)$, because $U$ is dense in $X$.  Let $\bL$ be a $\bZ_p$-local system on $X_\ket$.  By \cite[\aCor \logadiclissepr]{Diao/Lan/Liu/Zhu:lasfr}, $R^i f_{\ket, *}(\bL)$ is a $\bZ_p$-local system on $Y_\ket$.  By \cite[\aDef \logadicdefproketlocsyst{} and \aProp \propdirimketvsproket]{Diao/Lan/Liu/Zhu:lasfr}, we have $\widehat{R^i f_{\ket, *}(\bL)} \cong R^i f_{\proket, *}(\widehat{\bL})$.  By Corollary \ref{cor-log-dR-cplx-rel}, $\BBdRX{X} \otimes_{f_\proket^{-1}(\BBdRX{Y})} f_\proket^{-1}(\OBdlX{Y})$ is quasi-isomorphic to $\OBdlX{X} \otimes_{\cO_{X_\proket}} \Omega^{\log, \bullet}_{X / Y}$, and hence the canonical morphism $f_\proket^{-1}(\OBdlX{Y}) \to \BBdRX{X} \otimes_{f_\proket^{-1}(\BBdRX{Y})} f_\proket^{-1}(\OBdlX{Y})$ induces a canonical morphism
\[
    R f_{\proket, *}(\widehat{\bL}) \otimes_{\widehat{\bZ}_p} \OBdlX{Y} \to R f_{\proket, *}\bigl(\widehat{\bL} \otimes_{\widehat{\bZ}_p} \OBdlX{X} \otimes_{\cO_{X_\proket}} \Omega^{\log, \bullet}_{X / Y}\bigr).
\]
By applying $R \mu'_{Y, *}$ to both sides of this morphism, we obtain $\RHl\bigl(R f_{\ket, *}(\bL)\bigr)$ on the left-hand side, and $R(\mu'_{Y, *} \circ f_{\proket, *})\bigl(\widehat{\bL} \otimes_{\widehat{\bZ}_p} \OBdlX{X} \otimes_{\cO_{X_\proket}} \Omega^{\log, \bullet}_{X / Y}\bigr) \cong R(f_{\an, *} \circ \mu'_{X, *})\bigl(\widehat{\bL} \otimes_{\widehat{\bZ}_p} \OBdlX{X} \otimes_{\cO_{X_\proket}} \Omega^{\log, \bullet}_{X / Y}\bigr) \cong R f_{\log \dR, *}\bigl(\RHl(\bL)\bigr)$ on the right-hand side, by Theorem \ref{thm-log-RH-geom}\Refenum{\ref{thm-log-RH-geom-main}} and the projection formula.  Therefore, we obtain a canonical morphism \Pth{of coherent sheaves with log connections}
\begin{equation}\label{eq-compat-fil-push-RHl}
    \RHl\bigl(R^i f_{\ket, *}(\bL)\bigr) \to R^i f_{\log \dR, *}\bigl(\RHl(\bL)\bigr),
\end{equation}
which is compatible with the filtrations on both sides.  If $\bL|_{U_\et}$ is \emph{de Rham}, then
we can identify $R^i f_{\log \dR, *}\bigl(\RHl(\bL)\bigr)$ with $R^i f_{\log \dR, *}\bigl(\Ddl(\bL) \ho_k \BdR\bigr)$, by Corollary \ref{cor-Ddl-to-RHl-fil-strict}, and hence with $R^i f_{\log \dR, *}\bigl(\Ddl(\bL)\bigr) \ho_k \BdR$, by the following lemma:
\begin{lemma}\label{lem-log-dR-coh-bc}
    Let $g: Z \to Z'$ be a proper morphism of rigid analytic varieties over $k$, and let $\cF$ be a complex of vector bundles on $Z'$ \Pth{whose differentials are not necessarily $\cO_{Z'}$-linear}.  Then we have $R g_{\an, *}\bigl(\cF \ho_k \BdR) \cong R g_{\an, *}(\cF) \ho_k \BdR$.
\end{lemma}
\begin{proof}
    By considering spectral sequences associated with the filtration on $\BdR$ and the stupid \Pth{\Qtn{b\^ete}} filtration on $\cF$, it suffices to note that, for any coherent $\cO_X$-module $F$, we have $R g_{\an, *}\bigl(F \ho_k \gr^r \BdR) \cong R g_{\an, *}(F) \ho_k \gr^r \BdR$, where $\gr^r \BdR \cong K(r)$, because $g$ is a proper morphism \Pth{\Refcf{} \cite[the proof of \aLem 7.13]{Scholze:2013-phtra}}.
\end{proof}

Thus, when $\bL|_{U_\et}$ is de Rham, \Refeq{\ref{eq-compat-fil-push-RHl}} can be rewritten as
\[
    \RHl\bigl(R^i f_{\ket, *}(\bL)\bigr) \to R^i f_{\log \dR, *}\bigl(\Ddl(\bL)\bigr) \ho_k \BdR.
\]
By taking $\Gal(K / k)$-invariants, we obtain a canonical morphism
\begin{equation}\label{eq-compat-fil-push-Ddl}
    \Ddl\bigl(R^i f_{\ket, *}(\bL)\bigr) \to R^i f_{\log \dR, *}\bigl(\Ddl(\bL)\bigr),
\end{equation}
which is also compatible with the filtrations on both sides.

\begin{lemma}\label{lem-compat-fil-push}
    Under the assumption that $\bL|_{U_\et}$ is de Rham, $\bigl(R^i f_{\ket, *}(\bL)\bigr)|_{V_\et} \cong R^i (f|_U)_{\et, *}(\bL|_{U_\et})$ is also de Rham, and \Refeq{\ref{eq-compat-fil-push-Ddl}} is defined and induces a morphism
    \begin{equation}\label{eq-lem-compat-fil-push-Ddl}
        \Ddl\bigl(R^i f_{\ket, *}(\bL)\bigr) \to \Bigl(R^i f_{\log \dR, *}\bigl(\Ddl(\bL)\bigr)\Bigr)_\free
    \end{equation}
    which is injective and strictly compatible with the filtrations on both sides.  That is, for each $r$, \Refeq{\ref{eq-lem-compat-fil-push-Ddl}} induces an injective morphism
    \begin{equation}\label{eq-lem-compat-fil-push-Ddl-gr}
        \gr^r \Ddl\bigl(R^i f_{\ket, *}(\bL)\bigr) \to \gr^r \Bigl(\Bigl(R^i f_{\log \dR, *}\bigl(\Ddl(\bL)\bigr)\Bigr)_\free\Bigr).
    \end{equation}
\end{lemma}
\begin{proof}
    Since $\bL|_{U_\et}$ is de Rham, $\bigl(R^i f_{\ket, *}(\bL)\bigr)|_{V_\et} \cong R^i(f|_U)_{\et, *}(\bL|_{U_\et})$ is also de Rham, by \cite[\aThm 8.8]{Scholze:2013-phtra} and \cite[\aThm 3.8(v)]{Liu/Zhu:2017-rrhpl}.  Therefore, by Corollary \ref{cor-Ddl-gr-vec-bdl}, $\gr \Ddl\bigl(R^i f_{\ket, *}(\bL)\bigr)$ is a vector bundle on $Y$.  Since $V$ is dense in $Y$, the morphism \Refeq{\ref{eq-lem-compat-fil-push-Ddl-gr}} \Pth{which is defined as soon as \Refeq{\ref{eq-lem-compat-fil-push-Ddl}} is compatible with the filtrations of both sides} is injective because the corresponding morphism for $f|_U: U \to V$ is an isomorphism, by \cite[\aThm 8.8]{Scholze:2013-phtra}.
\end{proof}

\begin{cor}\label{cor-compat-conn-push}
    Under the assumption that $\bL|_{U_\et}$ is de Rham, \Refeq{\ref{eq-lem-compat-fil-push-Ddl}} is an isomorphism compatible with the log connections and filtrations on both sides.
\end{cor}
\begin{proof}
    By Lemma \ref{lem-compat-fil-push}, it suffices to show that \Refeq{\ref{eq-lem-compat-fil-push-Ddl}} is an isomorphism \Pth{compatible with the log connections on both sides}.  By the same argument as in \cite[\aSec VII]{Katz:1970-rtag}, the eigenvalues of the residues of $\Bigl(R^i f_{\log \dR, *}\bigl(\Ddl(\bL)\bigr)\Bigr)_\free$ are still in $\bQ \cap [0, 1)$.  Hence, by Proposition \ref{prop-conn-isom-res} and Theorem \ref{thm-log-RH-arith}\Refenum{\ref{thm-log-RH-arith-main}}, the assertion follows from the corresponding one for $f|_U: U \to V$ in \cite[\aThm 8.8]{Scholze:2013-phtra}.
\end{proof}

The proof of Theorem \ref{thm-log-RH-arith}\Refenum{\ref{thm-log-RH-arith-push}} is now complete.

\subsection{Comparison of cohomology}\label{sec-comp-coh}

In this subsection, we prove the remaining Theorems \ref{thm-log-RH-geom}\Refenum{\ref{thm-log-RH-geom-comp}}, \ref{thm-log-Simp}\Refenum{\ref{thm-log-Simp-comp}}, and \ref{thm-log-RH-arith}\Refenum{\ref{thm-log-RH-arith-comp}}.  We shall assume that $X$ is proper over $k$, and that $K = \widehat{\AC{k}}$.  \Pth{In this case, $\BdRp$ and $\BdR$ are the usual Fontaine's rings.}

\begin{lemma}\label{lem-ket-BdR}
    For each $\bZ_p$-local system $\bL$ on $X_\ket$, and for each $i \geq 0$, we have a canonical $\Gal(K / k)$-equivariant isomorphism of $\BdRp$-modules
    \[
        H^i(X_{K, \ket}, \bL) \otimes_{\bZ_p} \BdRp \cong H^i(X_{K, \proket}, \widehat{\bL} \otimes_{\widehat{\bZ}_p} \BBdRp),
    \]
    compatible with the filtrations on both sides, and also \Pth{by taking $0$-th graded pieces} a canonical $\Gal(K / k)$-equivariant isomorphism of $K$-modules
    \[
        H^i(X_{K, \ket}, \bL) \otimes_{\bZ_p} K \cong H^i(X_{K, \proket}, \widehat{\bL} \otimes_{\widehat{\bZ}_p} \cO_{X_{K, \proket}}).
    \]
\end{lemma}
\begin{proof}
    The proof is the same as \cite[\aThm 8.4]{Scholze:2013-phtra}, with the input \cite[\aThm 5.1]{Scholze:2013-phtra} there replaced with \cite[\aThm \logadicthmprimcomp]{Diao/Lan/Liu/Zhu:lasfr}.
\end{proof}

\begin{lemma}\label{lem-comp-dR-log}
    Let $\bL$ be any $\bZ_p$-local system on $X_\ket$.  For each $i \geq 0$, we have a canonical $\Gal(K / k)$-equivariant isomorphism of $\BdR$-modules
    \[
        H^i\bigl(X_{K, \proket}, \widehat{\bL} \otimes_{\widehat{\bZ}_p} \BBdR\bigr) \cong H^i_{\log \dR}\bigl(\cX, \RHl(\bL)\bigr)
    \]
    and also a canonical $\Gal(K / k)$-equivariant isomorphism of $K$-modules
    \[
        H^i\bigl(X_{K, \proket}, \widehat{\bL} \otimes_{\widehat{\bZ}_p} \cO_{X_{K, \proket}}\bigr) \cong H^i_{\log \Hi}\bigl(X_{K, \an}, \Hl(\bL)\bigr).
    \]
\end{lemma}
\begin{proof}
    Let us simply denote the complexes $(\widehat{\bL} \otimes_{\widehat{\bZ}_p} \OBdl \otimes_{\cO_{X_\proket}} \Omega^{\log, \bullet}_X, \nabla)$ and $(\widehat{\bL} \otimes_{\widehat{\bZ}_p} \OCl \otimes_{\cO_{X_\proket}} \Omega^{\log, \bullet}_X(-\bullet), \theta)$ \Pth{where the two $\bullet$ in the latter complex are equal to each other} by $\DRl(\widehat{\bL} \otimes_{\widehat{\bZ}_p} \OBdl)$ and $\Hil(\widehat{\bL} \otimes_{\widehat{\bZ}_p} \OCl)$, respectively.  By Corollary \ref{cor-log-dR-cplx}, we have quasi-isomorphisms $\widehat{\bL} \otimes_{\widehat{\bZ}_p} \BBdR \Mi \DRl(\widehat{\bL} \otimes_{\widehat{\bZ}_p} \OBdl)$ and \Pth{by taking the $0$-th graded pieces} $\widehat{\bL} \otimes_{\widehat{\bZ}_p} \cO_{X_{K, \proket}} \Mi \Hil(\widehat{\bL} \otimes_{\widehat{\bZ}_p} \OCl)$ over $X_{K, \proket}$.  By Theorem \ref{thm-log-RH-geom}\Refenum{\ref{thm-log-RH-geom-main}} and Proposition \ref{prop-L-OCl}, and by the projection formula \Pth{\Refcf{} \Refeq{\ref{eq-RHl-proj}}}, $R\mu'_*\bigl( \DRl(\widehat{\bL} \otimes_{\widehat{\bZ}_p} \OBdl) \bigr) \cong \DRl(\RHl(\bL))$ and $R\mu'_*\bigl( \Hil(\widehat{\bL} \otimes_{\widehat{\bZ}_p} \OCl) \bigr) \cong \Hil(\Hl(\bL))$, and the lemma follows.
\end{proof}

Thus, Theorems \ref{thm-log-RH-geom}\Refenum{\ref{thm-log-RH-geom-comp}} and \ref{thm-log-Simp}\Refenum{\ref{thm-log-Simp-comp}} follow from Lemmas \ref{lem-ket-BdR} and \ref{lem-comp-dR-log}.

It remains to complete the proof of Theorem \ref{thm-log-RH-arith}\Refenum{\ref{thm-log-RH-arith-comp}}.  In the remainder of this subsection, we shall assume in addition that $\bL|_{U_\et}$ is a de Rham local system.  Firstly, the isomorphism \Refeq{\ref{eq-thm-log-RH-arith-comp-dR}} is given by Theorem \ref{thm-log-RH-geom}\Refenum{\ref{thm-log-RH-geom-comp}} and the following:
\begin{lemma}\label{lem-log-RH-arith-comp-dR}
    With assumptions as above, there is a canonical isomorphism
    \[
        H^i_{\log \dR}\bigl(\cX, \RHl(\bL)\bigr) \cong H^i_{\log \dR}\bigl(X_\an, \Ddl(\bL)\bigr) \otimes_k \BdR.
    \]
\end{lemma}
\begin{proof}
    Combine Corollary \ref{cor-Ddl-to-RHl-fil-strict} \Pth{with $\bL|_{U_\et}$ de Rham} and Lemma \ref{lem-log-dR-coh-bc}.
\end{proof}

Secondly, $\gr \Ddl(\bL)$ is a vector bundle of rank $\rank_{\bZ_p}(\bL)$ by Corollary \ref{cor-Ddl-gr-vec-bdl}, and the isomorphism \Refeq{\ref{eq-thm-log-RH-arith-comp-HT}} is given by Theorem \ref{thm-log-Simp}\Refenum{\ref{thm-log-Simp-comp}} and the following:
\begin{lemma}\label{lem-log-RH-arith-comp-HT}
    With assumptions as above, there is a canonical isomorphism
    \[
        H^i_{\log \Hi}\bigl(X_{K, \an}, \Hl(\bL)\bigr) \cong \oplus_{a + b = i} \, \Bigl( H^{a, b}_{\log \Hdg}\bigl(X_\an, \Ddl(\bL)\bigr) \otimes_k K(-a) \Bigr).
    \]
\end{lemma}
\begin{proof}
    Since $\bL|_{U_\et}$ is de Rham, by Corollary \ref{cor-Ddl-to-RHl-fil-strict} and Lemma \ref{lem-log-dR-coh-bc}, we have
    \[
    \begin{split}
        & H^i_{\log \Hi}\bigl(X_{K, \an}, \Hl(\bL)\bigr) = H^i\bigl(X_{K, \an}, \Hil(\Hl(\bL)) \bigr) \\
        & \cong H^i\bigl(X_{K, \an}, \gr^0 \bigl(\DRl(\Ddl(\bL)) \ho_k \BdR\bigr)\bigr) \\
        & \cong \oplus_a \, \Bigl( H^i\bigl(X_\an, \gr^a \DRl(\Ddl(\bL))\bigr) \otimes_k K(-a) \Bigr) \\
        & \cong \oplus_{a + b = i} \, \Bigl( H^{a, b}_{\log \Hdg}\bigl(X_\an, \Ddl(\bL)\bigr) \otimes_k K(-a) \Bigr).  \qedhere
    \end{split}
    \]
\end{proof}

Finally, the \Pth{log} Hodge--de Rham spectral sequence for $\Ddl(\bL)$ degenerates on the $E_1$ page because, by \Refeq{\ref{eq-thm-log-RH-arith-comp-dR}} and \Refeq{\ref{eq-thm-log-RH-arith-comp-HT}}, we have
\[
    \dim_k H^i_{\log \dR}\bigl(X_\an, \Ddl(\bL)\bigr) = \sum_{a + b = i} \dim_k H^{a, b}_{\log \Hdg}\bigl(X_\an, \Ddl(\bL)\bigr).
\]
The proof of Theorem \ref{thm-log-RH-arith}\Refenum{\ref{thm-log-RH-arith-comp}} is now complete.

\subsection{Compatibility with nearby cycles}\label{sec-comp-nearby}

Let $f: X \to \bD = \Spa(k\Talg{T}, k^+\Talg{T})$ be a morphism of smooth rigid analytic varieties such that $D := f^{-1}(0)$ is a normal crossings divisor.  We endow $X$ with the log structure defined by $\imath: D_\red \Em X$ as in Example \ref{ex-log-adic-sp-ncd}.  Let $U := X - D$.  Recall that we have introduced in \cite[\aDef \logadicdefnearby]{Diao/Lan/Liu/Zhu:lasfr} the functors of unipotent and quasi-unipotent nearby cycles $R\Psi^\unip_f(\bL|_U)$ and $R\Psi^\qunip_f(\bL|_U)$, respectively, for $\bQ_p$-local systems $\bL$ on $X_\ket$.  In this subsection, we show that their formation is compatible with the log Riemann--Hilbert functors, in the simplest situation to which the methods of this paper are directly applicable.

As usual, for any $\cO_Y$-module $F$ on a locally ringed space $Y$ and any closed immersion $\imath: Z \Em Y$ such that $\cI_Z := \ker\bigl(\cO_Y \to \imath_*(\cO_Z)\bigr)$ is an invertible $\cO_Y$-ideal, let $F(n Z) := F \otimes_{\cO_Y} \cI_Z^{\otimes (-n)}$, for each $n \in \bZ$.  Also, if we have compatible inclusions $F(n Z) \Em F(m Z)$ extending the identity morphism on $F|_U$, for all $m \geq n$, then we let $F(* Z) := \varinjlim_{n \in \bZ} \, F(n Z) = \cup_{n \in \bZ} \, F(n Z)$.  The following lemma is elementary:
\begin{lemma}\label{lem-V-fil}
    Let $(F, \nabla)$ be any vector bundle with an integrable log connection on $X_\an$.  Let $Z \subset D_\red$ be an irreducible component, and suppose that all the eigenvalues of the residue \Refeq{\ref{eq-def-res}} belong to $\bQ \cap [0, 1)$.  Then $F(* Z)$ is defined, and there is a unique decreasing $\bQ$-filtration $V^\bullet$ on $F(* Z)$ by locally free $\cO_X$-submodules equipped with compatible integrable log connections, characterized by the following properties:
    \begin{enumerate}
        \item We have $V^0 F(* Z) = F$ and $V^{\alpha + 1} F(* Z) = \bigl(V^\alpha F(* Z)\bigr)(-Z)$, for all $\alpha \in \bQ$.

        \item The isomorphism $\bigl(V^0 F(* Z)\bigr) \big/ \bigl(V^1 F(* Z)\bigr) \cong F|_Z$ canonically induces, for each $\alpha \in \bQ \cap [0, 1)$, an isomorphism
            \begin{equation}\label{eq-V-gr-vs-res-eigen}
                \gr_V^\alpha F(* Z) := \bigl(V^\alpha F(* Z)\bigr) / \bigl(V^{> \alpha} F(* Z)\bigr) \cong F|_Z^\alpha,
            \end{equation}
            where $V^{> \alpha} F(* Z) := \cup_{\beta > \alpha} \, V^\beta F(* Z)$ and where $F|_Z^\alpha$ is as in \Refeq{\ref{eq-res-gen-eigen}}.
    \end{enumerate}
    By using Lemma \ref{lem-OXBdR-mod}, when $Z_{\AC{k}}$ is also irreducible, we have analogues of the above for any vector bundle with an integrable log connection $(\cF, \nabla)$ on $\cX$ such that all the eigenvalues of $\Res_Z(\nabla)$ belong to $\bQ \cap [0, 1)$.  When $F \ho_k \BdR \cong \cF$, for each $\alpha$, we have $\bigl(V^\alpha F(* Z)\bigr) \ho_k \BdR \cong V^\alpha \cF(* Z)$ and $\bigl(\gr_V^\alpha F(* Z)\bigr) \ho_k \BdR \cong \gr_V^\alpha \cF(* Z)$.
\end{lemma}

\begin{rk}\label{rem-V-fil}
    Since $\nabla: F \to F \otimes_{\cO_X} \Omega^{\log}_X$ induces a connection $\nabla: F(* Z) \to F(* Z) \otimes_{\cO_X} \Omega_X$ satisfying $\nabla^2 = 0$, we can view $F(* Z)$ as a $D$-module, and view the filtration in Lemma \ref{lem-V-fil} as a special case of the Kashiwara--Malgrange $V$-filtrations \Pth{\Refcf{} \cite[\aSec 3.1]{Saito:1988-mhp}, with $V^\bullet$ here corresponding to $V_{-1 - \bullet}$ there}.  \Pth{Note that there are different conventions of indices in the literature.}
\end{rk}

By Theorem \ref{thm-log-RH-geom}\Refenum{\ref{thm-log-RH-geom-res}}, when $Z_{\AC{k}}$ is irreducible, the above construction applies to $\RHl(\bL)$.  Recall that there is a decreasing filtration $\Fil^\bullet \RHl(\bL)$ on $\RHl(\bL)$ by locally free $\cO_X \ho_k \BdR^+$-submodules.  This induces, for each $\alpha > -1$, the filtration
\[
     \Fil^i V^\alpha \RHl(\bL)(* Z) := \bigl(\Fil^i \RHl(\bL)\bigr)(Z) \cap \bigl(V^\alpha \RHl(\bL)(* Z)\bigr)
\]
on $V^\alpha \RHl(\bL)(* Z)$ by $\cO_X \ho_k \BdR^+$-submodules.  By construction, for $\alpha \geq \beta > -1$, the inclusion $V^\alpha \RHl(\bL)(* Z) \Em V^\beta \RHl(\bL)(* Z)$ is strictly compatible with the filtrations.  For each $\alpha \geq -1$, we similarly define the filtration on $V^{> \alpha} \RHl(\bL)(* Z)$.  Then we have an induced quotient filtration on $\gr_V^\alpha \RHl(\bL)(* Z)$, for each $\alpha > -1$.

\begin{rk}\label{rem-V-gr-vs-res-eigen-fil}
    For each $\alpha \in \bQ \cap [0, 1)$, in general, the isomorphism \Refeq{\ref{eq-V-gr-vs-res-eigen}} is compatible with filtrations only if we view $\RHl(\bL)|_Z^\alpha$ as a quotient \Pth{rather than a subsheaf} of $\RHl(\bL)|_Z$ with its induced filtration.  We emphasize that it is the quotient filtration on $\gr_V^\alpha \RHl(\bL)(* Z)$ that will be important in the following.
\end{rk}

\begin{thm}\label{thm-nearby-comp-geom}
    Assume that $D$ is smooth and $D_{\AC{k}}$ is irreducible.  Then, for each $\bQ_p$-local system $\bL$ on $X_\ket$, there is a canonical $\Gal(K / k)$-equivariant isomorphism of $\cO_D \ho_k \BdR$-modules
    \begin{equation}\label{eq-thm-nearby-comp-geom-qunip}
        \RH\bigl(R\Psi^\qunip_f(\bL|_U)\bigr) \cong \oplus_{\alpha \in (-1, 0]} \, \bigl(\gr_V^\alpha \RHl(\bL)(*)\bigr),
    \end{equation}
    which restricts to an isomorphism
    \begin{equation}\label{eq-thm-nearby-comp-geom-unip}
        \RH\bigl(R\Psi^\unip_f(\bL|_U)\bigr) \cong \gr_V^0 \RHl(\bL)(*),
    \end{equation}
    compatible with filtrations and integrable connections.  Here $\RH$ is the functor defined in \cite[\aThm 3.8]{Liu/Zhu:2017-rrhpl} \Pth{see Remark \ref{rem-compat-triv}}, and we write $\RHl(\bL)(*)$ instead of $\RHl(\bL)(* D)$ for simplicity.
\end{thm}
\begin{proof}
    Let us first prove \Refeq{\ref{eq-thm-nearby-comp-geom-unip}}.  Besides the trivial log structure, there is another natural log structure on $D$ given by the pullback of the log structure on $X$.  Let $D^\partial$ denote the corresponding log adic space.  Then we have a correspondence of log adic spaces $D \xleftarrow{\varepsilon^\partial} D^\partial \xrightarrow{\imath} X$.  Let $\widehat{\bL}^\partial := \imath_{\proket}^{-1}(\widehat{\bL})$, which is associated with $\bL^\partial := \imath^{-1}_\ket(\bL)$ by \cite[\aLem \logadiclemproketlisse]{Diao/Lan/Liu/Zhu:lasfr}.  Let $\bJ_r^\partial$ denote the pullback to $D^\partial$ of the $\bZ_p$-local system denoted by the same symbols $\bJ_r^\partial$ in \cite[second last paragraph preceding \aLem \logadiclemnearbycohstab]{Diao/Lan/Liu/Zhu:lasfr}, and let $\widehat{\bJ}_r^\partial$ denote the associated $\widehat{\bZ}_p$-local system.

    By Corollary \ref{cor-OBdRp-loc-cl-imm} and Lemma \ref{lem-L-OBdl-a-b-van}, we have canonical morphisms of sheaves $\widehat{\bL} \otimes_{\widehat{\bQ}_p} \OBdlX{X} \to \imath_{\proket, *}\bigl(\widehat{\bL}^\partial \otimes_{\widehat{\bQ}_p} \OBdlX{D^\partial}\bigr) \cong R\imath_{\proket, *}\bigl(\widehat{\bL}^\partial \otimes_{\widehat{\bQ}_p} \OBdlX{D^\partial}\bigr)$ on $X_\proket$.  By applying $R\mu'_{X, *}$, we obtain a morphism of $\cO_X \ho_k \BdR$-modules
    \begin{equation}\label{eq-RHl-to-bd}
        \RHl(\bL) \to \imath_* \, R\mu'_{D^\partial, *}(\widehat{\bL}^\partial \otimes_{\widehat{\bQ}_p} \OBdlX{D^\partial}).
    \end{equation}
    By Corollaries \ref{cor-OBdlp-loc-gr} and \ref{cor-OBdRp-loc-cl-imm}, and by matching a basis of $\bJ_r$ with binomial monomials up to degree $r - 1$ in $W$ as in the proof of \cite[\aLem \logadiclemnearbycohstab]{Diao/Lan/Liu/Zhu:lasfr}, for any $-\infty < a \leq b < \infty$, there is a natural isomorphism
    \[
        \varepsilon_\proket^{\partial, -1}(\OBdlX{D}^{[a, b]}) \otimes_{\widehat{\bZ}_p} \varinjlim_r(\widehat{\bJ}_r^\partial) \cong \OBdlX{D^\partial}^{[a, b]}.
    \]
    By \cite[\aDef \logadicdefnearby{} and \aProp \logadicpropproketvsketadj]{Diao/Lan/Liu/Zhu:lasfr}, we obtain canonical morphisms
    \[
    \begin{split}
        R \Psi^\unip_f(\bL|_U) \otimes_{\widehat{\bZ}_p} \OBdlX{D}^{[a, b]} & \to R\varepsilon^\partial_{\proket, *} \, \Bigl(\widehat{\bL}^\partial \otimes_{\widehat{\bZ}_p} \varinjlim_r(\widehat{\bJ}_r^\partial) \otimes_{\widehat{\bZ}_p} \varepsilon_\proket^{\partial, -1}(\OBdlX{D}^{[a, b]})\Bigr) \\
        & \cong R\varepsilon^\partial_{\proket, *} \, \bigl(\widehat{\bL}^\partial \otimes_{\widehat{\bZ}_p} \OBdlX{D^\partial}^{[a, b]}\bigr).
    \end{split}
    \]
    Since $R\mu'_{D^\partial, *} \cong R\mu'_{D, *} \circ R\varepsilon^\partial_{\proket, *}$ \Pth{as $D^\partial_\an \cong D_\an$}, by applying $R\mu'_{D^\partial, *}$ to the above, and by taking colimit and limit, we obtain a canonical morphism of sheaves
    \begin{equation}\label{eq-RH-nearby-to-bd}
    \begin{split}
        \RH\bigl(R\Psi^\unip_f(\bL|_U)\bigr) \to R\mu'_{D^\partial, *}(\widehat{\bL}^\partial \otimes_{\widehat{\bQ}_p} \OBdlX{D^\partial}).
    \end{split}
    \end{equation}
    We claim that \Refeq{\ref{eq-RH-nearby-to-bd}} is an isomorphism, and that the combination of \Refeq{\ref{eq-RHl-to-bd}} and \Refeq{\ref{eq-RH-nearby-to-bd}} induces a canonical isomorphism $\gr_V^0 \RHl(\bL)(*) \cong \RH\bigl(R\Psi^\unip_f(\bL|_U)\bigr)$.

    Since the question is local, we may assume that $X$ is affinoid, and that $f$ factors as $X \to \bD^n \cong \bD \times \bD^{n - 1} \to \bD$, where the first map is a smooth toric chart, and where the last map is the first projection.  Accordingly, we have $\Gamma_\geom \cong (\widehat{\bZ}\gamma_1) \times \widehat{\bZ}(1)^{n - 1}$.  By Lemma \ref{lem-Gamma-geom-inv}, Corollary \ref{cor-OBdRp-loc-cl-imm} and \Refeq{\ref{lem-L-B-dR-Z-i}}, in the notation there, the evaluation of \Refeq{\ref{eq-RHl-to-bd}} at $X$ can be identified with the natural map
    \begin{equation}\label{eq-RHl-to-bd-eval}
        N \cong \Bigl(\varprojlim_r \bigl((N_\infty / \xi^r)^\Utext{unip}\bigr)\Bigr)[\tfrac{1}{t}] \to \Bigl(\varprojlim_r \Bigl(\bigl((N_\infty / \xi^r) \big/ ([T_1^{s\flat}])^\wedge_{s \in \bQ_{> 0}}\bigr)^\Utext{unip}\Bigr)\Bigr)[\tfrac{1}{t}],
    \end{equation}
    where $(\,\cdot\,)^\Utext{unip}$ denotes the maximal quotient spaces on which $\Gamma_\geom$ acts unipotently.  By the arguments in the proof of Lemma \ref{lem-poss-eval-N+-xi-r-T}, the right-hand side of \Refeq{\ref{eq-RHl-to-bd-eval}} can be identified with the quotient of $N / T_1$ on which $t^{-1} \log(\gamma_1)$ acts nilpotently, which is $\imath_*\bigl(\gr_V^0 \RHl(\bL)(*)\bigr)(X)$ by Lemma \ref{lem-res-explicit} and \Refeq{\ref{eq-V-gr-vs-res-eigen}}.  This also gives the left-hand side of the evaluation of \Refeq{\ref{eq-RH-nearby-to-bd}} at $D$, by \cite[\aProp \logadicpropnearby]{Diao/Lan/Liu/Zhu:lasfr}, Lemma \ref{lem-L-OCl-coh}, and Remark \ref{rem-compat-L-m-0-res}.  Since $(N_\infty / \xi)^\Utext{unip} \to \bigl((N_\infty / \xi) \big/ ([T_1^{s\flat}])^\wedge_{s \in \bQ_{> 0}}\bigr)^\Utext{unip}$ is surjective by Lemma \ref{lem-L-OCl-coh} again, so is \Refeq{\ref{eq-RHl-to-bd-eval}}.  Thus, the claim and \Refeq{\ref{eq-thm-nearby-comp-geom-unip}} follow.

    Next, we reduce \Refeq{\ref{eq-thm-nearby-comp-geom-qunip}} to \Refeq{\ref{eq-thm-nearby-comp-geom-unip}}.  Since $\gamma_1$ acts quasi-unipotently on $\bL|_{D^\partial}$ \Pth{as in the proof of Lemma \ref{lem-poss-eval-N-infty-xi-r}}, there is some degree $m$ standard Kummer \'etale cover of $\bD \to \bD$ \Pth{inducing an isomorphism between the origins of $\bD$}, with base change $g: X_m \to X$, such that $g^{-1}(\bL|_U)$ has purely unipotent geometric monodromy along $D$ \Pth{which we also identify as a subspace of $X_m$}.  By \cite[\aLem \logadiclemearby]{Diao/Lan/Liu/Zhu:lasfr}, $R\Psi^\qunip_f(\bL|_U) \cong R\Psi^\unip_{g \circ f}\bigl(g^{-1}(\bL|_U)\bigr)$.  By Lemma \ref{lem-compat-fil-pull}, we have a canonical morphism $g^*\bigl(\RHl(\bL)\bigr) \to \RHl\bigl(g^{-1}(\bL)\bigr)$, strictly compatible with filtrations, which restricts to an isomorphism over $U_m := X_m - D$ \Pth{and can be viewed as a ``meromorphic isomorphism''}.  Hence, we obtain an induced canonical isomorphism $\Bigl(g^*\bigl(\RHl(\bL)\bigr)\Bigr)(*) \Mi \Bigl(\RHl\bigl(g^{-1}(\bL)\bigr)\Bigr)(*)$.  By Theorem \ref{thm-unip-vs-nilp}, the residue of $\RHl\bigl(g^{-1}(\bL)\bigr)$ along $D$ is nilpotent.  Since the eigenvalues of the residues of $\RHl(\bL)$ belong to $\bQ \cap [0, 1)$, those of $\RHl(\bL)(D)$ belong to $\bQ \cap [-1, 0)$.  Since pulling back by $g$ multiplies the eigenvalues of residues by $m$ \Pth{as explained in the proof of Corollary \ref{cor-compat-conn-pull}}, those of $g^*\bigl(\RHl(\bL)\bigr)$ and $g^*\bigl(\RHl(\bL)(D)\bigr)$ belong to $\bQ \cap [0, m)$ and $\bQ \cap [-m, 0)$, respectively. Hence, by applying Lemma \ref{lem-V-fil} to $\cF = \RHl\bigl(g^{-1}(\bL)\bigr)$, we obtain inclusions of sheaves
    \[
        g^*\bigl(\RHl(\bL)\bigr) \Em \RHl\bigl(g^{-1}(\bL)\bigr) \Em g^*\bigl(\RHl(\bL)(D)\bigr).
    \]
    which are strictly compatible with the filtrations \Pth{see the paragraph preceding Remark \ref{rem-V-gr-vs-res-eigen-fil}}.  By pushing forward to $X$, we obtain the following inclusions of sheaves
    \[
        \RHl(\bL) \otimes_{\cO_X} g_*(\cO_{X_m}) \Em g_*\RHl\bigl(g^{-1}(\bL)\bigr) \Em \RHl(\bL)(D) \otimes_{\cO_X} g_*(\cO_{X_m}),
    \]
    which are strictly compatible with the filtrations.  We can identify the above with
    \[
    \begin{split}
        \oplus_{i = 0}^{m - 1} \, T_1^{\frac{i}{m}} \RHl(\bL) & \Em \oplus_{i = 0}^{m - 1} \, \bigl(T_1^{\frac{i}{m}} V^{-\frac{i}{m}} \RHl(\bL)(*)\bigr) \\
        & \Em \oplus_{i = 0}^{m - 1} \, \bigl(T_1^{\frac{i}{m}} V^{-1} \RHl(\bL)(*)\bigr),
    \end{split}
    \]
    because $g_*(\cO_{X_m}) \cong \oplus_{i = 0}^{m - 1} \, T_1^{\frac{i}{m}} \cO_X$ \Pth{which explains the first and third terms}, and because the second term is exactly the submodule of the third term whose eigenvalues of residues are zero.  Thus, we obtain the desired isomorphisms
    \[
    \begin{split}
        \oplus_{i = 0}^{m - 1} \, \bigl(\gr_V^{-\frac{i}{m}} \RHl(\bL)(*)\bigr) & \cong \gr_V^0 \Bigl(g_* \, \RHl\bigl(g^{-1}(\bL)\bigr)\Bigr)(*) \\
        & \cong \gr_V^0 \RHl\bigl(g^{-1}(\bL)\bigr)(*),
    \end{split}
    \]
    which are compatible with filtrations.
\end{proof}

Now suppose moreover that $\bL|_{U_\et}$ is de Rham.  By Theorem \ref{thm-log-RH-arith}\Refenum{\ref{thm-log-RH-arith-res}} and Lemma \ref{lem-V-fil}, we also have the $V$-filtration $V^\bullet \Ddl(\bL)(*)$ on $\Ddl(\bL)(*)$, and the filtration $\Fil^\bullet \Ddl(\bL)$ on $\Ddl(\bL)$ induces, for each $\alpha > -1$, the filtration
\[
     \Fil^i V^\alpha \Ddl(\bL)(*) := \bigl(\Fil^i \Ddl(\bL)\bigr)(D) \cap \bigl(V^\alpha \Ddl(\bL)(*)\bigr)
\]
on $V^\alpha \Ddl(\bL)(*)$ by $\cO_X$-submodules, and similar filtrations on $V^{> \alpha} \Ddl(\bL)(*)$ and $\gr_V^\alpha \Ddl(\bL)(*)$.
\begin{thm}\label{thm-nearby-comp-arithm}
    In Theorem \ref{thm-nearby-comp-geom}, suppose moreover that $\bL|_{U_\et}$ is \emph{de Rham}.  Then there is a canonical isomorphism
    \begin{equation}\label{eq-thm-nearby-comp-arithm-qunip}
        \DdR\bigl(R\Psi^\qunip_f(\bL|_U)\bigr) \cong \oplus_{\alpha \in (-1, 0]} \, \bigl(\gr_V^\alpha \Ddl(\bL)(*)\bigr),
    \end{equation}
    which restricts to an isomorphism
    \begin{equation}\label{eq-thm-nearby-comp-arithm-unip}
        \DdR\bigl(R\Psi^\unip_f(\bL|_U)\bigr) \cong \gr_V^0 \Ddl(\bL)(*),
    \end{equation}
    compatible with filtrations and integrable connections.  Here $\DdR$ is the functor defined in \cite[\aThm 3.9]{Liu/Zhu:2017-rrhpl} \Pth{see Remark \ref{rem-compat-triv}}.  Moreover, $R\Psi^\qunip_f(\bL|_U)$ and its direct summand $R\Psi^\unip_f(\bL|_U)$ are both de Rham $\bQ_p$-local systems on $D_\et$.
\end{thm}
\begin{proof}
    By taking $\Gal(K / k)$-invariants, we obtain \Refeq{\ref{eq-thm-nearby-comp-arithm-qunip}} and \Refeq{\ref{eq-thm-nearby-comp-arithm-unip}} from \Refeq{\ref{eq-thm-nearby-comp-geom-qunip}} and \Refeq{\ref{eq-thm-nearby-comp-geom-unip}}, respectively.  Since $\bL|_{U_\et}$ is de Rham, by Corollary \ref{cor-Ddl-to-RHl-fil-strict} and Lemma \ref{lem-V-fil}, the canonical morphism $\bigl(\gr_V^\alpha \Ddl(\bL)(*)\bigr) \ho_k \BdR \to \gr_V^\alpha \Ddl(\bL)(*)$ is an isomorphism, for each $\alpha$.  Hence, by \Refeq{\ref{eq-thm-nearby-comp-geom-qunip}} and \Refeq{\ref{eq-thm-nearby-comp-arithm-qunip}}, the canonical morphism $\DdR\bigl(R\Psi^\qunip_f(\bL|_U)\bigr) \ho_k \BdR \to \RH\bigl(R\Psi^\qunip_f(\bL|_U)\bigr)$ is also an isomorphism.  It follows that $R\Psi^\qunip_f(\bL|_U)$ and therefore its summand $R\Psi^\unip_f(\bL|_U)$ are de Rham, as desired.
\end{proof}

\section{Riemann--Hilbert functor for {$p$}-adic algebraic varieties}\label{sec-alg}

\subsection{The functor {$\DdRalg$}}\label{sec-DdR-alg}

In this subsection, we shall prove Theorem \ref{thm-intro-main} and record some byproducts.  Let $X$ be a smooth algebraic variety over a $p$-adic field $k$.  By \cite{Nagata:1962-iavcv, Hironaka:1964-rsavz-1, Hironaka:1964-rsavz-2}, there is a smooth compactification $\jmath: X \Em \overline{X}$ such that the boundary $D = \overline{X} - X$ \Pth{with its reduced subscheme structure} is a normal crossings divisor.  Let $X^\an$, $\overline{X}^\an$, $\jmath^\an$, and $D^\an$ denote the analytifications \Pth{realized in the category of adic spaces over $\Spa(k, k^+)$, where $k^+ = \cO_k$}.  We shall equip $\overline{X}^\an$ with the log structure defined by $D^\an$, as in Example \ref{ex-log-adic-sp-ncd}.

In order to simplify the language, we shall use the term \emph{filtered connection} \Pth{\resp \emph{filtered regular connection}} to mean a filtered vector bundle on $X$ equipped with an integrable connection \Pth{\resp an integrable connection with regular singularities} satisfying the Griffiths transversality.  Likewise, we shall use the term \emph{filtered log connection} to mean a filtered vector bundle on $\overline{X}$ \Pth{\resp $\overline{X}^\an$} equipped with an integrable connection satisfying the Griffiths transversality.  In addition, by abuse of language, we shall say that a $\bZ_p$-local system on $X_\et$ is de Rham if its analytification is, and that a $\bZ_p$-local system on $\overline{X}^\an_\ket$ is de Rham if its restriction to $X^\an$ is.

Let $\bL$ be a $\bZ_p$-local system on $X_\et$, with analytification $\bL^\an$.  Let $\overline{\bL}^\an := \jmath^\an_{\ket, *}(\bL^\an)$ be its extension to a $\bZ_p$-local system on $\overline{X}^\an_\ket$ \Pth{by \cite[\aCor \logadiccorpuritylisse]{Diao/Lan/Liu/Zhu:lasfr}}.  By Theorem \ref{thm-log-RH-arith}, we obtain a filtered log connection $\Ddl(\overline{\bL}^\an)$ on $\overline{X}^\an$, which is the analytification of an algebraic one, by GAGA \Pth{see \cite{Kopf:1974-efava}}, which we abusively denote by $\Ddlalg(\bL)$.  Then its restriction to $X$ is a filtered regular connection
\begin{equation}\label{eq-def-DdR-alg}
    \DdRalg(\bL) := \bigl(\Ddlalg(\bL)\bigr)\big|_X.
\end{equation}

Let us summarize the constructions in the following commutative diagram:
\[
    \xymatrix{ {\bigl\{ \Utext{de Rham $\bZ_p$-local systems on $X_\et$} \bigr\}} \ar@{.>}^-{\DdRalg}[r] \ar_-{(\,\cdot\,)^{\an}}[d] & {\bigl\{ \Utext{filtered regular connections on $X$} \bigr\}} \\
    {\bigl\{ \Utext{de Rham $\bZ_p$-local systems on $X^\an_\et$} \bigr\}} \ar_-{\jmath^\an_{\ket, *}}[d] & {\bigl\{ \Utext{filtered log connections on $\overline{X}$} \bigr\}} \ar_-{\jmath^*}[u] \ar_-{(\,\cdot\,)^{\an}}^-{\Utext{$\cong$ by GAGA}}[d] \\
    {\bigl\{ \Utext{de Rham $\bZ_p$-local systems on $\overline{X}^\an_\ket$} \bigr\}} \ar^-{\Ddl}[r] & {\bigl\{ \Utext{filtered log connections on $\overline{X}^\an$} \bigr\}} }
\]
Note that the de Rham assumptions on the local systems ensure that the associated regular connections or log connections are of the right ranks, and are filtered by vector subbundles \Pth{rather than more general coherent subsheaves} with vector bundles as associated graded pieces.

\begin{lemma}\label{lem-DdR-tensor}
    The functor $\DdRalg$ is a tensor functor, and is independent of the choice of the compactification $\overline{X}$.
\end{lemma}
\begin{proof}
    By Proposition \ref{prop-res-Ddl}, and by \cite[\aCh 1, \aProp 6.2.2]{Andre/Baldassarri:2001-DDA} or \cite[\aSec 11.1.3]{Andre/Baldassarri/Cailotto:2020-DDA(2)}, for all $\bL$, the exponents of the integrable connection $\DdRalg(\bL)$ consist of only rational numbers, which are not Liouville numbers.  \Pth{See, for example, \cite[\aSec 32.1]{Andre/Baldassarri/Cailotto:2020-DDA(2)} for a review.}  Then the lemma follows from the following two facts:
    \begin{enumerate}
        \item\label{lem-DdR-tensor-1} By \cite{Baldassarri:1988-ccapc-2}, \cite[\aCh 4, \aThm 4.1]{Andre/Baldassarri:2001-DDA}, or \cite[\aThm 32.2.1]{Andre/Baldassarri/Cailotto:2020-DDA(2)}; and by the same argument as in the proof of \cite[\aCh 4, \aCors 3.6]{Andre/Baldassarri:2001-DDA} or \cite[\aCor 31.4.6]{Andre/Baldassarri/Cailotto:2020-DDA(2)}, the analytification functor from the category of algebraic regular connections on $X$ whose exponents contain no Liouville numbers to the category of analytic ones on $X^{\an}$ is fully faithful.

        \item\label{lem-DdR-tensor-2} The composition of $\DdRalg$ with the analytification functor is the functor $\DdR$ in \cite[\aThm 3.9(v)]{Liu/Zhu:2017-rrhpl}, a tensor functor independent of the choice of $\overline{X}$.  \qedhere
    \end{enumerate}
\end{proof}

It remains to establish the comparison isomorphism in Theorem \ref{thm-intro-main}.  As in Section \ref{sec-comp-coh}, let $K = \widehat{\AC{k}}$, so that the rings $\BdRp$ and $\BdR$ in Definition \ref{def-OXBdR}\Refenum{\ref{def-OXBdR-1}} have their usual meaning as Fontaine's rings.  By \cite[\aProp 2.1.4 and \aThm 3.8.1]{Huber:1996-ERA}, if $\bL$ is an \'etale $\bZ_p$-local system on $X$, and if $\bL^\an$ is its analytification on $X^\an$, then we have a canonical $\Gal(\AC{k} / k)$-equivariant isomorphism
\begin{equation}\label{eq-comp-Huber-alg-an}
    H^i_\et\bigl(X_{\AC{k}}, \bL\bigr) \cong H^i_\et\bigl(X^\an_{\AC{k}}, \bL^\an\bigr).
\end{equation}
By \cite[\aCor \logadiccorpuritylisse]{Diao/Lan/Liu/Zhu:lasfr} and Theorem \ref{thm-log-RH-arith}\Refenum{\ref{thm-log-RH-arith-comp}}, we have a canonical isomorphism $H^i_\et\bigl(X^\an_{\AC{k}}, \bL^\an\bigr) \otimes_{\bZ_p} \BdR \cong H^i_{\log \dR}\bigl(\overline{X}^\an, \Ddl(\overline{\bL}^\an)\bigr) \otimes_k \BdR$, compatible with the filtrations and $\Gal(\AC{k} / k)$-actions on both sides.  Finally, by GAGA again \Pth{see \cite{Kopf:1974-efava}} and by Deligne's comparison result in \cite[II, 6]{Deligne:1970-EDR}, we have
\[
    H^i_{\log \dR}\bigl(\overline{X}^\an, \Ddl(\overline{\bL}^\an)\bigr) \cong H^i_{\log \dR}\bigl(\overline{X}, \Ddlalg(\bL)\bigr) \cong H^i_\dR\bigl(X, \DdRalg(\bL)\bigr).
\]
This completes the proof of Theorem \ref{thm-intro-main}.

By combining \Refeq{\ref{eq-comp-Huber-alg-an}}, \cite[\aCor \logadiccorpuritylisse]{Diao/Lan/Liu/Zhu:lasfr}, and GAGA \Pth{see \cite{Kopf:1974-efava}} with the other assertions in Theorem \ref{thm-log-RH-arith}\Refenum{\ref{thm-log-RH-arith-comp}}, we also obtain the following:
\begin{thm}\label{thm-HT-degen-comp}
    In the above setting, the \Pth{log} Hodge--de Rham spectral sequence
    \[
        E_1^{a, b} = H^{a, b}_{\log \Hdg}\bigl(\overline{X}, \Ddlalg(\bL)\bigr) \Rightarrow H^{a + b}_{\log \dR}\bigl(\overline{X}, \Ddlalg(\bL)\bigr)
    \]
    degenerates on the $E_1$ page, and the $0$-th graded piece of \Refeq{\ref{eq-thm-intro-main}} can be identified with a canonical $\Gal(\AC{k} / k)$-equivariant comparison isomorphism
    \[
        H^i_\et\bigl(X_{\AC{k}}, \bL\bigr) \otimes_{\bQ_p} \widehat{\AC{k}} \cong \oplus_{a + b = i} \, \Bigl( H^{a, b}_{\log \Hdg}\bigl(\overline{X}, \Ddlalg(\bL)\bigr) \otimes_k \widehat{\AC{k}}(-a) \Bigr).
    \]
\end{thm}

\subsection{Generalizations of Kodaira--Akizuki--Nakano vanishing}\label{sec-van-gen}

This subsection will be devoted to the proof of the following theorem:
\begin{thm}\label{thm-van-gen}
    Let $X$ be a proper smooth algebraic variety of pure dimension $d$ over a $p$-adic field $k$, with a reduced normal crossings divisor $D$.  Let $U := X - D$.  Let $\bL$ be a de Rham $\bQ_p$-local system on $U_\et$.  Let $\overline{\cE} = \Ddlalg(\bL)$ be as in Section \ref{sec-DdR-alg} \Pth{with $X$ and $\overline{X}$ there given by $U$ and $X$ here, respectively}.  Let $\DRl(\overline{\cE})$ and $\gr \DRl(\overline{\cE})$ be as in Definition \ref{def-log-conn-etc}.  Let $\cL$ be an invertible sheaf on $X$, with a \Pth{possibly empty} effective divisor $D'$ supported on \Pth{a subdivisor of} $D$ such that
    \begin{equation}\label{eq-thm-van-gen-cond-EV}
        \Utext{$\cL^N(-D')$ is ample for all sufficiently large $N$.}
    \end{equation}
    Then we have
    \begin{align}
        H^i\bigl(X, \cL^{-1} \otimes_{\cO_X} \gr \DRl(\overline{\cE})\bigr) = 0, \; \Utext{for all $i < d$}; \label{eq-thm-van-gen} \\
        H^i\bigl(X, \cL(-D) \otimes_{\cO_X} \gr \DRl(\dual{\overline{\cE}})\bigr) = 0, \; \Utext{for all $i > d$}. \label{eq-thm-van-gen-dual}
    \end{align}
    If $\bL$ has \emph{unipotent} geometric monodromy along $D$, then we also have
    \begin{equation}\label{eq-thm-van-gen-dual-unip}
        H^i\bigl(X, \cL(-D) \otimes_{\cO_X} \gr \DRl(\overline{\cE})\bigr) = 0, \; \Utext{for all $i > d$}.
    \end{equation}
\end{thm}

\begin{rk}\label{rem-van-cond}
    The condition \Refeq{\ref{eq-thm-van-gen-cond-EV}} implies that $\cL$ is \emph{nef} and \emph{big}---see \cite[\aRem 11.6 a)]{Esnault/Viehweg:1992-LVT-B}.  In fact, it is equivalent to being nef and big up to applying embedded resolution of singularities as in \cite{Hironaka:1964-rsavz-1, Hironaka:1964-rsavz-2}---see \cite[footnote 1]{Suh:2018-vmhma}.
\end{rk}

\begin{rk}\label{rem-van-prev}
    When $\bL$ is trivial, in which case $\overline{\cE} = \cO_X$, our $p$-adic Hodge-theoretic proof of Theorem \ref{thm-van-gen} provides new proofs for the classical vanishing theorems \Pth{in characteristic zero} due to Kodaira, Akizuki, and Nakano \cite{Kodaira:1953-odgma, Akizuki/Nakano:1954-nkspl} \Pth{when $D = \emptyset$}; Deligne, Illusie, and Raynaud \cite{Deligne/Illusie:1987-rdcdr} \Pth{when $D' = \emptyset$}; and Esnault and Viehweg \cite{Esnault/Viehweg:1992-LVT-B}.  Also, when $\bL$ is of the form $R^a f_*(\bQ_p)$ for some $a$ and some proper smooth morphism $f: V \to U$, Theorem \ref{thm-van-gen} provides a $p$-adic Hodge-theoretic generalization \Pth{as opposed to the complex analytic one in \cite{Suh:2018-vmhma}} of the characteristic-zero consequences in \cite{Illusie:1990-rsdcd} and \cite[\aSec 3]{Lan/Suh:2013-vttag}, without having to assume that $f$ extends to a proper morphism $Y \to X$ with very good properties.
\end{rk}

\begin{proof}[Proof of Theorem \ref{thm-van-gen}]
    It suffices to prove \Refeq{\ref{eq-thm-van-gen}}, since \Refeq{\ref{eq-thm-van-gen-dual}} follows by Serre duality, and since \Refeq{\ref{eq-thm-van-gen-dual-unip}} follows because $\dual{\overline{\cE}} \cong \Ddlalg(\dual{\bL})$ under the unipotency assumption, by Theorem \ref{thm-unip-vs-nilp} and GAGA \cite{Kopf:1974-efava}.  We will closely follow the first strategy in \cite[\aSec 2]{Suh:2018-vmhma}, but with the input from Saito's direct image theorem \Pth{see \cite[\aThm 2.14]{Saito:1990-mhm}} replaced with our $p$-adic Hodge-theoretic results.

    We claim that, up to replacing $D'$ with a positive multiple, we may assume that there exists some $N_0$ such that $\cL^N(-D')$ is very ample for all $N \geq N_0$.  When $D' = \emptyset$, the claim is clearly true, and the remainder of this proof establishes the special case of this theorem when $D' = \emptyset$.  When $D' \neq \emptyset$, the claim follows from the same argument as in the proof of \cite[(**) in the proof of \aProp 11.5]{Esnault/Viehweg:1992-LVT-B}, with the input \cite[\aCor 11.3]{Esnault/Viehweg:1992-LVT-B} of \cite[\aCor 11.4]{Esnault/Viehweg:1992-LVT-B} there replaced with the special case of this theorem when $D' = \emptyset$, whose proof we have just explained.

    We may enlarge $N_0$ and assume that, along each irreducible component $Z$ of $D$ which has multiplicity $e_Z$ in $D'$, the eigenvalues of the residue of $\overline{\cE}$ are contained in $\bQ \cap [0, 1 - \frac{e_Z}{N_0})$.  By the same Bertini-type argument as in \cite[\aSec 2.1]{Lan/Suh:2013-vttag}, there exist $N \geq N_0$ and $s \in H^0\bigl(X, \cL^N(-D')\bigr)$ such that the corresponding hyperplane section $H \subset X$ is smooth and meets $D$ transversally, so that $D + H$ and $D|_H$ are normal crossings divisors on $X$ and $H$, respectively.  Up to replacing $k$ with a finite extension, we may assume that $k$ contains all the $N$-th roots of unity in $\AC{k}$.

    Let $\imath: H \to X$ denote the canonical closed immersion.  For the sake of clarity, we shall denote by $\DR_{\log D}(\,\cdot\,)$ the log de Rham complex associated with $\Omega^\bullet_X(\log D)$, and similarly denote complexes associated with log structures defined by other normal crossings divisors.  In order to prove \Refeq{\ref{eq-thm-van-gen}}, by considering the long exact sequence associated with the following twist of the adjunction exact sequence
    \[
    \begin{split}
        0 & \to \cL^{-1} \otimes_{\cO_X} \gr \DR_{\log D}(\overline{\cE}) \to \cL^{-1} \otimes_{\cO_X} \gr \DR_{\log (D + H)}(\overline{\cE}) \\
        & \to \imath_*\bigl(\cL|_H^{-1} \otimes_{\cO_H} \gr \DR_{\log(D|_H)}(\overline{\cE}|_H)(-1)\bigr)[-1] \to 0
    \end{split}
    \]
    \Pth{in which the Tate twist $(-1)$ is just a shift of grading by $-1$}, and by induction on the dimension of $X$ \Pth{since $\overline{\cE}|_H \cong \Ddlalg(\bL|_{U \cap H})$ by Theorem \ref{thm-log-RH-arith}\Refenum{\ref{thm-log-RH-arith-mor}}, and since pulling back under the immersion $\imath$ preserves ampleness}, it suffices to prove that
    \begin{equation}\label{eq-thm-van-gen-mod}
        H^i\bigl(X, \cL^{-1} \otimes_{\cO_X} \gr \DR_{\log(D + H)}(\overline{\cE})\bigr) = 0, \; \Utext{for all $i < d$}.
    \end{equation}

    As in \cite[\aSec 3]{Esnault/Viehweg:1992-LVT-B}, consider $\cL^{(a)^{-1}} := \cL^{-a}(\lfloor\tfrac{a (D' + H)}{N}\rfloor)$, which is equipped with an integrable log connection $\nabla^{(a)}$ such that the eigenvalues of the residue of $\nabla^{(a)}$ along $H$ \Pth{\resp each irreducible component $Z$ of $D$} are $\frac{a}{N}$ \Pth{\resp $\frac{a e_Z}{N} - \lfloor\frac{a e_Z}{N}\rfloor$}.  Let $Y$ denote the relative spectrum of the $\cO_X$-algebra $\oplus_{a = 0}^{N - 1} \, \cL^{(a)^{-1}}$, whose multiplicative structure is induced by the dual of $\cO_X \Mapn{s} \cL^{\otimes N}(-D') \subset \cL^{\otimes N}$.  Then the \emph{cyclic cover} $\pi: Y \to X$ is finite flat, and the pullback of $\pi$ to $W := X - (D + H)$ is a finite \'etale Galois cover $\pi_W: V \to W$ with Galois group $\Hom\bigl(\bZ / N \bZ, k^\times\bigr)$.

    By construction, $\cL^{(a)^{-1}}|_W \cong \pi_{W, *}(\cO_V)[\chi_a]$, where $[\chi_a]$ denotes the isotypical component for the character $\chi_a: \Hom\bigl(\bZ / N \bZ, k^\times\bigr) \to k^\times$ defined by evaluation at the image of $a$, which is compatible with the connections \Pth{and trivial filtrations} on both sides.  Consider $\bM_a := \pi_{W, \et, *}(k)[\chi_a]$, where $k$ denotes the constant $k$-local system on $V$ of rank one; \ie, a constant $\bQ_p$-local system of rank $[k : \bQ_p]$ equipped with the canonical action of $k$.  Then $\bM_a$ is a $k$-local system on $W$ of rank one.  Let $\tau: k \otimes_{\bQ_p} k \to k$ be the multiplication map, and let $\tau M$ denote the pushout via $\tau$ of any $k \otimes_{\bQ_p} k$-module $M$.  Since $\pi_{W, *}(\cO_V) \cong \tau \DdRalg\bigl(\pi_{W, \et, *}(k)\bigr)$, by Theorem \ref{thm-log-RH-arith}\Refenum{\ref{thm-log-RH-arith-push}}, $\cL^{(a)^{-1}}|_W \cong \tau \DdRalg(\bM_a)$, which uniquely extends to $\cL^{(a)^{-1}} \cong \tau \Ddlalg(\bM_a)$ by \cite[\aCh 1, \aProp 4.7]{Andre/Baldassarri:2001-DDA} or \cite[\aThm 11.2.2]{Andre/Baldassarri/Cailotto:2020-DDA(2)}, because both sides have eigenvalues of residues in $\bQ \cap [0, 1)$, by the above and Theorem \ref{thm-log-RH-arith}\Refenum{\ref{thm-log-RH-arith-res}}.

    Since $\frac{e_Z}{N} \leq \frac{e_Z}{N_0} < 1$ for each irreducible component $Z$ of $D$, we have $\cL^{(1)^{-1}} = \cL^{-1}$.  Since $\overline{\cE}|_U \cong \DdRalg(\bL)$, the residue of $\overline{\cE}$ along $H$ is zero.  By Lemma \ref{lem-DdR-tensor}, we have $\cL^{-1}|_W \otimes_{\cO_W} \overline{\cE}|_W \cong \tau \DdRalg(\bM_1 \otimes_{\bQ_p} \bL|_W)$, which uniquely extends to $\cL^{-1} \otimes_{\cO_X} \overline{\cE} \cong \tau \Ddlalg(\bM_1 \otimes_{\bQ_p} \bL|_W)$, again because both sides have eigenvalues of residues in $\bQ \cap [0, 1)$.  Thus, the Hodge--de Rham spectral sequences for $\Ddlalg(\bM_1 \otimes_{\bQ_p} \bL|_W)$ and $\cL^{-1} \otimes_{\cO_X} \overline{\cE}$ degenerate by Theorem \ref{thm-HT-degen-comp}, and \Refeq{\ref{eq-thm-van-gen-mod}} is equivalent to
    \begin{equation}\label{eq-thm-van-gen-mod-log-dR}
        H^i\bigl(X, \cL^{-1} \otimes_{\cO_X} \DR_{\log(D + H)}(\overline{\cE})\bigr) = 0, \; \Utext{for all $i < d$}.
    \end{equation}
    Since the eigenvalue of the residue of $\cL^{-1} \otimes_{\cO_X} \overline{\cE}$ along $H$ is positive, by \cite[\aLem 2.10]{Esnault/Viehweg:1992-LVT-B}, for any $b \geq 0$, the statement \Refeq{\ref{eq-thm-van-gen-mod-log-dR}} is in turn equivalent to
    \begin{equation}\label{eq-thm-van-gen-mod-log-dR-twist}
        H^i\bigl(X, \cL^{-1}(-b H) \otimes_{\cO_X} \DR_{\log(D + H)}(\overline{\cE})\bigr) = 0, \; \Utext{for all $i < d$}.
    \end{equation}
    Finally, by considering the filtration spectral sequence, it suffices to show that, for \emph{some} $b \geq 0$, we have
    \begin{equation}\label{eq-thm-van-gen-mod-log-Hodge-twist}
        H^i\bigl(X, \cL^{-1}(-b H) \otimes_{\cO_X} \gr \DR_{\log(D + H)}(\overline{\cE})\bigr) = 0, \; \Utext{for all $i < d$}.
    \end{equation}
    Since the divisor $H$ is ample, and since $\gr \DR_{\log(D + H)}(\overline{\cE})$ is a complex of finite locally free $\cO_X$-modules concentrated in degrees $[0, d]$, by considering the spectral sequence associated with the stupid \Pth{\Qtn{b\^ete}} filtration, the last statement \Refeq{\ref{eq-thm-van-gen-mod-log-Hodge-twist}} holds for some $b \gg 0$, by Serre vanishing and Serre duality, as desired.
\end{proof}

\subsection{De Rham local systems at the boundary}\label{sec-dR-bd}

In this subsection, we apply the results in Section \ref{sec-comp-nearby} to study nearby cycles in some simple cases.  We will leave a more general treatment to a future work.

Let $X$ be an algebraic variety with a divisor $D$ over $k$.  Suppose that there exist an \'etale neighborhood $D \to W \to X$ and a morphism $f: W \to \bA^1$ over $k$ such that $f^{-1}(0) = D$.  In this case, there is the notion of \emph{unipotent and quasi-unipotent nearby cycles} due to Beilinson \Pth{see \cite{Beilinson:1987-hgps}; \Refcf{} \cite{Reich:2010-nbhgp}}.  Let us briefly recall the definition.  Let $\bG_m := \bA^1 - \{0\}$ be the multiplicative group scheme over $k$.  We have a canonical isomorphism $\pi_1(\bG_m, 1) \cong \pi_1(\bG_{m, \AC{k}}, 1) \rtimes \Gal(\AC{k} / k)$, and $\pi_1(\bG_{m, \AC{k}}, 1) \cong \widehat{\bZ}(1)$ as $\Gal(\AC{k} / k)$-modules.  For each $r \geq 1$, let $\bJ_r$ denote the rank $r$ unipotent \'etale $\bQ_p$-local system on $\bG_m$ defined by the representation of $\pi_1(\bG_m, 1)$ on $\bQ_p^r$ such that a topological generator $\gamma \in \pi_1(\bG_{m, \AC{k}}, 1)$ acts as a principal unipotent matrix $J_r$ and such that $\Gal(\AC{k} / k)$ acts diagonally on $\bQ_p^r$ and trivially on $\ker(J_r - 1)$.  There is an obvious inclusion $\bJ_r \Em \bJ_{r + 1}$, and a projection $\bJ_{r + 1} \to \bJ_r(-1)$ such that the composition $\bJ_r \to \bJ_r(-1)$ is given by the monodromy action.  For each $m \geq 1$, let $[m]$ denote the $m$-th power homomorphism of $\bG_m$, and let $\bK_m := [m]_*(\bQ_p)$.  If $m \mid m'$, there is a natural inclusion $\bK_m \Em \bK_{m'}$ \Pth{defined by adjunction}.  Let $U := W - D$, and let $\imath: D \to W$ and $\jmath: U \to W$ denote the canonical morphisms.  We shall also denote by $\bJ_r$ and $\bK_m$ their pullbacks to $U$.  Then for each $\bQ_p$-perverse sheaf $\cF$ on $U_\et$, its unipotent and quasi-unipotent nearby cycles are
\[
    R\Psi^\unip_f(\cF) :=  \varinjlim_r \, \imath^{-1} \, R\jmath_*(\cF \otimes_{\bQ_p} \bJ_r) \quad \Utext{and} \quad R\Psi^\qunip_f(\cF) :=  \varinjlim_m \, R\Psi^\unip_f(\cF \otimes_{\bQ_p} \bK_m),
\]
respectively, where the limits are taken in the category of perverse sheaves on $D_\et$.

Let $\bL$ be a $\bQ_p$-local system on $U_\et$.  Let $f^\an: W^\an \to \bA^{1, \an}$ denote the analytification of $f$, whose pullback under $\bD \Em \bA^{1, \an}$ we denote by $f^\an_\bD$.  If the reduced subspace of $D^\an$ is a normal crossings divisor in $X^\an$, the quasi-unipotent nearby cycles $R\Psi^\qunip_{f^\an_\bD}(\bL^\an)$ has been introduced in \cite[\aDef \logadicdefnearby]{Diao/Lan/Liu/Zhu:lasfr}.

\begin{lemma}\label{prop-unip-nearby-ana-alg-comp}
    In the above setting, we have $\bigl(R\Psi_f^{\qunip}(\bL)\bigr)^\an \cong R\Psi^\qunip_{f^\an_\bD}(\bL^\an)$.
\end{lemma}
\begin{proof}
   This follows from \cite[\aProp 2.1.4 and \aThm 3.8.1]{Huber:1996-ERA} and \cite[\aLem \logadiclemclimmketmor{} and \aThm \logadicthmpurity]{Diao/Lan/Liu/Zhu:lasfr}.
\end{proof}

Suppose moreover that $f$ is smooth, and that $(F, \nabla)$ is a vector bundle with an integrable connection on $U = W - D$ that \Pth{necessarily uniquely} extends to a vector bundle $\overline{F}$ on $W$ with a log connection $\overline{\nabla}$ whose eigenvalues of residues along $\overline{D}$ belong to $\bQ \cap [0, 1)$.  Then we can define the $\bQ$-filtration $V^\bullet$ on $F(* D) := \cup_n \, F(n D)$, as in Lemma \ref{lem-V-fil}; and $R\Psi_f^\unip(F, \nabla)$ and $R\Psi_f^\qunip(F, \nabla)$ \Pth{which are defined much more generally using the theory of holonomic algebraic $D$-modules} are canonically isomorphic to $\gr_V^0 F(* D)$ \Pth{\resp $\oplus_{\alpha \in (-1, 0]} \, \bigl(\gr_V^\alpha F(* D)\bigr)$}, with canonically induced integrable connections and filtrations \Pth{\Refcf{} \cite[(5.1.3.3)]{Saito:1988-mhp}}.  By \cite{Nagata:1962-iavcv, Hironaka:1964-rsavz-1, Hironaka:1964-rsavz-2} again, we can compactify $(W, D)$ to some $(\overline{W}, \overline{D})$, where $\overline{W}$ is proper, and where $\overline{D}$ is a simple normal crossings divisor such that $D = W \cap \overline{D}$ and the closure of $D$ in $\overline{D}$ is a union of smooth irreducible components of $\overline{D}$.  Since we have a log connection $\Ddlalg(\bL)$ on $\overline{W}$ as in Section \ref{sec-DdR-alg}, its restriction to $W$ gives an extension of $\DdRalg(\bL)$ as in the last paragraph, and hence we have $R\Psi^\unip_f\bigl(\DdRalg(\bL)\bigr) \cong \gr_V^0 \DdRalg(\bL)(* D)$ and $R\Psi^\qunip_f\bigl(\DdRalg(\bL)\bigr) \cong \oplus_{\alpha \in (-1, 0]} \, \bigl(\gr_V^\alpha \DdRalg(\bL)(* D)\bigr)$.
\begin{thm}\label{thm-nearby-comp-alg}
    Assume that $f$ is smooth.  Let $\bL$ be a de Rham $\bQ_p$-local system on $U_\et$.  Then $R\Psi^{\qunip}_f(\bL)$ is a de Rham $\bQ_p$-local system on $D_\et$, and there is a canonical isomorphism $\DdRalg\bigl(R\Psi^\qunip_f(\bL)\bigr) \cong R\Psi^\qunip_f\bigl(\DdRalg(\bL)\bigr)$ which restricts to an isomorphism $\DdRalg\bigl(R\Psi^\unip_f(\bL)\bigr) \cong R\Psi^\unip_f\bigl(\DdRalg(\bL)\bigr)$, as filtered \Pth{integrable} connections.
\end{thm}
\begin{proof}
    As explained in Lemma \ref{lem-DdR-tensor}, all the exponents of $\DdRalg\bigl(R\Psi^\qunip_f(\bL)\bigr)$ are non-Liouville numbers.  Moreover, since the eigenvalues of the residues of the log connection $\Ddlalg(\bL)$ on $\overline{W}$ along the irreducible divisors of $\overline{D}$ are all in $\bQ \cap [0, 1)$, the exponents of the connection $\Ddlalg(\bL)|_D^0$ are also non-Liouville numbers.  Thus, the theorem follows from the algebraization of the canonical isomorphisms in Theorem \ref{thm-nearby-comp-arithm}, by using Lemma \ref{prop-unip-nearby-ana-alg-comp} and the fact \Refenum{\ref{lem-DdR-tensor-1}} in the proof of Lemma \ref{lem-DdR-tensor}.
\end{proof}

\begin{rk}\label{rem-nearby-full}
    As the geometric monodromy of $\bL$ along $D$ is quasi-unipotent \Pth{see \cite[\aDef \logadicdefunipqunipmonod{} and \aRem \logadicrkunipqunipmonodalgcl]{Diao/Lan/Liu/Zhu:lasfr} and the proof of Lemma \ref{lem-poss-eval-N-infty-xi-r}}, and as the eigenvalues of the residue of $\Ddlalg(\bL)$ along $D$ are in $\bQ \cap [0, 1)$, the quasi-unipotent nearby cycles of $\bL$ and $\DdRalg(\bL)$ coincide with their respective full nearby cycles.
\end{rk}

When $X$ is a smooth curve over $k$, Theorem \ref{thm-nearby-comp-alg} has the following concrete interpretation.  In this case, $D = x$ is a $k$-point, and $f = z$ is an \'etale local coordinate of $X$ at $x$.  We can identify $R\Psi_z(\bL)$ with the finite-dimensional $\bQ_p$-representation $\bL_{\AC{\eta}_x}$ of $\Gal(\AC{K}_x / K_x)$, where $K_x$ is the local field around $x$, and $\AC{\eta}_x$ is a geometric point above $\eta_x = \Spec(K_x)$, which specializes to a geometric point $\AC{x} = \Spec(\AC{k})$ above $x$.  The coordinate $z$ splits the natural projection $\Gal(\AC{K}_x / K_x) \to \Gal(\AC{k} / k)$, and so we may regard $R\Psi_z(\bL)$ as a representation of $\Gal(\AC{k}/k)$.
\begin{cor}\label{cor-nearby-curve}
    If $\bL$ is a de Rham $\bQ_p$-local system on $(X - x)_\et$, then $R\Psi_z(\bL)$ is a de Rham representation of $\Gal(\AC{k} / k)$ \Pth{with the choice of coordinate $z$}.
\end{cor}

\section{Application to Shimura Varieties}\label{sec-Sh-var}

In this section, we shall prove Theorem \ref{thm-intro-Sh}, which serves as an evidence of Conjecture \ref{conj-intro}, and also Corollary \ref{cor-intro-GM}.  In order to avoid confusion, the symbol $K$ will be reserved for levels \Pth{rather than fields}.  For simplicity, we shall continue to use the term \emph{filtered log connection} to mean a filtered vector bundle equipped with an integrable connection satisfying the Griffiths transversality, as in Section \ref{sec-DdR-alg}.

\subsection{The setup}\label{sec-loc-syst-setup}

Let $(\Grp{G}, \Shdom)$ be any Shimura datum.  That is, $\Grp{G}$ is a connected reductive $\bQ$-group, and $\Shdom$ is a hermitian symmetric domain parameterizing a conjugacy class of homomorphisms
\begin{equation}\label{eq-h}
    \hd: \DelS := \Res_{\bC / \bR} \Gm{\bC} \to \Grp{G}_\bR,
\end{equation}
satisfying a list of axioms \Pth{see \cite[2.1.1]{Deligne:1979-vsimc} and \cite[\aDef 5.5]{Milne:2005-isv}}.  For each \emph{neat} \Pth{see \cite[0.6]{Pink:1989-Ph-D-Thesis}} open compact subgroup $\levcp$ of $\Grp{G}(\bAi)$, we denote by $\Model_\levcp = \Sh_\levcp(\Grp{G}, \Shdom)$ the canonical model of the associated Shimura variety at level $\levcp$, which is a smooth quasi-projective algebraic variety over a number field $\ReFl\subset\bC$, called the \emph{reflex field} $\ReFl$ of $(\Grp{G}, \Shdom)$.  Recall that, essentially by definition, the analytification of its base change $\Model_{\levcp, \bC}$ from $\ReFl$ to $\bC$ is the \emph{complex manifold}
\begin{equation}\label{eq-Sh-var-C-an}
    \Model_{\levcp, \bC}^\an \cong \Grp{G}(\bQ) \Lquot \bigl( \Shdom \times \Grp{G}(\bAi) \bigr) / \levcp,
\end{equation}
where $\Grp{G}(\bQ)$ acts diagonally on $\Shdom \times \Grp{G}(\bAi)$ from the left, and where $\levcp$ acts trivially on $\Shdom$ and canonically on $\Grp{G}(\bAi)$ from the right.  Note that right multiplication by $g \in \Grp{G}(\bAi)$ induces an isomorphism $[g]: \Model_{g \levcp g^{-1}, \bC}^\an \Mi \Model_{\levcp, \bC}^\an$, which algebraizes and descends to an isomorphism $\Model_{g \levcp g^{-1}} \Mi \Model_\levcp$, still denoted by $[g]$.  \Pth{See \cite{Milne:2005-isv, Lan:2017-ebisv} and the references there for basic facts concerning Shimura varieties.}

Given neat open compact subgroups $\levcp_1$ and $\levcp_2$ such that $\levcp_1$ is a normal subgroup of $\levcp_2$, we obtain a finite \'etale cover $\Model_{\levcp_1} \to \Model_{\levcp_2}$ with a canonical $\levcp_2 / \levcp_1$-action.  It will be convenient to consider the projective system $\{ \Model_\levcp \}_\levcp$, which can be viewed as the scheme $\Model := \varprojlim_\levcp \Model_\levcp$ over $\ReFl$, which admits the canonical right action of $\Grp{G}(\bAi)$ described above.  We call these actions \Pth{and their various extensions to other objects} \emph{Hecke actions} of $\Grp{G}(\bAi)$ \Pth{sometimes with $\Grp{G}(\bAi)$ omitted}.

Let $\Grp{G}^c$ be the quotient of $\Grp{G}$ by the minimal subtorus $Z_s(\Grp{G})$ of the center $Z(\Grp{G})$ of $\Grp{G}$ such that the torus $Z(\Grp{G})^\circ / Z_s(\Grp{G})$ has the same split ranks over $\bQ$ and $\bR$.  \Pth{This is equivalent to the definition in \cite[\aCh III]{Milne:1990-cmsab} when $(\Grp{G}, \Shdom)$ satisfies \cite[(II.2.1.4)]{Milne:1990-cmsab}.}  Let $\Grp{G}^\der$ denote the derived group of $\Grp{G}$, and let $\Grp{G}^{\der, c}$ denote the image of $\Grp{G}^\der$ in $\Grp{G}^c$.  Let $\Grp{G}^\ad$ denote the adjoint quotient of $\Grp{G}$.  We have the canonical central isogenies $\Grp{G}^\der \to \Grp{G}^{\der, c} \to \Grp{G}^\ad$ of connected semisimple $\bQ$-algebraic groups.

For each field $F$, let $\Rep_F(\Grp{G}^c)$ denote the category of finite-dimensional algebraic representations of $\Grp{G}^c$ over $F$, which we also view as an algebraic representation of $\Grp{G}$ by pullback.  Let $\AC{\bQ}$ denote the algebraic closure of $\bQ$ in $\bC$, and let $\AC{\bQ}_p$ be an algebraic closure of $\bQ_p$, together with a fixed isomorphism $\ACMap: \AC{\bQ}_p \Mi \bC$, which induces an injective field homomorphism $\ACMap^{-1}|_{\AC{\bQ}}: \AC{\bQ} \Em \AC{\bQ}_p$.

\subsection{Local systems on Shimura varieties}\label{sec-loc-syst-constr}

Let us begin with the complex analytic constructions.  For any $\rep \in \Rep_{\AC{\bQ}}(\Grp{G}^c)$, we define the \emph{\Pth{Betti} $\AC{\bQ}$-local system}
\[
    \BSh{\rep} := \Grp{G}(\bQ) \Lquot \bigl( ( \Shdom \times \rep) \times \Grp{G}(\bAi) \bigr) / \levcp
\]
on $\Model_{\levcp, \bC}^\an$.  \Pth{See Proposition \ref{prop-loc-syst-infty} below for some formal properties.}

Let us also explain the construction of $\BSh{\rep}$ more concretely via the representation of the fundamental groups of \Pth{connected components} of $\Model_{\levcp, \bC}^\an$, under the classical correspondence between local systems and fundamental group representations.

Suppose that we have a connected component of $\Model_{\levcp, \bC}^\an$ \Pth{see \Refeq{\ref{eq-Sh-var-C-an}}} given by
\begin{equation}\label{eq-Sh-var-conn}
    \Gamma^+_{\levcp, g_0} \Lquot \Shdom^+ \cong \Grp{G}(\bQ)_+ \Lquot \bigl( \Shdom^+ \times ( \Grp{G}(\bQ)_+ g_0 \levcp ) \bigr) / \levcp
\end{equation}
\Pth{\Refcf{} \cite[2.1.2]{Deligne:1979-vsimc} or \cite[\aLem 5.13]{Milne:2005-isv}}, where $\Shdom^+$ is a fixed connected component of $\Shdom$ and $g_0 \in \Grp{G}(\bAi)$, and where $\Grp{G}(\bQ)_+$ is the stabilizer of $\Shdom^+$ in $\Grp{G}(\bQ)$ and
\[
    \Gamma^+_{\levcp, g_0} := \Grp{G}(\bQ)_+ \cap (g_0 \levcp g_0^{-1})
\]
is a \emph{neat} \Pth{see \cite[17.1]{Borel:1969-IGA}} arithmetic subgroup of $\Grp{G}(\bQ)$.  It follows from the definitions that $\Gamma^+_{\levcp, g_0}$ is neat when $\levcp$ is.  Let $\Gamma^{+, c}_{\levcp, g_0}$ and $\Gamma^{+, \ad}_{\levcp, g_0}$ denote the images of $\Gamma^+_{\levcp, g_0}$ in $\Grp{G}^c(\bQ)$ and $\Grp{G}^\ad(\bQ)$, respectively, so that we have surjective homomorphisms
\begin{equation}\label{eq-Gamma+-surj}
    \Gamma^+_{\levcp, g_0} \Surj \Gamma^{+, c}_{\levcp, g_0} \Surj \Gamma^{+, \ad}_{\levcp, g_0}.
\end{equation}

\begin{lemma}\label{lem-Gamma+-c-isom-ad}
    The subgroup $\Gamma^{+, c}_{\levcp, g_0}$ of $\Grp{G}^c(\bQ)$ is contained in $\Grp{G}^{\der, c}(\bQ)$, and the second homomorphism in \Refeq{\ref{eq-Gamma+-surj}} is an isomorphism $\Gamma^{+, c}_{\levcp, g_0} \Mi \Gamma^{+, \ad}_{\levcp, g_0}$.
\end{lemma}
\begin{proof}
    Since $\ker(\Grp{G} \to \Grp{G}^c)$ is the maximal $\bQ$-anisotropic $\bR$-split subtorus of the center of $\Grp{G}$, the quotient $\Grp{G}^c / \Grp{G}^{\der, c}$ is a torus isogenous to a product of a split torus and a torus of compact type \Pth{\ie, $\bR$-anisotropic} over $\bQ$.  Since all neat arithmetic subgroups of such a torus are trivial, the neat image $\Gamma^{+, c}_{\levcp, g_0}$ of $\Gamma^+_{\levcp, g_0}$ in $\Grp{G}^c(\bQ)$ is contained in $\Grp{G}^{\der, c}(\bQ)$.  Consequently, the second homomorphism in \Refeq{\ref{eq-Gamma+-surj}} is an isomorphism, because its kernel, being both neat and \emph{finite}, is trivial.
\end{proof}

\begin{cor}\label{cor-Gamma+-c-as-fund-grp}
    The connected component $\Gamma^+_{\levcp, g_0} \Lquot \Shdom^+$ is a smooth manifold whose fundamental group \Pth{with any base point of $\Shdom^+$} is canonically isomorphic to $\Gamma^{+, c}_{\levcp, g_0}$.
\end{cor}
\begin{proof}
    As $\Gamma^+_{\levcp, g_0}$ acts on $\Shdom^+$ via $\Gamma^{+, \ad}_{\levcp, g_0} \subset \Grp{G}^\ad(\bQ)$, this follows from Lemma \ref{lem-Gamma+-c-isom-ad}.
\end{proof}

\begin{rk}\label{rem-Gamma+-ad-not-again}
    We shall not write $\Gamma^{+, \ad}_{\levcp, g_0}$ again in what follows.
\end{rk}

By taking $\Shdom^+$ as a universal cover of $\Gamma^+_{\levcp, g_0} \Lquot \Shdom^+$, and by fixing the choice of a base point on $\Shdom^+$, the pullback of $\BSh{\rep}$ to $\Gamma^+_{\levcp, g_0} \Lquot \Shdom^+$ determines and is determined by the fundamental group representation
\begin{equation}\label{eq-fund-grp-rep}
    \Frep^+_{\levcp, g_0}(\rep): \Gamma_{\levcp, g_0}^{+, c} \to \GL_{\AC{\bQ}}(\rep),
\end{equation}
which coincides with the restriction of the representation of $\Grp{G}^c$ on $\rep$.  In particular, it is compatible with the change of levels $K' \subset K$.

Moreover, given $g \in g_0^{-1} \Grp{G}(\bQ)_+ g_0$, so that $g_0 g = \gamma g_0$ for some $\gamma \in \Grp{G}(\bQ)_+$, we have $\Gamma^{+, c}_{g \levcp g^{-1}, g_0} = \gamma \Gamma^{+, c}_{\levcp, g_0} \gamma^{-1}$, and the Hecke action $[g]$ induces a morphism
\begin{equation}\label{eq-Hecke-g-conn-comp}
    \Gamma^+_{g \levcp g^{-1}, g_0} \Lquot \Shdom^+ \Mi \Gamma^+_{\levcp, g_0} \Lquot \Shdom^+,
\end{equation}
which is nothing but the isomorphism defined by left multiplication by $\gamma^{-1}$.  It follows that the canonical isomorphism $[g]^{-1}(\BSh{\rep}) \Mi \BSh{\rep}$ of local systems corresponds to the following equality of fundamental group representations
\begin{equation}\label{eq-Hecke-g-fund-rep}
    \Frep^+_{g \levcp g^{-1}, g_0}(\rep) = \gamma \bigl( \Frep^+_{\levcp, g_0}(\rep) \bigr),
\end{equation}
where $\gamma \bigl( \Frep^+_{\levcp, g_0}(\rep) \bigr)$ means the representation of $\Gamma_{g \levcp g^{-1}, g_0}^{+, c} = \gamma \Gamma_{\levcp, g_0}^{+, c} \gamma^{-1}$ defined by conjugating the values of $\Frep^+_{\levcp, g_0}(\rep)$ by $\gamma$ in $\GL(\rep)$.

Now, by base change along $\AC{\bQ}\subset \bC$ via the canonical homomorphism, we obtain the object $\rep_\bC := \rep \otimes_{\AC{\bQ}} \bC$ in $\Rep_\bC(\Grp{G}^c)$, as well as the $\bC$-local system
\[
    \BSh{\rep}_\bC := \BSh{\rep} \otimes_{\AC{\bQ}} \bC
\]
on $\Model_\levcp(\bC)$, which via the classical Riemann--Hilbert correspondence \Pth{as reviewed in the introduction} corresponds to the \Pth{complex analytic} integrable connection
\[
    (\dRSh{\rep}_\bC^\an, \nabla) := (\BSh{\rep}_\bC \otimes_\bC \cO_{\Model_{\levcp, \bC}^\an}, 1 \otimes d).
\]
Moreover, any $\hd \in \Shdom$ \Pth{as in \Refeq{\ref{eq-h}}} induces a homomorphism $\hd_\bC: \Gm{\bC} \times \Gm{\bC} \to \Grp{G}_\bC$, whose restriction to the first factor defines the so-called \emph{Hodge cocharacter}
\begin{equation}\label{eq-mu-h}
    \hc_\hd: \Gm{\bC} \to \Grp{G}_\bC,
\end{equation}
inducing a \Pth{decreasing} filtration $\Fil^\bullet$ on $\dRSh{\rep}_\bC^\an$ satisfying the Griffiths transversality condition.  Then we obtain a \emph{filtered integrable connection} $(\dRSh{\rep}_\bC^\an, \nabla, \Fil^\bullet)$.

Let $\Torcpt{\Model}_\levcp$ be a toroidal compactification of $\Model_\levcp$ \Pth{as in \cite{Pink:1989-Ph-D-Thesis}}, which we assume to be projective and smooth, with the boundary divisor $\NCD := \Torcpt{\Model}_\levcp - \Model_\levcp$ \Pth{with its reduced subscheme structure} a normal crossings divisor, whose base change from $\ReFl$ to $\bC$ and whose further complex analytification are denoted by $\Torcpt{\Model}_{\levcp, \bC}$ and $\Model_{\levcp, \bC}^{\Tor, \an}$, respectively.  As explained in \cite[\aSec 6.1]{Lan/Suh:2013-vttag}, $\BSh{\rep}_\bC$ has \emph{unipotent monodromy} along $\NCD_\bC^\an$.  Therefore, by \cite[II, 5]{Deligne:1970-EDR} and \cite[\aSecs VI and VII]{Katz:1970-rtag}, $(\dRSh{\rep}_\bC^\an, \nabla)$ uniquely extends to an integrable log connection $(\dRSh{\rep}_\bC^{\canext, \an}, \nabla)$, with \emph{nilpotent residues} along $\NCD_\bC^\an$.  By \cite[II, 5.2(d)]{Deligne:1970-EDR}, $V \mapsto (\dRSh{\rep}_\bC^{\canext, \an}, \nabla)$ defines a tensor functor from $\Rep_\bC(\Grp{G}^c)$ to the category of integrable log connections on $\Model_{\levcp, \bC}^{\Tor, \an}$.  Moreover, by \cite{Schmid:1973-vhssp} \Pth{see also \cite{Cattani/Kaplan/Schmid:1987-vhsam}}, the filtration $\Fil^\bullet$ on $\dRSh{\rep}_\bC^\an$ uniquely extends to a filtration on $\dRSh{\rep}_\bC^{\canext, \an}$ \Pth{by subbundles}, still denote by $\Fil^\bullet$.  The extended $\nabla$ and $\Fil^\bullet$ still satisfy the Griffiths transversality, and therefore $(\dRSh{\rep}_\bC^{\canext, \an}, \nabla, \Fil^\bullet)$ is an analytic \emph{filtered log connection}.  By GAGA \Pth{see the proof of \cite[II, 5.9]{Deligne:1970-EDR}}, this triple canonically algebraizes to an algebraic filtered log connection
\[
    (\dRSh{\rep}_\bC^\canext, \nabla, \Fil^\bullet).
\]
\Pth{These $\dRSh{\rep}_\bC^{\canext, \an}$ and $\dRSh{\rep}_\bC^\canext$ agree with the canonical extensions defined differently in \cite[\aSec 4]{Harris:1989-ftcls}, and also \cite{Harris:1990-afcvs} and \cite{Milne:1990-cmsab}.}  The restriction of $(\dRSh{\rep}_\bC^\canext, \nabla, \Fil^\bullet)$ then defines an algebraic filtered regular connection
\[
    (\dRSh{\rep}_\bC, \nabla, \Fil^\bullet)
\]
on $\Model_{\levcp, \bC}$, whose complex analytification is isomorphic to $(\dRSh{\rep}_\bC^\an, \nabla, \Fil^\bullet)$.  We call $(\dRSh{\rep}_\bC, \nabla)$ the \emph{automorphic vector bundle} associated with $\rep_\bC$.  We summarize the above discussions as the following:
\begin{prop}\label{prop-loc-syst-infty}
    The assignment of $\BSh{\rep}$ \Pth{\resp $(\dRSh{\rep}_\bC, \nabla, \Fil^\bullet)$} to $\rep$ defines a tensor functor from $\Rep_{\AC{\bQ}}(\Grp{G}^c)$ to the category of $\Grp{G}(\bAi)$-equivariant $\AC{\bQ}$-local systems \Pth{\resp filtered regular connections} on $\{ \Model_{\levcp, \bC}^\an \}_\levcp$ \Pth{\resp $\{ \Model_{\levcp, \bC} \}_\levcp$}, which is functorial with respect to pullbacks under morphisms between Shimura varieties induced by morphisms between Shimura data.  Hence, the assignment of $(\dRSh{\rep}_\bC, \nabla)$ \Pth{\resp $(\dRSh{\rep}_\bC, \nabla, \Fil^\bullet)$} to $\rep$ defines a $\Grp{G}(\bAi)$-equivariant $\Grp{G}^c$-bundle with an integrable connection $(\cE_\bC, \nabla)$ \Pth{\resp a $\Grp{P}^c_\bC$-bundle $\cE_{\Grp{P}^c_\bC}$} on $\{ \Model_{\levcp, \bC} \}_\levcp$, where $\Grp{P}^c_\bC$ is the parabolic subgroup of $\Grp{G}^c_\bC$ defined by some $\hc_\hd$ as in \Refeq{\ref{eq-mu-h}} \Pth{\Refcf{} \cite[\aRem 4.1(i)]{Liu/Zhu:2017-rrhpl}}.  By forgetting filtrations, we obtain a $\Grp{G}(\bAi)$-equivariant morphism $\cE_{\Grp{P}^c_\bC} \to \cE_\bC$.
\end{prop}

\begin{rk}\label{rem-partial-flag}
    As explained in \cite[\aRem 4.1(i)]{Liu/Zhu:2017-rrhpl}, the conjugacy class of $\hc_\hd$ as in \Refeq{\ref{eq-mu-h}} defines a partial flag variety $\Fl_\bC \cong \Grp{G}^c_\bC / \Grp{P}^c_\bC$ parameterizing the associated conjugacy class of parabolic subgroups, which depends only on the Shimura datum $(\Grp{G}, \Shdom)$ and descends to a partial flag variety $\Fl$ of $\Grp{G}^c$ over the reflex field $\ReFl$.  Let $\cE_{\Fl_\bC} := \cE_\bC \times^{\Grp{G}^c_\bC} \Fl_\bC$.  Then the filtrations on $\dRSh{\rep}_\bC$'s as in Proposition \ref{prop-loc-syst-infty} define a section of $\cE_{\Fl_\bC}$ over $\{ \Model_{\levcp, \bC} \}_\levcp$.  For any particular choice of $\Grp{P}^c_\bC$ in $\Fl(\bC)$, this section amounts to the reduction of $\cE_\bC$ to a $\Grp{P}^c_\bC$-bundle $\cE_{\Grp{P}^c_\bC}$ as in Proposition \ref{prop-loc-syst-infty}.  Moreover, if $(\cE, \nabla)$ is the canonical model of $(\cE_\bC, \nabla)$ as in \cite[\aCh III, \aThm 4.3]{Milne:1990-cmsab}, we also have the canonical model $\cE_\Fl := \cE \times^{\Grp{G}^c} \Fl$ of $\cE_{\Grp{P}^c_\bC} \cong \cE_{\Fl_\bC}$, over $\Model_\levcp$.
\end{rk}

Next, let us turn to the $p$-adic analytic constructions.  Given any $\rep \in \Rep_{\AC{\bQ}}(\Grp{G}^c)$ as above, by base change via $\ACMap^{-1}|_{\AC{\bQ}}: \AC{\bQ} \Em \AC{\bQ}_p$, we obtain the object $\rep_{\AC{\bQ}_p} := \rep \otimes_{\AC{\bQ}} \AC{\bQ}_p$ in $\Rep_{\AC{\bQ}_p}(\Grp{G}^c)$.  As explained in \cite[\aSec 3]{Lan/Stroh:2018-ncaes-2} \Pth{see also \cite[\aSec 4.2]{Liu/Zhu:2017-rrhpl}}, given such a finite-dimensional representation $\rep_{\AC{\bQ}_p}$ of $\Grp{G}^c$ over $\AC{\bQ}_p$, there is a canonical \emph{automorphic $\AC{\bQ}_p$-\'etale local system} \Pth{\ie, lisse $\AC{\bQ}_p$-\'etale sheaf} $\etSh{\rep}_{\AC{\bQ}_p}$ on $\Model_\levcp$ \Pth{with stalks isomorphic to $\rep_{\AC{\bQ}_p}$}.  In fact, by the very construction of $\etSh{\rep}_{\AC{\bQ}_p}$, for each finite extension $\Coef$ of $\bQ_p$ in $\AC{\bQ}_p$ such that $\rep_{\AC{\bQ}_p}$ has a model $\rep_\Coef$ over $\Coef$, we have an $\Coef$-\'etale local system $\etSh{\rep}_\Coef$ on $\Model_\levcp$ \Pth{with stalks isomorphic to $\rep_\Coef$} such that
\begin{equation}\label{eq-loc-syst-et-coef-base-ch}
    \etSh{\rep}_\Coef \otimes_\Coef \AC{\bQ}_p \cong \etSh{\rep}_{\AC{\bQ}_p}.
\end{equation}
In addition, by \cite[XI, 4.4]{SGA:4} \Pth{or by using the canonical homomorphism from the fundamental group to the \'etale fundamental group}, its pullback to $\Model_{\levcp, \bC}$ induces a $\AC{\bQ}_p$-local system $\BSh{\rep}_{\AC{\bQ}_p}$, together with a canonical isomorphism
\begin{equation}\label{eq-loc-syst-comp-et-B}
    \BSh{\rep}_{\AC{\bQ}_p} \otimes_{\AC{\bQ}_p, \ACMap} \bC \cong \BSh{\rep}_\bC.
\end{equation}
Note that this implies that $\etSh{\rep}_{\AC{\bQ}_p}$ has unipotent geometric monodromy along $D_{\AC{\bQ}}$.

Suppose that $\rep_{\AC{\bQ}_p}$ has a model $\rep_\Coef$ over a finite extension $\Coef$ of $\bQ_p$ in $\AC{\bQ}_p$.  Let $\BFp$ be a finite extension of the composite of $\Coef$ and the image of $\ReFl \Emn{\can} \AC{\bQ} \Emn{~~\ACMap^{-1}} \AC{\bQ}_p$ in $\AC{\bQ}_p$.  Let us denote with an additional subscript \Qtn{$\BFp$} \Pth{\resp \Qtn{$\AC{\bQ}_p$}} the base changes of $\Model_\levcp$ \etc from $\ReFl$ to $\BFp$ \Pth{\resp $\AC{\bQ}_p$} via the above composition.  We will adopt a similar notation for sheaves.  We can view the $\Coef$-\'etale local system $\etSh{\rep}_\Coef$ as a $\bQ_p$-\'etale local system with compatible $\Coef$-actions.  By \cite[\aThm 1.2]{Liu/Zhu:2017-rrhpl}, the pullback of $\etSh{\rep}_\Coef$ to $\Model_{\levcp, \BFp}$, which we still denote by the same symbols, is de Rham.  By working as in Section \ref{sec-DdR-alg}, and by pushing out via the multiplication homomorphism
\begin{equation}\label{eq-coef-proj}
    \CoefMap: \Coef \otimes_{\bQ_p} \BFp \to \BFp: \; a \otimes b \mapsto a b,
\end{equation}
we obtain a filtered log connection $(\pdRSh{\rep}_\BFp^\canext := \Ddlalg(\etSh{\rep}_\Coef) \otimes_{(\Coef \otimes_{\bQ_p} \BFp), \CoefMap} \BFp, \nabla, \Fil^\bullet)$ on $\Torcpt{\Model}_{\levcp, \BFp}$, which has nilpotent residues along $D_k$ by \cite[\aCor \logadicpropnearby]{Diao/Lan/Liu/Zhu:lasfr}, Theorem \ref{thm-unip-vs-nilp}, and GAGA \Pth{see \cite{Kopf:1974-efava}}; and also a filtered regular connection $(\pdRSh{\rep}_\BFp := \DdRalg(\etSh{\rep}_\Coef) \otimes_{(\Coef \otimes_{\bQ_p} \BFp), \CoefMap} \BFp, \nabla, \Fil^\bullet)$.  These constructions are compatible with replacements of $\Coef$ and $\BFp$ with extension fields satisfying the same conditions.  Thus, we can canonically assign to each $\etSh{\rep}_{\AC{\bQ}_p}$ as above the filtered log connection
\begin{equation}\label{eq-loc-syst-p-log-dR}
    (\pdRSh{\rep}_{\AC{\bQ}_p}^\canext, \nabla, \Fil^\bullet) := (\pdRSh{\rep}_\BFp^\canext, \nabla, \Fil^\bullet) \otimes_\BFp \AC{\bQ}_p
\end{equation}
on $\Torcpt{\Model}_{\levcp, \AC{\bQ}_p}$, whose restriction to $\Model_{\levcp, \AC{\bQ}_p}$ is the filtered regular connection
\begin{equation}\label{eq-loc-syst-p-dR}
    (\pdRSh{\rep}_{\AC{\bQ}_p}, \nabla, \Fil^\bullet) :=  (\pdRSh{\rep}_\BFp, \nabla, \Fil^\bullet) \otimes_\BFp \AC{\bQ}_p.
\end{equation}
Both \Refeq{\ref{eq-loc-syst-p-log-dR}} and \Refeq{\ref{eq-loc-syst-p-dR}} are independent of the choices of $\Coef$ and $\BFp$ for a given $\rep$.

Since $(\pdRSh{\rep}_{\AC{\bQ}_p}, \nabla, \Fil^\bullet)$ is \emph{algebraic}, its base change under $\ACMap: \AC{\bQ}_p \Mi \bC$ is a filtered regular connection $(\pdRSh{\rep}_\bC, \nabla, \Fil^\bullet)$ on $\Model_{\levcp, \bC}$, the horizontal sections of whose \emph{complex analytification} defines a $\bC$-local system $\pBSh{\rep}_\bC$ on $\Model_{\levcp, \bC}^\an$.  Since $\DdRalg$ is a tensor functor uniquely determined by $\DdR$ via the analytification functor \Pth{see Lemma \ref{lem-DdR-tensor} and its proof}, by \cite[\aThm 3.9(ii)]{Liu/Zhu:2017-rrhpl}, we obtain the following:
\begin{prop}\label{prop-loc-syst-p}
    The analogue of Proposition \ref{prop-loc-syst-infty} holds for the assignments of $\pBSh{\rep}_\bC$, $(\pdRSh{\rep}_{\AC{\bQ}_p}, \nabla, \Fil^\bullet)$, and $(\pdRSh{\rep}_\bC, \nabla, \Fil^\bullet)$.  In particular, they define a $\Grp{G}(\bAi)$-equivariant $\Grp{G}^c$-bundle with an integrable connection $({}_p\cE_\bC, {}_p\nabla)$ \Pth{\resp a $\Grp{P}^c_\bC$-bundle ${}_p\cE_{\Grp{P}^c_\bC}$} on $\{ \Model_{\levcp, \bC} \}_\levcp$, with a $\Grp{G}(\bAi)$-equivariant morphism ${}_p\cE_{\Grp{P}^c_\bC} \to {}_p\cE_\bC$.
\end{prop}

Likewise, the base change of $(\pdRSh{\rep}_{\AC{\bQ}_p}^\canext, \nabla, \Fil^\bullet)$ under $\ACMap$ is a filtered log connection $(\pdRSh{\rep}_\bC^\canext, \nabla, \Fil^\bullet)$ on $\Torcpt{\Model}_{\levcp, \bC}$, with nilpotent residues along $D_\bC^\an$.  The analogues of Proposition \ref{prop-loc-syst-p} for $(\pdRSh{\rep}_{\AC{\bQ}_p}^\canext, \nabla, \Fil^\bullet)$ and $(\pdRSh{\rep}_\bC^\canext, \nabla, \Fil^\bullet)$ also hold.

\begin{rk}[{\Refcf{} Remark \ref{rem-partial-flag}}]\label{rem-loc-syst-p-model}
    By construction \Pth{based on \Refeq{\ref{eq-loc-syst-p-dR}}}, $({}_p\cE_\bC, {}_p\nabla)$ \Pth{\resp ${}_p\cE_{\Grp{P}^c_\bC} \cong {}_p\cE_{\Fl_\bC}$} canonically admits a model $({}_p\cE_\BFp, {}_p\nabla)$ \Pth{\resp ${}_p\cE_{\Fl_\BFp}$} over $\Model_{\levcp, \BFp}$, where $\BFp$ is the completion of $\ReFl$ at the place determined by $\ACMap$.
\end{rk}

\subsection{Statement of theorem}\label{sec-loc-syst-main-comp}

It is natural to ask whether the Betti local systems $\pBSh{\rep}_\bC$ and $\BSh{\rep}_\bC$ \Pth{\resp the filtered connections $(\pdRSh{\rep}_\bC, \nabla, \Fil^\bullet)$ and $(\dRSh{\rep}_\bC, \nabla, \Fil^\bullet)$} on $\Model_{\levcp, \bC}^\an$ \Pth{\resp $\Model_{\levcp, \bC}$} are canonically isomorphic to each other, as in the following summarizing diagram:
\[
    \xymatrix{ {\BSh{\rep}_\bC} \ar@{:}[rd]^(.55){?} & {\rep \in \Rep_{\AC{\bQ}}(\Grp{G}^c)} \ar@{|->}[l]_-{\begin{smallmatrix} \text{\tiny coefficient} \\ \text{\tiny base change via} \\ \text{\tiny $\can: \AC{\bQ} \Em \bC$} \end{smallmatrix}} \ar@{|->}[r]^-{\begin{smallmatrix} \text{\tiny coefficient} \\ \text{\tiny base change via} \\ \text{\tiny $\ACMap^{-1}|_{\AC{\bQ}}: \AC{\bQ} \Em \AC{\bQ}_p$} \end{smallmatrix}} & {\etSh{\rep}_{\AC{\bQ}_p}} \ar@{|->}[d]^-{\text{\tiny $p$-adic (log) RH}} \\
    {(\dRSh{\rep}_\bC, \nabla, \Fil^\bullet)} \ar@{|->}[u]^-{\text{\tiny classical RH}} \ar@{:}[rd]_(.45){?} & {\pBSh{\rep}_\bC} & {(\pdRSh{\rep}_{\AC{\bQ}_p}, \nabla, \Fil^\bullet)} \ar@/^1pc/@{|->}[ld]^-{\quad \begin{smallmatrix} \text{\tiny base change} \\ \text{\tiny via $\ACMap: \AC{\bQ}_p \Mi \bC$} \end{smallmatrix}} \\
    & {(\pdRSh{\rep}_\bC, \nabla, \Fil^\bullet)} \ar@{|->}[u]_(.55){\text{\tiny classical RH}} & }
\]

The following theorem provides affirmative \Pth{and finer} answers:
\begin{thm}\label{thm-loc-syst-comp}
    We have canonical isomorphisms $\pBSh{\rep}_\bC \cong \BSh{\rep}_\bC$ over $\Model_{\levcp, \bC}^\an$ and $(\pdRSh{\rep}_\bC, \nabla, \Fil^\bullet) \cong (\dRSh{\rep}_\bC, \nabla, \Fil^\bullet)$ over $\Model_{\levcp, \bC}$, compatible with each other under the complex Riemann--Hilbert correspondence.  Furthermore, we have canonical $\Grp{G}(\bAi)$-equivariant isomorphisms between the relevant pairs of tensor functors in Propositions \ref{prop-loc-syst-p} and \ref{prop-loc-syst-infty}, compatible with pullbacks under morphisms between Shimura varieties induced by morphisms of Shimura data, inducing compatible canonical $\Grp{G}(\bAi)$-equivariant isomorphisms $({}_p\cE_\bC, {}_p\nabla) \cong (\cE_\bC, \nabla)$ and ${}_p\cE_{\Grp{P}^c_\bC} \cong \cE_{\Grp{P}^c_\bC}$.

    These isomorphisms are compatible with the formation of canonical models in the sense that they descend to canonical $\Grp{G}(\bAi)$-equivariant isomorphisms $({}_p\cE_\BFp, {}_p\nabla) \cong (\cE, \nabla) \otimes_\ReFl \BFp$ and ${}_p\cE_{\Fl_\BFp} \cong \cE_\Fl \otimes_\ReFl \BFp$, respectively, if $(\cE, \nabla)$ and $\cE_\Fl$ are the canonical models of $(\cE_\bC, \nabla)$ and $\cE_{\Grp{P}^c_\bC} \cong \cE_{\Fl_\bC}$, respectively, as in \cite[\aCh III, \aThm 4.3]{Milne:1990-cmsab} and Remark \ref{rem-partial-flag}, and if $({}_p\cE_\BFp, {}_p\nabla)$ and ${}_p\cE_{\Fl_\BFp}$ are as in Remark \ref{rem-loc-syst-p-model}.

    The analogous assertions hold for the filtered log connections $(\pdRSh{\rep}_\bC^\canext, \nabla, \Fil^\bullet)$ and $(\dRSh{\rep}_\bC^\canext, \nabla, \Fil^\bullet)$ \Pth{and the associated torsors}.
\end{thm}

The proofs of Theorem \ref{thm-loc-syst-comp} will be given in the remaining subsections.  Note that it verifies, in particular, the conjecture in \cite[\aRem 4.1(ii)]{Liu/Zhu:2017-rrhpl}.

Assuming this theorem for the moment, since every irreducible algebraic representation of $\Grp{G}^c$ over $\AC{\bQ}_p$ has a model over $\AC{\bQ}$, we obtain the following:
\begin{cor}\label{cor-intro-Sh}
    Theorem \ref{thm-intro-Sh} also holds.
\end{cor}

Next, we turn to Corollary \ref{cor-intro-GM}.  Consider the \Pth{$p$-adic} analytic $\Grp{G}^{c, \an}_\BFp$-torsor with an integrable connection $({}_p\cE^\an_\BFp, {}_p\nabla)$ on $\Model_{\levcp, \BFp}^\an$ defined by the assignment of the $p$-adic analytification $\pdRSh{\rep}^\an_\BFp$ of $\pdRSh{\rep}_\BFp$ to $\rep \in \Rep_{\bQ_p}(\Grp{G}^c)$, where $k$ is the completion of $\ReFl$ with respect to the $p$-adic place determined by $\iota$.  As in Remark \ref{rem-partial-flag}, the filtrations on $\pdRSh{\rep}^\an_\BFp$, for all $\rep$, define a section of ${}_p\cE_{\Fl_\BFp}^\an := {}_p\cE^\an_\BFp \times^{\Grp{G}^{c, \an}_\BFp} \Fl_\BFp^\an$ over $\Model_{\levcp, \BFp}^\an$.  Now let $x \in \Model_\levcp(\BFp')$, where $\BFp'$ is a finite extension of $\BFp$ in $\AC{\bQ}_p$.  Then there is an analytic neighborhood $U$ of $x$ in $\Model_{\levcp, \BFp'}^\an$ trivializing \Pth{the pullback of} $({}_p\cE^\an_\BFp, {}_p\nabla)$ as a $\Grp{G}^{c, \an}_{\BFp'}$-torsor with an integrable connection.  Then the above section of ${}_p\cE_{\Fl_\BFp}^\an$ over $\Model_{\levcp, \BFp}^\an$ defines the so-called \emph{Grothendieck--Messing period map} $\pi_{\Utext{GM}}: U \to \Fl^\an_{\BFp'}$.

\begin{cor}[{restatement of Corollary \ref{cor-intro-GM}}]\label{cor-etale-GM-period}
    This morphism $\pi_{\Utext{GM}}$ is \'etale.
\end{cor}
\begin{proof}
    Let $\cO_{U, x}^\wedge$ \Pth{\resp $\cO_{\Fl, \pi_{\Utext{GM}}(x)}^\wedge$} denote the completion of the local ring $\cO_{U, x}$ \Pth{\resp $\cO_{\Fl^\an_{\BFp'}, \pi_{\Utext{GM}}(x)}$}.  It suffices to show that the homomorphism $\psi: \cO_{\Fl, \pi_{\Utext{GM}}(x)}^\wedge \to \cO_{U, x}^\wedge$ induced by $\pi_{\Utext{GM}}$ is an isomorphism \Pth{\Refcf{} \cite[\aProp 1.7.11]{Huber:1996-ERA}}.  Note that the composition of $\Spec(\cO_{U, x}^\wedge) \Mapn{\Spec(\psi)} \Spec(\cO_{\Fl, \pi_{\Utext{GM}}(x)}^\wedge) \Mapn{\can} \Fl_\BFp$ is determined by the induced section of ${}_p\cE_{\Fl_\BFp}$ over $\Spec(\cO_{U, x}^\wedge)$ and the universal property of the \emph{algebraic} partial flag variety $\Fl$.  By transporting via the isomorphisms $({}_p\cE_\bC, {}_p\nabla) \cong (\cE_\bC, \nabla)$ and ${}_p\cE_{\Grp{P}^c_\bC} \cong \cE_{\Grp{P}^c_\bC}$ in Theorem \ref{thm-loc-syst-comp} the pullbacks via $\BFp' \Emn{\can} \AC{\bQ}_p \Misn{\ACMap} \bC$ of the trivialization of $({}_p\cE_\BFp, {}_p\nabla)$ and the section of ${}_p\cE_{\Fl_\BFp}$ over $\Spec(\cO_{U, x}^\wedge)$, the pullback $\psi_\bC$ of $\psi$ can be identified with the corresponding homomorphism for the usual complex analytic period map defined by $\cE_{\Grp{P}^c_\bC}^\an$ on $\Model_{\levcp, \bC}^\an$, whose induced morphisms from the spectra of completions of local rings are also determined by the universal property of $\Fl$.  By the complex analytic construction in Section \ref{sec-loc-syst-constr}, this latter period map is locally an open immersion of complex analytic spaces, given by the Borel embedding $\Shdom^+ \Em \Fl_\bC^\an$ \Pth{see \cite[\aCh III, \aSec 1]{Milne:1990-cmsab} and \cite[\aCh VIII, \aSec 7]{Helgason:2001-DLS}}.  Thus, $\psi_\bC$ is an isomorphism, and so is $\psi$, as desired.
\end{proof}

\begin{rk}\label{rem-loc-syst-comp}
    Theorem \ref{thm-loc-syst-comp} and Corollary \ref{cor-etale-GM-period} are not surprising when there are families of motives whose relative Betti, de Rham, and $p$-adic \'etale realizations define the local systems $\BSh{\rep}_\bC$, $(\dRSh{\rep}_\bC, \nabla, \Fil^\bullet)$, and $\etSh{\rep}_{\AC{\bQ}_p}$, respectively.  This is the case, for example, when $\Model_\levcp$ is a Shimura variety of Hodge type.  \Pth{We will take advantage of this in Section \ref{sec-case-rank-one-} below.}  But Theorem \ref{thm-loc-syst-comp} and Corollary \ref{cor-etale-GM-period} also apply to Shimura varieties associated with exceptional groups, over which there are \Pth{as yet} no known families of motives defining our local systems as above.
\end{rk}

\begin{rk}\label{rem-loc-syst-Hodge-degen}
    By Theorem \ref{thm-loc-syst-comp} and Deligne's comparison result in \cite[II, 6]{Deligne:1970-EDR}, and by Theorem \ref{thm-HT-degen-comp}, the spectral sequence
    \[
        E_1^{a, b} = H^{a, b}_{\log \Hdg}(\Torcpt{\Model}_{\levcp, \bC}, \dRSh{\rep}^\canext_\bC) \Rightarrow H^{a + b}_{\log \dR}(\Torcpt{\Model}_{\levcp, \bC}, \dRSh{\rep}^\canext_\bC) \cong H^{a + b}_\dR(\Model_{\levcp, \bC}, \dRSh{\rep}_\bC)
    \]
    degenerates on the $E_1$ page.  While this degeneration was already known thanks to Saito's direct image theorem \Pth{see \cite[\aThm 2.14]{Saito:1990-mhm} and \cite[\aSec 4]{Suh:2018-vmhma}}, we have a new proof here based on $p$-adic Hodge theory.  Also, we can determine the Hodge--Tate weights of $H^i(\Model_{\levcp, \AC{\bQ}_p}, \etSh{\rep}_{\AC{\bQ}_p})$ in terms of $\dim_\bC H^{a, i - a}_{\log \Hdg}(\Torcpt{\Model}_{\levcp, \bC}, \dRSh{\rep}^\canext_\bC)$, for all $a \in \bZ$, which can be computed using the \emph{dual BGG decomposition} and relative Lie algebra cohomology.  \Pth{We will explain these in more detail in \cite{Lan/Liu/Zhu:dcpdr}.}
\end{rk}

\begin{rk}\label{rem-loc-sym-var}
    The comparison isomorphisms in Theorem \ref{thm-loc-syst-comp} over Shimura varieties induce similar isomorphisms on general locally symmetric varieties, by pullback and by finite \'etale descent.  Consequently, the analogue of the statements in Remark \ref{rem-loc-syst-Hodge-degen} for general locally symmetric varieties also hold.
\end{rk}

\begin{rk}\label{rem-van-aut}
    By replacing the input \cite{Suh:2018-vmhma} in the proof of \cite[\aThm 4.3]{Lan:2016-vtcac} with Theorem \ref{thm-van-gen}, and by Theorem \ref{thm-loc-syst-comp} and Remark \ref{rem-loc-sym-var}, we obtain new $p$-adic Hodge-theoretic proofs of the vanishing results for the coherent and de Rham cohomology in \cite[\aThms 4.1, 4.4, 4.7, and 4.10]{Lan:2016-vtcac}, generalizing the characteristic-zero cases of previous results in \cite{Lan/Suh:2012-vttac, Lan/Suh:2013-vttag, Lan/Stroh:2014-rcfps, Lan:2016-hkp}.
\end{rk}

\subsection{Proof of theorem: preliminary reductions}\label{sec-proof-comp-prelim}

Let us fix a connected component $\Gamma^+_{\levcp, g_0} \Lquot \Shdom^+$ of $\Model_{\levcp, \bC}^\an$ as in \Refeq{\ref{eq-Sh-var-conn}}, which is the analytification of a quasi-projective variety defined over some finite extension $\ReFl^+$ of $\ReFl$ in $\AC{\bQ}$.  Let $\hd \in \Shdom^+$ be a \emph{special point} such that \Refeq{\ref{eq-h}} factors through $\Grp{T}_\bR$ for some maximal torus $\Grp{T}$ of $\Grp{G}$ \Pth{over $\bQ$}.  \Pth{Recall that special points are dense in $\Shdom^+$---see the proof of \cite[\aLem 13.5]{Milne:2005-isv}.}  Up to replacing $\ReFl^+$ with a finite extension in $\AC{\bQ}$, we may assume that the image of $\hd \in \Shdom^+$ in $\Gamma^+_{\levcp, g_0} \Lquot \Shdom^+$ is defined over $\ReFl^+$.

The pullbacks of $\BSh{\rep}_\bC$ to $\hd \in \Shdom^+$ can be canonically identified with $\rep_\bC$ by its very construction.  On the other hand, the pullback of $\pBSh{\rep}_\bC$ can also be canonically identified with $\rep_\bC$.  In fact, in both cases, we have slightly more:
\begin{prop}\label{prop-ident-spec-pt}
    The pullbacks of $\BSh{\rep}_\bC$ and $\pBSh{\rep}_\bC$ to $(\Grp{G}(\bQ) \hd) \times \Grp{G}(\bAi)$ are canonically and $\Grp{G}(\bQ) \times \Grp{G}(\bAi)$-equivariantly isomorphic to the trivial local system $(\Grp{G}(\bQ) \hd) \times \rep_\bC \times \Grp{G}(\bAi)$ \Pth{on which $\Grp{G}(\bQ)$ acts by diagonal left multiplication on all three factors, and $\Grp{G}(\bAi)$ acts by right multiplication on the last factor.}
\end{prop}
\begin{proof}
    The assertion for $\BSh{\rep}_\bC$ follows from its very construction.  As for the assertion for $\pBSh{\rep}_\bC$, let us first identify the pullback of $\etSh{\rep}_{\AC{\bQ}_p}$ to the images of $(\hd, g)$, for $g \in \Grp{G}(\bAi)$, by recalling the arguments on which \cite[\aLem 4.8]{Liu/Zhu:2017-rrhpl} is based.  We shall write $\Gamma^+_{\levcp, g} := \Grp{G}(\bQ)_+ \cap (g \levcp g^{-1})$ \Pth{\Refcf{} \Refeq{\ref{eq-Sh-var-conn}}}, so that $\Gamma^+_{\levcp, g} \Lquot \Shdom^+$ gives the connected component of $\Model_{\levcp, \bC}^\an$ containing the image of $(\hd, g)$.

    By assumption, $\hd: \DelS = \Res_{\bC / \bR} \Gm{\bC} \to \Grp{G}_\bR$ \Pth{as in \Refeq{\ref{eq-h}}} factors through $\Grp{T}_\bR$, and the Hodge cocharacter \Pth{as in \Refeq{\ref{eq-mu-h}}} induces a cocharacter $\hc_\hd: \Gm{\bC} \to \Grp{T}_\bC$, which is the base change of some cocharacter $\hc: \Gm{F} \to \Grp{T}_F$ defined over some number field $F$ in $\AC{\bQ}$.  Then the composition of $\hc$ with the norm map from $\Grp{T}_F$ to $\Grp{T}$ defines a homomorphism $\OP{N} \hc: \Res_{F / \bQ} \Gm{F} \to \Grp{T}$ of tori over $\bQ$, and we have a composition of homomorphisms $F^\times \Lquot \bA_F^\times \Mapn{\OP{N} \hc} \Grp{T}(\bQ) \Lquot \Grp{T}(\bA) \to \overline{\Grp{T}(\bQ)} \Lquot \Grp{T}(\bAi)$, where $\overline{\Grp{T}(\bQ)}$ denotes the closure of $\Grp{T}(\bQ)$ in $\Grp{T}(\bAi)$, which factors through
    \begin{equation}\label{eq-Gal-F}
        F^\times \Lquot \bA_F^\times \Mapn{\Art_F} \Gal(F^\ab / F) \Mapn{r(\hc)} \overline{\Grp{T}(\bQ)} \Lquot \Grp{T}(\bAi),
    \end{equation}
    where $F^\ab$ is the maximal abelian extension of $F$ in $\AC{\bQ}$.  If $F_{\levcp, g}$ is the subfield of $F^\ab$ such that $\Gal(F^\ab / F_{\levcp, g})$ is the preimage of $(g \levcp g^{-1} \cap \overline{\Grp{T}(\bQ)}) \Lquot (g \levcp g^{-1} \cap \Grp{T}(\bAi))$ under \Refeq{\ref{eq-Gal-F}}, then we have an induced Galois representation
    \begin{equation}\label{eq-Gal-E+}
        r(\hc)^+_{\levcp, g}: \Gal(\AC{\bQ} / F_{\levcp, g}) \to (g \levcp g^{-1} \cap \overline{\Grp{T}(\bQ)}) \Lquot (g \levcp g^{-1} \cap \Grp{T}(\bAi)).
    \end{equation}
    \Pth{If $g = g_0$, then $F_{\levcp, g} \subset \ReFl^+$, since the image of $\hd$ in $\Gamma^+_{\levcp, g_0} \Lquot \Shdom^+$ is defined over $\ReFl^+$.}

    Since $\Grp{T}_\bR$ stabilizes the special point $\hd$, it is $\bR$-anisotropic modulo the center of $\Grp{G}$, and hence its maximal $\bQ$-anisotropic $\bR$-split subtorus is the same as that of the center of $\Grp{G}$.  Therefore, as explained in the proof of \cite[\aLem 4.5]{Liu/Zhu:2017-rrhpl}, the pullback of $\rep$ to $\Grp{T}$ satisfies the requirement that its restriction to $g \levcp g^{-1} \cap \overline{\Grp{T}(\bQ)}$ is trivial as in \cite[(4.4)]{Liu/Zhu:2017-rrhpl} \Pth{with the neatness of $g \levcp g^{-1}$ here implying that the open compact subgroup $\levcp$ there is sufficiently small}.  Thus, the composition $\Grp{T}(\bAi) \to \Grp{T}(\bQ_p) \to \Grp{G}(\bQ_p) \to \GL_{\AC{\bQ}_p}(\rep_{\AC{\bQ}_p})$ factors through $(g \levcp g^{-1} \cap \overline{\Grp{T}(\bQ)}) \Lquot \Grp{T}(\bAi)$ and induces, by composition with \Refeq{\ref{eq-Gal-E+}}, a Galois representation $r(\hc, \rep)^+_{\levcp, g, p}: \Gal(\AC{\bQ} / F_{\levcp, g}) \to \GL_{\AC{\bQ}_p}(\rep_{\AC{\bQ}_p})$ describing the pullback of $\etSh{\rep}_{\AC{\bQ}_p}$ to the geometric point above the image of $h$ in $\Gamma^+_{\levcp, g} \Lquot \Shdom^+$ given by the composition of $F_{\levcp, g} \Emn{\can} \AC{\bQ} \Emn{~~\ACMap^{-1}} \AC{\bQ}_p$.

    Let $\Coef$, $\rep_\Coef$, $\BFp$, and $\CoefMap$ be as in Section \ref{sec-loc-syst-constr}, giving us maps $\bQ_p \Em \Coef \Em \BFp \Em \AC{\bQ}_p \Misn{\ACMap} \bC$ in the remainder of this paragraph.  Without loss of generality, we may assume that $\BFp$ also contains the image of $F_{\levcp, g}$ in $\AC{\bQ}_p$.  Then the image of $r(\hc, \rep)^+_{\levcp, g, p}$ is contained in the subgroup $\GL_\Coef(\rep_\Coef)$ of $\GL_{\AC{\bQ}_p}(\rep_{\AC{\bQ}_p})$, and we can view this representation over $\Coef$ as a representation over $\bQ_p$ with an additional action of $\Coef$, as usual.  By \cite[\aLem 4.4]{Liu/Zhu:2017-rrhpl}, this representation is potentially crystalline.  Since it factors through the abelian group $\Grp{T}(\bQ_p)$, by \cite[\aThm 2.3.13 and its proof]{Patrikis:2019-VTT} \Pth{based on \cite[\aChs II and III]{Serre:1968-ARE} and \cite[\aCh IV]{Deligne/Milne/Ogus/Shih:1982-HMS}; \Refcf{} \cite[\S 6, \aProp{}, and its proof]{Fontaine/Mazur:1993-ggr}}, it is isomorphic to the $p$-adic \'etale realization $M_p$ of some object $M$ in the Tannakian category $(\textrm{CM})_{F_{\levcp, g}}$ of motives \Pth{for absolute Hodge cycles} over $F_{\levcp, g}$ generated by Artin motives and abelian varieties potentially of CM type, as defined in \cite[\aCh IV]{Deligne/Milne/Ogus/Shih:1982-HMS}.  Let $M_\B$ and $M_\dR$ denote the Betti and de Rham realizations of $M$.  Via the canonical $p$-adic \'etale--Betti and Betti--de Rham comparison isomorphisms, we have canonically induced actions of $\Coef$ on $M_{\B, \bQ_p} := M_\B \otimes_\bQ \bQ_p$, and hence on $M_{\B, \bC} := M_\B \otimes_\bQ \bC$ and $M_{\dR, \bC} := M_\dR \otimes_{F_{\levcp, g}} \bC$.  Moreover, we have canonically induced isomorphisms from $\rep_\bC$ to the pushouts of $M_{\B, \bC}$ and $M_{\dR, \bC}$ via the base change $\CoefMap_\bC: \Coef \otimes_{\bQ_p} \bC \to \bC$ of $\CoefMap: \Coef \otimes_{\bQ_p} \BFp \to \BFp$.  By \cite[\aThm 0.3]{Blasius:1994-phcav}, the absolute Hodge cycles defining objects in $(\textrm{CM})_{F_{\levcp, g}}$ are compatible with the $p$-adic \'etale--de Rham comparison isomorphism of Faltings's, and hence we have a canonically induced comparison isomorphism $\DdR(M_p) \Mi M_\dR \otimes_{F_{\levcp, g}} \BFp$, where $\DdR$ is defined over $k$.  Thus, we can canonically identify the pullbacks of $\dRSh{\rep}_\bC$ and $\pdRSh{\rep}_\bC$ to $\hd$ via the pushout of $M_{\dR, \bC}$ via $\CoefMap_\bC$.  \Pth{These identifications do not depend on the choices of $M$ and $r(\hc, \rep)^+_{\levcp, g, p} \Mi M_p$ because, by \cite[\aCh IV, (D)]{Deligne/Milne/Ogus/Shih:1982-HMS}, $\Hom_{(\textrm{CM})_{F_{\levcp, g}}}(M, M') \otimes_\bQ \bQ_p \Misn{\can} \Hom_{\Gal(\AC{\bQ} / F_{\levcp, g})}(M_p, M'_p)$, for any $M'$ in $(\textrm{CM})_{F_{\levcp, g}}$, and hence any isomorphism $r(\hc, \rep)^+_{\levcp, g, p} \Mi M'_p$ canonically induce isomorphisms $M_p \Mi M'_p$, $M_{\B, \bQ_p} \Mi M'_{\B, \bQ_p}$, $M_{\B, \bC} \Mi M'_{\B, \bC}$, and $M_{\dR, \bC} \Mi M'_{\dR, \bC}$, which are compatible with the various canonical comparison isomorphisms above for $M$ and $M'$.}  Accordingly, by using the trivial complex Riemann--Hilbert correspondence over the single $\bC$-point $\hd$, we can canonically identify the pullbacks of $\BSh{\rep}_\bC$ and $\pBSh{\rep}_\bC$ to $\hd$ via the corresponding pushout of $M_{B, \bC}$ via $\CoefMap_\bC$.  Consequently, we can canonically identify the pullback of $\pBSh{\rep}_\bC$ to $h$ with $\rep_\bC$.

    For each $\gamma \in \Grp{G}(\bQ)$, by conjugating all the actions on the base changes of $\rep$ and the model $\rep_\Coef$ by $\gamma$ in the above, we can also canonically identify the pullback of $\pBSh{\rep}_\bC$ to $\gamma \hd$ with $\rep_\bC$.  When put together as a canonical identification over the whole $\bigl(\Grp{G}(\bQ) \hd\bigr) \times \Grp{G}(\bAi)$, it is $\Grp{G}(\bQ)$-equivariant by what we just explained, and is $\Grp{G}(\bAi)$-equivariant because it does not involve the second factor $\Grp{G}(\bAi)$ at all.
\end{proof}

\begin{prop}\label{prop-loc-syst-B-comp-imply-dR-comp}
    Suppose that the assertions for Betti local systems in Theorem \ref{thm-loc-syst-comp} hold.  Then the remaining assertions in Theorem \ref{thm-loc-syst-comp} also hold.
\end{prop}
\begin{proof}
    Since $\pBSh{\rep}_\bC \cong \BSh{\rep}_\bC$ over $\Model_{\levcp, \bC}^\an$, we have $(\pdRSh{\rep}_\bC, \nabla) \cong (\dRSh{\rep}_\bC, \nabla)$ over $\Model_{\levcp, \bC}$, because both sides admit extensions $(\pdRSh{\rep}_\bC^\canext, \nabla)$ and $(\dRSh{\rep}_\bC^\canext, \nabla)$ over $\Torcpt{\Model}_{\levcp, \bC}$ \Pth{and hence have regular singularities along $\NCD_\bC = \Torcpt{\Model}_{\levcp, \bC} - \Model_{\levcp, \bC}$}.  Since both $(\pdRSh{\rep}_\bC^\canext, \nabla)$ and $(\dRSh{\rep}_\bC^\canext, \nabla)$ have nilpotent residues along $\NCD_\bC$ \Pth{as explained in Section \ref{sec-loc-syst-constr}}, these two extensions are also canonically isomorphic to each other \Pth{by \cite[\aCh 1, \aProp 4.7]{Andre/Baldassarri:2001-DDA} or \cite[\aThm 11.2.2]{Andre/Baldassarri/Cailotto:2020-DDA(2)}}.  To verify that the filtrations are respected by such isomorphisms, it suffices to do so at the special points, or just at the arbitrary special point $\hd$ we have chosen, because special points are dense in the complex analytic topology \Pth{see the proof of \cite[\aLem 13.5]{Milne:2005-isv}}.

    Consider the Galois representation $r(\hc, \rep)^+_{\levcp, g, p}$ in the proof of Proposition \ref{prop-ident-spec-pt}.  By decomposing the representation $\rep_{\AC{\bQ}_p}$ of $\Grp{T}(\AC{\bQ}_p)$ into a direct sum of characters of $\Grp{T}(\AC{\bQ}_p)$, we obtain a corresponding decomposition of $r(\hc, \rep)^+_{\levcp, g, p}$ into a direct sum of characters $\Gal(\AC{\bQ} / F_{\levcp, g}) \to \AC{\bQ}_p^\times$.  By construction, the restrictions of these characters of $\Gal(\AC{\bQ} / F_{\levcp, g})$ to the decomposition group at the place $v$ of $F_{\levcp, g}$ given by the composition $F_{\levcp, g} \Emn{\can} \AC{\bQ} \Emn{~~\ACMap^{-1}} \AC{\bQ}_p$ are \emph{locally algebraic} \Pth{in the sense of \cite[\aCh III, \aSec 1.1, \aDef{}]{Serre:1968-ARE}}, because they are induced \Pth{up to a sign convention} by the composition of the local Artin map, the cocharacter $\AC{\bQ}_p^\times \to T(\AC{\bQ}_p)$ given by the base change of $\hc$ under the same $F_{\levcp, g} \Em \AC{\bQ}_p$, and the corresponding characters of $\Grp{T}(\AC{\bQ}_p)$.  \Pth{This local algebraicity played a crucial role in the references we cited in the proof of Proposition \ref{prop-ident-spec-pt}.}  Thus, the Hodge filtrations on the pullbacks of $\pdRSh{\rep}_\bC$ and $\dRSh{\rep}_\bC$ to $\hd$ are both determined by the Hodge cocharacter $\hc_\hd: \Gm{\bC} \to \Grp{G}_\bC$.

    The remainder of Theorem \ref{thm-loc-syst-comp} follows from the fact that the formations of $(\cE_\ReFl, \nabla)$ \Pth{\resp $\cE_\Fl$} and $({}_p\cE_\BFp, \nabla)$ \Pth{\resp ${}_p\cE_{\Fl_\BFp}$} are determined by their pullbacks to special points, which are compatible with each other by the arguments in the proof of Proposition \ref{prop-ident-spec-pt} \Pth{\resp of this proposition}; and that the descent data for such torsors extend to their canonical extensions as in \cite[\aCh V, \aSec 6]{Milne:1990-cmsab}.
\end{proof}

By Proposition \ref{prop-loc-syst-B-comp-imply-dR-comp}, it remains to prove the assertions for Betti local systems in Theorem \ref{thm-loc-syst-comp}.  As explained before \Pth{\Refcf{} \Refeq{\ref{eq-fund-grp-rep}}}, the pullbacks of $\pBSh{\rep}_\bC$ and $\BSh{\rep}_\bC$ to $\Gamma^+_{\levcp, g_0} \Lquot \Shdom^+$ determine and are determined by the fundamental group representations
\[
    \Frep^{+, (p)}_{\levcp, g_0}(\rep): \Gamma_{\levcp, g_0}^{+, c} \to \GL_\bC(\rep_\bC) \quad \Utext{and} \quad \Frep^+_{\levcp, g_0}(\rep): \Gamma_{\levcp, g_0}^{+, c} \to \GL_\bC(\rep_\bC),
\]
respectively, by canonically identifying the pullbacks of the local systems $\pBSh{\rep}_\bC$ and $\BSh{\rep}_\bC$ to the image of $\hd \in \Shdom^+$ in $\Gamma_{\levcp, g_0}^+ \Lquot \Shdom^+$ using Proposition \ref{prop-ident-spec-pt}.  Then it suffices to show that $\Frep^{+, (p)}_{\levcp, g_0}(\rep)$ and $\Frep^+_{\levcp, g_0}(\rep)$ \emph{coincide} as representations of $\Gamma_{\levcp, g_0}^+$.  In this case, they are isomorphic via the identity morphism on $\rep_\bC$, and therefore the choice of such an isomorphism is functorial in $\rep$ and compatible with tensor products and duals because the assignment of $\Frep^+_{\levcp, g_0}(\rep)$ to $\rep$ is, with Hecke actions because of Proposition \ref{prop-ident-spec-pt}, and with morphisms between Shimura varieties induced by morphisms of Shimura data because all constructions involved are.

Since $\Gamma_{\levcp, g_0}^{+, c}$ is contained in $\Grp{G}^{\der, c}(\bQ)$, and since $\Frep^+_{\levcp, g_0}(\rep)$ depends only on the restriction $\rep_\bC|_{\Grp{G}^{\der, c}}$ of $\rep_\bC$ to $\Grp{G}^{\der, c}$, it remains to prove the following proposition.
\begin{prop}\label{prop-ext-der}
    The representation $\Frep^{+, (p)}_{\levcp, g_0}(\rep)$ extends to an algebraic representation of $\Grp{G}^{\der, c}$ that coincides with the representation $\rep_\bC|_{\Grp{G}^{\der, c}}$ of $\Grp{G}^{\der, c}$ on $\rep_\bC$.
\end{prop}

\begin{rk}\label{rem-Borel-density}
    Such an extension is necessarily unique, as arithmetic subgroups of semisimple groups without compact $\bQ$-simple factors are Zariski dense, by the \emph{Borel density theorem} \Pth{see, for example, \cite[\aProp 15.12]{Borel:1969-IGA}}.
\end{rk}

\begin{lemma}\label{lem-isog}
    Let $(\breve{\Grp{G}}, \breve{\Shdom})$, $\breve{\levcp}$, $\breve{g}_0$, and $\breve{\hd} \in \breve{\Shdom}^+$ be a Shimura datum and some additional choices of data analogous to $(\Grp{G}, \Shdom)$, $\levcp$, $g_0$, and $\hd \in \Shdom^+$.  Let $\breve{\Grp{G}}^\der \to \Grp{G}^\der$ be a central isogeny inducing some $\breve{\Shdom}^+ \Mi \Shdom^+$ mapping $\breve{\hd}$ to $\hd$ and a finite covering map $f: \Gamma_{\breve{\levcp}, \breve{g}_0}^+ \Lquot \breve{\Shdom}^+ \to \Gamma_{\levcp, g_0}^+ \Lquot \Shdom^+$.  Let $\breve{\Gamma} := \Gamma_{\breve{\levcp}, \breve{g}_0}^+ \subset \breve{\Grp{G}}^{\der, c}(\bQ)$, which \Pth{by neatness} is mapped isomorphically to a finite index subgroup $\overline{\Gamma}$ of $\Gamma := \Gamma_{\levcp, g_0}^+ \subset \Grp{G}^{\der, c}(\bQ)$.  Let $\breve{\rep} \in \Rep_{\AC{\bQ}}(\breve{\Grp{G}}^c)$, with an isomorphism between the pullbacks to $\breve{\Grp{G}}^{\der, c}$ of $\breve{\rep}$ and $\rep$.  Then the similarly defined representations $\breve{\Frep}^+_{\breve{\levcp}, \breve{g}_0}(\breve{\rep}), \breve{\Frep}^{+, (p)}_{\breve{\levcp}, \breve{g}_0}(\breve{\rep}): \breve{\Gamma} \to \GL_\bC(\breve{\rep}_\bC)$ are canonically induced, respectively, by the pullbacks of $\Frep^+_{\levcp, g_0}(\rep), \Frep^{+, (p)}_{\levcp, g_0}(\rep): \Gamma \to \GL_\bC(\rep_\bC)$.  Moreover, if the analogue of Proposition \ref{prop-ext-der} holds for $\breve{\Frep}^{+, (p)}_{\breve{\levcp}, \breve{g}_0}(\breve{\rep})$ and $\breve{\rep}_\bC|_{\breve{\Grp{G}}^{\der, c}}$, then Proposition \ref{prop-ext-der} holds for $\Frep^{+, (p)}_{\levcp, g_0}(\rep)$ and $\rep_\bC|_{\Grp{G}^{\der, c}}$.
\end{lemma}
\begin{proof}
    All statements but the last one follow immediately from the functoriality of the constructions.  Let $\BSh{\breve{\rep}}_\bC$ and $\pBSh{\breve{\rep}}_\bC$ denote the corresponding local system on $\Sh_{\breve{\levcp}, \bC}^\an$.  By Proposition \ref{prop-ident-spec-pt}, and by considering the descent data for local systems with respect to the covering map $\breve{\Gamma} \Lquot \breve{\Shdom}^+ \to \Gamma \Lquot \Shdom^+$, the pullbacks to $\hd$ of $f_*(\BSh{\breve{\rep}}_\bC|_{\breve{\Gamma} \Lquot \breve{\Shdom}^+})$ and $f_*(\pBSh{\breve{\rep}}_\bC|_{\breve{\Gamma} \Lquot \breve{\Shdom}^+})$ can be canonically identified with
    \[
        \Ind_{\overline{\Gamma}}^{\Gamma}(\rep_\bC|_{\overline{\Gamma}}) \cong \{ \varphi: \Gamma \to \rep_\bC : \varphi(\overline{\gamma} \gamma') = \overline{\gamma}^{-1} \bigl(\varphi(\gamma')\bigr), \text{for all $\overline{\gamma} \in \overline{\Gamma}$ and $\gamma' \in \Gamma$} \},
    \]
    and the pullbacks to $\hd$ of the sub-local systems $\BSh{\rep}_\bC|_{\Gamma \Lquot \Shdom^+}$ and $\pBSh{\rep}_\bC|_{\Gamma \Lquot \Shdom^+}$ can be canonically identified with the equalizer of $\Ind_{\overline{\Gamma}}^{\Gamma}(\rep_\bC|_{\overline{\Gamma}}) \rightrightarrows \Ind_{\overline{\Gamma}}^{\Gamma}\bigl( \Ind_{\overline{\Gamma}}^{\Gamma}(\rep_\bC|_{\overline{\Gamma}}) \bigr)$, where the two morphisms are canonical, which can be in turn canonically identified with
    \[
        \rep_\bC|_\Gamma \cong \{ \varphi: \Gamma \to \rep_\bC : \varphi(\gamma \gamma') = \gamma^{-1} \bigl(\varphi(\gamma')\bigr), \text{for all $\gamma, \gamma' \in \Gamma$} \},
    \]
    where $\overline{\Gamma}$ and $\Gamma$ act on the codomains $\rep_\bC$ by restrictions of $\rep_\bC|_{\Grp{G}^{\der, c}}$.  If the analogue of Proposition \ref{prop-ext-der} holds for $\breve{\Frep}^{+, (p)}_{\breve{\levcp}, \breve{g}_0}(\breve{\rep})$ and $\breve{\rep}_\bC|_{\breve{\Grp{G}}^{\der, c}}$, then $\BSh{\breve{\rep}}_\bC|_{\breve{\Gamma} \Lquot \breve{\Shdom}^+} \cong \pBSh{\breve{\rep}}_\bC|_{\breve{\Gamma} \Lquot \breve{\Shdom}^+}$, and $f_*(\BSh{\breve{\rep}}_\bC|_{\breve{\Gamma} \Lquot \breve{\Shdom}^+}) \cong f_*(\pBSh{\breve{\rep}}_\bC|_{\breve{\Gamma} \Lquot \breve{\Shdom}^+})$ matches $\BSh{\rep}_\bC|_{\Gamma \Lquot \Shdom^+}$ and $\pBSh{\rep}_\bC|_{\Gamma \Lquot \Shdom^+}$ as sub-local systems.  As a result, Proposition \ref{prop-ext-der} also holds for $\Grp{G}^{\der, c}$ and $\rep_\bC|_{\Grp{G}^{\der, c}}$.
\end{proof}

\begin{lemma}\label{lem-scc}
    It suffices to prove Proposition \ref{prop-ext-der} in the special case where $\Grp{G}^\der$ and $\Grp{G}^{\der, c}$ are $\bQ$-simple and simply-connected as algebraic groups over $\bQ$.
\end{lemma}
\begin{proof}
    By \cite[\aLem 2.5.5]{Deligne:1979-vsimc}, there exists a connected Shimura datum with the semisimple algebraic group over $\bQ$ being the simply-connected cover $\widetilde{\Grp{G}}$ of $\Grp{G}^\der$.  Moreover, there is a decomposition $\widetilde{\Grp{G}} \cong \prod_{i \in I} \widetilde{\Grp{G}}_i$ of $\widetilde{\Grp{G}}$ into a product of its $\bQ$-simple factors such that each $\widetilde{\Grp{G}}_i$ is part of a connected Shimura datum.  Thus, we can choose some Shimura data $(\breve{\Grp{G}}, \breve{\Shdom})$ and $(\breve{\Grp{G}}_i, \breve{\Shdom}_i)$ with $\breve{\Grp{G}}^{\der, c} \Mi \widetilde{\Grp{G}}$ and $\breve{\Grp{G}}_i^{\der, c} \Mi \widetilde{\Grp{G}}_i$, choose levels and additional data such that the corresponding $\widetilde{\Gamma} \subset \widetilde{\Grp{G}}(\bQ)$ is of the form $\widetilde{\Gamma} = \prod_{i \in I} \widetilde{\Gamma}_i$ for some neat arithmetic subgroups $\widetilde{\Gamma}_i$ of $\widetilde{\Grp{G}}_i(\bQ)$ and such that its image $\overline{\Gamma}$ in $\Grp{G}^{\der, c}(\bQ)$ is a subgroup of $\Gamma_{\levcp, g_0}^{+, c}$, and apply Lemma \ref{lem-isog}.
\end{proof}

Consequently, in what follows, we may and we shall assume that $\Grp{G}^{\der, c}$ is simply-connected as an algebraic group over $\bQ$, so that $\Grp{G}^\der \cong \Grp{G}^{\der, c}$ also is.

\subsection{Cases of real rank one, or of abelian type}\label{sec-case-rank-one-}

In this subsection, we assume that $\Grp{G}^{\der, c}$ is $\bQ$-simple and simply-connected as an algebraic group over $\bQ$ \Pth{so that $\Grp{G}^\der \cong \Grp{G}^{\der, c}$}, and that either $\Grp{G}^\der_\bR$ is of type $\mathrm{A}$ or $\rank_\bR(\Grp{G}^\der_\bR) \leq 1$.

\begin{lemma}\label{lem-ab-type}
    Under the above assumptions, the Shimura datum $(\Grp{G}, \Shdom)$ is necessarily of abelian type \Pth{see, for example, \cite[\aDef 5.2.2.1]{Lan:2017-ebisv}}.  Up to replacing $\Grp{G}$ with another reductive algebraic group over $\bQ$ with the same derived group $\Grp{G}^\der$, we may assume that $(\Grp{G}, \Shdom)$ is of Hodge type \Pth{see, for example, \cite[\aDef 5.2.1.1]{Lan:2017-ebisv}}.
\end{lemma}
\begin{proof}
    Let $\Grp{G}^\der(\bR)_\nc$ denote the product of all noncompact simple factors of $\Grp{G}^\der(\bR)$ \Pth{as a real Lie group}.  By the classification of Hermitian symmetric domains \Pth{see \cite[\aCh X, \aSec 6]{Helgason:2001-DLS}}, any $\Grp{G}^\der(\bR)_\nc$ here satisfying $\rank_\bR(\Grp{G}^\der_\bR) \leq 1$ is isomorphic to $\SU_{n, 1}$, for some $n \geq 1$.  Thus, every $\Grp{G}^\der_\bR$ considered by this lemma is of type $\mathrm{A}$.  By the classification in \cite[2.3]{Deligne:1979-vsimc}, $(\Grp{G}, \Shdom)$ is necessarily of abelian type.  The last statement then also follows, essentially by definition.
\end{proof}

Consequently, for our purpose, we may assume that the Shimura datum $(\Grp{G}, \Shdom)$ is of Hodge type.  Note that $\Grp{G} \cong \Grp{G}^c$ in this case.  Moreover, there exists some faithful representation $\rep_0$ of $\Grp{G} \cong \Grp{G}^c$ over $\AC{\bQ}$, together with a perfect alternating pairing
\begin{equation}\label{eq-pairing}
    \rep_0 \times \rep_0 \to \AC{\bQ}(-1),
\end{equation}
where $(-1)$ denotes the formal Tate twist \Pth{induced by the pullback of the symplectic similitude character}, which are defined by some Siegel embedding in the definition of a Shimura datum of Hodge type, together with an abelian scheme $\AVstr: \AV \to \Model_\levcp$ with a polarization $\Pol$, whose $m$-fold self-fiber product we denote by $\AVstr^m: \AV^m \to \Model_\levcp$, such that, for all $i \geq 0$, we have $R^i\AVstr^m_{\bC, *}(\AC{\bQ}) \cong \Ex^i(\BSh{\rep}_0^m)$ over $\Model_{\levcp, \bC}$; $R^i\AVstr^m_{\et, *}(\AC{\bQ}_p) \cong \Ex^i(\etSh{\rep}_{0, \AC{\bQ}_p}^m)$ over $\Model_\levcp$, where $\rep_{0, \AC{\bQ}_p} := \rep_0 \otimes_{\AC{\bQ}} \AC{\bQ}_p$; and $\bigl(R^i\AVstr^m_*(\Omega^\bullet_{\AV^m / \Model_\levcp}) \otimes_\ReFl \bC, \nabla, \Fil^\bullet\bigr) \cong \bigl(\Ex^i(\dRSh{\rep}_{0, \bC}^m), \nabla, \Fil^\bullet\bigr)$ over $\Model_{\levcp, \bC}$, where $\rep_{0, \bC} := \rep_0 \otimes_{\AC{\bQ}} \bC$ and the $\nabla$ and $\Fil^\bullet$ at the left-hand side are the Gauss--Manin connection and the relative Hodge filtration, respectively.  The polarization $\Pol$ then compatibly induces \Pth{as in \cite[1.5]{Deligne/Pappas:1994-smhcd}} the pairings $\BSh{\rep}_0 \times \BSh{\rep}_0 \to \BSh{\AC{\bQ}}(-1)$, $\etSh{\rep}_{0, \AC{\bQ}_p} \times \etSh{\rep}_{0, \AC{\bQ}_p} \to \etSh{\AC{\bQ}_p}(-1)$, and $\dRSh{\rep}_{0, \bC} \times \dRSh{\rep}_{0, \bC} \to \dRSh{\bC}(-1)$ defined by \Refeq{\ref{eq-pairing}}, with $(-1)$ denoting the Tate twists in the respective categories.

\begin{lemma}\label{lem-std-comp}
    We have a canonical isomorphism
    \begin{equation}\label{eq-lem-std-comp-dR}
        \bigl(\dRSh{\rep}_{0, \bC}^{\otimes m}(-t), \nabla, \Fil^\bullet\bigr) \cong \bigl(\pdRSh{\rep}_{0, \bC}^{\otimes m}(-t), \nabla, \Fil^\bullet\bigr)
    \end{equation}
    for all $i \geq 0$, $m \geq 0$, and $t \in \bZ$.  Accordingly, we have a canonical isomorphism
    \begin{equation}\label{eq-lem-std-comp-B}
        \BSh{\rep}_{0, \bC}^{\otimes m}(-t) \cong \pBSh{\rep}_{0, \bC}^{\otimes m}(-t).
    \end{equation}
    Moreover, the pullback of \Refeq{\ref{eq-lem-std-comp-B}} to the image of the special point $\hd \in \Shdom^+$ in $\Gamma^+_{\levcp, g_0} \Lquot \Shdom^+$ \Pth{see the beginning of Section \ref{sec-proof-comp-prelim}}, which is defined over a finite extension $\ReFl^+$ of $\ReFl$ in $\AC{\bQ}$, is given by the identity morphism of $\rep_{0, \bC}^{\otimes m}(-t)$.
\end{lemma}
\begin{proof}
    Since $\Model_\levcp$ is defined over $\ReFl$ and so the $p$-adic analytification functor from the category of algebraic filtered connections to the category of $p$-adic analytic ones is fully faithful, by \cite[\aCh 4, \aCor 6.8.2]{Andre/Baldassarri:2001-DDA} or \cite[\aCor 34.6.2]{Andre/Baldassarri/Cailotto:2020-DDA(2)}; since the $p$-adic analytification of $\bigl(R^i\AVstr^m_*(\Omega^\bullet_{\AV^m / \Model_\levcp}) \otimes_\ReFl \BFp, \nabla, \Fil^\bullet\bigr) \cong \bigl(R^i\AVstr^m_{\BFp, *}(\Omega^\bullet_{\AV^m_\BFp / \Model_{\levcp, \BFp}}), \nabla, \Fil^\bullet\bigr)$ is canonically isomorphic to $\bigl(\DdR(R^i\AVstr^m_{\BFp, \et, *}(\AC{\bQ}_p)), \nabla, \Fil^\bullet\bigr)$, for any finite extension $\BFp$ of the composite of $\ReFl$ and $\bQ_p$ in $\AC{\bQ}_p$, by \cite[\aThms 8.8 and 9.1]{Scholze:2013-phtra}; and since such an isomorphism is functorial, we have \Refeq{\ref{eq-lem-std-comp-dR}}, from which \Refeq{\ref{eq-lem-std-comp-B}} follows by taking horizontal sections, because both sides of \Refeq{\ref{eq-lem-std-comp-dR}} can be canonically identified \Pth{up to the same Tate twist $(-t)$} with the image of $\bigl(R^m\AVstr^m_*(\Omega^\bullet_{\AV^m / \Model_\levcp}) \otimes_\ReFl \bC, \nabla, \Fil^\bullet\bigr)$ under $\varepsilon_m^*$, for some endomorphism $\varepsilon_m$ of the abelian scheme $\AVstr^m: \AV^m \to \Model_\levcp$, by Lieberman's trick \Pth{see, \eg, \cite[\aSec 3.2]{Lan/Suh:2012-vttac}}.  Since the comparison isomorphisms among the Betti, \'etale, and de Rham cohomology of an abelian variety defined over $\ReFl^+$ are all compatible with each other, the second assertion also follows.
\end{proof}

\begin{lemma}\label{lem-rep-summand}
    For each irreducible representation $\rep$ of $\Grp{G}$ over $\AC{\bQ}$, there exist integers $m_\rep \geq 0$ and $t_\rep$ \Pth{depending noncanonically on $\rep$} such that $\rep$ is a direct summand of $\rep_0^{\otimes m_\rep}(-t_\rep)$, where $(-t_\rep)$ denotes the formal Tate twist \Pth{which has no effect when restricted to the subgroup $\Grp{G}^\der$ of $\Grp{G}$}, so that $\rep = s_\rep \bigl(\rep_0^{\otimes m_\rep}(-t_\rep)\bigr)$ for some Hodge tensor $s_\rep \in \rep_0^\otimes$ \Pth{\ie, a tensor of weight $(0, 0)$ with respect to the induced Hodge structure on $\rep_0^\otimes$; \Refcf{} \cite[\aCh I, \aProp 3.4]{Deligne/Milne/Ogus/Shih:1982-HMS}}.
\end{lemma}
\begin{proof}
    See \cite[\aProp 3.2]{Lan/Stroh:2018-ncaes}, which is based on \cite[\aCh I, \aProp 3.1(a)]{Deligne/Milne/Ogus/Shih:1982-HMS}.
\end{proof}

By combining Lemmas \ref{lem-std-comp} and \ref{lem-rep-summand}, we obtain the following:
\begin{cor}\label{cor-rep-summand}
    For each irreducible $\rep \in \Rep_{\AC{\bQ}}(\Grp{G})$, there exist some integers $m_\rep \geq 0$ and $t_\rep$ such that the local systems $\BSh{\rep}_\bC$ and $\pBSh{\rep}_\bC$ are direct summands of $\BSh{\rep}_{0, \bC}^{\otimes m_\rep}(-t_\rep)$ and $\pBSh{\rep}_{0, \bC}^{\otimes m_\rep}(-t_\rep)$, respectively.  Consequently, there is a morphism
    \begin{equation}\label{eq-comp-B}
        \BSh{\rep}_\bC \to \pBSh{\rep}_\bC
    \end{equation}
    defined by composing $\BSh{\rep}_\bC \Emn{\can} \BSh{\rep}_{0, \bC}^{\otimes m_\rep}(-t_\rep) \Misn{\Refeq{\ref{eq-lem-std-comp-B}}} \pBSh{\rep}_{0, \bC}^{\otimes m_\rep}(-t_\rep) \Surn{\can} \pBSh{\rep}_\bC$.
\end{cor}

\begin{prop}\label{prop-comp-B}
    The above morphism \Refeq{\ref{eq-comp-B}} is an isomorphism over the connected component $\Gamma_{\levcp, g_0}^+ \Lquot \Shdom^+$ of $\Model_{\levcp, \bC}$ that induces the identity morphism between the two representations $\Frep^{+, (p)}_{\levcp, g_0}(\rep)$ and $\Frep^+_{\levcp, g_0}(\rep)$.  In particular, Proposition \ref{prop-ext-der} holds under the assumptions of this subsection.
\end{prop}
\begin{proof}
    It suffices to show that, via the canonical isomorphisms in Proposition \ref{prop-ident-spec-pt}, the pullback of \Refeq{\ref{eq-comp-B}} to the image of $\hd$, as in Lemma \ref{lem-std-comp}, is given by the identity morphism of a subspace of $\rep_{0, \bC}^{\otimes m}(-t)$.  Since the comparison isomorphisms thus far are functorial and compatible with pullbacks to special points, over which the pullbacks of $\AV$ are abelian varieties potentially of CM type over number fields \Pth{see \cite[\aCor 14.11]{Milne:2005-isv}}, the pullback of \Refeq{\ref{eq-comp-B}} is induced by the comparison isomorphisms for the cohomology of such abelian varieties, which are compatible with the ones used in the proof of Proposition \ref{prop-ident-spec-pt} by Remark \ref{rem-dR-comp-for-CM-ab-var} below, and it suffices to note that the Hodge tensor $s_\rep$ in Lemma \ref{lem-rep-summand} are respected by such comparison isomorphisms, because Hodge cycles on such abelian varieties are \emph{absolute Hodge} \Pth{by \cite[\aCh I, Main \aThm 2.11]{Deligne/Milne/Ogus/Shih:1982-HMS}} and \emph{de Rham} \Pth{by \cite[\aThm 0.3]{Blasius:1994-phcav}}.
\end{proof}

\begin{rk}\label{rem-dR-comp-for-CM-ab-var}
    By \cite[\aSec 5]{Serre/Tate:1968-grav}, abelian varieties potentially of CM type over number fields have potential good reduction everywhere.  By \cite[\aSec 11]{Ito/Ito/Koshikawa:2021-ckfat}, the $p$-adic de Rham comparison isomorphisms of Faltings's \Pth{used in \cite{Blasius:1994-phcav} and the proof of Proposition \ref{prop-ident-spec-pt}} and Scholze's \Pth{used in the proof of Lemma \ref{lem-std-comp}} coincide \Pth{at least} for abelian varieties with potential good reductions over $p$-adic fields.
\end{rk}

\subsection{Cases of real rank at least two}\label{sec-case-rank-two+}

In this subsection, we assume that $\Grp{G}^{\der, c}$ is $\bQ$-simple and simply-connected as an algebraic group over $\bQ$ \Pth{so that $\Grp{G}^\der \cong \Grp{G}^{\der, c}$}, and that $\rank_\bR(\Grp{G}^\der_\bR) \geq 2$.  By Proposition \ref{prop-comp-B}, we may and we shall assume in addition that $\Grp{G}^\der_\bR$ is not of type $\mathrm{A}$.  We shall make use of the following special case of Margulis's \emph{superrigidity theorem}:
\begin{thm}\label{thm-superrig}
    Let $\Grp{H}$ be a $\bQ$-simple simply-connected connected algebraic group over $\bQ$, and let $\Gamma$ be an arithmetic subgroup of $\Grp{H}(\bQ)$.  Suppose that $\rank_\bR(\Grp{H}_\bR) \geq 2$.  Then, given any representation $\rho: \Gamma \to \GL_m(\bC)$, there exists a finite index normal subgroup $\Gamma_0$ of $\Gamma$ such that $\rho|_{\Gamma_0}$ extends to a \Pth{unique} group homomorphism $\widetilde{\rho}: \Grp{H}(\bC) \to \GL_m(\bC)$ that is induced by an algebraic group homomorphism $\Grp{H}_\bC \to \GL_{m, \bC}$, and such that $\rho(\gamma) = \delta(\gamma) \widetilde{\rho}(\gamma)$, for all $\gamma \in \Gamma$, for some representation $\delta: \Gamma / \Gamma_0 \to \GL_m(\bC)$ whose image commutes with $\widetilde{\rho}(\Grp{H}(\bC))$.
\end{thm}
\begin{proof}
    This follows from \cite[\aCh VIII, \aThm (B), part (iii)]{Margulis:1991-DSL} with $S = \{ \infty \}$, $\Lambda = \Gamma$, $K = \bQ$, and $\ell = \bC$ \Pth{in the notation there}.
\end{proof}

By applying Theorem \ref{thm-superrig} with $\Grp{H} = \Grp{G}^{\der, c}$, $\Gamma = \Gamma_{\levcp, g_0}^{+, c}$, and $\rho = \Frep^{+, (p)}_{\levcp, g_0}(\rep)$ as in Section \ref{sec-proof-comp-prelim}, we see that there is a finite-index subgroup $\Gamma_{\levcp, g_0, 0}^{+, c}$ of $\Gamma_{\levcp, g_0}^{+, c}$ such that the restriction $\Frep^{+, (p)}_{\levcp, g_0}(\rep)|_{\Gamma_{\levcp, g_0, 0}^{+, c}}$ extends to an algebraic representation $\widetilde{\Frep}^{+, (p)}_{\levcp, g_0}(\rep): \Grp{G}^{\der, c}_\bC \to \GL_\bC(\rep_\bC)$.  If $\levcp_1 \subset \levcp_2$ are neat open compact subgroups of $\Grp{G}(\bAi)$, then $\Gamma_{\levcp_1, g_0}^{+, c} \subset \Gamma_{\levcp_2, g_0}^{+, c}$ and $\Frep^{+,(p)}_{\levcp_1, g_0}(\rep) = \Frep^{+,(p)}_{\levcp_2, g_0}(\rep)|_{\Gamma_{ \levcp_1, g_0}^{+, c}}$, and therefore
\begin{equation}\label{eq-rigidity-level}
    \Frep^{+,(p)}_{\levcp_1, g_0}(\rep)|_{\Gamma_{\levcp_1, g_0, 0}^{+, c} \cap \Gamma_{\levcp_2, g_0, 0}^{+, c}} = \Frep^{+,(p)}_{\levcp_2, g_0}(\rep)|_{\Gamma_{\levcp_1, g_0, 0}^{+, c} \cap \Gamma_{\levcp_2, g_0, 0}^{+, c}}.
\end{equation}
Since $\Gamma_{\levcp_1, g_0, 0}^{+, c} / (\Gamma_{\levcp_1, g_0, 0}^{+, c} \cap \Gamma_{\levcp_2, g_0, 0}^{+, c}) \subset \Gamma_{\levcp_2, g_0}^{+, c} / \Gamma_{\levcp_2, g_0, 0}^{+, c}$ is finite, $\Gamma_{\levcp_1, g_0, 0}^{+, c} \cap \Gamma_{\levcp_2, g_0, 0}^{+, c}$ is an arithmetic subgroup of $\Grp{G}^{\der, c}(\bQ)$.  By \Refeq{\ref{eq-rigidity-level}} and the Borel density theorem \Pth{see Remark \ref{rem-Borel-density}}, $\widetilde{\Frep}^{+, (p)}_{\levcp_1, g_0}(\rep) = \widetilde{\Frep}^{+, (p)}_{\levcp_2, g_0}(\rep)$.  Since $\levcp_1$ and $\levcp_2$ are arbitrary, there is a well-defined assignment \Pth{to $\rep$} of an algebraic representation
\[
    \widetilde{\Frep}^{+, (p)}_{g_0}(\rep): \Grp{G}^{\der, c}_\bC \to \GL_\bC(\rep_\bC)
\]
such that $\widetilde{\Frep}^{+, (p)}_{\levcp, g_0}(\rep) = \widetilde{\Frep}^{+, (p)}_{g_0}(\rep)$ for all neat open compact subgroups $\levcp$ of $\Grp{G}(\bAi)$.

Since $\Grp{G}^{\der, c}$ is $\bQ$-simple and simply-connected, since $\rank_\bR(\Grp{G}^\der_\bR) \geq 2$, and since $\Grp{G}^\der_\bR$ is not of type $\mathrm{A}$, by the known positive answers to the congruence subgroup problem \Pth{see \cite[\aProp 9.10, \aThm 9.15, and the summary in the last two pages of \aSec 9.5]{Platonov/Rapinchuk:1994-AGN}}, any $\Gamma_{\levcp, g_0, 0}^{+, c}$ obtained above is a congruence subgroup.  By taking any $\levcp' \subset \levcp$ such that $\Gamma_{\levcp', g_0}^{+, c} \subset \Gamma_{\levcp, g_0, 0}^{+, c}$, we see that $\Frep^{+, (p)}_{\levcp', g_0}(\rep) = \widetilde{\Frep}^{+, (p)}_{g_0}(\rep)|_{\Gamma_{\levcp', g_0}^{+, c}}$.  Thus, by Lemma \ref{lem-isog}, it suffices to show that $\widetilde{\Frep}^{+, (p)}_{g_0}(\rep)$ coincides with $\rep_\bC|_{\Grp{G}^{\der, c}_\bC}$.

By the above construction, and by Proposition \ref{prop-loc-syst-p}, the assignment of $\widetilde{\Frep}^{+, (p)}_{g_0}(\rep)$ to $\rep \in \Rep_\bC(\Grp{G}^c)$ defines a tensor functor from $\Rep_\bC(\Grp{G}^c)$ to $\Rep_\bC(\Grp{G}^{\der, c})$, and hence induces \Pth{as in \cite[\aCh II, \aCor 2.9]{Deligne/Milne/Ogus/Shih:1982-HMS}} a group homomorphism $\Grp{G}^{\der, c}_\bC \to \Grp{G}^c_\bC$.  Since $\Grp{G}^{\der, c}_\bC$ is semisimple, this homomorphism factors through
\begin{equation}\label{eq-B-p-B-twist}
    \Grp{G}^{\der, c}_\bC \to \Grp{G}^{\der, c}_\bC.
\end{equation}

For each $\gamma \in \Grp{G}^{\der, c}(\bQ)$, the Hecke action of $g = g_0^{-1} \gamma g_0 \in \Grp{G}(\bAi)$ induces an isomorphism $\Gamma^+_{g \levcp g^{-1}, g_0} \Lquot \Shdom^+ \Mi \Gamma^+_{\levcp, g_0} \Lquot \Shdom^+$ \Pth{see \Refeq{\ref{eq-Hecke-g-conn-comp}}} defined by left multiplication by $\gamma^{-1}$, compatible with the isomorphism $\Gamma^+_{g \levcp g^{-1}, g_0} \Mi \Gamma^+_{\levcp, g_0}$ induced by conjugation by $\gamma^{-1}$.  By Proposition \ref{prop-ident-spec-pt} and \Refeq{\ref{eq-Hecke-g-fund-rep}}, we have the compatibility
\begin{equation}\label{eq-Hecke-conj}
\begin{split}
    \widetilde{\Frep}^{+, (p)}_{g_0}(\rep)(\gamma \gamma' \gamma^{-1})
    & = \bigl(\widetilde{\Frep}^{+, (p)}_{g_0}(\rep)(\gamma)\bigr) \bigl(\widetilde{\Frep}^{+, (p)}_{g_0}(\rep)(\gamma')\bigr) \bigl(\widetilde{\Frep}^{+, (p)}_{g_0}(\rep)(\gamma)\bigr)^{-1} \\
    & = \bigl(\Hrep(\rep)(\gamma)\bigr) \bigl(\widetilde{\Frep}^{+, (p)}_{g_0}(\rep)(\gamma')\bigr) \bigl(\Hrep(\rep)(\gamma)\bigr)^{-1},
\end{split}
\end{equation}
for all $\gamma' \in \Gamma_{\levcp, g_0, 0}^{+, c} \cap \gamma^{-1} \Gamma_{g \levcp g, g_0, 0}^{+, c} \gamma$, where $g = g_0^{-1} \gamma g_0$, where
\[
    \Hrep(\rep): \Grp{G}^{\der, c}_\bC \to \GL_\bC(\rep_\bC)
\]
denotes the algebraic representation given by the restriction $\rep_\bC|_{\Grp{G}^{\der, c}_\bC}$.

\begin{lemma}\phantomsection\label{lem-appl-Schur}
    Suppose that the representation $\widetilde{\Frep}^{+, (p)}_{g_0}(\rep)$ is irreducible when $\rep$ is.  Then $\widetilde{\Frep}^{+, (p)}_{g_0}(\rep) = \Hrep(\rep)$ as algebraic representations of $\Grp{G}^{\der, c}_\bC$.
\end{lemma}
\begin{proof}
    By Proposition \ref{prop-loc-syst-p}, we may assume that $\Hrep := \Hrep(\rep)$ and hence $\widetilde{\Frep} := \widetilde{\Frep}^{+, (p)}_{g_0}(\rep)$ are irreducible.  Let us measure their difference by the algebraic morphism $\epsilon: \Grp{G}^{\der, c}_\bC \to \GL_\bC(\rep_\bC)$ \Pth{which is not shown to be a group homomorphism yet} defined by $\epsilon(g) = \widetilde{\Frep}(g)^{-1} \Hrep(g)$, for all $g \in \Grp{G}^{\der, c}(\bC)$.  By \Refeq{\ref{eq-Hecke-conj}}, we have
    \begin{equation}\label{eq-Hecke-comm}
        \widetilde{\Frep}(\gamma') = \epsilon(\gamma) \widetilde{\Frep}(\gamma') \epsilon(\gamma)^{-1}
    \end{equation}
    for all $\gamma' \in \Gamma' := \Gamma_{\levcp, g_0, 0}^{+, c} \cap \gamma^{-1} \Gamma_{g \levcp g, g_0, 0}^{+, c} \gamma$.  Since $\Gamma'$ is a neat arithmetic subgroup of $\Grp{G}^{\der, c}(\bQ)$, by the Borel density theorem \Pth{see Remark \ref{rem-Borel-density}}, we also have \Refeq{\ref{eq-Hecke-comm}} for all $\gamma' \in \Grp{G}^{\der, c}(\bC)$.  Then the morphism $\epsilon$ is a group homomorphism, because $\epsilon(\gamma \gamma') = \widetilde{\Frep}(\gamma \gamma')^{-1} \Hrep(\gamma \gamma') = \widetilde{\Frep}(\gamma')^{-1} \widetilde{\Frep}(\gamma)^{-1} \Hrep(\gamma) \Hrep(\gamma') = \widetilde{\Frep}(\gamma')^{-1} \epsilon(\gamma) \Hrep(\gamma') = \epsilon(\gamma) \widetilde{\Frep}(\gamma')^{-1} \Hrep(\gamma') = \epsilon(\gamma) \epsilon(\gamma')$, for all $\gamma, \gamma' \in \Grp{G}^{\der, c}(\bQ)$, and because $\Grp{G}^{\der, c}(\bQ)$ is Zariski dense in $\Grp{G}^{\der, c}$ \Pth{by \cite[\aCor 13.3.10]{Springer:1998-LAG(2)}, or still the Borel density theorem}.  By Schur's lemma \Pth{and this Zariski density}, $\epsilon$ factors through an algebraic group homomorphism $\Grp{G}^{\der, c}_\bC \to \Gm{\bC}$, which is trivial because $\Grp{G}^{\der, c}_\bC$ is semisimple.
\end{proof}

\begin{lemma}\label{lem-rep-tensor-equiv}
    The above homomorphism \Refeq{\ref{eq-B-p-B-twist}} is an automorphism.  In particular, the representation $\widetilde{\Frep}^{+, (p)}_{g_0}(\rep)$ is indeed irreducible when $\rep$ is.
\end{lemma}
\begin{proof}
    Since $\Grp{G}^{\der, c}_\bC$ is semisimple and simply-connected, it suffices to show that the Lie algebra of the kernel of \Refeq{\ref{eq-B-p-B-twist}}, which is a priori a product of $\bC$-simple factors of the Lie algebra of $\Grp{G}^{\der, c}_\bC$, is trivial.  Therefore, it suffices to show that \Refeq{\ref{eq-B-p-B-twist}} has nontrivial restrictions to all $\bC$-simple factors of $\Grp{G}^{\der, c}_\bC$, and it suffices to find some $\rep \in \Rep_\bC(\Grp{G}^c)$ such that $\widetilde{\Frep}^{+, (p)}_{g_0}(\rep)$ is nontrivial on all simple $\bC$-simple factors.

    As explained in \cite{Borovoi:1983/84-lcccs, Milne:1983-aacsv}, based on a construction due to Piatetski-Shapiro, there exist morphisms $\varphi_1: (\Grp{G}, \Shdom) \Em (\Grp{G}_1, \Shdom_1)$ and $\varphi_2: (\Grp{G}_2, \Shdom_2) \Em (\Grp{G}_1, \Shdom_1)$ between Shimura data such that the following hold:
    \begin{itemize}
        \item $\Grp{G}_1^{\der, c}$ is $\bQ$-simple, and we have $\Grp{G}^{\der, c} \Emn{~~\varphi_1} \Grp{G}_1^{\der, c} \Em \Res_{F / \bQ} \Grp{G}^{\der, c}_F$ for some totally real number field $F$, identifying $\Grp{G}_\bC^{\der, c}$ as a direct factor of $\Grp{G}_{1, \bC}^{\der, c}$.

        \item $\Grp{G}_2^{\der, c}$ is also $\bQ$-simple, and all $\bC$-simple factors of $\Grp{G}_{2, \bC}^{\der, c}$ are of type $A_1$.  In this case, as in the proof of Lemma \ref{lem-ab-type}, the Shimura datum $(\Grp{G}_2, \Shdom_2)$ is of abelian type, for which Proposition \ref{prop-comp-B} and hence Theorem \ref{thm-loc-syst-comp} hold.  Moreover, the homomorphism $\Grp{G}_{2, \bC}^{\der, c} \to \Grp{G}_{1, \bC}^{\der, c}$ induced by $\varphi_2$ embeds distinct $\bC$-simple factors of $\Grp{G}_{2, \bC}^{\der, c}$ into distinct $\bC$-simple factors of $\Grp{G}_{1, \bC}^{\der, c}$, and every simple factor of $\Grp{G}_{1, \bC}^{\der, c}$ meets $\Grp{G}_{2, \bC}^{\der, c}$ nontrivially.
    \end{itemize}
    Therefore, there exists some $\rep_1 \in \Rep_\bC(\Grp{G}_1^c)$ such that its pullback $\rep_2 \in \Rep_\bC(\Grp{G}_2^c)$ is nontrivial on all $\bC$-simple factors of $\Grp{G}_{2, \bC}^{\der, c}$.  By Proposition \ref{prop-loc-syst-p}, $\pBSh{\rep}_\bC$ and $\pBSh{\rep}_{2, \bC}$ are canonically isomorphic to pullbacks of $\pBSh{\rep}_{1, \bC}$, and we already know that the fundamental group representations associated with $\pBSh{\rep}_{2, \bC} \cong \BSh{\rep}_{2, \bC}$ are given by the restrictions of $\rep_{2, \bC}$.  Let $\widetilde{\Frep}^{+, (p)}_{g_0}(\rep_1)$ be associated with $\pBSh{\rep}_{1, \bC}$ as in the case of $\widetilde{\Frep}^{+, (p)}_{g_0}(\rep)$.  By \cite[\aCh I, \aSec 3, \aLem 3.13]{Margulis:1991-DSL}, the pullbacks of arithmetic subgroups of $\Grp{G}_1^{\der, c}(\bQ)$ to $\Grp{G}^{\der, c}(\bQ)$ and $\Grp{G}_2^{\der, c}(\bQ)$ contain arithmetic subgroups.  Therefore, by the Borel density theorem \Pth{see Remark \ref{rem-Borel-density}}, the pullback of $\widetilde{\Frep}^{+, (p)}_{g_0}(\rep_1)$ to $\Grp{G}_{2, \bC}^{\der, c}$ is nontrivial on all $\bC$-simple factors of $\Grp{G}_{2, \bC}^{\der, c}$, and hence $\widetilde{\Frep}^{+, (p)}_{g_0}(\rep_1)$ is nontrivial on all $\bC$-simple factors of $\Grp{G}_{1, \bC}^{\der, c}$.  By the Borel density theorem again, $\widetilde{\Frep}^{+, (p)}_{g_0}(\rep)$ is isomorphic to the pullback of $\widetilde{\Frep}^{+, (p)}_{g_0}(\rep_1)$, which is then nontrivial on all simple $\bC$-simple factors of $\Grp{G}^{\der, c}_\bC$, as desired.
\end{proof}

Thus, Proposition \ref{prop-ext-der} also holds under the assumptions of this subsection, by Lemmas \ref{lem-appl-Schur} and \ref{lem-rep-tensor-equiv}.  The proof of Theorem \ref{thm-loc-syst-comp} is now complete.

\appendix

\section{A formalism of decompletion}\label{sec-decompl}

In this appendix, we generalize the formalism of decompletion developed in \cite[\aSec 5]{Kedlaya/Liu:2016-RPH-2}, in order to treat the general Kummer towers.

\subsection{Results}\label{sec-decompl-results}

For a topological ring $A$ with a continuous action by a topological group $\Gamma$, let $\Proj_A(\Gamma)$ \Pth{\resp $\Rep_A(\Gamma)$} denote the category of finite projective \Pth{\resp finite free} $A$-modules equipped with a semilinear continuous $\Gamma$-action.  For simplicity, they are also called finite projective \Pth{\resp finite free} $\Gamma$-modules over $A$.  Note that, given a finite free $\Gamma$-module $L$ of rank $l$, after choosing a basis of $L$, the action of $\Gamma$ on $L$ can be represented by a $1$-cocycle $f \in C^1\bigl(\Gamma, \GL_l(A)\bigr)$, and any change of basis only results in a change of the cocycle by a coboundary.  It follows that the isomorphism classes of finite free $\Gamma$-modules of rank $l$ over $A$ are classified by the cohomology set $H^1\bigl(\Gamma, \GL_l(A)\bigr)$.  In addition, for $L_1, L_2 \in \Proj_A(\Gamma)$, we have
\begin{equation}\label{eq-Hom-Gamma}
    \Hom_{\Proj_A(\Gamma)}(L_1, L_2) \cong H^0(\Gamma, \dual{L}_1 \otimes_A L_2).
\end{equation}

From now on, we shall assume that our topological rings are all commutative, unless otherwise specified.  Let $\{ A_i \}_{i \in I}$ be a direct system of topological rings, where $I$ is a small filtered index category; and $\widehat{A}_\infty$ a complete topological ring with compatible continuous homomorphisms $A_i \to \widehat{A}_\infty$ such that the induced map $\varinjlim_{i \in I} A_i \to \widehat{A}_\infty$ has dense image.  Let $\Gamma$ be a topological group acting continuously and compatibly on $\{ A_i \}_{i \in I}$ and $\widehat{A}_\infty$.

\begin{defn}\label{def-decompl-syst}
    We call the triple $(\{ A_i \}_{i \in I}, \widehat{A}_\infty, \Gamma)$ a \emph{decompletion system} \Pth{\resp \emph{weak decompletion system}} if the following two conditions hold:
    \begin{enumerate}
        \item\label{def-decompl-syst-1}  For each finite projective \Pth{\resp finite free} $\Gamma$-module $L_\infty$ over $\widehat{A}_\infty$, there exist some $i \in I$, some finite projective \Pth{\resp finite free} $\Gamma$-module $L_i$ over $A_i$, and some $\Gamma$-equivariant continuous $A_i$-linear morphism $\iota_i: L_i \to L_\infty$ inducing an isomorphism $\iota_i \otimes 1: L_i \otimes_{A_i} \widehat{A}_\infty \Mi L_\infty$ of $\Gamma$-modules over $\widehat{A}_\infty$.  We shall call such a pair $(L_i, \iota_i)$ \Pth{or simply $L_i$} a \emph{model} of $L_\infty$ over $A_i$.

        \item\label{def-decompl-syst-2}  For each model $(L_i, \iota_i)$ over $A_i$, there exists some $i_0 \geq i$ such that, for every $i' \geq i_0$, the model $(L_{i'}, \iota_{i'}) := (L_i \otimes_{A_i} A_{i'}, \iota_i \otimes 1)$ is \emph{good} in the sense that the natural map $H^\bullet(\Gamma, L_{i'}) \to H^\bullet(\Gamma, L_\infty)$ is an isomorphism.
    \end{enumerate}
\end{defn}

\begin{rk}\label{rem-def-decompl-syst-equiv-cat}
    If $(\{ A_i \}_{i \in I}, \widehat{A}_\infty, \Gamma)$ is a decompletion system \Pth{\resp weak decompletion system}, then the natural functor $\varinjlim_{i \in I} \Proj_{A_i}(\Gamma) \to \Proj_{\widehat{A}_\infty}(\Gamma)$ \Pth{\resp $\varinjlim_{i \in I} \Rep_{A_i}(\Gamma) \to \Rep_{\widehat{A}_\infty}(\Gamma)$} is an equivalence of categories.  Indeed, the condition \Refenum{\ref{def-decompl-syst-1}} implies that the functor is essentially surjective, and \Refeq{\ref{eq-Hom-Gamma}} and the condition \Refenum{\ref{def-decompl-syst-2}} imply that the functor is fully faithful.  Similarly, the condition \Refenum{\ref{def-decompl-syst-2}} implies that, for any two models $(L_{i, 1}, \iota_{i, 1})$ and $(L_{i, 2}, \iota_{i, 2})$ of $L_\infty$ over $A_i$, there exists some $i' \geq i$ such that $(L_{i, 1} \otimes_{A_i} A_{i'}, \iota_{i, 1} \otimes 1) \cong (L_{i, 2} \otimes_{A_i} A_{i'}, \iota_{i, 2} \otimes 1)$ over $A_{i'}$.
\end{rk}

To give criterions when a triple $(\{ A_i \}_{i \in I}, \widehat{A}_\infty, \Gamma)$ is a \Pth{weak} decompletion system, we shall work with \Pth{nonarchimedean} Banach rings, as in \cite[\aSec 2.2]{Kedlaya/Liu:2015-RPH}.  For a Banach $A$-module $N$ and a closed $A$-submodule $M$, we shall equip $M$ with the induced norm and $N / M$ with the quotient norm.  Both are Banach $A$-modules.  The following lemma is straightforward.

\begin{lemma}\label{lem-Banach-module-split}
    Let $M \to N$ be an isometric homomorphism of Banach modules over a Banach ring $A$.  Then the following are equivalent:
    \begin{enumerate}
        \item The natural projection $\pi: N \to N / M$ admits an isometric section.

        \item The embedding $M \to N$ admits a submetric splitting $\pr: N \to M$.

        \item $N$ admits a closed $A$-submodule $L$ such that $M \oplus L \to N$ is an isometric isomorphism, where $M \oplus L$ is equipped with the supremum norm.
\end{enumerate}
\end{lemma}

If $L$ and $M$ are Banach modules over a Banach ring, we will often equip the completed tensor product $L \ho_A M$ with the \emph{product norm}, as in \cite[\aDef 2.1.10]{Kedlaya/Liu:2015-RPH}. If a Banach ring $A$ admits a continuous action by a profinite group $\Gamma$, and $M$ is a Banach $A$-module with a semilinear isometric $\Gamma$-action, then we equip the $A$-module $C^s(\Gamma, M)$ of continuous maps $\Gamma^s \to M$ with the supremum norm given by $\| f \| = \sup_{\gamma \in \Gamma^s} |f(\gamma)|$, for each degree $s$.  Then $C^\bullet(\Gamma, M)$ is a complex of Banach $A$-modules.  We will also make use of the following terminology.

\begin{defn}\label{def-unif-str-ex}
    A complex $(C^\bullet, d)$ of Banach modules over a Banach ring $A$ is called \emph{uniformly strict exact} with respect to some constant $c \geq 0$ if, for each degree $s$ and each cocycle $f \in C^s$, there exists $g \in C^{s - 1}$ such that $f = d g$ and $|g| \leq c |f|$.
\end{defn}

Now the following definition makes sense.

\begin{defn}\label{def-decompl-tuple}
    Let $(\{ A_i \}_{i \in I}, \widehat{A}_\infty, \Gamma)$ be as in the paragraph preceding Definition \ref{def-decompl-syst}.  Suppose moreover that each $A_i \to \widehat{A}_\infty$ is a closed embedding and $\Gamma$ is profinite.  We say that $(\{ A_i \}_{i \in I}, \widehat{A}_\infty, \Gamma)$ is \emph{weakly decompleting} if there exist:
    \begin{itemize}
        \item a norm $|\,\cdot\,|$ on $\widehat{A}_\infty$ making it a Banach ring \Pth{and therefore making each $A_i$ a Banach subring}; and

        \item an inverse system $\{ \Gamma_i \}_{i \in I}$ of closed normal subgroups converging to $1$ \Pth{\ie, each open neighborhood of $1$ in $\Gamma$ contains $\Gamma_i$, for some $i \in I$} such that the canonical map $\Gamma \to \Gamma / \Gamma_i$ admits a continuous section \Pth{which is not necessarily a homomorphism}, for each $i \in I$,
    \end{itemize}
    satisfying the following conditions:
    \begin{enumerate}
        \item\label{def-decompl-tuple-isometric}  The $\Gamma$-action on $\widehat{A}_\infty$ is isometric.

        \item\label{def-decompl-tuple-split}  \Pth{\emph{Splitting}.}  For each $i \in I$, the natural projection $\widehat{A}_\infty \to \widehat{A}_\infty / A_i$ admits an isometric section as Banach $A_i$-modules.

        \item\label{def-decompl-tuple-unif-str-ex}  \Pth{\emph{Uniform strict exactness}.}  There exists some $c > 0$ such that, for all $i \in I$, the complex $C^\bullet(\Gamma_i, \widehat{A}_\infty / A_i)$ is uniform strict exact with respect to $c$, as in Definition \ref{def-unif-str-ex}.  In particular, $\widehat{A}_\infty / A_i$ has totally trivial $\Gamma_i$-cohomology.
    \end{enumerate}
\end{defn}

\begin{rk}\phantomsection\label{rem-def-decompl-tuple}
    \begin{enumerate}
        \item\label{rem-def-decompl-tuple-norm}  If we keep the same choices of $\{ \Gamma_i \}_{i \in I}$, but modify the norm on $\widehat{A}_\infty$ up to equivalence such that the conditions \Refenum{\ref{def-decompl-tuple-isometric}} and \Refenum{\ref{def-decompl-tuple-split}} still hold, then the condition \Refenum{\ref{def-decompl-tuple-unif-str-ex}} also holds up to adjusting the constant $c > 0$.

        \item\label{rem-def-decompl-tuple-HS}  Since $\Gamma \to \Gamma / \Gamma_i$ admits a continuous section, the Hochschild--Serre spectral sequence holds for the continuous cohomology of the subgroup $\Gamma_i$ of $\Gamma$.

        \item\label{rem-def-decompl-tuple-split} By Lemma \ref{lem-Banach-module-split} above, the condition \Refenum{\ref{def-decompl-tuple-split}} in Definition \ref{def-decompl-tuple} is equivalent to the existence of a submetric splitting $\pr_i: \widehat{A}_\infty \Surj A_i$.  But unlike the classical Tate trace maps, the map $\pr_i$ is not required to be $\Gamma_i$-equivariant.
    \end{enumerate}
\end{rk}

Our first main result of this appendix is the following:
\begin{thm}\label{thm-decompl-weak}
    A weakly decompleting triple is a weak decompletion system.
\end{thm}

In order to obtain decompletion systems rather than weak decompletion systems, we shall consider those $A$ underlying stably uniform Huber pairs $(A, A^+)$ \Pth{as in \cite[\aDef 5.2.4]{Scholze/Weinstein:2020-BLG}} or stably uniform adic Banach rings \Pth{as in \cite[\aRem 2.8.5]{Kedlaya/Liu:2015-RPH}} over nonarchimedean fields.   For simplicity, we shall say such $A$ are \emph{stably uniform}.

\begin{defn}\label{def-decompl-tuple-stab}
    A triple $(\{ A_i \}_{i \in I}, \widehat{A}_\infty, \Gamma)$ as in the first two sentences of Definition \ref{def-decompl-tuple} is called \emph{stably decompleting} if:
    \begin{enumerate}[label=(\alph*), ref=\alph*]
        \item\label{def-decompl-tuple-stab-a} $A_i$'s and $\widehat{A}_\infty$ are stably uniform over a nonarchimedean field $k$.

        \item\label{def-decompl-tuple-stab-b} each rational subset $U$ of $\Spa(A_i, A_i^\circ)$ is stabilized by some open normal subgroup $\Gamma_U$ of $\Gamma$; and the pullback of $(\{ A_j \}_{j \geq i}, \widehat{A}_\infty, \Gamma_U)$ to each such $U$ is weakly decompleting.
    \end{enumerate}
\end{defn}

The second main result of this appendix is the following:
\begin{thm}\label{thm-decompl}
    A stably decompleting triple is a decompletion system.
\end{thm}

Now we start to prove Theorems \ref{thm-decompl-weak} and \ref{thm-decompl}.

\begin{lemma}\label{lem-basic-est-1}
    Let $(C^\bullet, d)$ be as in Definition \ref{def-unif-str-ex}.  Then, for each $f \in C^s$, there exists some $h \in C^{s - 1}$ such that $|h| \leq \max\bigl\{ c |f|, c^2 |df| \bigr\}$ and $|f - dh| \leq c |df|$.
\end{lemma}
\begin{proof}
    Since $d f$ is a cocycle, there exists some $g \in C^s$ such that $d f = d g$ and $|g| \leq c |d f|$.  Since $d (f - g) = 0$, there exists some $h \in C^{s - 1}$ such that $d h = f - g$ and $|h| \leq c |f - g| \leq \max\bigl\{ c |f|, c^2 |d f| \bigr\}$.
\end{proof}

\begin{lemma}\label{lem-weak-tot-triv}
    Let $A$ be a Banach ring with a continuous action by a profinite group $\Gamma$, and $M$ a Banach $A$-module with a semilinear isometric $\Gamma$-action.  Let $L = \oplus_{j = 1}^l A e_j$ be an object of $\Rep_A(\Gamma)$, equipped with the supremum norm.  Suppose that $C^\bullet(\Gamma, M)$ is uniformly strict exact with respect to some $c > 0$, as in Definition \ref{def-unif-str-ex}; and that there exists some $r > 1$ such that $|\gamma(e_j) - e_j| < \frac{1}{r c}$, for all $j$ and all $\gamma \in \Gamma$.  Then $C^\bullet(\Gamma, L \otimes_A M)$ is uniformly strict exact with respect to the same $c$.
\end{lemma}
\begin{proof}
    Let $f = \sum_{j = 1}^l (e_j \otimes f_j)$ be a cocycle with $f_j \in C^s(\Gamma, M)$ for all $1 \leq j \leq l$.  Note that the norm of $\sum_{j = 1}^l (e_j \otimes df_j) = \sum_{j = 1}^l (e_j \otimes d f_j) - d f = \sum_{j = 1}^l (e_j \otimes d f_j) - d \bigl(\sum_{j = 1}^l (e_j \otimes f_j) \bigr)$ is bounded by $\frac{\| f \|}{r c}$.  That is, for each $j$, we have $\| d f_j \| \leq \frac{\| f \|}{r c}$.  By Lemma \ref{lem-basic-est-1}, there exist $h_j \in C^{s - 1}(\Gamma, M)$ with $\| f_j - d h_j \| \leq \frac{\| f \|}{r}$ and $\| h_j \| \leq c \| f \|$, for all $j = 1, \ldots, l$.  Put $h = \sum_{j = 1}^l (e_j \otimes h_j)$.  Then $\| h \| \leq c \| f \|$, and the norm of $f - d h = \sum_{j = 1}^l (e_j \otimes (f_j - d h_j)) + \Bigl(\sum_{j = 1}^l (e_j \otimes d h_j) - d \bigl(\sum_{j = 1}^l (e_j \otimes h_j) \bigr)\Bigr)$ is bounded by $\frac{\| f \|}{r}$.  By iterating this process, we obtain cochains $H_1, H_2, \ldots \in C^{s - 1}(\Gamma, L \otimes_A M)$ satisfying $\| H_n \| \leq \frac{c \| f \|}{r^{n - 1}}$ and $\|f - d H_1 - \cdots - d H_n\| \leq \frac{\| f \|}{r^n}$, for all $n \geq 1$.  Then $f = d H$ for $H = \sum_{i = 1}^\infty H_i \in C^{s - 1}\bigl(\Gamma, L \otimes_A M)$, and $\| H \| \leq c \| f \|$, as desired.
\end{proof}

\begin{cor}\label{cor-weak-good-model}
    Let $(\{ A_i \}_{i \in I}, \widehat{A}_\infty, \Gamma)$ be weakly decompleting.  Let $L_i$ be a model of a finite free $\Gamma$-module $L_\infty$ over $\widehat{A}_\infty$, as in Definition \ref{def-decompl-syst}\Refenum{\ref{def-decompl-syst-1}}.  Then there exist some $i_0 \geq i$ such that $L_{i'}$ is a good model, as in Definition \ref{def-decompl-syst}\Refenum{\ref{def-decompl-syst-2}}, for each $i' \geq i_0$.
\end{cor}
\begin{proof}
    Take any basis $\{ e_j \}_{1 \leq j \leq l}$ of $L_i$ over $A_i$, and equip $L_i$ with the supremum norm, as in Lemma \ref{lem-weak-tot-triv}.  For any $r > 1$, since $\{ \Gamma_i \}_{i \in I}$ converges to $1$ in $\Gamma$, there exists some $i_0 \geq i$ such that $|\gamma(e_j) - e_j| < \frac{1}{r c}$, for all $\gamma \in \Gamma_{i_0}$ and $j$.  Hence, for each $i' \geq i_0$, by the assumption that $L_i$ is finite free, by the conditions in Definition \ref{def-decompl-tuple}, and by applying Lemma \ref{lem-weak-tot-triv} to $(A_{i'}, \Gamma_{i'}, \widehat{A}_\infty / A_{i'})$, we see that $L_i \otimes_{A_i} (\widehat{A}_\infty / A_{i'})$ has totally trivial $\Gamma_{i'}$-cohomology, and therefore has totally trivial $\Gamma$-cohomology, by the Hochschild--Serre spectral sequence \Pth{see Remark \ref{rem-def-decompl-tuple}\Refenum{\ref{rem-def-decompl-tuple-HS}}}.
\end{proof}

\begin{lemma}\label{lem-Banach-split}
    Let $A \to B$ be an isometric homomorphism of Banach rings.  If the natural projection $\pi: B \to B / A$ admits an isometric Banach $A$-module section $s: B / A \to B$, then $|\pi(b_1 b_2)| \leq \max\bigl\{|\pi(b_1)| \, |b_2|, |b_1| \, |\pi(b_2)| \bigr\}$, for all $b_1, b_2 \in B$.
\end{lemma}
\begin{proof}
    Write $b_i = a_i + s(\pi(b_i))$, for $i = 1, 2$, so that $\pi(b_1 b_2) = a_1 \pi(b_2) +  \pi(b_1) a_2 + \pi(s(\pi(b_1)) s(\pi(b_2)))$.  Then $|a_1 \pi(b_2)| \leq |a_1| \, |\pi(b_2)| \leq |b_1| \, |\pi(b_2)|$, by Lemma \ref{lem-Banach-module-split}.  Similarly, $|\pi(b_1) a_2| \leq |b_2| \, |\pi(b_1)|$.  Consequently, we have $|\pi(s(\pi(b_1)) s(\pi(b_2)))| \leq |s(\pi(b_1)) s(\pi(b_2))| \leq |(s(\pi(b_1))| \, |s(\pi(b_2))| = |\pi(b_1)| \, |\pi(b_2)|$.  Since $|\pi(b_1)| \, |\pi(b_2)| \leq \min\bigl\{ |\pi(b_1)| \, |b_2|, |b_1| \, |\pi(b_2)| \bigr\}$, the lemma follows.
\end{proof}

\begin{lemma}\label{lem-basic-est-2}
    Let $A \to B$ be a $\Gamma$-equivariant isometric homomorphism of Banach rings with isometric $\Gamma$-actions.  Assume that the natural projection $\pi: B \to B / A$ admits an isometric Banach $A$-module section $s$, and that $C^\bullet(\Gamma, B / A)$ is uniformly strict exact with respect to some constant $c \geq 1$ \Pth{as in Definition \ref{def-unif-str-ex}}.  Let $\M_l(B / A) := \M_l(B) / \M_l(A)$ \Pth{as Banach $A$-modules}.  Let $f$ be a cocycle in $C^1\bigl(\Gamma, \GL_l(B)\bigr)$.  Suppose that there exists some $r > 1$ such that $|f(\gamma) - 1| \leq \frac{1}{r c}$ for all $\gamma \in \Gamma$ and that $\| \overline{f} \| \leq \frac{1}{r c^2}$, where $\overline{f}$ is the image of $f$ in $C^1\bigl(\Gamma, \M_l(B / A)\bigr)$ \Pth{which is merely a cochain}.  \Pth{We shall also denote similar images by overlines in the proof below.}  Then $f$ is equivalent to a cocycle in $C^1\bigl(\Gamma, \GL_l(A)\bigr)$.
\end{lemma}
\begin{proof}
    We claim that there exists some $\varsigma \in \M_l(B)$ with $|\varsigma| \leq c \|\overline{f}\|$ such that the cocycle $f': \gamma \mapsto \gamma(1 + \varsigma) f(\gamma) (1 + \varsigma)^{-1}$ satisfies $|f'(\gamma) - 1| \leq \frac{1}{r c}$ for all $\gamma \in \Gamma$ and $\|\overline{f'}\| \leq \frac{\| \overline{f} \|}{r}$ in $C^1\bigl(\Gamma, \M_l(B / A)\bigr)$.  Granting the claim, by iterating this process, we can find a sequence $\varsigma_1, \varsigma_2, \ldots$ in $\M_l(B)$ with $|\varsigma_n| \leq \frac{c \| \overline{f} \|}{r^{n - 1}} \leq \frac{1}{r^n}$ such that $\bigl| \overline{\gamma(1 + \varsigma_n) \cdots \gamma(1 + \varsigma_1) f(\gamma) (1 + \varsigma_1)^{-1} \cdots (1 + \varsigma_n)^{-1}} \bigr| \leq \frac{\| \overline{f} \|}{r^n}$, for all $n \geq 1$.  Put $\varsigma_\infty = \lim_{n\rightarrow \infty} \bigl((1 + \varsigma_n) \cdots (1 + \varsigma_1)\bigr) \in \GL_l(B)$.  It follows that the cocycle $\widetilde{f}: \gamma \mapsto \gamma(\varsigma_\infty) f (\gamma) \varsigma_\infty^{-1}$ takes values in $\M_l(A) \cap \GL_l(B)$ and satisfies $|\widetilde{f}(\gamma) - 1| \leq \frac{1}{r c} \leq \frac{1}{r}$ for $\gamma \in \Gamma$.  This implies that $\widetilde{f}$ is a cocycle in $C^1\bigl(\Gamma, \GL_l(A)\bigr)$, and the lemma follows.

    It remains to prove the claim.  Note that $f(\gamma_1 \gamma_2) = \gamma_1(f(\gamma_2)) f(\gamma_1)$ because $f$ is cocycle in $C^1\bigl(\Gamma, \GL_l(B)\bigr)$.  By Lemma \ref{lem-Banach-split}, we have
    \[
        |d \overline{f}(\gamma_1, \gamma_2)| = |\overline{\gamma_1 f(\gamma_2) + f(\gamma_1) - f(\gamma_1 \gamma_2)}| = |\overline{(\gamma_1 f(\gamma_2) - 1)(f(\gamma_1) - 1)}| \leq \tfrac{\| \overline{f} \|}{r c}.
    \]
    By Lemma \ref{lem-basic-est-1} \Pth{applied to $-f$ instead}, there exists $\overline{h} \in \M_l(B / A)$ such that
    \begin{equation}\label{eq-lem-basic-est-2-lambda}
        |\overline{h}| \leq \max\bigl\{ c \| \overline{f} \|, c^2 \| d \overline{f} \|\bigr \} = c \| \overline{f} \| \leq \tfrac{1}{r c} \leq \tfrac{1}{r}
    \end{equation}
    and
    \begin{equation}\label{eq-lem-basic-est-2-g-bar}
        \| \overline{f} + d \overline{h} \| \leq c \| d \overline{f} \| \leq \tfrac{\| \overline{f} \|}{r}.
    \end{equation}
    By assumption, we can lift $\overline{h}$ to some $h \in \M_l(B)$ such that $|h| = |\overline{h}| \leq c \| \overline{f} \|$.  For $\gamma \in \Gamma$, by \Refeq{\ref{eq-lem-basic-est-2-lambda}}, we have $|\gamma(1 + h) f(\gamma) (1 + h)^{-1} - f(\gamma)| \leq |h| \leq \frac{1}{r c}$, and therefore $|\gamma(1 + h) f(\gamma) (1 + h)^{-1} -1 | \leq \frac{1}{r c}$.  Moreover, by \Refeq{\ref{eq-lem-basic-est-2-lambda}} again, we have $|\overline{\gamma(1 + h) f(\gamma) (1 + h)^{-1} - \gamma(1 + h) f(\gamma) (1 - h)}| \leq |h|^2 \leq \frac{1}{r c}\bigl(c \| \overline{f} \|\bigr) = \frac{\| \overline{f} \|}{r}$.  Also, by Lemma \ref{lem-Banach-split}, we get $|\overline{\gamma(1 + h) f(\gamma) (1 - h)} - \overline{f}(\gamma) - \gamma(\overline{h}) + \overline{h}| = |\overline{\gamma(h) (f(\gamma) - 1)} - \overline{(f(\gamma) - 1) h}) - \overline{\gamma(h) f(\gamma) h}| \leq \frac{\| \overline{f} \|}{r}$.  By combining these and \Refeq{\ref{eq-lem-basic-est-2-g-bar}}, we obtain the desired estimate $|\overline{\gamma(1 + h) f(\gamma) (1 + h)^{-1}}| \leq \frac{\| \overline{f} \|}{r}$, and the claim follows.
\end{proof}

\begin{proof}[Proof of Theorem \ref{thm-decompl-weak}]
    By Corollary \ref{cor-weak-good-model}, the condition \Refenum{\ref{def-decompl-syst-2}} in Definition \ref{def-decompl-syst} holds.  Hence, our main task is to verify the condition \Refenum{\ref{def-decompl-syst-1}} in Definition \ref{def-decompl-syst}.

    Let $L_\infty$ be a finite free $\Gamma$-module over $\widehat{A}_\infty$.  As before, by choosing an $\widehat{A}_\infty$-basis of $L_\infty$, the $\Gamma$-module structure of $L_\infty$ amounts to a cocycle $f\in C^1\bigl(\Gamma, \GL_l(\widehat{A}_\infty)\bigr)$.  By taking $i$ sufficiently large such that Lemma \ref{lem-basic-est-2} applies to the restriction of $f$ to a $1$-cocycle of $\Gamma_i$, we obtain a \Pth{free} model $L_i$ of the $\Gamma_i$-module $L_\infty$ over $A_i$.  Since the $\Gamma$-action on $L_\infty$ is $A_i$-semilinear, and since $\Gamma_i$ is normal in $\Gamma$, for each $g \in \Gamma$, the subset $g L_i := \{ g(x) : x \in L_i \}$ of $L_\infty$ is not only an $A_i$-submodule, but also a $\Gamma_i$-submodule.  Moreover, the canonical map $(g L_i) \otimes_{A_i} \widehat{A}_\infty \to L_\infty$ is an isomorphism of $\Gamma_i$-modules over $\widehat{A}_\infty$, as the canonical map $L_i \otimes_{A_i} \widehat{A}_\infty \to L_\infty$ is.  We would like to find some $i' \geq i$ such that $L_{i'} = g L_{i'}$ in $L_\infty$, for all $g \in \Gamma$.  \Pth{We emphasize that we need the same $i'$ to work for all $g \in \Gamma$.}  If so, then the semilinear $\Gamma_i$-action on $L_{i'}$ extends to a semilinear $\Gamma$-action, which makes $L_{i'}$ a model of the $\Gamma$-module $L_\infty$ over $A_{i'}$.  We shall adapt the proof of Corollary \ref{cor-weak-good-model}.

    Take any bases $\{ e_j \}_{1 \leq j \leq l}$ and $\{ e_j' \}_{1 \leq j \leq l}$ of $L_i$ and $\dual{L}_i$ over $A_i$, respectively.  Since the $\Gamma$-action on $L_\infty$ is $A_i$-semilinear, $\{ g e_j \}_{1 \leq j \leq l}$ is a basis of $g L_i$ over $A_i$, and $\{ e_j' \otimes g(e_{j'}) \}_{1 \leq j, j' \leq l}$ is a basis of $M := \dual{L}_i \otimes_{A_i} (g L_i)$ over $A_i$, for each $g \in \Gamma$.  We shall equip these modules with the supremum norms given by the chosen bases.  Then the norm on $M = \dual{L}_i \otimes_{A_i} (g L_i)$ is also the product norm.  Since $\Gamma$ acts isometrically on $A_i$, the $A_i$-semilinear map $g: L_i \to g L_i$ is also isometric.  Fix any $r > \max\{1, \frac{1}{c}\}$, and take any sufficiently large $i' \geq i$ such that $|\gamma g(e_j) - g(e_j)|_{g L_i} = |g(g^{-1} \gamma g)(e_j) - g(e_j)|_{g L_i} = |(g^{-1} \gamma g)(e_j) - e_j|_{L_i} < \frac{1}{r c}$ and $|\gamma(e_j') - e_j'|_{\dual{L}_i} < \frac{1}{r c}$, for all $g \in \Gamma$, $\gamma \in \Gamma_{i'}$, and $j$.  As a result, we have $\bigl|\gamma\bigl(e_j' \otimes g(e_{j'})\bigr) - e_j' \otimes g(e_{j'})\bigr|_M = \bigl|\bigl(\gamma(e_j') - e_j'\bigr) \otimes \bigl(\gamma g(e_{j'}) - g(e_{j'})\bigr) + e_j' \otimes \bigl(\gamma g(e_{j'}) - g(e_{j'})\bigr) + \bigl(\gamma(e_j') - e_j'\bigr) \otimes g(e_{j'})\bigr|_M < \frac{1}{r c}$.  Therefore, $(\dual{L}_i \otimes_{A_i} g L_i) \otimes_{A_i} (\widehat{A}_\infty / A_{i'})$ has totally trivial $\Gamma_i$-cohomology, by Lemma \ref{lem-weak-tot-triv} and the conditions in Definition \ref{def-decompl-tuple}.  Thus, since the canonical inclusions $L_{i'} \Em L_\infty$ and $g L_{i'} \Em L_\infty$ induce the isomorphisms $\Hom_{\Proj_{A_{i'}}(\Gamma_i)}(L_{i'}, g L_{i'}) \Mi \Hom_{\Proj_{\widehat{A}_\infty}(\Gamma_i)}(L_i \otimes_{A_i} \widehat{A}_\infty, g L_i \otimes_{A_i} \widehat{A}_\infty) \Mi \Hom_{\Proj_{\widehat{A}_\infty}(\Gamma_i)}(L_\infty, L_\infty)$ \Pth{\Refcf{} \Refeq{\ref{eq-Hom-Gamma}}}, we obtain $\Id_{L_\infty}(L_{i'}) = g L_{i'}$; \ie, $L_{i'} = g L_{i'}$ in $L_\infty$, as desired.
\end{proof}

\begin{lemma}\label{lem-cplx-cpt-Banach-surj}
    For any continuous surjective map $U \to V$ of Banach spaces over a nonarchimedean field $k$, and any compact topological space $X$, the natural map $C(X, U) \to C(X, V)$ between spaces of continuous functions on $X$ is also surjective.
\end{lemma}
\begin{proof}
    Let $f \in C(X, V)$.  Then the subspace $f(X)$ of $V$ is a compact metric space, which is separable and admits a countable dense subset.  Therefore, $f(X)$ is contained in a closed subspace $W$ of $V$ with a dense countable-dimensional $k$-subspace.  By \cite[\aSec 2.7.2, \aProp 8]{Bosch/Guntzer/Remmert:1984-NAA}, up to replacing the norm on $W$ with an equivalent one, $W$ admits an orthonormal Schauder $k$-basis $\{ e_j \}_{j \in J}$.  Thus, in order to lift $f$ to $C(X, U)$, it suffices to find some elements $\widetilde{e}_j$'s of $U$ lifting $e_j$'s such that $\sup_{j \in J}\{|\widetilde{e}_j|\} < \infty$.  Such elements $\widetilde{e}_j$'s exist because, by the open mapping theorem \Pth{see, for example, \cite[\aThm 2.2.8]{Kedlaya/Liu:2015-RPH}}, the unit ball of $W$ is contained in the image of the ball of some radius $C > 0$ of the preimage of $W$ in $U$.
\end{proof}

\begin{lemma}\label{lem-tot-triv}
    Let $(\{ A_i \}_{i \in I}, \widehat{A}_\infty, \Gamma)$ be stably decompleting.  Let $L$ be a finite projective $\Gamma$-module over $A_i$, for some $i \in I$.  Then there exists some $i_0 \geq i$ such that $L \otimes_{A_i} (\widehat{A}_\infty / A_{i'})$ has totally trivial $\Gamma$-cohomology, for each $i' \geq i_0$.
\end{lemma}
\begin{proof}
    Take any finite covering $\mathfrak{B}$ of $\Spa(A_i, A_i^\circ)$ by rational subsets over which the pullbacks of $L_i$ are free, and let $\Gamma'$ be an open normal subgroup of $\Gamma$ stabilizing every rational subset in $\mathfrak{B}$.  By Corollary \ref{cor-weak-good-model}, there exists some $i_0 \geq i$ such that, for every $i' \geq i_0$, the restrictions of $L \otimes_{A_i} (\widehat{A}_\infty / A_{i'})$ to all the rational subsets in $\mathfrak{B}$ as well as their intersections have totally trivial $\Gamma'$-cohomology.  Since $A_{i'}$ and $\widehat{A}_\infty$ are stably uniform, by \cite[\aThms 2.7.7, and 2.8.10]{Kedlaya/Liu:2015-RPH}, their \v{C}ech complexes over $\mathfrak{B}$ are acyclic.  Therefore, the \v{C}ech complex $\mathscr{C}^\bullet$ for $L \otimes_{A_i} (\widehat{A}_\infty / A_{i'})$ over $\mathfrak{B}$ is also acyclic.  Equip rational localizations of $A_{i'}$ and $\widehat{A}_\infty$ with spectral norms \Pth{as in \cite[\aDefs 2.1.9 and 2.8.1]{Kedlaya/Liu:2015-RPH}}, and $L$ with any Banach $A_i$-module structure.  Then $\mathscr{C}^\bullet$ becomes, in particular, a complex of Banach $k$-spaces.  Consider the double complex $\mathscr{C}^{\bullet, \bullet}$ with $\mathscr{C}^{\bullet, b} := C^\bullet(\Gamma', \mathscr{C}^b)$ the standard complex of continuous maps from $(\Gamma')^\bullet$ to $\mathscr{C}^b$ that computes the continuous group cohomology, which is exact, as explained above, for each $b \geq 0$; and with $\mathscr{C}^{a, \bullet}$ induced by $\mathscr{C}^\bullet$, which is acyclic, for each $a \geq 0$, by applying Lemma \ref{lem-cplx-cpt-Banach-surj} to $X = (\Gamma')^a$ and $U = \mathscr{C}^b \to V = \OP{im}(\mathscr{C}^b \to \mathscr{C}^{b + 1})$, for each $b \geq 0$.  By an elementary diagram chasing, $C^\bullet(\Gamma', H^0(\mathscr{C}^\bullet))$ is also exact.  Thus, $L \otimes_{A_i} (\widehat{A}_\infty / A_{i'}) \cong H^0(\mathscr{C}^\bullet)$ has totally trivial $\Gamma'$-cohomology, and therefore totally trivial $\Gamma$-cohomology, by the Hochschild--Serre spectral sequence, as desired.
\end{proof}

\begin{lemma}[{\Refcf{} \cite[\aLem 5.6.8]{Kedlaya/Liu:2016-RPH-2}, with simplified statements and a more detailed proof here}]\label{lem-proj-dense-subring}
    Let $\widehat{A}_\infty$ be a Banach ring, with a direct system of closed subrings $\{ A_i \}_{i \in I}$ such that $\cup_i A_i$ is dense in $\widehat{A}_\infty$.  Then each finite projective $\widehat{A}_\infty$-module arises by base change from some finite projective $A_i$-module, for some $i \in I$.
\end{lemma}
\begin{proof}
    Let $M$ be a finite projective $\widehat{A}_\infty$-module, and choose an $\widehat{A}_\infty$-linear surjection $F \to M$ with $F$ a finite free $\widehat{A}_\infty$-module.  Choose a projector on $F$ corresponding to a splitting of $F \to M$, and represent this projector by a matrix $U$ over $\widehat{A}_\infty$.  Note that $U^2 = U$ and $\bigl| U \bigr| \geq 1$.  Hence, we may choose a matrix $V$ over some $A_i$ such that $\bigl| U - V \bigr| < \bigl| U \bigr|^{-3}$, and so that $\bigl| V^2 - V \bigr| = \bigl| V (V - U) + (V - U) U + (U - V) \bigr| \leq \bigl| U - V \bigr| \bigl| U \bigr| < \bigl| U \bigr|^{-2}$.  Let us define a sequence $W_0, W_1, \ldots$ by Newton iteration, by taking $W_0 = V$ and $W_{l + 1} = 3 W_l^2 - 2 W_l^3$, for all $l \geq 0$.  Since $W_{l + 1} - W_l = (W_l^2 - W_l) (1 - 2 W_l)$ and $W_{l + 1}^2 - W_{l + 1} = (W_l^2 - W_l)^2 (4 W_l^2 - 4 W_l - 3)$, by induction on $l \geq 0$, we obtain $\bigl| W_l - U \bigr| < \bigl| V^2 - V \bigr| \bigl| U \bigr|$ and $\bigl| W_l^2 - W_l \bigr| \leq \bigl| U \bigr|^{-2} \bigl(\bigl| V^2 - V \bigr| \bigl| U \bigr|^2\bigr)^{2^l}$. Consequently, the matrices $W_l$ converge to a matrix $W$.  By the last inequality, $W^2 = W$, and so $W$ is a projector over $A_i$.  Let $F_i$ be the free $A_i$-module on the same basis as $F$.  Then $W$ represents a projector on $F_i$, whose image is a finite projective $A_i$-module which we denote by $M_i$.

    It remains to exhibit an isomorphism $M_i \otimes_{A_i} \widehat{A}_\infty \Mi M$.  Note that $M_i \otimes_{A_i} \widehat{A}_\infty$ and $M$ are the images of the projectors on $F$ represented by the matrices $U$ and $W$, respectively; and that $\bigl| U - W \bigr| < \bigl| U  \bigr|^{-2}$.  Then the matrix $X := U W + (1 - U) (1 - W)$ is invertible, because $X - 1 = U W + (1 - U) (1 - W) - 1 = 2 U W - U - W = U (W - U) + (U - W) W$ satisfies $\bigl| X - 1 \bigr| < 1$.  Since $U X = U W = X W$, the isomorphism $F \Mapn{X} F$ induces an isomorphism $M_i \otimes_{A_i} \widehat{A}_\infty \Mi M$, as desired.
\end{proof}

\begin{proof}[Proof of Theorem \ref{thm-decompl}]
    By Lemma \ref{lem-tot-triv}, the condition \Refenum{\ref{def-decompl-syst-2}} in Definition \ref{def-decompl-syst} follows from the same argument based on \Refeq{\ref{eq-Hom-Gamma}} as in the proof of Theorem \ref{thm-decompl-weak}.  It remains to verify the condition \Refenum{\ref{def-decompl-syst-1}} in Definition \ref{def-decompl-syst}, by constructing a model for each finite projective $\Gamma$-module $L_\infty$ over $\widehat{A}_\infty$.  By Lemma \ref{lem-proj-dense-subring}, $L_\infty$ is the base change to $\widehat{A}_\infty$ of a finite projective $A_i$-module $\widetilde{L}_i$ \Pth{without $\Gamma$-action}, for some $i \in I$.  Take any finite covering $\mathfrak{B}$ of $\Spa(A_i, A_i^\circ)$ by rational subsets over which the pullbacks of $\widetilde{L}_i$ are free, and let $\Gamma'$ be an open normal subgroup of $\Gamma$ stabilizing every rational subset in $\mathfrak{B}$.  By Theorem \ref{thm-decompl-weak}, for some $i' \geq i$, the \Pth{finite free} pullback of $L_\infty$ \Pth{as a $\Gamma'$-module} to each rational subset in $\mathfrak{B}$ admits a model \Pth{as a $\Gamma'$-module} over $A_{i'}$, and these models coincide on intersections of rational subsets in $\mathfrak{B}$.  Thus, they glue to a \Pth{finite projective} model $L_{i'}$ of $L_\infty$ \Pth{as a $\Gamma'$-module} over $A_{i'}$, by the Kiehl gluing property for stably uniform adic Banach rings \Pth{see, again, \cite[\aThms 2.7.7 and 2.8.10]{Kedlaya/Liu:2015-RPH}}.  For each $g \in \Gamma$, consider the $A_{i'}$-submodule $g L_{i'}$ of $L_\infty$, which is also a $\Gamma'$-submodule because $\Gamma'$ is a normal subgroup of $\Gamma$, as in the second paragraph of the proof of Theorem \ref{thm-decompl-weak}.  It suffices to show that there exists some $i'' \geq i'$ such that $L_{i''} = g L_{i''}$ in $L_\infty$ for all $g \in \Gamma$, and we may verify this after pullback to the rational subsets in $\mathfrak{B}$.  Since $\Gamma / \Gamma'$ and $\mathfrak{B}$ are both finite, this follows from Theorem \ref{thm-decompl-weak}, as desired.
\end{proof}

\begin{cor}\label{cor-decompl-char-twist}
    Let $(\{ A_i \}_{i \in I}, \widehat{A}_\infty, \Gamma)$ be weakly \Pth{\resp stably} decompleting.  Let $0 \in I$ be an initial object, and $\{ \psi_s: \Gamma \to A_0^\times \}_{s \in S}$ a collection of continuous group homomorphisms such that, for each open neighborhood $\mathscr{U}$ of $1$ in $A_0$, there exists some open neighborhood $\mathscr{V}$ of $1$ in $\Gamma$ such that $\psi_s(\mathscr{V}) \subset \mathscr{U}$ for all $s \in S$.  Let $L_\infty$ be a finite free \Pth{\resp finite projective} $\Gamma$-module over $\widehat{A}_\infty$.  For each $s \in S$, let $L_\infty(\psi_s) := L_\infty \otimes_{A_0} A_0(\psi_s)$, where $A_0(\psi_s)$ is $A_0$ equipped with the action of $\Gamma$ via $\psi_s$.  Let $L_i$ be a model of $L_\infty$ over $A_i$, for some $i \in I$.  Then we have the following:
    \begin{enumerate}
        \item\label{cor-decompl-char-twist-1} $L_i(\psi_s) := L_i \otimes_{A_0} A_0(\psi_s)$ is a model of $L_\infty(\psi_s)$ over $A_i$, for all $s \in S$.

        \item\label{cor-decompl-char-twist-2} There exists some $i_0 \geq i$ such that $L_{i'}(\psi_s) := L_{i'} \otimes_{A_0} A_0(\psi_s)$ is a good model of $L_\infty(\psi_s)$ over $A_{i'}$, for all $i' \geq i_0$ and all $s \in S$.
    \end{enumerate}
\end{cor}
\begin{proof}
    The assertion \Refenum{\ref{cor-decompl-char-twist-1}} is clear.  As for the assertion \Refenum{\ref{cor-decompl-char-twist-2}}, by the same argument as in the proof of Theorem \ref{thm-decompl}, we are reduced to the case where $L_i$ is a finite free $A_i$-module.  Given the argument in the proof of Corollary \ref{cor-weak-good-model}, it suffices to show that, if $L_i$ admits a basis $\{ e_j \}_{j \in J}$ over $A_i$ such that $|\gamma(e_j) - e_j| < \frac{1}{r c}$, for all $\gamma \in \Gamma_i$, for some $r$ and $c$ as in Lemma \ref{lem-weak-tot-triv}, then there exists some $i_0' \geq i$ such that the corresponding basis $\{ e_j' := e_j \otimes 1 \}_{j \in J}$ of $L_{i'}(\psi_s) = L_{i'} \otimes_{A_0} A_0(\psi_s)$ over $A_{i'}$ also satisfies $|\gamma(e_j') - e_j'| < \frac{1}{r c}$, for all $i' \geq i_0'$, $s \in S$, and $\gamma \in \Gamma_{i'}$.  To see this, note that $|\gamma(e_j') - e_j'| = |(\gamma(e_j) - e_j) \otimes (\psi_s(\gamma) - 1) + e_j \otimes (\psi_s(\gamma) - 1) + (\gamma(e_j) - e_j) \otimes 1|$.  Thus, it suffices to take any $i_0' \geq i$ such that $|\psi_s(\gamma) - 1| < \min\{1, \frac{1}{r c}\}$, for all $s \in S$, $\gamma \in \Gamma_{i_0'}$, and $j \in J$, which exists by our assumption on $\{ \psi_s \}_{s \in S}$.
\end{proof}

\subsection{Examples}\label{sec-decompl-ex}

We present three examples of decompletion systems.

\numberwithin{equation}{subsubsection}

\subsubsection{Arithmetic towers}\label{sec-arith-tower}

Consider $\bQ_p(\Grpmu_{p^\infty}) := \cup_{l \geq 0} \, \bQ_p(\Grpmu_{p^l})$ and $\bZ_p(\Grpmu_{p^\infty}) := \cup_{l \geq 0} \, \bZ_p(\Grpmu_{p^l})$, equipped with the $p$-adic norms extending the standard ones on $\bQ_p$ and $\bZ_p$.  Let $(A, A^+)$ be any Huber pair over $(\bQ_p, \bZ_p)$.  For each $l \geq 0$, let $(A_{p_l}, A^+_{p^l}) := \bigl(A \otimes_{\bQ_p} \bQ_p(\Grpmu_{p^l}), A^+ \otimes_{\bZ_p} \bZ_p(\Grpmu_{p^l})\bigr)$.  Let $(\widehat{A}_{p^\infty}, \widehat{A}_{p^\infty}^+)$ be the $p$-adic completion of $(\cup_{l \geq 0} \, A_{p^l}, \cup_{l \geq 0} \, A_{p^l}^+)$.  Suppose that $(A_{p^l}, A_{p^l}^+)$ and $(\widehat{A}_{p^\infty}, \widehat{A}_{p^\infty}^+)$ are stably uniform.  Let $\Gamma_{p^l} := \Gal\bigl(\bQ_p(\Grpmu_{p^\infty}) / \bQ_p(\Grpmu_{p^l})\bigr)$.  By the Tate--Sen formalism with Banach-algebra coefficients developed in \cite{Berger/Colmez:2008-frdrm}, we have the following:

\begin{prop}\label{prop-arith-decompl}
    $(\{ A_{p^l} \}_{l \geq 0}, \widehat{A}_{p^\infty}, \Gamma_1)$ is stably decompleting.
\end{prop}
\begin{proof}
    It suffices to show that any $(\{ A_{p^l} \}_{l \geq 0}, \widehat{A}_{p^\infty}, \Gamma_1)$ as above is weakly decompleting, since its pullbacks to rational localizations of $A$ satisfy the same assumptions.  We may use the product norm on $\widehat{A}_\infty \cong A \ho_{\bQ_p} \bQ_p(\Grpmu_{p^\infty})$, where $A$ is equipped with the spectral norm with unit ball $A^+$ \Pth{as in \cite[\aDefs 2.1.9 and 2.8.1]{Kedlaya/Liu:2015-RPH}} and use the open subgroups $\{ \Gamma_{p^l} \}_{l \geq 0}$ of $\Gamma_1$.  Then the condition \Refenum{\ref{def-decompl-tuple-isometric}} of Definition \ref{def-decompl-tuple} holds.  As for the condition \Refenum{\ref{def-decompl-tuple-split}} of Definition \ref{def-decompl-tuple}, it suffices to note that $\bQ_p(\Grpmu_{p^m})$ admits a norm-direct supplement in $\bQ_p(\Grpmu_{p^n})$ as normed $\bQ_p(\Grpmu_{p^l})$-vector spaces, whenever $l \leq m \leq n$, by \cite[\aSec 2.4.2, \aProp 3, and \aSec 2.4.1, \aProp 5]{Bosch/Guntzer/Remmert:1984-NAA}.  It remains to verify the condition \Refenum{\ref{def-decompl-tuple-unif-str-ex}} of Definition \ref{def-decompl-tuple}.  By \cite[\aProps 4.1.1 and 3.1.4, and, in particular, TS(3) in \aDef 3.1.3]{Berger/Colmez:2008-frdrm}, for any $c > \frac{1}{p-1}$, and for all sufficiently large $l$ \Pth{depending on $c$} and any topological generator $\gamma$ of $\Gamma_{p^l}$, the endomorphism $1-\gamma: \widehat{A}_{p^\infty} / A_{p^l} \to \widehat{A}_{p^\infty} / A_{p^l}$ admits a continuous inverse of norm $\leq p^c$.  As a result, $H^0(\Gamma_{p^l}, \widehat{A}_{p^\infty}^+ / A_{p^l}^+) = 0$, and $H^1(\Gamma_{p^l}, \widehat{A}_{p^\infty}^+ / A_{p^l}^+)$ is annihilated by $p^2$.  Since $\Gamma_{p^l}$ is procyclic, $H^i(\Gamma_{p^l}, \widehat{A}_{p^\infty}^+ / A_{p^l}^+) = 0$, for $i \geq 2$.  In this case, we claim that the condition \Refenum{\ref{def-decompl-tuple-unif-str-ex}} of Definition \ref{def-decompl-tuple} holds with $c = p^2$.  To see this, let $f$ be a cocycle in $C^i(\Gamma_{p^l}, \widehat{A}_{p^\infty} / A_{p^l})$.  Up to replacing $f$ with a scalar multiple, we may suppose that $f$ lies in $C^i(\Gamma_{p^l}, \widehat{A}_{p^\infty}^+ / A_{p^l}^+)$.  Since $H^i(\Gamma_{p^l}, \widehat{A}_{p^\infty}^+ / A_{p^l}^+)$ is annihilated by $p^2$, there exists some $h \in C^{i - 1}(\Gamma_{p^l}, \widehat{A}_{p^\infty}^+ / A_{p^l}^+)$ such that $\| h \| \leq \| f \|$ and $d h = p^2 f$.  Thus, $g := p^{-2} h \in C^{i - 1}(\Gamma_{p^l}, \widehat{A}_{p^\infty}^+ / A_{p^l}^+)$ satisfies $\| g \| \leq p^2 \| f \|$ and $d g = f$, as desired.
\end{proof}

\begin{thm}\label{thm-arith-decompl}
    $(\{ A_{p^l} \}_{l \geq 0}, \widehat{A}_{p^\infty}, \Gamma_1)$ is a decompletion system.
\end{thm}
\begin{proof}
    Combine Proposition \ref{prop-arith-decompl} and Theorem \ref{thm-decompl}.
\end{proof}

\begin{rk}\label{rem-arith-decompl-DVF}
    Let $k$ be a $p$-adic field, and let $A$ be a Banach $k$-algebra.  For each $l \geq 0$, put $A_{p^l} := A \otimes_k k(\Grpmu_{p^l})$ and $\Gamma_{p^l} := \Gal\bigl(k(\Grpmu_{p^\infty}) / k(\Grpmu_{p^l})\bigr)$.  Note that there exists some sufficiently large $l_0 \geq 0$ such that, for all $l \geq l_0$, we have $A_{p^l} \cong A_{p^{l_0}} \otimes_{\bQ_p(\Grpmu_{p^{l_0}})} \bQ_p(\Grpmu_{p^l})$, and $\Gamma_{p^l} \cong \Gal\bigl(\bQ_p(\Grpmu_{p^\infty}) / \bQ_p(\Grpmu_{p^l})\bigr)$.  In this case, by Theorem \ref{thm-arith-decompl}, $(\{ A_{p^l} \}_{l \geq 0}, \widehat{A}_{p^\infty}, \Gamma_1)$ is still a decompletion system.
\end{rk}

\subsubsection{Geometric towers}\label{sec-geom-tower}

In this example, we shall follow the setup in Section \ref{sec-OBdl-explicit}, with $X = \Spa(A, A^+)$ and $\widehat{\widetilde{X}} = \Spa(\widehat{A}_\infty, \widehat{A}_\infty^+)$, in the notation there.  Write $X_{m, \widehat{k}_\infty} = \Spa(A_{m, \widehat{k}_\infty}, A_{m, \widehat{k}_\infty}^+)$, for all $m \geq 1$, so that $(\widehat{A}_\infty, \widehat{A}_\infty^+)$ is the $p$-adic completion of $\varinjlim_m (A_{m, \widehat{k}_\infty}, A_{m, \widehat{k}_\infty}^+)$.  Let $\Gamma_1 := \Hom(P^\gp_\bQ / P^\gp, \Grpmu_\infty) \cong \Hom\bigl(P^\gp, \widehat{\bZ}(1)\bigr)$, as in \cite[(\logadiceqgeomtowerGal)]{Diao/Lan/Liu/Zhu:lasfr}, and let $\Gamma_m := \Hom(P^\gp_\bQ / \frac{1}{m} P^\gp, \Grpmu_\infty) \subset \Gamma_1$, which acts on $A_{m, \widehat{k}_\infty}$'s and $\widehat{A}_\infty$ by $\gamma T^a = \gamma(a) \, T^a$, for all $\gamma \in \Gamma_1$ and $a \in P_{\bQ_{\geq 0}}$ \Pth{and acts trivially on $A$ and $\widehat{k}_\infty$}.  Note that the actions of $\Gamma_1$ on $A_{m, \widehat{k}_\infty}$'s and $\widehat{A}_\infty$ naturally extend to $\widetilde{\Gamma} := \Gamma_1 \rtimes \Gal(k_\infty / k)$, with $\Gal(k_\infty / k)$ acting on $\Gamma_1$ via the cyclotomic character.

\begin{prop}\label{prop-geom-decompl}
    $(\{ A_{m, \widehat{k}_\infty} \}_{m \geq 1}, \widehat{A}_\infty, \widetilde{\Gamma})$ is stably decompleting.
\end{prop}
\begin{proof}
    Since $X_m$'s are reduced rigid analytic spaces over $k$, and since $\widehat{\widetilde{X}}$ is perfectoid, $A_{m, \widehat{k}_\infty}$'s are closed subrings of $\widehat{A}_\infty$ satisfying the condition \Refenum{\ref{def-decompl-tuple-stab-a}} in Definition \ref{def-decompl-tuple-stab}.  By \cite[\aLem 2.1.3]{Berkovich:2007-IPA}, each rational subspace of $X_{m, \widehat{k}_\infty}$ is the base change of a rational subspace of $X_{m, k'}$, for some $[k' : k] < \infty$, and hence is stabilized by the open subgroup $\Gamma_m \rtimes \Gal(k_\infty / k')$ of $\widetilde{\Gamma}$.  Since rational subsets of $X_m$ are also strictly \'etale over $\bE_m$, in order to verify the condition \Refenum{\ref{def-decompl-tuple-stab-b}} in Definition \ref{def-decompl-tuple-stab}, up to replacing $X$ with $X_m$, and replacing $k$ with a finite extension, it suffices to show that any $(\{ A_{m, \widehat{k}_\infty} \}_{m \geq 1}, \widehat{A}_\infty, \widetilde{\Gamma})$ as above is weakly decompleting.  We shall use the spectral norm on $\widehat{A}_\infty$, and use the subgroups $\{ \Gamma_m \}_{m \geq 1}$ of $\widetilde{\Gamma}$.  Firstly, the condition \Refenum{\ref{def-decompl-tuple-isometric}} of Definition \ref{def-decompl-tuple} is satisfied, by \cite[\aThm 2.3.10 and \aRem 2.8.3]{Kedlaya/Liu:2015-RPH}.  Secondly, according to the $\Gamma_1$-action, we have a canonical decomposition of $k_\infty^+[P]$-modules
    \begin{equation}\label{eq-kP-decomp}
        k_\infty^+[P_{\bQ_{\geq 0}}] = k_\infty^+[P] \oplus \bigl(\oplus_{\chi \neq 1} (k_\infty^+[P_{\bQ_{\geq 0}}]_\chi)\bigr),
    \end{equation}
    where $\chi$ runs over all nontrivial finite-order characters of $\Gamma_1$; and so the condition \Refenum{\ref{def-decompl-tuple-split}} of Definition \ref{def-decompl-tuple} follows from the completed tensor product of \Refeq{\ref{eq-kP-decomp}} with $A^+$.  Finally, note that \Refeq{\ref{eq-kP-decomp}} also induces \Pth{by completed tensor product as above} a $\Gamma_m$-equivariant isomorphism $\widehat{A}_\infty^+ / A_m^+ \cong \varprojlim_n \, (\oplus_{\chi \neq 1} \, M_{n, \chi})$, where $M_{n, \chi} := (A_m^+ / p^n) \otimes_{(\widehat{k}_\infty^+ / p^n)[\frac{1}{m}\overline{P}]} \bigl((k^+ / p^n)[\overline{P}_{\bQ_{\geq 0}}]_\chi\bigr)$ and the direct sum is over all nontrivial finite-order characters $\chi$ of $\Gamma_m$.  Thus, by using \cite[\aLem \logadiclemGalcohafg]{Diao/Lan/Liu/Zhu:lasfr} and proceeding as in the proof of Proposition \ref{prop-arith-decompl}, we see that the condition \Refenum{\ref{def-decompl-tuple-unif-str-ex}} of Definition \ref{def-decompl-tuple} holds with $c = \max_{m > 1}\bigl\{ |\zeta_m - 1|^{-1} \bigr\} = p^{\frac{1}{p - 1}}$, as desired.
\end{proof}

\begin{thm}\label{thm-geom-decompl}
    $(\{ A_m \}_{m \geq 1}, \widehat{A}_\infty, \widetilde{\Gamma})$ is a decompletion system.
\end{thm}
\begin{proof}
    Combine Proposition \ref{prop-geom-decompl} and Theorem \ref{thm-decompl}.
\end{proof}

\begin{rk}\label{rem-geom-decompl}
    Both Proposition \ref{prop-geom-decompl} and Theorem \ref{thm-geom-decompl} remain true if we replace $\widetilde{\Gamma}$ with a closed subgroup $\widetilde{\Gamma}'$ containing $\Gamma_1$, because if the conditions in Definitions \ref{def-decompl-tuple} and \ref{def-decompl-tuple-stab} hold for $\widetilde{\Gamma}$, then they also hold for $\widetilde{\Gamma}'$.
\end{rk}

\subsubsection{Deformation of geometric towers}\label{sec-geom-tower-deform}

In this example, we shall continue to follow the setup in Section \ref{sec-OBdl-explicit}.  Let us fix the choice of a uniformizer $\varpi$ of $k$.  For each $r \geq 1$, equip $\BdRp / \xi^r$ with the norm $|x| := \inf\bigl\{ |\varpi|^n : n \in \bZ, \, \varpi^{-n} x \in \Ainf / \xi^r \bigr\}$, for each $r \geq 1$.  This norm on $\BdRp / \xi^r$ extends the norm on $k$ and makes $\BdRp / \xi^r$ a Banach $k$-algebra.  For any toric monoid $P$, let us equip $(\BdRp / \xi^r)\Talg{P}$ and $(\BdRp / \xi^r)\Talg{\tfrac{1}{m}P}$ with the supremum norm, and equip $A \ho_k (\BdRp / \xi^r)$ and
\begin{equation}\label{eq-def-B-r-m}
    \bB_{r, m} := \bigl(A \ho_k (\BdRp / \xi^r)\bigr) \otimes_{(\BdRp / \xi^r)\Talg{P}} (\BdRp / \xi^r)\Talg{\tfrac{1}{m} P}
\end{equation}
with the product norms, as in \cite[\aDef 2.1.10]{Kedlaya/Liu:2015-RPH}.  Note that $\bB_{r, m}$ is equipped with a natural isometric action of $\Gamma_1$, by \Refeq{\ref{eq-Gamma-act-mono-alg}}.  Then $\{ \bB_{r, m} \}_{r \geq 1}$ is a direct system of Banach $k$-algebras with compatible actions of $\Gamma_1$, with completed direct limit
\[
    \widehat{\bB}_{r, \infty} := \bigl(A \ho_k (\BdRp / \xi^r)\bigr) \ho_{(\BdRp / \xi^r)\Talg{\overline{P}}} (\BdRp / \xi^r)\Talg{\overline{P}_{\bQ_{\geq 0}}}
\]
canonically isomorphic to $\BBdRp(\widetilde{X}) / \xi^r$ as topological rings, by Lemma \ref{lem-str-BdR}.  Since \Refeq{\ref{eq-kP-decomp}} induces \Pth{by completed tensor product with $A \ho_k (\BdRp / \xi^r)$} submetric surjections $\widehat{\bB}_{r, \infty} \to \bB_{r, m}$, the system $\{ \bB_{r, m} \}_{r \geq 1}$ satisfies the condition \Refenum{\ref{def-decompl-tuple-split}} in Definition \ref{def-decompl-tuple}.  Moreover, when $r = 1$, the topological rings $\bB_{1, m}$ and $\widehat{\bB}_{1, \infty}$ can be identified with $A_{m, \widehat{k}_\infty}$ and $\widehat{A}_\infty$, respectively, with compatibly equivalent norms.  Therefore, by Remarks \ref{rem-geom-decompl} and \ref{rem-def-decompl-tuple}\Refenum{\ref{rem-def-decompl-tuple-norm}}, if we use the same subgroups $\{ \Gamma_m \}_{m \geq 1}$ of $\Gamma_1$ as in the proof of Proposition \ref{prop-geom-decompl}, then $(\{ \bB_{1, m} \}_{m \geq 1}, \widehat{\bB}_{1, \infty}, \Gamma_1)$ is stably decompleting, and hence is a decompletion system, as $(\{ A_{m, \widehat{k}_\infty} \}_{m \geq 1}, \widehat{A}_\infty, \Gamma_1)$ is.

\begin{lemma}\label{lem-sect-theta}
    The natural projection $\theta: \BdRp / \xi^r \to \widehat{k}_\infty$ admits a section $s$ in the category of $k$-Banach spaces whose operator seminorm satisfies $|s| \leq 2 |\varpi|^{-1}$.
\end{lemma}
\begin{proof}
    By \cite[\aSec 2.7.2, \aProp 3]{Bosch/Guntzer/Remmert:1984-NAA}, we can find a Schauder $k$-basis $\{ e_j \}_{j \in J}$ of $\widehat{k}_\infty$ such that $\max_{j \in J}\bigl\{ |b_j e_j| \bigr\} \leq 2 \bigl|\sum_{j \in J} b_j e_j\bigr|$, for every convergent sum $\sum_{j \in J} b_j e_j$.  Moreover, we can rescale $e_j$ such that $|\varpi| < |e_j| \leq 1$ for all $j \in J$, and lift each $e_j$ to some element $\widetilde{e}_j$ in $\Ainf$.  Then we can define the desired section $s$ by mapping each convergent sum $\sum_{j \in J} b_j e_j$ to $\sum_{j \in J} b_j \widetilde{e}_j$.
\end{proof}

\begin{prop}\label{prop-BdR-decompl}
    $(\{ \bB_{r, m} \}_{m \geq 1}, \widehat{\bB}_{r, \infty}, \Gamma_1)$ is weakly decompleting, for all $r \geq 1$.
\end{prop}
\begin{proof}
    With the chosen norms, and with the open subgroups $\{ \Gamma_m \}_{m \geq 1}$ of $\Gamma_1$, it remains to verify the uniform strict exactness condition.  As explained above, $(\{ \bB_{1, m} \}_{m \geq 1}, \widehat{\bB}_{1, \infty}, \Gamma_1)$ is stably \Pth{and hence weakly} decompleting, which satisfies the condition with some constant $c \geq 1$ \Pth{when we equip norms compatibly as above and use the same subgroups $\{ \Gamma_m \}_{m \geq 1}$ of $\Gamma_1$}.  We shall show by induction that $(\{ \bB_{r, m} \}_{m \geq 1}, \widehat{\bB}_{r, \infty}, \Gamma_1)$ satisfies the condition with the constant $(2 |\varpi|^{-1})^{r - 1} c^r$, starting with the known case $r = 1$.  For each $r > 1$, let $f$ be a cocycle in $C^\bullet(\Gamma_m, \widehat{\bB}_{r, \infty} / \bB_{r, m})$.  Then its image $\overline{f}$ in $C^\bullet(\Gamma_m, \widehat{\bB}_{1, \infty} / \bB_{1, m})$ satisfies $\| \overline{f} \| \leq \| f \|$.  Let $\overline{g}$ be a cochain satisfying  $d \overline{g} = \overline{f}$ with $\| \overline{g} \| \leq c \| \overline{f} \| \leq c \| f \|$.  By Lemma \ref{lem-sect-theta}, we can lift $\overline{g}$ to a cochain $\widetilde{g} \in C^\bullet(\Gamma_m, \widehat{\bB}_{r, \infty} / \bB_{r, m})$ with $\|\widetilde{g}\| \leq 2 |\varpi|^{-1} \|\overline{g}\| \leq 2 |\varpi|^{-1} c \| f \|$.  Accordingly, there is a cochain $f_1 \in C^\bullet(\Gamma_m, \widehat{\bB}_{r - 1, \infty} / \bB_{r - 1, m})$ such that $f - d \widetilde{g} = \xi f_1$ via the isometry $\BdRp / \xi^{r - 1} \cong \xi \BdRp / \xi^r$ induced by multiplication by $\xi$, and $\| f_1 \| = \| \xi f_1 \| = \| f - d \widetilde{g} \| \leq 2 |\varpi|^{-1} c \| f \|$.  By the inductive hypothesis, we can find a cochain $g_1 \in C^\bullet(\Gamma_m, \widehat{\bB}_{r - 1, \infty} / \bB_{r - 1, m})$ satisfying $d g_1 = f_1$ with $\| g_1 \| \leq (2 |\varpi|^{-1})^{r - 2} c^{r - 1} \| f_1 \| \leq (2 |\varpi|^{-1})^{r - 1} c^r \| f \|$.  Now put $g := \widetilde{g} + \xi g_1$; here again $\xi g_1$ is a cochain in $C^\bullet(\Gamma_m, \widehat{\bB}_{r, \infty} / \bB_{r, m})$ via the isometry $\BdRp / \xi^{r - 1} \cong \xi \BdRp / \xi^r$.  Then we have $d g = f$ and $\| g \| \leq \max\bigl\{ \| \widetilde{g} \|, \| g_1 \| \bigr\} \leq (2 |\varpi|^{-1})^{r - 1} c^r \| f \|$, as desired.
\end{proof}

\begin{thm}\label{thm-BdR-decompl}
    $(\{ \bB_{r, m} \}_{m \geq 1}, \widehat{\bB}_{r, \infty}, \Gamma)$ is a decompletion system, for all $r \geq 1$.
\end{thm}
\begin{proof}
    Let $L_\infty$ be a finite projective $\Gamma$-module over $\widehat{\bB}_{r, \infty}$.  Note that, if $(L_m, \iota_m)$ is a model of $L_\infty$ over $\bB_{r, m}$, then $(\xi^{s - 1} L_m / \xi^s L_m, \overline{\iota}_m)$ is a model of $L_\infty / \xi L_\infty$ over $\bB_{1, m}$, for all $1 \leq s \leq r$.  Since $(\{ \bB_{1, m} \}_{m \geq 1}, \widehat{\bB}_{1, \infty}, \Gamma)$ is a decompletion system, there exists some sufficiently divisible $m' \geq 1$ such that, for all $1 \leq s \leq r$, the base change of $(\xi^{s - 1} L_m / \xi^s L_m, \overline{\iota}_m)$ to $\bB_{1, m'}$ is a good model.  Thus, the base change of $(L_m, \iota_m)$ to $\bB_{r, m'}$ is a good model, and we have verified the condition \Refenum{\ref{def-decompl-syst-2}} in Definition \ref{def-decompl-syst}.  Moreover, by the same argument based on \Refeq{\ref{eq-Hom-Gamma}} as in the proof of Theorem \ref{thm-decompl-weak}, this property also ensures that any two models over $\bB_{r, m}$ becomes identical in $L_\infty$ after base change to $\bB_{r, m'}$ for some sufficiently large multiple $m'$ of $m$.

    It remains to show the existence of a model of $L_\infty$.  Firstly, by the same argument as in the proof of Theorem \ref{thm-decompl}, for some $m \geq 1$, we can find a finite covering of $(X_m)_{\widehat{k}_\infty}$ by rational subsets over which $L_\infty / \xi L_\infty$ are free.  \Pth{More precisely, we mean $L_\infty / \xi L_\infty$ is free over the pullbacks of $\widehat{\widetilde{X}}$ to these rational subsets.  Since rational subsets of $X_m$ are also strictly \'etale over $\bE_m$, we still have compatible actions of $\Gamma_m$ on such pullbacks.  For simplicity, we shall adopt a similar language in the following.}  As explained in the proof of Proposition \ref{prop-geom-decompl}, by \cite[\aLem 2.1.3]{Berkovich:2007-IPA}, there is an open subgroup of $\Gal(k_\infty / k)$ stabilizing all rational subsets in the above finite covering.  Hence, up to replacing $X$ with $X_m$, and replacing $k$ with a finite extension, we may assume that there exists a finite covering $X = \cup_{i \in I} \, \Spa(R_i, R_i^+)$ by rational subsets such that $L_\infty / \xi L_\infty$ is free over each $\Spa(R_i, R_i^+)_{\widehat{k}_\infty}$.  Since $\xi$ is a nilpotent element, the base change $L_{\infty, i}$ of $L_\infty$ under $A \ho_k (\BdRp / \xi^r) \to R_i \ho_k (\BdRp / \xi^r)$ is also free.  By Proposition \ref{prop-BdR-decompl} and Theorem \ref{thm-decompl-weak}, for some $m \geq 1$, each $L_{\infty, i}$ admits a free model $(L_{m, i}, \iota_{m, i})$ over $\bigl(R_i \ho_k (\BdRp / \xi^r)\bigr) \otimes_{(\BdRp / \xi^r)\Talg{P}} (\BdRp / \xi^r)\Talg{\frac{1}{m} P}$, and we may assume that these models coincide on the intersections of rational subsets in the covering.  Thus, by \cite[\aProp 3.3]{Liu/Zhu:2017-rrhpl}, these models glue to a model of $L_\infty$, as desired.
\end{proof}

% bibliography

%\bibliographystyle{amsalpha}
%\bibliography{log-RH}

\providecommand{\bysame}{\leavevmode\hbox to3em{\hrulefill}\thinspace}
\providecommand{\MR}{\relax\ifhmode\unskip\space\fi MR }
% \MRhref is called by the amsart/book/proc definition of \MR.
\providecommand{\MRhref}[2]{%
  \href{http://www.ams.org/mathscinet-getitem?mr=#1}{#2}
}
\providecommand{\href}[2]{#2}

\end{document}